\documentclass[a4paper,11pt]{amsart}

\usepackage{graphicx}
\usepackage{amsmath}
\usepackage{amssymb}

\usepackage{longtable}

\newtheorem{thm}{Theorem}[section]
\newtheorem{lem}[thm]{Lemma}%
\newtheorem{prop}[thm]{Proposition}%
\newtheorem{cor}[thm]{Corollary}%
\theoremstyle{remark}
\newtheorem{remark}[thm]{Remark}
\theoremstyle{plain}

\numberwithin{equation}{section}

\def\QQ{{\mathbb Q}}
\def\RR{{\mathbb R}}

\def\ZZ{{\mathbb Z}}

\def\veca{{\text{\boldmath$a$}}}
\def\vecb{{\text{\boldmath$b$}}}

\def\vecc{{\text{\boldmath$c$}}}

\def\vece{{\text{\boldmath$e$}}}
\def\vech{{\text{\boldmath$h$}}}
\def\veck{{\text{\boldmath$k$}}}

\def\vecm{{\text{\boldmath$m$}}}
\def\vecn{{\text{\boldmath$n$}}}
\def\vecq{{\text{\boldmath$q$}}}
\def\vecQ{{\text{\boldmath$Q$}}}
\def\vecp{{\text{\boldmath$p$}}}

\def\vecv{{\text{\boldmath$v$}}}
\def\vecV{{\text{\boldmath$V$}}}
\def\vecw{{\text{\boldmath$w$}}}
\def\vecx{{\text{\boldmath$x$}}}

\def\vecy{{\text{\boldmath$y$}}}
\def\vecz{{\text{\boldmath$z$}}}
\def\vecalf{{\text{\boldmath$\alpha$}}}
\def\vecbeta{{\text{\boldmath$\beta$}}}

\def\vecomega{{\text{\boldmath$\omega$}}}

\def\vecxi{{\text{\boldmath$\xi$}}}

\def\vecnull{{\text{\boldmath$0$}}}

\def\scrA{{\mathcal A}}
\def\scrB{{\mathcal B}}

\def\scrE{{\mathcal E}}

\def\scrK{{\mathcal K}}
\def\scrL{{\mathcal L}}

\def\scrN{{\mathcal N}}
\def\scrQ{{\mathcal Q}}
\def\scrP{{\mathcal P}}

\def\scrT{{\mathcal T}}

\def\scrW{{\mathcal W}}

\def\fA{{\mathfrak A}}
\def\fB{{\mathfrak B}}
\def\fC{{\mathfrak C}}
\def\fD{{\mathfrak D}}

\def\fU{{\mathfrak U}}

\def\fZ{{\mathfrak Z}}

\def\e{\mathrm{e}}
\def\i{\mathrm{i}}

\def\diag{\operatorname{diag}}

\def\C{\operatorname{C{}}}

\def\L{\operatorname{L{}}}
\def\M{\operatorname{M{}}}
\def\GL{\operatorname{GL}}

\def\S{\operatorname{S{}}}

\def\SL{\operatorname{SL}}
\def\ASL{\operatorname{ASL}}
\def\SO{\operatorname{SO}}

\def\O{\operatorname{O{}}}
\def\T{\operatorname{T{}}}

\def\vol{\operatorname{vol}}

\def\Proj{\operatorname{Proj}}

\def\GamG{\Gamma\backslash G}
\def\ASLASL{\ASL(d,\ZZ)\backslash\ASL(d,\RR)}
\def\ASLZ{\ASL(d,\ZZ)}
\def\ASLR{\ASL(d,\RR)}
\def\SLSL{\SL(d,\ZZ)\backslash\SL(d,\RR)}
\def\SLZ{\SL(d,\ZZ)}
\def\SLR{\SL(d,\RR)}

\def\trans{\,^\mathrm{t}\!}

\def\Lalf{\scrL_\vecalf}

\def\Onder#1#2#3#4#5{#1 \setbox0=\hbox{$#1$}\setbox1=\hbox{$#2$}
       \dimen0=.5\wd0 \dimen1=\dimen0 \dimen2=\dp0 \dimen3=\dimen2
       \advance\dimen0 by .5\wd1 \advance\dimen0 by -#4
       \advance\dimen1 by -.5\wd1 \advance\dimen1 by -#4
       \advance\dimen2 by -#3 \advance\dimen2 by \ht1
       \advance\dimen2 by 0.3ex \advance\dimen3 by #5
        \kern-\dimen0\raisebox{-\dimen2}[0ex][\dimen3]{\box1}
       \kern\dimen1}

\newcommand{\Edomain}{D}

\newcommand{\Q}{\mathbb{Q}}
\newcommand{\R}{\mathbb{R}}
\newcommand{\Z}{\mathbb{Z}}
\newcommand{\Far}{\mathfrak{F}}
\newcommand{\minmod}{\text{ mod }}
\newcommand{\HS}{{{\S'_1}^{d-1}}}
\renewcommand{\aa}{\mathsf{a}}
\newcommand{\kk}{\mathsf{k}}
\newcommand{\nn}{\mathsf{n}}
\newcommand{\sfrac}[2]{{\textstyle \frac {#1}{#2}}}
\newcommand{\col}{\: : \:}
\newcommand{\Si}{\mathcal{S}}
\newcommand{\bn}{\mathbf{0}}

\newcommand{\tN}{M}
\newcommand{\txi}{\hat{\xi}}

\newcommand{\tx}{\widetilde{\vecx}}
\newcommand{\ve}{\varepsilon}
\newcommand{\F}{\mathcal{F}}

\newcommand{\matr}[4]{\left( \begin{matrix} #1 & #2 \\ #3 & #4 \end{matrix} \right) }
\newcommand{\smatr}[4]{\left( \begin{smallmatrix} #1 & #2 \\ #3 & #4 \end{smallmatrix} \right) }

\title[The distribution of free path lengths in the periodic Lorentz gas]{The distribution of free path lengths in the periodic Lorentz gas and related lattice point problems}
\author{Jens Marklof}
\author{Andreas Str\"ombergsson}
\address{School of Mathematics, University of Bristol,
Bristol BS8 1TW, U.K.\newline
\rule[0ex]{0ex}{0ex} \hspace{8pt}{\tt j.marklof@bristol.ac.uk}}
\address{Department of Mathematics, Royal Institute of Technology,
SE-100 44 Stockholm, Sweden\newline
\rule[0ex]{0ex}{0ex} \hspace{8pt}{\tt astrombe@math.kth.se}\newline
\rule[0ex]{0ex}{0ex} \hspace{8pt}\textit{Present address:} 
Dept.\ of Mathematics, Box 480, 
Uppsala University,
SE-75106 Uppsala, Sweden\newline
\rule[0ex]{0ex}{0ex} \hspace{8pt}{\tt astrombe@math.uu.se}}
\date{29 June 2007, revised 14 January 2008, to appear in the Annals of Mathematics}
\thanks{J.M.\ has been supported by an EPSRC Advanced Research Fellowship and a Philip Leverhulme Prize.
A.S.\ is a Royal Swedish Academy of Sciences Research Fellow supported by
a grant from the Knut and Alice Wallenberg Foundation.}

\begin{document}

\begin{abstract}
The periodic Lorentz gas describes the dynamics of a point particle in a periodic array of spherical scatterers, and is one of the fundamental models for chaotic diffusion. In the present paper we investigate the Boltzmann-Grad limit, where the radius of each scatterer tends to zero, and prove the existence of a limiting distribution for the free path length. We also discuss related problems, such as the statistical distribution of directions of lattice points that are visible from a fixed position. 
\end{abstract}

\maketitle
\tableofcontents

\section{Introduction}\label{secIntro}
\subsection{The periodic Lorentz gas}

The Lorentz gas, originally introduced by Lorentz \cite{Lorentz} in 1905
to model the motion of electrons in a metal, describes an ensemble of
non-interacting point particles in an infinite array of
spherical scatterers. Lorentz was in particular interested in the
stochastic properties of the dynamics that emerge in the 
Boltzmann-Grad limit, where the radius $\rho$ of each scatterer tends to
zero. 

In the present and subsequent papers \cite{partII}, \cite{partIII}, \cite{partIV} we investigate the periodic set-up, where the scatterers are placed at the vertices of a euclidean lattice $\scrL\subset\RR^d$ (Figure \ref{figLorentz}). We will identify a new random process that governs the macroscopic dynamics of a particle cloud in the Boltzmann-Grad limit. In the case of a Poisson-distributed (rather than periodic) configuration of scatterers, the limiting process is described by the linear Boltzmann equation, see Gallavotti \cite{Gallavotti69}, Spohn \cite{Spohn78}, and Boldrighini, Bunimovich and Sinai \cite{Boldrighini83}. It already follows from the estimates in \cite{Bourgain98}, \cite{Golse00} that the linear Boltzmann equation does not hold in the periodic set-up; this was pointed out recently by Golse \cite{Golse07}. 

The first step towards the proof of the existence of a limiting process for the periodic Lorentz gas is the understanding of the distribution of the free path length in the limit $\rho\to 0$, which is the key result of the present paper. The distribution of the free path lengths in the periodic Lorentz gas was already investigated by Polya, who rephrased the problem in terms of the visibility in a (periodic) forest \cite{Polya18}. We complete the analysis of the limiting process in \cite{partII}, \cite{partIII} and \cite{partIV}, where we establish a Markov property, and provide explicit formulas and asymptotic estimates for the limiting distributions. 

Our results complement classical studies in ergodic theory, where one is interested in the stochastic properties in the limit of long times, with the radius of each scatterer being {\em fixed}. Here Bunimovich and Sinai \cite{Bunimovich80} proved, in the case of a finite horizon and in dimension $d=2$, that the dynamics is diffusive in the limit of large times, and satisfies a central limit theorem. ``Finite horizon'' means that the scatterers are sufficiently large so that the path length between consecutive collisions is bounded; this hypothesis was recently removed by Szasz and Varju \cite{Szasz07} after initial work by Bleher \cite{Bleher92}. For related recent studies of statistical properties of the two-dimensional periodic Lorentz gas, see also \cite{dolgopyat}, \cite{Melbourne05}, \cite{Melbourne07}. There is at present no proof of the central limit theorem for higher dimensions, even in the case of finite horizon \cite{Chernov94}, \cite{Balint07}.

\begin{figure}
\begin{center}
\framebox{
\begin{minipage}{0.4\textwidth}
\unitlength0.1\textwidth
\begin{picture}(10,10)(0,0)
\put(0.5,1){\includegraphics[width=0.9\textwidth]{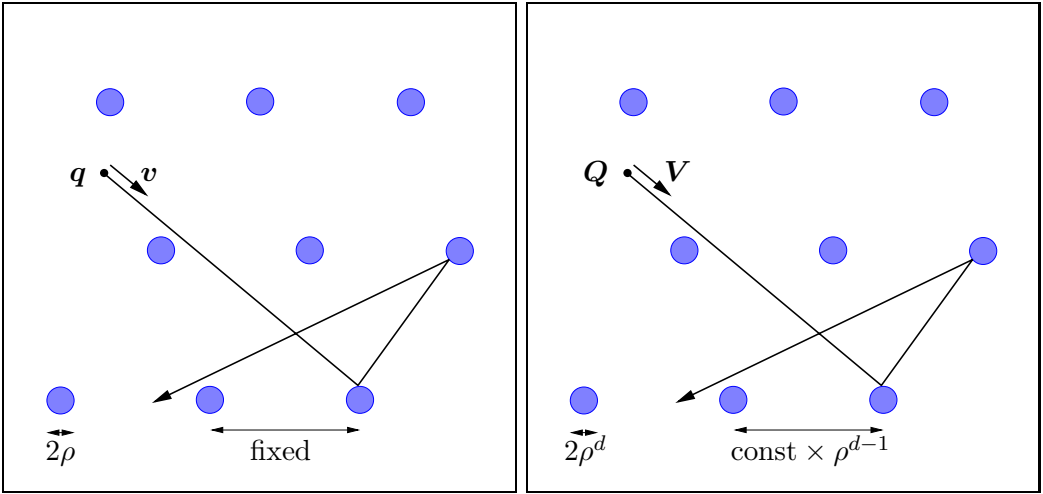}}
\put(1,6.4){$\vecq$} \put(2.5,6.4){$\vecv$}
\put(0.5,0.5){$2\rho$} \put(4.8,0.5){fixed}
\end{picture}
\end{minipage}
}
\framebox{
\begin{minipage}{0.4\textwidth}
\unitlength0.1\textwidth
\begin{picture}(10,10)(0,0)
\put(0.5,1){\includegraphics[width=0.9\textwidth]{lorentzgas.eps}}
\put(0.8,6.4){$\vecQ$} \put(2.5,6.4){$\vecV$}
\put(0.4,0.5){$2\rho^d$} \put(3.9,0.5){$\text{const}\times\rho^{d-1}$}
\end{picture}
\end{minipage}
}
\end{center}
\caption{Left: The periodic Lorentz gas in ``microscopic'' coordinates---the lattice $\scrL$ remains fixed as the radius $\rho$ of the scatterer tends to zero. Right: The periodic Lorentz gas in ``macroscopic'' coordinates ---both the lattice constant and the radius of each scatter tend to zero, in such a way that the mean free path length remains finite.} \label{figLorentz}
\end{figure}

Since the point particles of the Lorentz gas are non-interacting, we can reduce the problem to the study of the billiard flow
\begin{equation}\label{Lorentzflow}
	\varphi_t : \T^1(\scrK_\rho) \to \T^1(\scrK_\rho), \qquad (\vecq_0,\vecv_0) \mapsto (\vecq(t),\vecv(t))
\end{equation}
where $\scrK_\rho\subset\RR^d$ \label{Krho}
is the complement of the set $\scrB^d_\rho + \scrL$ (the ``billiard domain''), and $\T^1(\scrK_\rho)=\scrK_\rho\times\S_1^{d-1}$ is its unit tangent bundle (the ``phase space''). $\scrB^d_\rho$ \label{openball}
denotes the open ball of radius $\rho$, centered at the origin. A point in $\T^1(\scrK_\rho)$ is parametrized by $(\vecq,\vecv)$, with $\vecq\in\scrK_\rho$ denoting the position and $\vecv\in\S_1^{d-1}$ the velocity of the particle. The Liouville measure of $\varphi_t$ is 
\begin{equation}\label{Liouville}
	d\nu(\vecq,\vecv)=d\!\vol_{\RR^d}(\vecq)\, d\!\vol_{\S_1^{d-1}}(\vecv)
\end{equation}
where $\vol_{\RR^d}$ and $\vol_{\S_1^{d-1}}$ refer to the Lebesgue measures on $\RR^d$ (restricted to $\scrK_\rho$) and $\S_1^{d-1}$, respectively. 

The free path length for the initial condition $(\vecq,\vecv)\in\T^1(\scrK_\rho)$ is defined as
\begin{equation} \label{TAU1DEF0}
	\tau_1(\vecq,\vecv;\rho) = \inf\{ t>0 : \vecq+t\vecv \notin\scrK_\rho \}. 
\end{equation}
That is, $\tau_1(\vecq,\vecv;\rho)$ is the first time at which a particle with initial data $(\vecq,\vecv)$ hits a scatterer.

From now on we will assume, without loss of generality, that $\scrL$ has
covolume one.

\begin{thm}\label{freeThm1}
Fix a lattice $\scrL$ of covolume one,
let $\vecq\in\RR^d\setminus\scrL$,
and let $\lambda$ be a Borel probability measure on $\S_1^{d-1}$ absolutely continuous with respect to Lebesgue measure.\footnote{The condition $\vecq\in\RR^d\setminus\scrL$ ensures that $\tau_1$ is defined for $\rho$ sufficiently small. In Section \ref{secLorentz} we also consider variants of Theorem \ref{freeThm1} where the initial position is near $\scrL$, e.g., $\vecq\in\partial\scrK_\rho$.} Then there exists a continuous probability density $\Phi_{\scrL,\vecq}$ on $\R_{>0}$ such that, for every $\xi\geq 0$,
\begin{equation}
\lim_{\rho\to 0} \lambda(\{ \vecv\in\S_1^{d-1} : \rho^{d-1} 
\tau_1(\vecq,\vecv;\rho)\geq \xi \})
= \int_\xi^\infty \Phi_{\scrL,\vecq}(\xi') d\xi' .
\end{equation}
\end{thm}

The limiting density is in fact ``universal'' for generic $\vecq$, i.e., 
\begin{equation}\label{uniPhi}
	\Phi(\xi):=\Phi_{\scrL,\vecq}(\xi)
\end{equation}
is independent of $\scrL$ and $\vecq$, for Lebesgue-almost every $\vecq$. Theorem~\ref{freeThm1} is proved in Section~\ref{secLorentz}, it is closely related to the lattice point problem studied in Section \ref{secVisible}.
The asymptotic tails of the limiting distribution $\Phi_{\scrL,\vecq}(\xi)$ are calculated in \cite{partIV}. In Section \ref{secLorentz} we generalize Theorem~\ref{freeThm1} in several ways. We consider for instance the distribution of free paths that hit a given point on the scatterer, which will be crucial in the characterization of the limiting random process in \cite{partII}.

Theorem~\ref{freeThm1} shows that the free path length scales like $\rho^{-(d-1)}$. This suggests to re-define position and time and use the ``macroscopic'' coordinates
\begin{equation}
	(\vecQ(t),\vecV(t)) = (\rho^{d-1} \vecq(\rho^{-(d-1)} t),\vecv(\rho^{-(d-1)} t)) .
\end{equation}
We now state a macroscopic version of Theorem \ref{freeThm1}, 
which is a corollary of the proof of Theorem \ref{freeThm1}
(see Section \ref{fT1corproofsec}). Here 
\begin{equation} \label{TAU1TAU1DEF}
\scrT_1(\vecQ,\vecV;\rho)=\rho^{d-1} \tau_1(\rho^{-(d-1)}\vecQ,\vecV;\rho)
\end{equation}
is the corresponding macroscopic free path length. 

\begin{thm}\label{freeThm1cor}
Fix a lattice $\scrL$ of covolume one and
let $\Lambda$ be a Borel probability measure on $\T^1(\RR^d)$ absolutely continuous with respect to Lebesgue measure. Then, for every $\xi\geq 0$,
\begin{equation} \label{freeThm1correl}
\lim_{\rho\to 0} \Lambda(\{ (\vecQ,\vecV)\in\T^1(\rho^{d-1}\scrK_\rho) \col \scrT_1(\vecQ,\vecV;\rho)\geq \xi \})
= \int_\xi^\infty \Phi(\xi') \, d\xi' 
\end{equation}
with $\Phi(\xi)$ as in \eqref{uniPhi}.
\end{thm}

Variants of Theorem \ref{freeThm1cor} were recently established by Boca and Zaharescu \cite{Boca07} in dimension $d=2$, using methods from analytic number theory; cf. also their earlier work with Gologan \cite{Boca03}, and the paper by Calglioti and Golse \cite{Caglioti03}. Our approach uses dynamics and equidistribution of flows on homogeneous spaces (the details are developed in Section \ref{secEqui}), and works in arbitrary dimension. Previous work in higher dimension $d>2$ includes the papers by Bourgain, Golse and Wennberg \cite{Bourgain98}, \cite{Golse00} who provide tail estimates of possible limiting distributions of converging subsequences. More details on the existing literature can be found in the survey \cite{Golse06}.

\subsection{Related lattice point problems}

\begin{figure}
\begin{center}
\framebox{
\begin{minipage}{0.4\textwidth}
\unitlength0.1\textwidth
\begin{picture}(10,10)(0,0)
\put(0.5,1){\includegraphics[width=0.9\textwidth]{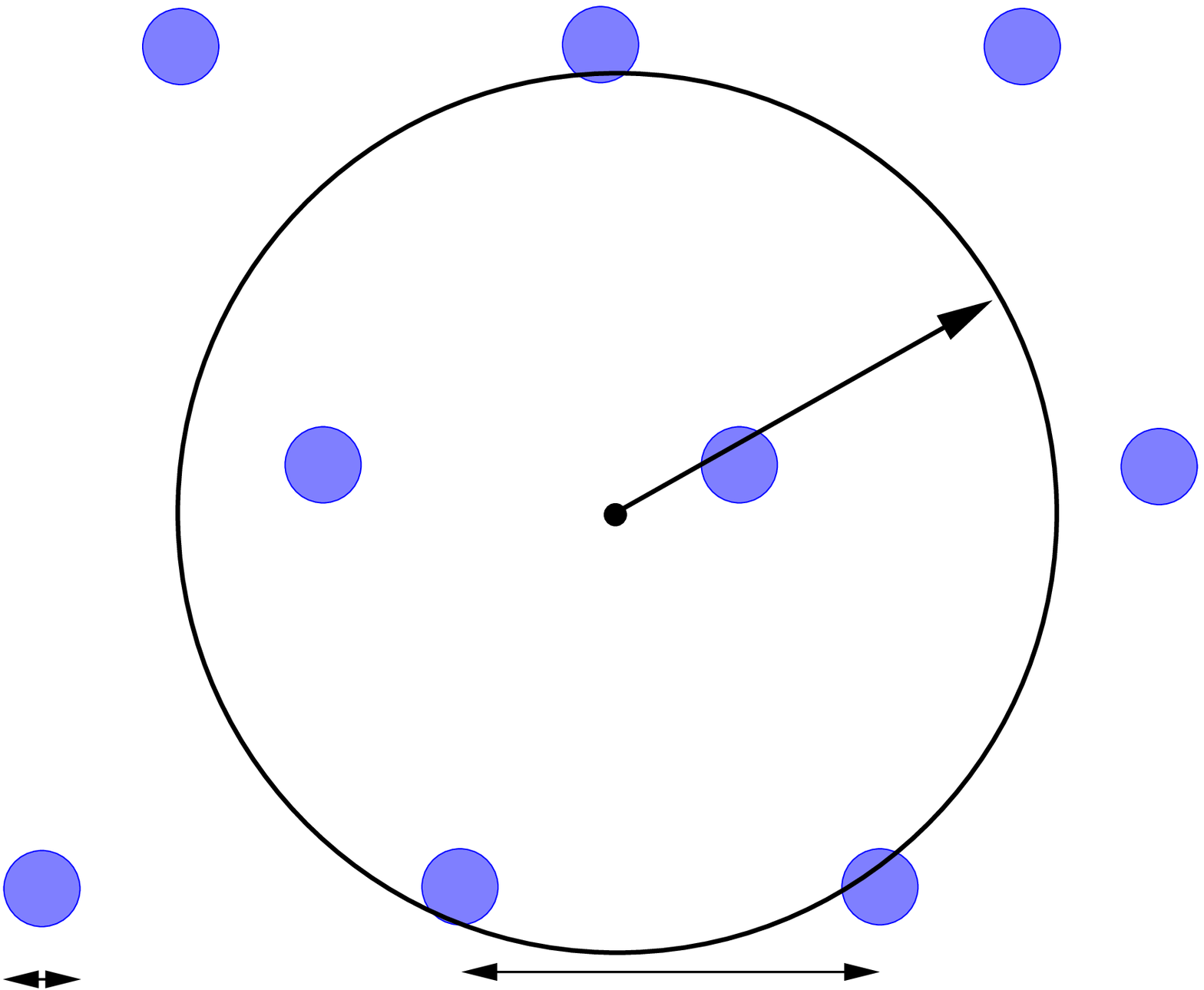}}
\put(5,4){$\vecq$}\put(7.2,5.2){$T\vecv$}
\put(0.5,0.5){$2\rho$} \put(4.8,0.5){fixed}
\end{picture}
\end{minipage}
}
\framebox{
\begin{minipage}{0.4\textwidth}
\unitlength0.1\textwidth
\begin{picture}(10,10)(0,0)
\put(0.5,1.5){\includegraphics[width=0.9\textwidth]{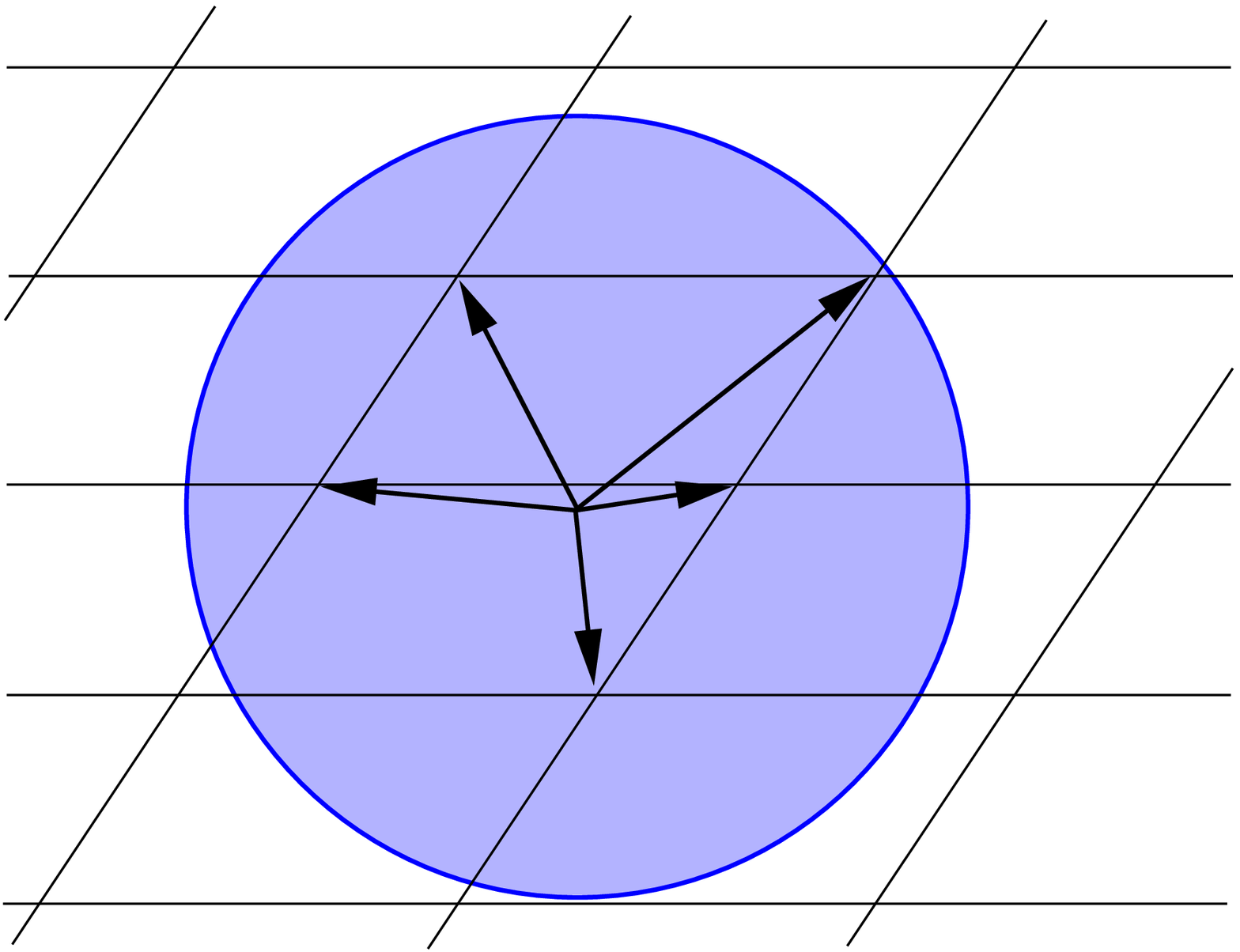}}
\put(4.3,4.2){$\vecnull$}\put(5.8,4.5){$\vecalf$}
\end{picture}
\end{minipage}
}
\end{center}
\caption{Left: How many lattice balls of radius $\rho$ does a random ray of length $T=\text{const}\times \rho^{-(d-1)}$ intersect? Right: What are the statistical properties of the directions of the affine lattice points $\scrL+\vecalf$ inside a large ball?}\label{figLattice}
\end{figure}

The key to the understanding of the Boltzmann-Grad limit of the periodic Lorentz gas are lattice point problems for thinly stretched domains, which are randomly rotated or sheared. In Sections \ref{secVisible0} and \ref{secVisible} we discuss two problems of independent interest that fall into this category: the distribution of spheres that intersect a randomly directed ray, and the statistical properties of the directions of lattice points (Figure~\ref{figLattice}). Section~\ref{secThin} discusses the general class of problems of this type.

Let us for example consider the affine lattice $\ZZ^2+\vecalf$, with the observer located at the origin. The directions of all lattice points with distance $< T$ are represented by points on the unit circle,
\begin{equation}
\frac{\vecm+\vecalf}{\|\vecm+\vecalf\|}\in\S_1^1, \qquad \text{for }\:
\vecm\in\ZZ^2\setminus\{-\vecalf\}, \quad \|\vecm+\vecalf\|< T. 
\end{equation}
We identify the circle with the unit interval via the map $(x,y)\mapsto (2\pi)^{-1}\arg(x+\i y)$, and therefore the distribution of directions is reformulated as a problem of distribution mod 1 of the numbers
\begin{equation}\label{ll}
\sfrac 1{2\pi} %
\arg(m+\alpha_1+\i (n+\alpha_2)), \qquad \text{for }\:
(m,n)\in\ZZ^2\setminus\{-\vecalf\}, \quad (m+\alpha_1)^2+(n+\alpha_2)^2 < T^2.
\end{equation}
We label these $N=N(T)$ numbers in order by
\begin{equation} \label{XISEQ}
-\tfrac12 < \xi_{N,1} \leq \xi_{N,2} \leq \ldots \leq \xi_{N,N} \leq \tfrac12
\end{equation}
and define in addition $\xi_{N,0}=\xi_{N,N}-1$. It is not hard to see that this sequence (or rather: this sequence of sequences) is uniformly distributed mod 1, i.e., for every $-\tfrac12\leq a< b \leq \tfrac12$,
\begin{equation}
	\lim_{N\to\infty} \frac{\#\{ 1\leq j\leq N : \xi_{N,j} \in [a,b) \}}{N}  = b-a.
\end{equation}
This (classical) equidistribution statement follows from the fact that the asymptotic number of lattice points in a fixed sector of a large disc is proportional to the volume of the sector.

A popular way to characterize the ``randomness'' of a uniformly distributed sequence is the statistics of gaps. The following theorem, which is a corollary of more general results in Section \ref{secVisible0}, shows that there is a limiting gap distribution when $N\to\infty$.

\begin{thm}\label{thmGap}
For every $\vecalf\in\RR^2$ there exists a distribution function 
$P_\vecalf(s)$ on $\R_{\geq 0}$ 
(continuous except possibly at $s=0$) 
such that for every $s\geq 0$,
\begin{equation} \label{thmGapformula}
	\lim_{N\to\infty} \frac{\#\{ 1\leq j\leq N : N(\xi_{N,j}-\xi_{N,j-1})\geq s\}}{N} = P_\vecalf(s).
\end{equation}
\end{thm}

We will provide explicit formulas for $P_\vecalf(s)$, which clearly deviate from the statistics of independent random variables from a Poisson process, where $P(s)=\exp(-s)$. It is remarkable that, for $\vecalf\notin\QQ^2$, the limiting distribution $P_\vecalf(s)$ is independent of $\vecalf$ and coincides with the gap distribution for the fractional parts of $\sqrt n$ calculated by Elkies and McMullen \cite{Elkies04}; cf. Figure \ref{figStats}. There is a deep reason for this apparent coincidence, which we will return to in the next section. 

The statistics are different for $\vecalf\in\QQ^2$. In particular 
$P_\vecalf(s)$ has a jump discontinuity at $s=0$ for every $\vecalf\in\Q^2$,
which exactly accounts for the multiplicities in the sequence \eqref{XISEQ};
removing all repetitions from that sequence results in a limiting gap distribution
which is continuous on all $\R_{\geq 0}$, see Corollary \ref{prim-thmGap}
below. In the particular case $\vecalf=\vecnull$ this recovers a result of
Boca, Cobeli and Zaharescu \cite{Boca00}, which is closely related to
the statistical distribution of Farey fractions (see also Boca and Zaharescu \cite{Boca05}). 

The only previously known result for non-zero values of $\vecalf$ is by Boca and Zaharescu \cite{Boca06}, who calculated the limit of the pair correlation function {\em on average} over $\vecalf$. (The pair correlation function is essentially the variance of the probability $E_{0,\vecalf}(r,\sigma)$ studied in Section \ref{secVisible0}.) Contrary to the behaviour of the gap probability $P_\vecalf(s)$, the limiting pair correlation function is the same as for random variables from a Poisson process.\footnote{Boca and Zaharescu consider a slightly different sequence of directions, which is obtained by replacing the last condition in \eqref{ll} with $\max(|m+\alpha_1|,|n+\alpha_2|) < T$. This sequence is however not uniformly distributed modulo one, which explains the discrepancy with the Poisson pair correlation function observed in \cite{Boca06}.}

\begin{figure}
\begin{center}
\begin{minipage}{0.49\textwidth}
\unitlength0.1\textwidth
\begin{picture}(10,8)(0,0)
\put(-0.2,-1.5){\includegraphics[width=\textwidth]{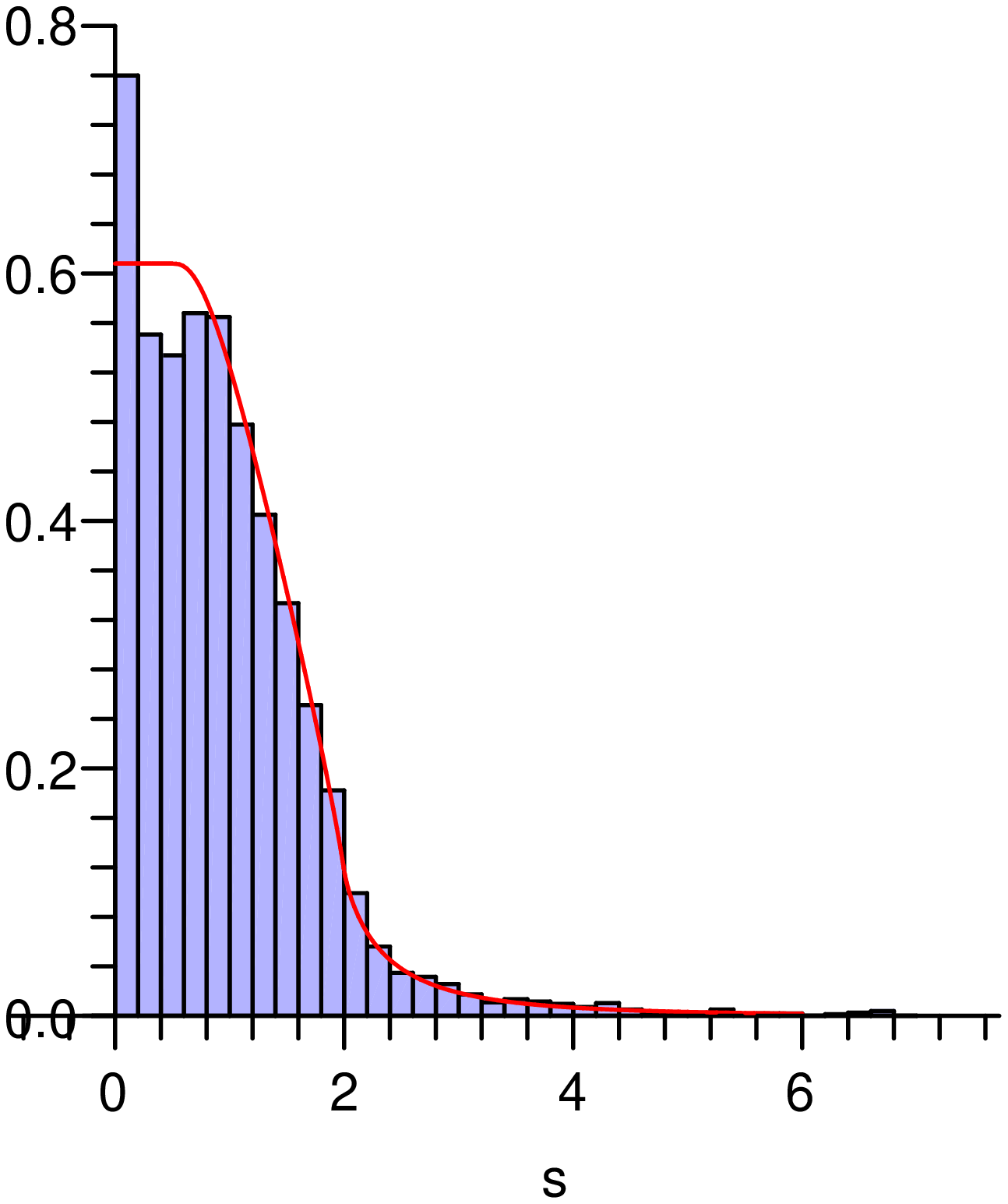}}
\end{picture}
\end{minipage}
\begin{minipage}{0.49\textwidth}
\unitlength0.1\textwidth
\begin{picture}(10,8)(0,0)
\put(-0.2,-1.5){\includegraphics[width=\textwidth]{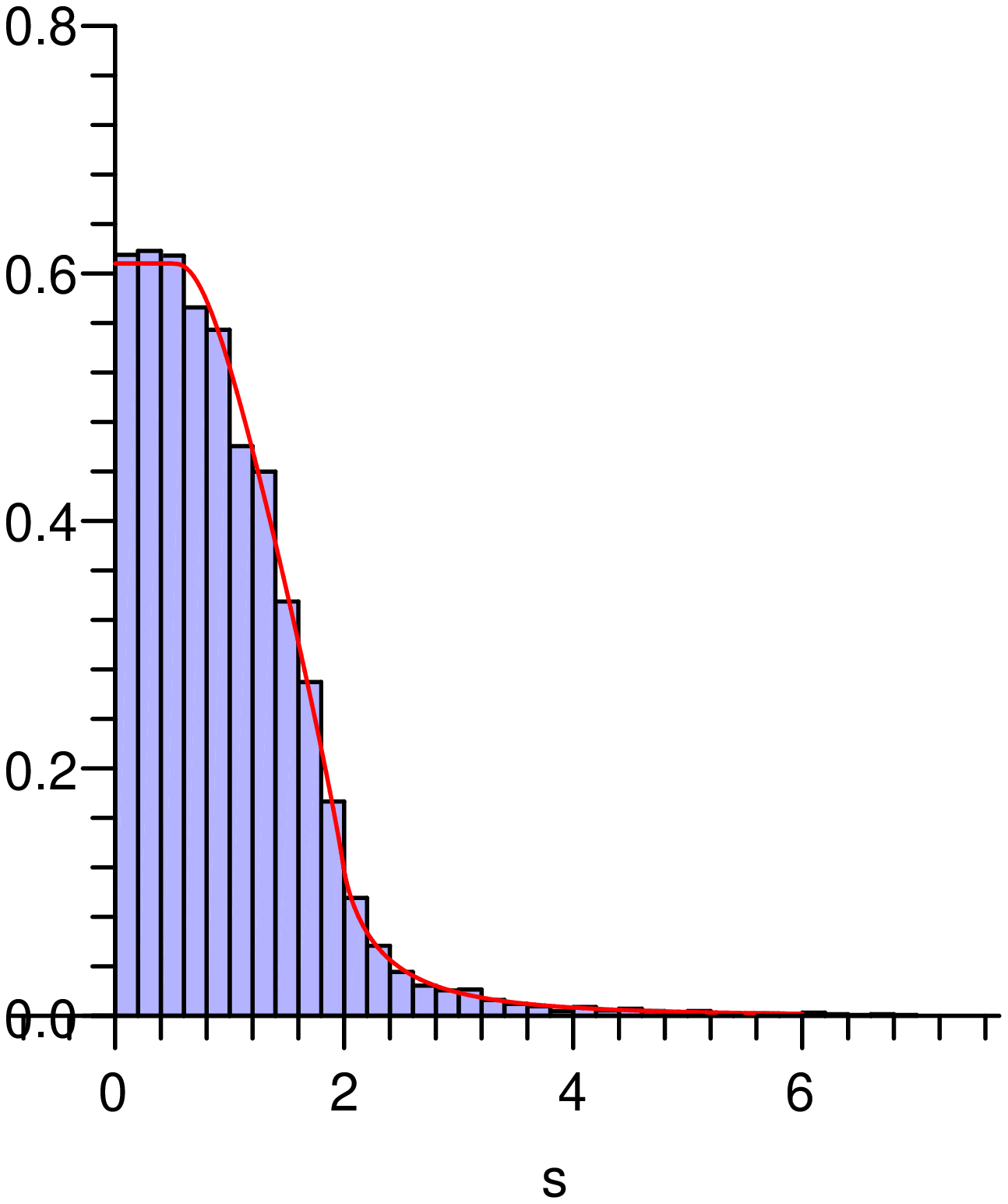}}
\end{picture}
\end{minipage}
\end{center}
\caption{Left: The distribution of gaps in the sequence $\sqrt{n}\bmod 1$, $n=1,\ldots,7765$, vs.~the Elkies-McMullen distribution. Right: Gap distribution for the directions of the vectors $(m-\sqrt{2},n)\in\RR^2$ with $m\in\ZZ$, $n\in\ZZ_{\geq 0}$, $(m-\sqrt2)^2+n^2< 4900$. The continuous curve is the Elkies-McMullen distribution.}\label{figStats}
\end{figure}

\subsection{Outline of the paper}

Sections \ref{secVisible0}--\ref{secLorentz} give a detailed account of the main results of this paper. Section \ref{secVisible0} discusses the statistical properties of affine lattice points inside a large sphere that are projected onto the unit sphere. A dual problem is the question of the probability that a ray of length $T$ pointing in a random direction intersects exactly $r$ lattice spheres whose radius scales as $T^{-1/(d-1)}$. The solution of the latter problem is provided in Section \ref{secVisible}, and applied in Section \ref{secLorentz} to the distribution of the free path lengths of the Lorentz gas. Both of the above lattice point problems fall into a general class of lattice point problems in randomly sheared or rotated domains, which are discussed in Section \ref{secThin}. The central idea for the solution of such questions is to exploit equidistribution results for flows on the homogeneous spaces $\SLSL$ and $\ASLASL$, which represent the space of lattices (resp. affine lattices) of covolume one. We establish the required ergodic-theoretic results in Section \ref{secEqui}. The key ingredient is Ratner's theorem \cite{Ratner91} on the classification of ergodic measures invariant under a unipotent flow. We provide useful integration formulas on $\SLSL$ and $\ASLASL$ in Section \ref{FOLIATIONSEC} and in Section \ref{PROPLIMFCNSEC} we apply these to our limit functions. Detailed proofs of the main limit theorems in Sections \ref{secVisible0}--\ref{secLorentz} are given in Section \ref{secProofs}. The proofs for Section \ref{secVisible0} are virtually identical to those of the corresponding theorems in Section \ref{secVisible}.

\section{Distribution of visible lattice points}\label{secVisible0}

\subsection{Lattices}

Let $\scrL\subset \RR^d$ be a euclidean lattice of covolume one. Recall that $\scrL=\ZZ^d M$ for some $M\in\SLR$ and that therefore the homogeneous space $X_1=\SLSL$ parametrizes the space of lattices of covolume one. 

Let $\ASLR=\SLR\ltimes \RR^d$ be the semidirect product group with multiplication law
\begin{equation} \label{ASLMULTLAW}
	(M,\vecxi)(M',\vecxi')=(MM',\vecxi M' +\vecxi') .
\end{equation}
An action of $\ASLR$ on $\RR^d$ can be defined as
\begin{equation}
	\vecy \mapsto \vecy(M,\vecxi):=\vecy M+\vecxi .
\end{equation}
Each \textit{affine} lattice (i.e.\ translate of a lattice) 
of covolume one in $\RR^d$ can then be expressed as
$\ZZ^d g$ for some $g\in\ASLR$, 
and the space of affine lattices is then represented by $X=\ASLASL$ where $\ASLZ=\SLZ\ltimes \ZZ^d$.
We denote by $\mu_1$ and $\mu$ the Haar measure on $\SLR$ and $\ASLR$, respectively, normalized in such a way that they represent probability measures on $X_1$ and $X$. 

If $\vecalf\in\QQ^d$, say $\vecalf=\vecp/q$ for $\vecp\in\ZZ^d$, 
$q\in\Z_{>0}$, we see that
\begin{equation}
	\bigg(\ZZ^d + \frac{\vecp}{q}\bigg) \gamma M = \bigg(\ZZ^d + \frac{\vecp}{q}\bigg) M 
\end{equation}
for all
\begin{equation}\label{Gamq}
	\gamma \in \Gamma(q):= \{ \gamma\in\SLZ \col \gamma \equiv 1_d \bmod q \},
\end{equation}
the principal congruence subgroup. This means that the space of affine lattices with $\vecalf=\vecp/q$ can be parametrized by the homogeneous space $X_q=\Gamma(q)\backslash\SLR$ (this is not necessarily one-to-one). We denote by $\mu_q$ the Haar measure on $\SLR$ which is normalized as a probability measure on $X_q$.

\subsection{Basic set-up}

We fix a lattice $\scrL\subset \RR^d$ of covolume one,
and fix, once and for all, a choice of $M_0\in\SL(d,\RR)$ such that
$\scrL=\ZZ^d M_0$.\label{latt} Given $\vecalf \in \RR^d$ we then define the
affine lattice
\begin{equation}\label{afflatt}
	\Lalf:=(\ZZ^d+\vecalf)M_0 = \ZZ^d (1,\vecalf)(M_0,\vecnull).
\end{equation}

Consider the set $\scrP_T$ of lattice points $\vecy\in\Lalf$ inside the ball $\scrB^d_T$ of radius $T$, or, more generally, the spherical shell
\begin{equation}\label{shell}
	\scrB^d_T(c)=\{ \vecx\in\RR^d : c T \leq \|\vecx\| < T \} , \qquad 0\leq c < 1.
\end{equation}
For $T$ large there are asymptotically $(1-c^d)\vol(\scrB^d_1) T^d$ such points, where $\vol(\scrB^d_1)=\pi^{d/2}/\Gamma(\tfrac{d+2}{2})$ is the volume of the unit ball. 
For each $T$, we study the corresponding \textit{directions,}
\begin{equation}
\|\vecy\|^{-1}\vecy\in\S_1^{d-1},\qquad
\text{for }\: \vecy \in\scrP_T=\Lalf\cap \scrB^d_T(c)\setminus\{\vecnull\},
\end{equation}
where $\S_\rho^{d-1}\subset\RR^d$ denotes the $(d-1)$-sphere of radius $\rho$.
It is well known that, as $T\to\infty$, these points become uniformly distributed on $\S_1^{d-1}$: For any set $\fU\subset\S_1^{d-1}$ with boundary of measure zero (with respect to the volume element $\vol_{\S_1^{d-1}}$ on $\S_1^{d-1}$) we have
\begin{equation}\label{udi}
\lim_{T\to\infty} \frac{\# \{\vecy\in\scrP_T\col\|\vecy\|^{-1}\vecy\in\fU\}}
{\#\scrP_T} 
= \frac{\vol_{\S_1^{d-1}}(\fU)}{\vol_{\S_1^{d-1}}(\S_1^{d-1})} . 
\end{equation}
Recall that $\vol_{\S_1^{d-1}}(\S_1^{d-1})=d \vol(\scrB^d_1)$.

\subsection{Distribution in small discs}\label{distrinsmalldiscssec}

We are interested in the fine-scale distribution of the directions to points
in $\scrP_T$, e.g., in the probability of finding $r$ directions in a small 
disc with random center $\vecv\in\S_1^{d-1}$.
We define $\fD_T(\sigma,\vecv)\subset\S_1^{d-1}$ to be the open disc
with center $\vecv$ and volume 
\begin{equation} \label{DISCSIZE}
\vol_{\S_1^{d-1}}(\fD_T(\sigma,\vecv)) =\frac{\sigma d}{1-c^d} T^{-d}.
\end{equation}
The radius of $\fD_T(\sigma,\vecv)$ is thus 
$\asymp T^{-d/(d-1)}$ (if $\sigma>0$). 
We introduce the counting function
\begin{equation}\label{disc00}
	\scrN_{c,T}(\sigma,\vecv)=\#\{\vecy\in\scrP_T \col
\|\vecy\|^{-1}\vecy\in \fD_T(\sigma,\vecv)\}
\end{equation}
for the number of points in $\fD_T(\sigma,\vecv)$.
The motivation for the definition \eqref{DISCSIZE} is that it implies, 
via \eqref{udi}, that
the expectation value for the counting function is asymptotically equal to 
$\sigma$ (for $T\to\infty$ and $\sigma$ fixed):
\begin{equation}\label{expval}
	\int_{\S_1^{d-1}} \scrN_{c,T}(\sigma,\vecv)\, d\lambda(\vecv) \sim 
\vol(\scrB^d_T(c)) \frac{\vol_{\S_1^{d-1}}(\fD_T(\sigma))}{\vol_{\S_1^{d-1}}(\S_1^{d-1})}  
	= \frac{1-c^d}{d}  \vol_{\S_1^{d-1}}(\fD_T(\sigma)) T^d=\sigma,
\end{equation}
where $\lambda$ is any probability measure on $\S_1^{d-1}$ with continuous density, and
$\fD_T(\sigma)=\fD_T(\sigma,\vece_1)$ is the disc centered at $\vece_1=(1,0,\ldots,0)$.

\begin{thm}\label{visThm0}
Let $\lambda$ be a Borel probability measure on $\S_1^{d-1}$ absolutely continuous with respect to Lebesgue measure. Then, for every $\sigma \geq 0$ and $r\in\ZZ_{\geq 0}$, the limit 
\begin{equation} \label{EDEF}
	E_{c,\vecalf}(r,\sigma):=\lim_{T\to\infty} \lambda(\{ \vecv\in\S_1^{d-1} : \scrN_{c,T}(\sigma,\vecv)=r \})
\end{equation}
exists, and for fixed $c,\vecalf,r$
the convergence is uniform with respect to $\sigma$ in any 
compact subset of $\RR_{\geq 0}$.
The limit function is given by
\begin{equation}
E_{c,\vecalf}(r,\sigma)=
	\begin{cases}
	\mu_1(\{ M\in X_1: \#( \ZZ_*^d M \cap \fC(c,\sigma))= r \}) & \text{if $\vecalf\in\ZZ^d$}\\
	\mu_q(\{ M\in X_q: \#( (\ZZ^d+\frac{\vecp}{q}) M \cap \fC(c,\sigma))= r \}) & \text{if $\vecalf=\frac{\vecp}{q}\in\QQ^d\setminus\ZZ^d $}\\
	\mu(\{ (M,\vecxi)\in X: \#((\ZZ^d M +\vecxi) \cap \fC(c,\sigma))= r \}) & \text{if $\vecalf\notin\QQ^d$,}
	\end{cases}
\end{equation}
where
\begin{equation} \label{FCCSDEF}
	\fC(c,\sigma) =\bigg\{(x_1,\ldots,x_d)\in\RR^d \col c < x_1 < 1, \: \|(x_2,\ldots,x_d)\|\leq x_1 A(c,\sigma)  \bigg\} ,
\end{equation}
\begin{equation} \label{ACSDEF}
	A(c,\sigma)=\bigg(\frac{\sigma d}{(1-c^d)\vol(\scrB_1^{d-1})}\bigg)^{\frac{1}{d-1}}, \qquad \vol(\scrB_1^{d-1})=\frac{\pi^{(d-1)/2}}{\Gamma(\tfrac{d+1}{2})}.
\end{equation}
In particular, $E_{c,\vecalf}(r,\sigma)$ is continuous in $\sigma$ and independent of $\scrL$ and $\lambda$.
\end{thm}

In the above, we use the notation $\ZZ_*^d:=\ZZ^d\setminus\{\vecnull\}$.\label{ZZSTARDDEF} Although the use of $\ZZ_*^d$ is superfluous at this point (since $\fC(c,\sigma)$ does not contain zero), it appears as the natural object in the proof. This subtlety is due to the fact that for generic $M$ we have $\ZZ^d M \cap \fC(0,\sigma)\neq\ZZ^d M \cap \overline{\fC(0,\sigma)}$ but $\ZZ_*^d M \cap \fC(0,\sigma)=\ZZ_*^d M \cap \overline{\fC(0,\sigma)}$.

Theorem \ref{visThm0} says that the limiting distribution $E_{c,\vecnull}(r,\sigma)$ is given by the probability that there are $r$ points of a random lattice in the cone $\fC(c,\sigma)$, and $E_{c,\vecalf}(r,\sigma)$ for $\vecalf\notin\QQ^d$ is the corresponding probability for a random affine lattice. Hence in particular $E_{c,\vecalf}(r,\sigma)$ is independent of $\vecalf$ when $\vecalf\notin\QQ^d$. 

\begin{remark} \label{visThm0rem}
We will furthermore prove that when $c=0$ the function $E_{c,\vecalf}(r,\sigma)$ is $\C^1$ with 
respect to $\sigma>0$; see Section~\ref{IMPORTANTSECV0}.
We expect that the same statement should also be true for any fixed $0<c<1$.
\end{remark}

\begin{remark}
In the case $c=0$, $d=2$ and $\vecalf\notin\QQ^2$ our distribution coincides with Elkies and McMullen's limiting distribution \cite{Elkies04} for the probability of finding $r$ elements of the sequence $\sqrt n$ \mbox{mod 1} ($n=1,\ldots,N$) in a randomly shifted interval of length $\sigma/N$ ($N\to\infty$).
Although the two problems are seemingly unrelated, the reason for this coincidence is that both results use equidistribution of translates of different orbits on the space of affine lattices $X$ with respect to the same test functions.
\end{remark}

\begin{remark} \label{GENSTATREM}
By a general statistical argument, cf.~e.g.~\cite{Elkies04}, \cite{Marklof07}, Theorem \ref{thmGap} is an immediate corollary of Theorem \ref{visThm0} in 
the case $d=2$, $r=0$, 
with the limit function $P_\vecalf(s)$ explicitly given by
\begin{equation} \label{PALFFORMULA}
	P_\vecalf(s):= -\frac{d}{ds} E_{0,\vecalf}(0,s) \quad(s>0);
\qquad P_\vecalf(0):=1.
\end{equation}
The continuity of $P_\vecalf(s)$ for $s>0$ follows from 
Remark \ref{visThm0rem}.
\end{remark}

To exhibit explicitly the group action which will play a central role in the proof of the above statements, it is convenient to realize $\S_1^{d-1}$ as the homogeneous space $\SO(d-1)\backslash\SO(d)$ by setting $\vecv=\vece_1 K$ with $\vece_1=(1,0,\ldots,0)$ and $K\in\SO(d)$. The stabilizer of $\vece_1$ is isomorphic to $\SO(d-1)$ (acting from the right), where $\SO(d-1)$ is identified with the subgroup
\begin{equation} \label{SODINBED}
	\begin{pmatrix}
	1 &  \vecnull \\
	\trans\vecnull & \SO(d-1)
	\end{pmatrix}
	\subset  \SO(d).
\end{equation} 
Then
\begin{equation}
	\fD_T(\sigma,\vecv)=\fD_T(\sigma) K=\{ \vecx : \vecx K^{-1}\in \fD_T(\sigma) \}
\end{equation}
and
\begin{equation}\label{disc1}
	\scrN_{c,T}(\sigma,K)=\#(\scrP_T \cap \fD_T(\sigma) K)
\end{equation}
is the number of points in $\fD_T(\sigma) K$. Note that $\scrN_{c,T}(\sigma,K)$ is left-invariant under the action of $\SO(d-1)$ and thus may be viewed as a function on $\SO(d-1)\backslash\SO(d)$. 
The statement equivalent to Theorem \ref{visThm0} is now that, if $\lambda$ is a Borel probability measure on $\SO(d)$ absolutely continuous with respect to Haar measure, then
\begin{equation}
	\lim_{T\to\infty} \lambda(\{ K\in\SO(d) : \scrN_{c,T}(\sigma,K)=r \}) = E_{c,\vecalf}(r,\sigma).
\end{equation}

\subsection{Visible lattice points} \label{VISLPTS}
In the study of directions of affine lattice points it is natural to
restrict our attention to those points that are visible from the origin.
That is, we consider the set of directions without counting
multiplicities. Non-trivial multiplicities are %
only obtained when the $\Q$-linear span of %
$1$ and the components of $\vecalf$ has dimension $\leq 2$.
If $\vecalf\notin\Q^d$ then the multiplicities are statistically
insignificant; in fact they can only occur along at most a single line
through the origin,
and thus restricting to considering only the visible lattice 
points still yields the same limit distribution as in Theorem \ref{visThm0}.

Hence from now on we will assume $\vecalf\in\Q^d$.
If $\vecalf=\bn$ then the visible lattice points are exactly the
\textit{primitive} lattice points, i.e.\ those points $\vecm M_0\in\scrL$
for which $\vecm\in\Z^d_*$, $\gcd(\vecm)=1$.
In the general case $\vecalf=\frac{\vecp}q\in\Q^d$ 
($q\in\Z_{>0}$, $\vecp\in\Z^d$),
the set of visible lattice points is:
\begin{align} \label{VISIBLECHAR}
\widehat\scrL_\vecalf=\widehat\Z_\vecalf^d M_0, \qquad 
\widehat\Z_\vecalf^d := \{\vecx\in(\Z^d+\vecalf)\setminus\{\bn\} \col
\gcd(q\vecx)\leq q\}.
\end{align}
From now on in this section we will assume that $q\in\Z_{>0}$ is 
the \textit{minimal} integer which gives $q\vecalf\in\Z^d$. 
Given $0\leq c<1$ we set
$\widehat\scrP_T=\widehat\scrL_\vecalf \cap \scrB_T^d(c)$; \label{WHPDEF}
then
by a sieving argument using \eqref{VISIBLECHAR} and \eqref{udi} one shows that
for any set $\fU\subset\S_1^{d-1}$ with boundary of measure zero,
\begin{align} \label{KAPPAQDEF}
& \lim_{T\to\infty} \frac{\#\{\vecy\in\widehat\scrP_T\cap\scrB_T^d(c)\col
\|\vecy\|^{-1}\vecy\in\fU\}}{\vol{\scrB_T^d(c)}}
=\kappa_q \frac{\vol_{\S_1^{d-1}}(\fU)}{\vol_{\S_1^{d-1}}(\S_1^{d-1})},
\\ \notag
& \text{with}\qquad 
\kappa_q:=\Bigl(\sum_{\substack{n\geq 1\\(n,q)=1}} \mu(n) n^{-d}\Bigr)
\sum_{\substack{1\leq t\leq q\\(t,q)=1}} t^{-d}
=\Bigl(\sum_{\substack{n\geq 1\\(n,q)=1}} n^{-d}\Bigr)^{-1} 
\sum_{\substack{1\leq t\leq q\\(t,q)=1}} t^{-d}.
\end{align}
When $\vecalf\in\Z^d$ this specializes to the well-known fact that
the asymptotic density of the primitive points in $\Z^d$ is $\zeta(d)^{-1}$.
It follows from \eqref{KAPPAQDEF} that if we introduce the following analogue 
of \eqref{disc00} for visible lattice points:
\begin{equation} \label{WHNDEF}
\widehat{\scrN}_{c,T}(\sigma,\vecv)=\#\{\vecy\in\widehat\scrP_T \col
\|\vecy\|^{-1}\vecy\in \fD_T(\kappa_q^{-1}\sigma,\vecv)\},
\qquad
\widehat\scrP_T=\widehat\scrL_\vecalf \cap \scrB_T^d(c);
\end{equation}
then the expectation value for $\widehat\scrN$ is again asymptotically equal
to $\sigma$:
\begin{align}
\lim_{T\to\infty}
\int_{\S_1^{d-1}} \widehat \scrN_{c,T}(\sigma,\vecv)\,d\lambda(\vecv)=\sigma,
\end{align}
for any fixed $\sigma\geq 0$, $0\leq c<1$ and $\lambda$ as in \eqref{expval}.

\begin{thm}\label{prim-visThm0}
Let $\lambda$ be a Borel probability measure on $\S_1^{d-1}$ absolutely continuous with respect to Lebesgue measure. Then, for every $\sigma \geq 0$ and $r\in\ZZ_{\geq 0}$, the limit 
\begin{equation}
	\widehat E_{c,\vecalf}(r,\sigma):=\lim_{T\to\infty} \lambda(\{ \vecv\in\S_1^{d-1} : \widehat\scrN_{c,T}(\sigma,\vecv)=r \})
\end{equation}
exists, and for fixed $c,r$
the convergence is uniform with respect to $\sigma$ in any 
compact subset of $\RR_{\geq 0}$.
The limit function is given by
\begin{equation}
	\mu_q(\{ M\in X_q\col 
\#( \widehat\Z_\vecalf^d M \cap \fC(c,\kappa_q^{-1}\sigma))= r \}) 
\qquad (\vecalf=\sfrac{\vecp}q\in\Q^d).
\end{equation}
In particular, $\widehat E_{c,\vecalf}(r,\sigma)$ is continuous in $\sigma$ and independent of $\scrL$ and $\lambda$.
\end{thm}

\begin{remark} \label{prim-visThm0rem}
The function $\widehat E_{0,\vecalf}(r,\sigma)$ is $\C^1$
with respect to $\sigma>0$. This is proved by
adapting the arguments of Sections \ref{SUBMSECTION}  
and \ref{IMPORTANTSECV0} to the setting of visible lattice points.
\end{remark}

In dimension $d=2$, considering only visible lattice points gives a 
variant of Theorem \ref{thmGap} with
an everywhere continuous distribution function:
Take $\vecalf\in\Q^2$, and consider the set of rescaled directions
\begin{align}
\bigl\{
\sfrac 1{2\pi} %
\arg(x_1+\i x_2)\col
\vecx=(x_1,x_2)\in\widehat\Z^2_\vecalf, \quad x_1^2+x_2^2<T^2\bigr\}.
\end{align}
Let us label these $\tN=\tN(T)$ numbers 
in order by
\begin{align}
-\tfrac12 < \txi_{\tN,1}<\txi_{\tN,2}<\ldots<\txi_{\tN,\tN}\leq \tfrac12
\end{align}
and define in addition $\txi_{\tN,0}=\txi_{\tN,\tN}-1$.
Note that this is exactly the sequence which is obtained from \eqref{XISEQ}
by removing all repetitions. We now have:
\begin{cor} \label{prim-thmGap}
There exists a distribution function $\widehat P_\vecalf(s)$ on 
$\R_{\geq 0}$, continuous on all of $\R_{\geq 0}$, such that for every 
$s\geq 0$,
\begin{equation} \label{TXITHM}
\lim_{\tN\to\infty} \tN^{-1}
\#\bigl\{ 1\leq j\leq \tN : \tN(\txi_{\tN,j}-\txi_{\tN,j-1}) \geq s\bigr\} 
=\widehat P_\vecalf(s).
\end{equation}
\end{cor}
\begin{proof}
Just as in Remark \ref{GENSTATREM},
the limit relation \eqref{TXITHM} follows from
Theorem \ref{prim-visThm0} together with the fact 
$\tN\sim \kappa_q\pi T^2$ as $T\to\infty$ (cf.\ \eqref{KAPPAQDEF}),
and $\widehat P_\vecalf(s)$ is explicitly given by
\begin{align}
\widehat P_\vecalf(s):= -\frac{d}{ds}\widehat E_{0,\vecalf}(0,s)
\quad(s>0);
\qquad \widehat P_\vecalf(0):=1.
\end{align}
Note that $\widehat E_{0,\vecalf}(0,s)
=E_{0,\vecalf}(0,\kappa_q^{-1}s)$ for all $s\geq 0$, 
since $\fC(0,\kappa_q^{-1}s)$ is star shaped. Hence
\begin{align}
\widehat P_\vecalf(s)=
\kappa_q^{-1} P_\vecalf(\kappa_q^{-1}s) \qquad \text{for } \: s>0.
\end{align}
The continuity of $\widehat P_\vecalf(s)$ for $s>0$ follows from 
Remark \ref{visThm0rem}, or Remark \ref{prim-visThm0rem}.
Furthermore, in Section \ref{IMPORTANTSECV0} we will prove that (for $d=2$),
\begin{align} \label{E0ALF0SMALLS}
E_{0,\vecalf}(0,\sigma)=1-\kappa_q \sigma, \qquad \forall \sigma\in \bigl[0,(2q)^{-1}\bigr],
\end{align}
and this implies that $\widehat P_\vecalf(s)$ is also continuous at $s=0$.
\end{proof}

When $\vecalf=\bn$, Corollary \ref{prim-thmGap} specializes to give the 
limiting
gap distribution for directions of primitive lattice points in $\Z^2$,
which was proved earlier by Boca, Cobeli and Zaharescu \cite{Boca00}.

\vspace{5pt}

The proofs of Theorems \ref{visThm0} and \ref{prim-visThm0} are virtually identical
to those of Theorems \ref{visThm} and \ref{prim-visThm2}; we will therefore only outline the differences in Section \ref{diff0}. In \cite{visible} we carry out a more detailed statistical analysis of the distribution of visible lattice points, which yields generalizations of Theorems \ref{visThm0} and \ref{prim-visThm0}, and also provide explicit formulas and tail estimates of the limiting distributions.

\section{The number of spheres in a random direction} \label{secVisible}

We now turn to a lattice point problem that is in some sense dual to the one studied in the previous Section \ref{secVisible0}. Its solution will also answer the question of the distribution of free path lengths in the periodic Lorentz gas, see Section \ref{secLorentz} below for details.

\subsection{Spheres centered at lattice points} \label{secVisiblefirst}

We place at each lattice point $\vecy\in\Lalf$ a ball of small radius $\rho$ and consider the set $\scrB^d_\rho + \Lalf$. 
The set of balls with centers inside the shell \eqref{shell} is
\begin{equation}\label{BBB}
	\{ \vecx\in \scrB^d_\rho + \vecy\col \vecy\in\Lalf\cap \scrB^d_T(c)\setminus\{\vecnull\} \}. 
\end{equation}
Note that we remove any ball at $\vecy=\vecnull$ (this is only relevant in the case $\vecalf\in\ZZ^d$). 
Furthermore we will always keep 
$\rho \leq m(\Lalf):=\min \{\|\vecy\| \col \vecy\in\Lalf\setminus\{\bn\}\}$, 
so that $\bn$ lies outside each of the balls in our set.
We are interested in the number $\scrN_{c,T}(\rho,\vecv)$ of intersections of this set with a ray starting at the origin $\vecnull$ that points in the random direction $\vecv\in\S_1^{d-1}$ distributed according to the probability measure $\lambda$. That is
\begin{equation} \label{NCTRHOVDEF}
	\scrN_{c,T}(\rho,\vecv):=\#\big\{ \vecy\in\Lalf\cap \scrB^d_T(c)\setminus\{\vecnull\}\col
	\R_{>0}\vecv \cap (\scrB^d_\rho + \vecy) \neq \emptyset \big\} .
\end{equation}

If $\rho\leq\|\vecy\|$, then
a ray in direction $\vecv$ hits the ball $\scrB^d_\rho + \vecy$ if and only if
\begin{equation}
	\|\vecy\|^{-1} \vecy \in \fD(\|\vecy\|^{-1}\rho,\vecv)
\end{equation}
with the disc
\begin{align}\label{disceps}
& \fD(\epsilon,\vecv) = (\scrB^d_{\epsilon}+\vecv)(1-\epsilon^2)^{-1/2} \cap \S_1^{d-1} 
\qquad (0<\epsilon<1);
\\ \notag
& \fD(1,\vecv)=\{\vecw\in\S_1^{d-1}\col \vecw\cdot\vecv>0\}.
\end{align}
We will again use the shorthand $\fD(\epsilon)=\fD(\epsilon,\vece_1)$.
The radius of this disc is $\sim\epsilon$, for $\epsilon\to 0$.
Hence the number of balls hit by a ray in direction $\vecv$ is 
\begin{equation}\label{asdef}
	\scrN_{c,T}(\rho,\vecv)=\#\bigg\{ \vecy\in\Lalf\cap \scrB^d_T(c)\setminus\{\vecnull\}\col  \frac{\vecy}{\|\vecy\|} \in \fD(\|\vecy\|^{-1}\rho,\vecv)\bigg\} ,
\end{equation}
compare \eqref{disc1}.

For any $\lambda$ as in \eqref{expval}, 
one finds for the expectation value as $T\to\infty$, $\rho\to 0$
\begin{equation}\label{expval2}
\begin{split}
	\int_{\S_1^{d-1}} \scrN_{c,T}(\rho,\vecv) \, d\lambda(\vecv) 
	& \sim \int_{\vecy\in\scrB^d_T(c),\: \|\vecy\|>\rho} \frac{\vol(\fD(\|\vecy\|^{-1}\rho))}{\vol(\S_1^{d-1})} \, d\!\vol(\vecy) \\
	& \sim \frac{\vol(\scrB^{d-1}_1)}{\vol(\S_1^{d-1})} \rho^{d-1} \int_{\scrB^d_T(c)} \frac{d\!\vol(\vecy)}{\|\vecy\|^{d-1}} \\
	& = \vol(\scrB^{d-1}_1) (1-c) \rho^{d-1} T .
\end{split}
\end{equation}
This suggests the scaling $\rho=\sigma T^{-1/(d-1)}$ with $\sigma\geq 0$ fixed.

\begin{thm}\label{visThm}
Let $\lambda$ be a Borel probability measure on $\S_1^{d-1}$ absolutely continuous with respect to $\vol_{\S_1^{d-1}}$. 
Then, for every $\sigma\geq 0$ and $r\in\ZZ_{\geq 0}$, the limit 
\begin{equation}
	F_{c,\vecalf}(r,\sigma):=\lim_{T\to\infty} \lambda(\{ \vecv\in\S_1^{d-1} : \scrN_{c,T}(\sigma T^{-1/(d-1)},\vecv)=r \})
\end{equation}
exists, and for fixed $\vecalf,r$
the convergence is uniform with respect to $\sigma$ in any 
compact subset of $\RR_{\geq 0}$, and with respect to $c\in [0,1]$.
The limit function is given by
\begin{equation} \label{FCALFRSIGMDEF}
F_{c,\vecalf}(r,\sigma)=
	\begin{cases}
	\mu_1(\{ M\in X_1: \#( \ZZ_*^d M \cap \fZ(c,\sigma))= r \}) & \text{if $\vecalf\in\ZZ^d$}\\
	\mu_q(\{ M\in X_q: \#( (\ZZ^d+\frac{\vecp}{q}) M \cap \fZ(c,\sigma))= r \}) & \text{if $\vecalf=\frac{\vecp}{q}\in\QQ^d\setminus\ZZ^d $}\\
	\mu(\{ (M,\vecxi)\in X: \#((\ZZ^d M +\vecxi) \cap \fZ(c,\sigma))= r \}) & \text{if $\vecalf\notin\QQ^d$,}
	\end{cases}
\end{equation}
where
\begin{equation}\label{zyl}
	\fZ(c,\sigma) =\big\{(x_1,\ldots,x_d)\in\RR^d : c < x_1 < 1, \|(x_2,\ldots,x_d)\|< \sigma \big\} .
\end{equation}
In particular, $F_{c,\vecalf}(r,\sigma)$ is continuous in $\sigma$ and independent of $\scrL$ and $\lambda$.
\end{thm}

\begin{remark} \label{visThmrem}
In the case $c=0$ the function $F_{c,\vecalf}(r,\sigma)$ is $\C^1$ with 
respect to $\sigma>0$; we will prove this in Section \ref{subsecC1}.
(We expect the same should be true also for any fixed $0<c<1$.)
If $\vecalf\notin\Q^d$ then $F_{c,\vecalf}(r,\sigma)$
is independent of $\vecalf$; we denote this ``universal''
limit function simply by $F_c(r,\sigma)$.\label{FCRSDEF}
We prove in Section \ref{subsecC1} that $F_c(r,\sigma)$ is $\C^2$
with respect to $\sigma>0$, for any fixed $0\leq c<1$. 
\end{remark}

\begin{remark} \label{visThmremexpl}
We will give tail estimates for 
$F_{c,\vecalf}(r,\sigma)$ for general dimension $d$ in \cite{partIV}.
In the special case $d=2$, explicit formulas for $F_0(r,\sigma)$ and 
$F_{0,\bn}(r,\sigma)$ were given in \cite{SV}, where these 
limit functions came up in a different set of problems.
Specifically, $F_0(r,\sigma)=f_r^{\text{box,ASL}_2}(2\sigma)$ and
$F_{0,\bn}(r,\sigma)=f_{2r+1}^{\text{box,SL}_2}(4\sigma)$ in the notation of
\cite[Section 7]{SV}.
\end{remark}

\subsection{A variation}\label{secVar}

Instead of rays emerging from the origin we consider now the family of rays starting at the points $\rho \vecbeta(\vecv)$ in direction 
$\vecv$, where $\vecbeta:\S_1^{d-1}\to\RR^{d}$ is some fixed continuous function.
We will keep $\rho$ so small that,
for all $\vecy\in\Lalf\setminus\{\bn\}$ and all $\vecv\in\S^{d-1}_1$,
the point $\rho\vecbeta(\vecv)$ lies outside the ball $\scrB_\rho^d+\vecy$.
Then the ray $\rho \vecbeta(\vecv)+\R_{>0} \vecv$ hits the ball 
$\scrB^d_\rho + \vecy$ if and only if
\begin{equation}
	\frac{\vecy-\rho \vecbeta(\vecv)}{\|\vecy-\rho \vecbeta(\vecv)\|}  \in 
	\fD(\|\vecy-\rho \vecbeta(\vecv)\|^{-1}\rho,\vecv),
\end{equation}
compare the analogous argument in the previous section.
Hence the number of balls in \eqref{BBB} intersecting this ray is
$\scrN_{c,T}(\rho,\vecv,\vecbeta(\vecv))$, where
\begin{equation} \label{asin}
	\scrN_{c,T}(\rho,\vecv,\vecw) := \#\bigg\{ \vecy \in (\Lalf \cap \scrB^d_T(c)\setminus\{\vecnull\})-\rho \vecw \col 
	\frac{\vecy}{\|\vecy\|}   \in \fD(\|\vecy\|^{-1}\rho,\vecv) \bigg\} .
\end{equation}

\begin{thm}\label{visThm2}
Let $\lambda$ be a Borel probability measure on $\S_1^{d-1}$ absolutely continuous with respect to Lebesgue measure. Then, for every $\sigma\geq 0$ and $r\in\ZZ_{\geq 0}$, the limit 
\begin{equation} \label{visThm2limit}
	F_{c,\vecalf,\vecbeta}(r,\sigma):=\lim_{T\to\infty} \lambda(\{ \vecv\in\S_1^{d-1} : \scrN_{c,T}(\sigma T^{-1/(d-1)},\vecv,\vecbeta(\vecv))=r \})
\end{equation}
exists, and for fixed $\vecalf,\vecbeta,\lambda,r$ 
the convergence is uniform with respect to $\sigma$ in any 
compact subset of $\RR_{\geq 0}$, and with respect to $c\in [0,1]$.
The limit function is given by
\begin{equation} \label{defF}
\begin{split}
&F_{c,\vecalf,\vecbeta}(r,\sigma)
\\
&=	\begin{cases}
	(\mu_1\times\lambda)(\{ (M,\vecv)\in X_1\times\S_1^{d-1}: 
	\#( \ZZ_*^d M \cap \fZ_\vecv(c,\sigma))= r \}) & \text{if $\vecalf\in\ZZ^d$}\\
	(\mu_q\times\lambda)(\{ (M,\vecv)\in X_q\times\S_1^{d-1}: 
	\#( (\ZZ^d+\frac{\vecp}{q}) M \cap \fZ_\vecv(c,\sigma))= r \}) 
	& \text{if $\vecalf=\frac{\vecp}{q}\in\QQ^d\setminus\ZZ^d $}\\
	\mu(\{ (M,\vecxi)\in X: \#((\ZZ^d M +\vecxi) \cap \fZ(c,\sigma))= r \}) & \text{if $\vecalf\notin\QQ^d$,}
	\end{cases}
\end{split}
\end{equation}
where
\begin{equation} \label{FZVCSDEF}
	\fZ_\vecv(c,\sigma)=\fZ(c,\sigma)+\sigma
\bigl\|\Proj_{\{\vecv\}^\perp} \vecbeta(\vecv)\bigr\|\cdot \vece_2.
\end{equation}
($\Proj_{\{\vecv\}^\perp}$ denotes the orthogonal projection from
$\R^d$ onto the orthogonal complement of $\vecv$, and 
$\vece_2=(0,1,0,\ldots,0)$.)
In particular $F_{c,\vecalf,\vecbeta}(r,\sigma)$ is continuous in $\sigma$ and independent of $\scrL$, and if $\vecalf\notin\QQ^d$ then 
$F_{c,\vecalf,\vecbeta}(r,\sigma)%
=F_{c}(r,\sigma)$,
independently of $\vecbeta$ and $\lambda$.
\end{thm}

\begin{remark} \label{visThm2rem}
Again, we prove in the case $c=0$ that the function $F_{c,\vecalf,\vecbeta}(r,\sigma)$ is $\C^1$ 
with respect to $\sigma>0$; see Section \ref{subsecC1}.
\end{remark}

\begin{remark} \label{GENBETAUPPERBOUNDREM}
It will be useful for several of the results in Section \ref{secLorentz} below,
as well as in the proofs in \cite{partII}, to know that 
$\lim_{\sigma\to 0} F_{c,\vecalf,\vecbeta}(0,\sigma)=1$ 
and $\lim_{\sigma\to\infty} F_{c,\vecalf,\vecbeta}(r,\sigma)=0$, and that
this holds uniformly with respect to the various parameters.
This follows from the following two basic bounds, which we prove
in Section \ref{GENBETAUPPERBOUNDPROOFSEC}.
More exact asymptotic formulas will be given in \cite{partIV}.

Let $v_d:=\vol(\scrB_1^{d-1})=\pi^{(d-1)/2}/\Gamma(\frac{d+1}2)$.
Then for all $\sigma>0$ we have
\begin{align} \label{GENBETAUPPERBOUND1}
F_{c,\vecalf,\vecbeta}(0,\sigma)\geq 1-v_d(1-c)\sigma^{d-1}
\qquad\text{and thus}\qquad
\sum_{r=1}^\infty F_{c,\vecalf,\vecbeta}(r,\sigma)\leq v_d(1-c)\sigma^{d-1}.
\end{align}
Furthermore, there exists a constant $C>0$ which only depends on
$r,d$ (thus $C$ is independent of $c$, $\vecalf$, $\vecbeta$, $\lambda$) 
such that for all $\sigma>0$ we have
\begin{align} \label{GENBETAUPPERBOUND2}
F_{c,\vecalf,\vecbeta}(r,\sigma) \leq C (1-c)^{-1}\sigma^{1-d}.
\end{align}
\end{remark}

\subsection{Spheres centered at visible lattice points}

Now assume $\vecalf=\frac{\vecp}q\in\Q^d$ and set
\begin{equation}
	\widehat\scrN_{c,T}(\rho,\vecv,\vecw) := \#\bigg\{ \vecy \in (\widehat\scrL_\vecalf \cap \scrB^d_T(c))-\rho \vecw \col 
	\frac{\vecy}{\|\vecy\|}   \in \fD(\|\vecy\|^{-1}\rho,\vecv) \bigg\} .
\end{equation}

\begin{thm}\label{prim-visThm2}
Let $\lambda$ be a Borel probability measure on $\S_1^{d-1}$ absolutely continuous with respect to Lebesgue measure. Then, for every $\sigma\geq 0$ and $r\in\ZZ_{\geq 0}$, the limit 
\begin{equation}
	\widehat F_{c,\vecalf,\vecbeta}(r,\sigma):=\lim_{T\to\infty} \lambda(\{ \vecv\in\S_1^{d-1} : \widehat\scrN_{c,T}(\sigma T^{-1/(d-1)},\vecv,\vecbeta(\vecv))=r \})
\end{equation}
exists, and for fixed $\vecalf,\vecbeta,\lambda,r$ 
the convergence is uniform with respect to $\sigma$ in any 
compact subset of $\RR_{\geq 0}$, and with respect to $c\in [0,1]$.
The limiting function is given by
\begin{equation}
\widehat F_{c,\vecalf,\vecbeta}(r,\sigma)=
(\mu_q\times\lambda)(\{ (M,\vecv)\in X_q\times\S_1^{d-1} \col
\#( \widehat\ZZ_\vecalf^d M \cap \fZ_\vecv(c,\sigma))= r \}) .
\end{equation}
In particular, $\widehat F_{c,\vecalf,\vecbeta}(r,\sigma)$ is continuous in $\sigma$ and independent of $\scrL$.
\end{thm}

\begin{remark}
The function $\widehat F_{0,\vecalf,\vecbeta}(r,\sigma)$ 
is $\C^1$ with respect to $\sigma>0$.
This is proved by adapting the arguments of Sections \ref{SUBMSECTION},
\ref{IMPORTANTVOLUMESEC} and \ref{subsecC1}
to the setting of visible lattice points.
\end{remark}

\subsection{Non-spherical objects}

Instead of balls we now consider more general objects 
\begin{equation} \label{QTDEF}
	\scrQ_T = T^{-1/(d-1)} \scrQ= \{ \vecx\in\RR^d \col T^{1/(d-1)} \vecx\in\scrQ\}
\end{equation}
where $\scrQ$ is a bounded open subset of $\RR^d$ which satisfies the technical
condition that, for Lebesgue-almost every $\vecv\in\S_1^{d-1}$, the subset
$\text{Proj}_{\{\vecv\}^\perp}\scrQ\subset\{\vecv\}^\perp$ has 
boundary of ($(d-1)$-dimensional) volume measure zero. 
This assumption is readily verified to
hold for any ``nice'' set $\scrQ$; for instance it certainly holds whenever
$\scrQ$ is convex, 
but also for much more general sets $\scrQ$.

As before we place translates of $\scrQ$ at lattice points, 
and consider the set
\begin{equation}\label{BBB2}
	\{ \vecx\in \scrQ_T + \vecy\col \vecy\in\Lalf\cap \scrB^d_T(c)\setminus\{\vecnull\} \}.
\end{equation} 
The number of intersections with a ray starting at the origin in direction $\vecv$ is
\begin{equation}\label{asin2}
	\scrN_{c,T}(\scrQ,\vecv):=\#\big\{ \vecy\in\Lalf\cap \scrB^d_T(c)\setminus\{\vecnull\}\col 
	\R_{>0}\vecv \cap (\scrQ_T + \vecy) \neq \emptyset \big\} .
\end{equation}

\begin{thm}\label{visThm3}
Let $\lambda$ be a Borel probability measure on $\S_1^{d-1}$ absolutely continuous with respect to $\vol_{\S_1^{d-1}}$. 
Then, for every $r\in\ZZ_{\geq 0}$, the limit 
\begin{equation} \label{visThm3limit}
	F_{c,\vecalf}(r,\scrQ):=\lim_{T\to\infty} \lambda(\{ \vecv\in\S_1^{d-1} \col  \scrN_{c,T}(\scrQ,\vecv)=r \})
\end{equation}
exists, and is given by
\begin{equation}
	\begin{cases}
	(\lambda\times\mu_1)(\{ (\vecv,M)\in \S_1^{d-1}\times X_1\col  \#( \ZZ_*^d M \cap \fZ(c,\scrQ,\vecv)) = r \}) & \text{if $\vecalf\in\ZZ^d$}\\
	(\lambda\times\mu_q)(\{ (\vecv,M) \in \S_1^{d-1}\times X_q\col  \#( (\ZZ^d+\frac{\vecp}{q}) M \cap \fZ(c,\scrQ,\vecv))= r \}) & \text{if $\vecalf=\frac{\vecp}{q}\in\QQ^d\setminus\ZZ^d $}\\
	(\lambda\times\mu)(\{ (\vecv,g) \in \S_1^{d-1}\times X\col  \#(\ZZ^d g \cap \fZ(c,\scrQ,\vecv))= r \}) & \text{if $\vecalf\notin\QQ^d$,}
	\end{cases}
\end{equation}
where
\begin{equation}
	\fZ(c,\scrQ,\vecv) =\big\{\vecx \in\RR^d \col  c < \vecx\cdot\vecv < 1,
\: \RR\vecv \cap (\scrQ+\vecx)\neq\emptyset  \big\} .
\end{equation}
In particular $F_{c,\vecalf}(r,\scrQ)$ is independent of $\scrL$.
\end{thm}

The analogous statement holds for visible lattice points. 
Assume $\vecalf=\frac{\vecp}q\in\Q^d$ and set
\begin{equation}
	\widehat\scrN_{c,T}(\scrQ,\vecv):=\#\big\{ \vecy\in\widehat\scrL_\vecalf\cap \scrB^d_T(c)\col  
	\R_{>0}\vecv \cap (\scrQ_T + \vecy) \neq \emptyset \big\} .
\end{equation}

\begin{thm}\label{prim-visThm3}
Let $\lambda$ be a Borel probability measure on $\S_1^{d-1}$ absolutely continuous with respect to Haar measure. Then, for every $\sigma>0$ and $r\in\ZZ_{\geq 0}$, the limit 
\begin{equation}
	\widehat F_{c,\vecalf}(r,\scrQ):=\lim_{T\to\infty} \lambda(\{ \vecv\in\S_1^{d-1} \col  \widehat \scrN_{c,T}(\scrQ,\vecv)=r \})
\end{equation}
exists, and is given by
\begin{equation}
	(\lambda\times\mu_q)(\{ (\vecv,M)\in \S_1^{d-1}\times X_q\col  \#( \widehat\ZZ_\vecalf^d M \cap \fZ(c,\scrQ,\vecv)) = r \}) .
\end{equation}
In particular, $\widehat F_{c,\vecalf}(r,\scrQ)$ is %
independent of $\scrL$.
\end{thm}

All statements in this section are proved in Section \ref{secProofs}.

\section{The periodic Lorentz gas \label{secLorentz}}

We now show how the results of the previous Section \ref{secVisible} can be applied to the distribution of free path lengths (Section \ref{secFree}). We will then generalize these results to provide joint distributions of free path lengths and exact location of impact on the scatterer (Section \ref{firstcoll}), and the distribution of the velocity vector after the first hit (Section \ref{secVelo}).

\subsection{Free path lengths}\label{secFree}
 
Recall that the free path length for the initial condition $(\vecq,\vecv)\in\T^1(\scrK_\rho)$ is defined as  
\begin{equation} \label{TAU1DEF}
	\tau_1(\vecq,\vecv;\rho) = \inf\{ t>0 \col  \vecq+t\vecv \notin\scrK_\rho \}. 
\end{equation}

The crucial observation is that if $\lambda$ is any given probability measure 
on $\S_1^{d-1}$ and $0<\rho<T$, $(\vecq,\vecv)\in\T^1(\scrK_\rho)$, 
then we have
\begin{multline} \label{CRUCIALINEQ}
\lambda(\{ \vecv\in\S_1^{d-1} \col  \scrN_{0,T+\rho}(\rho,\vecv)=0 \}) \\
\leq
\lambda(\{ \vecv\in\S_1^{d-1} \col  \tau_1(\vecq,\vecv;\rho)\geq T \}) \\
\leq \lambda(\{ \vecv\in\S_1^{d-1} \col  \scrN_{0,T-\rho}(\rho,\vecv)=0 \}),
\end{multline}
where $\scrN_{0,T}$ is as defined in \eqref{asdef} with affine lattice 
$\Lalf=\scrL-\vecq$ (thus $\vecalf\equiv-\vecq M_0^{-1}\bmod\ZZ^d$).

Let
\begin{equation} \label{P1}
\Phi_\vecalf(\xi) = -\frac{d}{d\xi} F_{0,\vecalf}(0,\xi^{1/(d-1)}) .
\end{equation}
This defines a continuous probability density on $\R_{>0}$
(cf.\ Remark \ref{visThmrem}).
If $\vecalf\notin\Q^d$ then $\Phi_\vecalf(\xi)$ is independent
of $\vecalf$ and we write $\Phi(\xi)$ for this function 
(as in \eqref{uniPhi}).

The following is a restatement of Theorem \ref{freeThm1}.

\begin{cor} \label{freeCor1}
Fix a lattice $\scrL=\Z^d M_0$.
Let $\vecq\in\RR^d\setminus\scrL$ 
and $\vecalf=-\vecq M_0^{-1}$,
and let $\lambda$ be a Borel probability measure on $\S_1^{d-1}$ absolutely continuous with respect to Lebesgue measure. Then, for every $\xi\geq 0$,
\begin{equation}
\lim_{\rho\to 0} \lambda(\{ \vecv\in\S_1^{d-1} \col  \rho^{d-1} 
\tau_1(\vecq,\vecv;\rho)\geq \xi \})
= \int_\xi^\infty \Phi_\vecalf(\xi') d\xi' .
\end{equation}
\end{cor}

Note here that the condition $\vecq\notin\scrL$ is ensures that 
$\tau_1(\vecq,\vecv;\rho)$ is defined for all sufficiently small $\rho$.
Corollary \ref{freeCor1} follows directly from \eqref{CRUCIALINEQ} and
Theorem \ref{visThm}; cf.\ the proof of Corollary~\ref{freeCor2} below.

The analogous result corresponding to the set-up of Section \ref{secVar} is as follows.
As in that section we let $\vecbeta:\S_1^{d-1}\to\RR^d$ be a continuous
function, and again
let $\lambda$ be a Borel probability measure on $\S_1^{d-1}$ absolutely continuous with respect to Lebesgue measure.

If $\vecq\in\scrL$, it is possible that the trajectory $\vecq+\rho\vecbeta(\vecv)+\R_{>0}\vecv$ starts inside the scatterer (if $\|\vecbeta(\vecv)\|<1$), or will hit the scatterer at $\vecq$ (if $\|\vecbeta(\vecv)\|\geq 1$ and $\vecv$ is suitably chosen). In the first case the corresponding free path length is undefined; in the second case $\tau_1(\vecq+\rho\vecbeta(\vecv),\vecv;\rho)=O(\rho)$. The measure of directions with short free path lengths,
\begin{equation}
	\lambda\bigl(\bigl\{ \vecv\in\S_1^{d-1}\col  \tau_1(\vecq+\rho\vecbeta(\vecv),\vecv;\rho)\leq 
\sfrac 12 m(\Lalf) \bigr\}\bigr)
\end{equation}
is independent of $\rho$, for $\rho$ sufficiently small.

In order to avoid these pathological cases we will from now on assume that $\vecbeta$ is such that \textit{if $\vecq\in\scrL$, then the ray $\vecbeta(\vecv)+\R_{>0}\vecv$ lies completely outside 
$\scrB_1^d$, for each $\vecv\in\S_1^{d-1}$}.
This assumption will be in force throughout the remainder of 
Section \ref{secLorentz}.

Set
\begin{equation} \label{PALFBETDEF}
\Phi_{\vecalf,\vecbeta}(\xi) =  -\frac{d}{d\xi} F_{0,\vecalf,\vecbeta}(0,\xi^{1/(d-1)}) ,
\end{equation}
which, unlike $\Phi_\vecalf$, depends on the choice of the measure $\lambda$; 
cf.\ \eqref{defF}.
The function $\Phi_{\vecalf,\vecbeta}(\xi)$ again defines a continuous 
probability density on $\RR_{>0}$, see Remark \ref{visThm2rem}.

\begin{cor}\label{freeCor2}
For every $\xi\geq 0$,
\begin{equation} \label{freeCor2rel}
\lim_{\rho\to 0} \lambda(\{ \vecv\in\S_1^{d-1} \col  \rho^{d-1} \tau_1(\vecq+\rho\vecbeta(\vecv),\vecv;\rho)\geq \xi \})
= \int_\xi^\infty \Phi_{\vecalf,\vecbeta}(\xi') d\xi' .
\end{equation}
\end{cor}

In this statement, $\tau_1(\vecq+\rho\vecbeta(\vecv),\vecv;\rho)$ 
is well-defined for all $\vecv\in\S^{d-1}_1$ so long as $\rho$ is 
sufficiently small. (For if $\vecq\in\scrL$ then, by our assumptions
on $\vecbeta$, we have in particular 
$\|\vecbeta(\vecv)\|\geq 1$ for all $\vecv$.)

\begin{proof}[Proof of Corollary \ref{freeCor2}]
Set $C=1+\sup_{\S_1^{d-1}} \|\vecbeta\|$.
Generalizing \eqref{CRUCIALINEQ} we note that when $\rho$ is sufficiently
small and $T$ is sufficiently large, we have
\begin{multline} \label{CRUCIALINEQBETA}
\lambda(\{ \vecv\in \S_1^{d-1} \col
\scrN_{0,T+C\rho}(\rho,\vecv,\vecbeta(\vecv))=0 \}) \\
\leq
\lambda(\{ \vecv\in \S_1^{d-1} \col
\tau_1(\vecq+\rho\vecbeta(\vecv),\vecv;\rho)\geq T \}) \\
\leq \lambda(\{ \vecv\in \S_1^{d-1} \col 
\scrN_{0,T-C\rho}(\rho,\vecv,\vecbeta(\vecv))=0 \}),
\end{multline}
where $\scrN_{0,T}$ is as defined in \eqref{asin}
with affine lattice $\Lalf=\scrL-\vecq$
(in \eqref{CRUCIALINEQBETA} we used our assumption that if $\vecq\in\scrL$ then
$(\vecbeta(\vecv)+\R_{>0}\vecv)\cap\scrB_1^d=\emptyset$
for all $\vecv\in\S_1^{d-1}$). 
In particular, 
writing $T_1=\xi\rho^{1-d}+C\rho$ and $\sigma(\rho)=T_1^{\frac 1{d-1}}\rho$
we have, for any $\rho>0$ sufficiently small,
\begin{multline}
\lambda(\{  \vecv\in \S_1^{d-1} \col 
\rho^{d-1} \tau_1(\vecq+\rho\vecbeta(\vecv),\vecv;\rho)\geq \xi \})
\\
\geq \lambda(\{  \vecv\in \S_1^{d-1} \col  
\scrN_{0,T_1}(\sigma(\rho) T_1^{-\frac 1{d-1}},
\vecv,\vecbeta(\vecv))=0\}).
\end{multline}
But $T_1\to\infty$ and $\sigma(\rho)\to\xi^{1/(d-1)}$ as 
$\rho\to 0^+$; hence by Theorem \ref{visThm2} the right hand side
above tends to $F_{0,\vecalf,\vecbeta}(0,\xi^{1/(d-1)})$.
This equals $\int_\xi^\infty \Phi_{\vecalf,\vecbeta}(\xi') \, d\xi'$,
because of \eqref{PALFBETDEF} and 
$\lim_{\sigma\to\infty}F_{0,\vecalf,\vecbeta}(0,\sigma)=0$
(see Remark \ref{GENBETAUPPERBOUNDREM}).
Hence we have proved
\begin{equation}
\liminf_{\rho\to 0}
\lambda(\{  \vecv\in \S_1^{d-1} \col \rho^{d-1} \tau_1(\vecq+\rho\vecbeta(\vecv),\vecv;\rho)\geq \xi \})
\geq \int_\xi^\infty \Phi_{\vecalf,\vecbeta}(\xi') \, d\xi'.
\end{equation}
But using the last inequality in \eqref{CRUCIALINEQBETA}
we obtain the same upper bound for the corresponding $\limsup$, 
and hence \eqref{freeCor2rel} is proved.
\end{proof}

\begin{remark}
When $\scrL=\Z^2$, $\vecq=\bn$, $\vecbeta(\vecv)=\vecv$ (say) and 
$\lambda=\text{uniform measure on }\S_1^1$, 
Corollary~\ref{freeCor2} specializes to the limit result proved in
Boca, Gologan and Zaharescu \cite{Boca03}.
Similarly for $\scrL=\Z^2$, Theorem \ref{freeThm1cor}
(which is basically a $\vecq$-averaged version of Corollary \ref{freeCor1};
cf.\ also Corollary~\ref{freeCor1-ave} below)
specializes to the limit result proved in Boca and Zaharescu \cite{Boca07}.
The known explicit formulas 
for the volumes $F_{0,\bn}(0,\sigma)$ and $F_{0}(0,\sigma)$ in 
\eqref{FCALFRSIGMDEF} in the case $d=2$ 
(cf.\ \cite{SV} and Remark \ref{visThmremexpl})
indeed agree, via \eqref{P1} and \eqref{PALFBETDEF}, with 
the limit formulas obtained in \cite{Boca03} and \cite{Boca07}
using methods of analytic number theory.
\end{remark}

Analogous results are valid for non-spherical scatterers, as direct corollaries of Theorem \ref{visThm3}.

\subsection{Location of the first collision}\label{firstcoll}

The position of the particle when hitting the first scatterer is
\begin{equation}
	\vecq_1(\vecq,\vecv;\rho) := \vecq+\tau_1(\vecq,\vecv;\rho) \vecv .
\end{equation}
We are now interested in the joint distribution of the free path length (considered in the previous section), and the precise location {\em on} the scatterer where the particle hits.

By definition there is a unique $\vecm\in\scrL$ such that
$\vecq_1(\vecq,\vecv;\rho)\in \S_\rho^{d-1} + \vecm$;
hence there is a unique point
$\vecw_1=\vecw_1(\vecq,\vecv;\rho)\in \S_1^{d-1}$ such that
$\vecq_1(\vecq,\vecv;\rho)=\rho \vecw_1+\vecm$.
Let us fix a map $K:\S_1^{d-1}\to\SO(d)$ such that
$\vecv K(\vecv)=\vece_1$ for all $\vecv\in\S_1^{d-1}$;
we assume that $K$ is smooth when restricted to $\S_1^{d-1}$ minus
one point.\footnote{For example, we may choose $K$ as 
$K(\vece_1)=I$, $K(-\vece_{1})=-I$ and
$K(\vecv)=
E\Bigl(-\frac{2\arcsin\bigl(\|\vecv-\vece_1\|/2\bigr)}
{\|\vecv_\perp\|} \vecv_\perp\Bigr)$ for
$\vecv\in\S_1^{d-1}\setminus\{\vece_1,-\vece_1\}$,
where $\vecv_\perp:=(v_2,\ldots,v_d)\in\R^{d-1}$ 
and $E(\vecw)=\exp\matr 0\vecw{-\trans\vecw}\bn\in\SO(d)$.
Then $K$ is smooth when restricted to $\S_1^{d-1}\setminus\{-\vece_1\}$.}
It is evident that $-\vecw_1 K(\vecv)\in \HS$, with the hemisphere $\HS=\{\vecv=(v_1,\ldots,v_d)\in\S_1^{d-1} \col v_1>0\}$. \label{HS}

Recall that we are assuming that $\vecbeta$ is a continuous function 
$\S_1^{d-1}\to\RR^d$ such that if $\vecq\in\scrL$ then
$(\vecbeta(\vecv)+\R_{>0}\vecv)\cap\scrB_1^d=\emptyset$
for all $\vecv\in\S_1^{d-1}$.
We will use the shorthand $\vecq_{\rho,\vecbeta}(\vecv)=\vecq+\rho\vecbeta(\vecv)$ for the initial position.\label{inpo}
For the statement of the theorem below, we define the following
submanifolds of $X_q$ and $X$, respectively:
\begin{align}
& X_q(\vecy):=\bigl\{M\in X_q \col \vecy\in (\Z^d+\vecalf)M\bigr\}
& & \text{(for $\vecy\in\R^d\setminus\{\bn\}$ and fixed 
$\vecalf\in q^{-1}\Z^d$);}
\\ \notag
& X(\vecy):=\bigl\{g\in X \col \vecy\in \Z^d g\bigr\}
& & \text{(for $\vecy\in\R^d$)}.
\end{align}
These submanifolds will be studied in Section \ref{FOLIATIONSEC},
where we will introduce a natural Borel probability measure
$\nu_\vecy$ on each of them.

We will also use the notation $\vecx_\perp=\vecx-(\vecx\cdot\vece_1)\vece_1$
for $\vecx\in\R^d$.\label{perpos}

\begin{thm}\label{exactpos1}
Fix a lattice $\scrL=\Z^d M_0$.
Let $\vecq\in\R^d$ and $\vecalf=-\vecq M_0^{-1}$.
There exists a function
$\Phi_{\vecalf}:\R_{>0}\times(\{0\}\times\scrB_1^{d-1})\times(\{0\}\times\R^{d-1})\to\R_{\geq 0}$ such that for any Borel probability measure 
$\lambda$ on $\S_1^{d-1}$ absolutely 
continuous with respect to $\vol_{\S_1^{d-1}}$, any
subset $\fU\subset\HS$ with $\vol_{\S_1^{d-1}}(\partial\fU)=0$, 
and any $0\leq \xi_1<\xi_2$, we have
\begin{multline} \label{exactpos1eq}
\lim_{\rho\to 0}  \lambda\bigl(\bigl\{ \vecv\in\S_1^{d-1} \col 
\rho^{d-1} \tau_1(\vecq_{\rho,\vecbeta}(\vecv),\vecv;\rho)\in [\xi_1,\xi_2), \:  
-\vecw_1(\vecq_{\rho,\vecbeta}(\vecv),\vecv;\rho)K(\vecv)\in\fU\bigr\}\bigr) \\
=\int_{\xi_1}^{\xi_2} \int_{\fU_\perp} \int_{\S_1^{d-1}} 
\Phi_{\vecalf}\bigl(\xi,\vecw,(\vecbeta(\vecv)K(\vecv))_\perp\bigr) 
\, d\lambda(\vecv) d\vecw \, d\xi,
\end{multline}
where $d\vecw$ denotes the $(d-1)$-dimensional
Lebesgue volume measure on $\{0\}\times\R^{d-1}$.
The function $\Phi_\vecalf$ is explicitly given by
\begin{align} \label{exactpos1limitrat}
\Phi_\vecalf(\xi,\vecw,\vecz)
=\begin{cases}
\nu_\vecy\bigl(\bigl\{M\in X_q(\vecy) \col 
(\Z^d+\vecalf)M \cap (\fZ(0,\xi,1)+\vecz)=\emptyset\bigr\}\bigr)
& \text{if } \: \vecalf\in q^{-1}\Z^d
\\
\nu_\vecy\bigl(\bigl\{g\in X(\vecy) \col \Z^d g \cap (\fZ(0,\xi,1)+\vecz)=\emptyset
\bigr\}\bigr)
& \text{if } \: \vecalf\notin \Q^d,
\end{cases}
\end{align}
where $\vecy=\xi\vece_1+\vecw+\vecz$, and 
\begin{align} \label{FZC1C2DEF}
	\fZ(c_1,c_2,\sigma) =\big\{(x_1,\ldots,x_d)\in\RR^d \col  c_1 < x_1 < c_2, \|(x_2,\ldots,x_d)\|< \sigma \big\} .
\end{align}
\end{thm}

\begin{remark} \label{PALFSYMMREMARK}
Note that $\Phi_\vecalf(\xi,\vecw,\vecz)$ is independent of $\vecbeta$.
For $\vecalf\in \Q^d$
the function $\Phi_\vecalf(\xi,\vecw,\vecz)$ is Borel measurable,
and in fact only depends on ($\vecalf$ and)
the four real numbers $\xi,\|\vecz\|,\|\vecw\|,\vecz\cdot\vecw$.
Also for $\vecalf\in \Q^d$,
if we restrict to $\|\vecz\|\leq 1$ [and if $d=2$: $\vecz+\vecw\neq \bn$],
then $\Phi_\vecalf(\xi,\vecw,\vecz)$ is jointly continuous
in the three variables $\xi,\vecw,\vecz$.
If $\vecalf\notin\Q^d$ then $\Phi_\vecalf(\xi,\vecw,\vecz)$
is everywhere continuous in the three variables,
and it is independent of both $\vecalf$ and $\vecz$;
in fact it only depends on $\xi$ and $\|\vecw\|$.
All these statements will be proved in Sections \ref{IMPORTANTVOLUMESEC}
and \ref{IMPORTANTVOLUMEIRRSEC}.
In particular, if $\vecalf\notin\Q^d$ then
the limit in \eqref{exactpos1eq} is independent of 
$\vecalf$, $\vecbeta$, $\lambda$.
\end{remark}

\begin{remark} \label{PALFBETSPECREM}
It follows from \eqref{exactpos1eq} that
\begin{align} \label{PALFINTEQ1}
\int_0^\infty \int_{\{0\}\times\scrB_1^{d-1}} 
\Phi_\vecalf(\xi,\vecw,\vecz)\,d\vecw\,d\xi=1
\end{align}
holds for almost all $\vecz\in\{0\}\times\R^{d-1}$,
and from \eqref{exactpos1eq} and Corollary \ref{freeCor2} that
\begin{equation} \label{PALFBETSPECREMFORMULA}
\int_{\{0\}\times\scrB_1^{d-1}} \int_{\S_1^{d-1}}
\Phi_{\vecalf}\bigl(\xi,\vecw,(\vecbeta(\vecv)K(\vecv))_\perp\bigr) 
\, d\lambda(\vecv) d\vecw
=\Phi_{\vecalf,\vecbeta}(\xi)
\end{equation}
holds for almost all $\xi>0$.
As a consistency check we derive in Section \ref{subsecC1}
(see Remark \ref{AFTERGENC1PROP}) the relations
\eqref{PALFINTEQ1} and \eqref{PALFBETSPECREMFORMULA} directly from the
explicit formula \eqref{exactpos1limitrat}. In fact it turns out that
\eqref{PALFINTEQ1} holds for \textit{all} $\vecz\in\{0\}\times\R^{d-1}$
and \eqref{PALFBETSPECREMFORMULA} holds for \textit{all} $\xi>0$.
\end{remark}

As a preparation for Theorem \ref{exactpos2-1hit} below and for the results
in \cite{partII},
we also state a version of Theorem \ref{exactpos1} involving an arbitrary
continuous test function.

\begin{cor}\label{exactpos2}
Let $\lambda$ be a Borel probability measure $\lambda$ on $\S_1^{d-1}$ absolutely 
continuous with respect to $\vol_{\S_1^{d-1}}$. For any bounded continuous function $f:\S_1^{d-1}\times \R_{>0} \times \S_1^{d-1} \to \RR$,
\begin{multline} \label{exactpos2eq}
\lim_{\rho\to 0}  \int_{\S_1^{d-1}} f\big(\vecv, \rho^{d-1} \tau_1(\vecq_{\rho,\vecbeta}(\vecv),\vecv;\rho), 
\vecw_1(\vecq_{\rho,\vecbeta}(\vecv),\vecv;\rho)\big) d\lambda(\vecv) \\
=\int_{\HS} \int_{\R_{>0}} \int_{ S_1^{d-1}}
f\bigl(\vecv,\xi,-\vecomega K(\vecv)^{-1}\bigr) \,
\Phi_\vecalf\bigl(\xi,\vecomega_\perp,(\vecbeta(\vecv)K(\vecv))_\perp\bigr)\,
\omega_1 \, d\lambda(\vecv) \, d\xi \, d\!\vol_{\S^{d-1}_1}(\vecomega),
\end{multline}
where $\vecomega=(\omega_1,\ldots,\omega_d)$.
\end{cor}
\begin{proof}
For $f$ with compact support the result follows in a standard way by 
approximating $f$ from above and below by linear combinations of 
characteristic functions and applying Theorem~\ref{exactpos1}.
When extending to arbitrary bounded continuous functions $f$ one uses
\eqref{PALFBETSPECREMFORMULA}, \eqref{PALFBETDEF} and
Remark \ref{GENBETAUPPERBOUNDREM}.
\end{proof}

\subsection{Velocity after the first collision}\label{secVelo}

If a particle moving with velocity $\vecv_0$ hits a spherical scatterer at the point $\vecq_1$ and is elastically reflected, its velocity changes to
\begin{equation} \label{VECV1DEF}
	\vecv_1 = \vecv_0 - 2 (\vecv_0\cdot\vecw_1) \vecw_1 ,
\end{equation}
where $\vecw_1\in\S_1^{d-1}$ is the location of the hit relative to the center of the sphere, as defined in Section \ref{firstcoll}.
This implies
\begin{equation}
	\vecw_1 = \frac{\vecv_1-\vecv_0}{\|\vecv_1-\vecv_0\|} .
\end{equation}

\begin{thm}\label{exactpos2-1hit}
Let $\lambda$ be a Borel probability measure on $\S_1^{d-1}$ absolutely 
continuous with respect to $\vol_{\S_1^{d-1}}$. For any bounded continuous function $f:\S_1^{d-1}\times \R_{>0} \times \S_1^{d-1} \to \RR$,
\begin{multline} \label{exactpos2eq-1hit}
\lim_{\rho\to 0}  \int_{\S_1^{d-1}} f\big(\vecv_0, \rho^{d-1} \tau_1(\vecq_{\rho,\vecbeta}(\vecv_0),\vecv_0;\rho), 
\vecv_1(\vecq_{\rho,\vecbeta}(\vecv_0),\vecv_0;\rho)\big) d\lambda(\vecv_0) \\
=\int_{\S_1^{d-1}} \int_{\R_{>0}} \int_{\S_1^{d-1}} f\big(\vecv_0,\xi,\vecv_1\big) 
p_{\vecalf,\vecbeta}(\vecv_0,\xi,\vecv_1) \, d\lambda(\vecv_0)\, d\xi 
d\!\vol_{\S_1^{d-1}}(\vecv_1) ,
\end{multline}
with the probability density $p_{\vecalf,\vecbeta}$ defined by
\begin{equation} \label{exactpos2-1hit-tpdef}
p_{\vecalf,\vecbeta}(\vecv_0,\xi,\vecv_1)\,d\!\vol_{\S_1^{d-1}}(\vecv_1)
=\Phi_\vecalf\bigl(\xi,\vecomega_\perp,(\vecbeta(\vecv_0)K(\vecv_0))_\perp\bigr)\,
\omega_1 \, d\!\vol_{\S^{d-1}_1}(\vecomega)
\end{equation}
where
\begin{equation} \label{V1OMEGASUBST}
\vecv_1=(\vece_1 - 2 (\vece_1\cdot\vecomega) \vecomega) K(\vecv_0)^{-1},
\qquad \vecomega\in\HS.
\end{equation}
\end{thm}
\begin{remark} \label{PTILDEEXPLICITREM}
The relationship between $p_{\vecalf,\vecbeta}(\vecv_0,\xi,\vecv_1)$
and $\Phi_\vecalf(\xi,\vecw,\vecz)$ can be expressed more explicitly as
\begin{align} \label{PTILDEEXPLICIT}
p_{\vecalf,\vecbeta}(\vecv_0,\xi,\vecv_1)
=
\frac 14 \,\|\vecv_1-\vecv_0\|^{3-d}\,
\Phi_\vecalf\Bigl(\xi,-\frac{(\vecv_1 K(\vecv_0))_\perp}{\|\vecv_1-\vecv_0\|},
(\vecbeta(\vecv_0)K(\vecv_0))_\perp\Bigr).
\end{align}
The function $p_{\vecalf,\vecbeta}(\vecv_0,\xi,\vecv_1)$ is independent 
of the choice of the function $K:\S_1^{d-1}\to\SO(d)$, since 
$\Phi_\vecalf(\xi,\vecw,\vecz)$ only depends on the four real numbers 
$\xi$, $\|\vecw\|$, $\|\vecz\|$, $\vecw\cdot\vecz$
(cf.\ Remark \ref{PALFSYMMREMARK}), which in \eqref{PTILDEEXPLICIT} 
can be expressed as
$\xi, \frac{\sqrt{1-(\vecv_0\cdot \vecv_1)^2}}{\|\vecv_1-\vecv_0\|},
\sqrt{1-(\vecbeta(\vecv_0)\cdot \vecv_0)^2},
\frac{(\vecv_1\cdot \vecv_0)(\vecbeta(\vecv_0)\cdot \vecv_0)
-\vecv_1\cdot\vecbeta(\vecv_0)}{\|\vecv_1-\vecv_0\|}$, respectively.
\end{remark}

\section{Equidistribution in homogeneous spaces} \label{secEqui}

This section provides the ergodic-theoretic results, which are the key ingredients in the proofs of the main theorems. These equidistribution theorems are consequences of Ratner's classification of measures that are invariant under the action of a unipotent flow \cite{Ratner91}, and may in particular be viewed as variants of Shah's %
Theorem 1.4 in \cite{Shah96}.

\subsection{Translates of expanding unipotent orbits} \label{transexporbsec}

The following is a special case of Shah's Theorem 1.4 in \cite{Shah96}.
Let $G$ be a connected Lie group and let $\Gamma$ be a lattice in $G$.

\begin{thm}\label{thmShah}
Suppose $G$ contains a Lie subgroup $H$ isomorphic to $\SLR$ (we denote the corresponding embedding by $\varphi:\SLR\to G$), such that the set $\Gamma\backslash\Gamma H$ is dense in $\GamG$. Let $\lambda$ be a Borel probability measure on $\RR^{d-1}$ %
which is absolutely continuous with respect to Lebesgue measure, and let $f:\GamG\to\RR$ be bounded continuous.
Then
\begin{equation}
	\lim_{t\to\infty} \int_{\RR} f\left(\varphi\left(
	\begin{pmatrix} 1 & \vecx \\ \trans\vecnull & 1_{d-1} \end{pmatrix}
	\begin{pmatrix} \e^{-(d-1) t} & \vecnull \\ \trans\vecnull & \e^{t}1_{d-1} \end{pmatrix}
	\right)\right) d\lambda(\vecx)
	= \int_{\GamG} f \, d\mu,
\end{equation}
where $\mu$ is the unique $G$-right-invariant probability measure on $\GamG$.
\end{thm}

Let us set
\begin{equation} \label{NMINUSDEF}
	n_-(\vecx)=\left(\begin{pmatrix} 1 & \vecx \\ \trans\vecnull & 1_{d-1} \end{pmatrix},\vecnull\right) \in \ASLR
\end{equation}
and
\begin{equation} \label{PHITDEF}
	\Phi^t = \left(\begin{pmatrix} \e^{-(d-1) t} & \vecnull \\ \trans\vecnull & \e^{t}1_{d-1} \end{pmatrix},\vecnull\right) \in \ASLR.
\end{equation}

Theorem \ref{thmShah} implies the following.
\begin{thm}\label{equiThm0}
Let $\lambda$ be a Borel probability measure on $\RR^{d-1}$ 
which is absolutely continuous with respect to Lebesgue measure, and let $f:X\to\RR$ be bounded continuous. Then, for every $\vecalf\in\RR^d\setminus\QQ^d$ and every $M\in\SL(d,\RR)$
\begin{equation}
	\lim_{t\to\infty} \int_{\RR^{d-1}} f\bigl((1_d,\vecalf)(M,\vecnull)n_-(\vecx) \Phi^t\bigr)\, d\lambda(\vecx) = \int_{X} f(g) \, d\mu(g) .
\end{equation}
\end{thm}

\begin{proof}
Let $G=\ASLR$, $\Gamma=\ASLZ$ and define the embedding
\begin{equation}
	\varphi: \SLR \to G, \quad \tilde M \mapsto (1_d,\vecalf) (M\tilde MM^{-1},\vecnull) (1_d,-\vecalf).
\end{equation}
We now wish to establish that $\Gamma\backslash\Gamma H$ with $H=\varphi(\SLR)$ is dense in $\GamG$. To this end it suffices to show that 
\begin{equation}
(\gamma,\vecm)(1_d,\vecalf) (M \tilde M,\vecnull)=(\gamma M \tilde M , (\vecalf+\vecm) M \tilde M)	
\end{equation}
are dense in $\ASLR$, as $\gamma$, $\vecm$ and $\tilde M$ vary over 
$\SLZ$, $\ZZ^d$ and $\SLR$, respectively. It is evident that this is 
in turn is equivalent to showing that $\{(\vecalf+\vecm)\gamma^{-1}\}$ is 
dense in $\RR^d$.

Letting $C\subset \RR^d/\ZZ^d$ be the closure of the image 
of $\vecalf\SLZ\subset \RR^d$ under the natural projection
$\RR^d\to\RR^d/\ZZ^d$, our task is to show $C=\RR^d/\ZZ^d$.
Since $\vecalf\notin \QQ^d$ there is a choice of $\gamma\in\SLZ$ 
either a permutation matrix or %
$\smatr 0{-1}10$ which gives
$\vecw=(w_1,\ldots,w_d):=\vecalf \gamma \in C$ with $w_1\notin\QQ$.
Then by choosing $\gamma'=\matr 1\veca{^t\bn}{1_{d-1}}\in\SLZ$
with appropriate $\veca\in\Z^{d-1}$, the point
$\vecw \gamma'$ can be made to lie arbitrarily close to 
$(w_1,0,\ldots,0)$ in $\RR^d/\ZZ^d$.
Hence since $C$ is closed we have $(w_1,0,\ldots,0)\in C$.
Now let $\vecy=(y_1,\ldots,y_d) \in \R^d$ and $\ve>0$ be given.
Then there is $m\in\Z\setminus\{0\}$ such that $\|mw_1-y_1\|<\ve$
(where $\|x\|=\inf_{n\in\Z} |x-n|$ as usual).
Letting $\gamma''$ be any matrix in $\SLZ$ with top left entry $m$
we have $(mw_1,*,\ldots,*)=(w_1,0,\ldots,0)\gamma''\in C$,
and hence since $C$ is right $\SL(d,\Z)$ invariant and $mw_1\notin\Q$,
an argument as above shows $(mw_1,0,\ldots,0)\in C$.
Finally by choosing (again)
$\gamma'''=\matr 1\veca{^t\bn}{1_{d-1}}\in\SLZ$
with appropriate $\veca\in\Z^{d-1}$, the point
$(mw_1,0,\ldots,0)\gamma'''\in C$ can be made to lie
arbitrarily close to $(mw_1,y_2,\ldots,y_d)$.
Since %
$\ve$ is arbitrary and $C$ is closed we obtain $\vecy\in C$.
Hence $C=\R^d/\Z^d$, as desired.

Having established the required density, Theorem \ref{thmShah} implies 
that for any bounded continuous $\tilde f:X\to\R$
\begin{equation}
	\lim_{t\to\infty} \int_{\RR^{d-1}} \tilde f((M,\vecalf M)n_-(\vecx) \Phi^t(M,\vecalf M)^{-1}) d\lambda(\vecx)  = \int_{X} \tilde f(g) d\mu(g) .
\end{equation}
Choosing the test function $\tilde f(g)=f(g (M,\vecalf M))$ completes the proof.
\end{proof}

We now extend Theorem \ref{equiThm0} by considering sequences of
test functions with additional parameter dependence.
\begin{thm}\label{equiThm}
Let $\lambda$ be a Borel probability measure on $\RR^{d-1}$ which is absolutely continuous with respect to Lebesgue measure. Let $f:\RR^{d-1}\times X\to\RR$ be bounded continuous and $f_t:\RR^{d-1}\times X\to\RR$ a family of uniformly bounded (i.e., $|f_t|<K$ for some absolute constant $K$), continuous functions such that $f_t\to f$ as $t\to\infty$, uniformly on compacta.
Then, for every $\vecalf\in\RR^d\setminus\QQ^d$, $M\in\SL(d,\RR)$,
\begin{equation} \label{equiThmformula}
	\lim_{t\to\infty} \int_{\RR^{d-1}} f_t\bigl(\vecx,(1_d,\vecalf)(M,\vecnull)n_-(\vecx) \Phi^t\bigr)\, d\lambda(\vecx) \\ = \int_{\RR^{d-1}\times X} f(\vecx,g) \, d\mu(g)d\lambda(\vecx) .
\end{equation}
\end{thm}

\begin{proof}
Let us first assume that $f_t$ and $f$ have support in the compact set $\scrK\subset\R^{d-1}\times X$. 
Hence the convergence $f_t\to f$ is uniform and all functions are uniformly continuous. Therefore, given $\delta>0$ there exist $\epsilon>0,t_0>0$ such that
\begin{equation}
	f(\vecx_0,g)-\delta \leq f(\vecx,g) \leq f(\vecx_0,g)+\delta 
\end{equation}
and
\begin{equation}
	f(\vecx_0,g)-\delta \leq f_t(\vecx,g) \leq f(\vecx_0,g)+\delta 
\end{equation}
for all $\vecx\in\vecx_0+[0,\epsilon)^{d-1}$, $t> t_0$. Now
\begin{equation}
\begin{split}
	\int_{\RR^{d-1}} & f_t(\vecx,(1_d,\vecalf)(M,\vecnull)n_-(\vecx) \Phi^t) d\lambda(\vecx) \\
	& = \sum_{\veck\in\ZZ^{d-1}} \int_{\epsilon\veck+[0,\epsilon)^{d-1}} f_t(\vecx,(1_d,\vecalf)(M,\vecnull)n_-(\vecx) \Phi^t) d\lambda(\vecx) \\
	& \leq \sum_{\veck\in\ZZ^{d-1}} \int_{\epsilon\veck+[0,\epsilon)^{d-1}} f(\epsilon\veck,(1_d,\vecalf)(M,\vecnull)n_-(\vecx) \Phi^t) d\lambda(\vecx) +\delta . 
\end{split}
\end{equation}
By Theorem \ref{equiThm0},
\begin{equation}
\begin{split}
\lim_{t\to\infty}	\int_{\epsilon\veck+[0,\epsilon)^{d-1}} & f(\epsilon\veck,(1_d,\vecalf)(M,\vecnull)n_-(\vecx) \Phi^t) d\lambda(\vecx) \\
  & =\int_{X} f(\epsilon\veck,g) d\mu(g) 	\int_{\epsilon\veck+[0,\epsilon)^{d-1}} d\lambda(\vecx) \\
  & \leq \int_{X} \int_{\epsilon\veck+[0,\epsilon)^{d-1}} [f(\vecx,g)+\delta] 	 d\lambda(\vecx)d\mu(g) ,
\end{split}
\end{equation}
and so
\begin{equation}
	\limsup_{t\to\infty} \int_{\RR^{d-1}}  f_t(\vecx,(1_d,\vecalf)(M,\vecnull)n_-(\vecx) \Phi^t) d\lambda(\vecx) \\
	\leq  \int_{X} \int_{\RR^{d-1}} f(\vecx,g)	 d\lambda(\vecx) d\mu(g)  + 2\delta .
\end{equation}
An analogous argument shows
\begin{equation}
	\liminf_{t\to\infty} \int_{\RR^{d-1}}  f_t(\vecx,(1_d,\vecalf)(M,\vecnull)n_-(\vecx) \Phi^t) d\lambda(\vecx) \\
	\geq  \int_{X} \int_{\RR^{d-1}} f(\vecx,g)	 d\lambda(\vecx) d\mu(g)  - 2\delta .
\end{equation}
It therefore follows that the limit exists and
\begin{equation}
	\lim_{t\to\infty} \int_{\RR^{d-1}}  f_t(\vecx,(1_d,\vecalf)(M,\vecnull)n_-(\vecx) \Phi^t) d\lambda(\vecx) \\
	=  \int_{X} \int_{\RR^{d-1}} f(\vecx,g)	 d\lambda(\vecx) d\mu(g) .
\end{equation}

We now extend the result to bounded continuous test functions 
$f_t$, uniformly bounded by $|f_t|<K$.
Given $\delta>0$ we choose compact sets $\scrK_1\subset\R^{d-1}$ and
$\scrK_2\subset X$ so large that 
$(1-\lambda(\scrK_1))+(1-\mu(\scrK_2)) \leq\delta/K$.
Let $c_1:\R^{d-1}\to [0,1]$ and $c_2:X\to [0,1]$ be continuous
functions which have compact support and satisfy
$\chi_{_{\scrK_1}}\leq c_1$ and
$\chi_{_{\scrK_2}}\leq c_2$, respectively. Write
\begin{equation}
f_t=f_t^1 + f_t^2,
\qquad\text{with} \quad
f_t^1(\vecx,g)=c_1(\vecx)c_2(g)f_t(\vecx,g),\quad
f_t^2=f_t-f_t^1. 
\end{equation}
Then $f_t^1$ is compactly supported as in the previous paragraph.
For $f_t^2$ we have, using Theorem~\ref{equiThm0},
\begin{align} \notag
&	\limsup_{t\to\infty} \int_{\RR^{d-1}} 
\bigl | f_t^2(\vecx,(1_d,\vecalf)(M,\vecnull)n_-(\vecx) \Phi^t) \bigr |
\, d\lambda(\vecx) 
\\ 
& \leq K(1-\lambda(\scrK_1))
+\limsup_{t\to\infty} \int_{\scrK_1}
K\bigl(1-c_2\bigl((1_d,\vecalf)(M,\vecnull)n_-(\vecx) \Phi^t\bigr)\bigr)
\, d\lambda(\vecx) 
\\ \notag
& = K(1-\lambda(\scrK_1))+ K \int_{X} \bigl(1-c_2(g)\bigr) \, d\mu(g)
\leq K\bigl(1-\lambda(\scrK_1)\bigr)+ K \bigl(1-\mu(\scrK_2)\bigr) 
\leq \delta .
\end{align}
This upper bound shows that the statement of the theorem can be extended from compactly supported  to bounded test functions.
\end{proof}

\subsection{Spherical averages}

We will now show that the statement of Theorem \ref{equiThm} 
(and thus of Theorem \ref{equiThm0}) holds when $n_-(\vecx)$ is replaced by
\begin{equation} \label{KANDEDEF}
	(E(\vecx),\vecnull)=\bigg(\exp\begin{pmatrix} 0 & \vecx \\ -\trans\vecx & 0_{d-1} \end{pmatrix},\vecnull\bigg).
\end{equation}
In fact we can prove a more general fact with almost no extra effort:
\begin{cor}\label{also2}
Let $\Edomain\subset \R^{d-1}$ be an open subset and
let $E_1:\Edomain\to\SO(d)$ be a smooth map such that the map
$\Edomain\ni\vecx\mapsto \vece_1 E_1(\vecx)^{-1}\in\S^{d-1}_1$ 
has non-singular differential at \mbox{(Lebesgue-)almost} all $\vecx\in\Edomain$.
Let $\lambda$ be a Borel probability measure on $\Edomain$, 
absolutely continuous with respect to Lebesgue measure.
We then have, for any bounded continuous function $f:\Edomain\times X\to \R$
and any family of uniformly bounded continuous functions 
$f_t:\Edomain\times X\to \R$ such that $f_t\to f$ as $t\to\infty$,
uniformly on compacta, and for every
$\vecalf\in\R^d\setminus\Q^d$, $M\in\SL(d,\R)$,
\begin{equation} \label{also2identity}
	\lim_{t\to\infty} \int_{\Edomain} f_t(\vecx,(1_d,\vecalf)(M,\vecnull)
(E_1(\vecx),\bn) \Phi^t) \, d\lambda(\vecx) \\ 
	= \int_{\Edomain\times X} f(\vecx,g) \, d\mu(g)d\lambda(\vecx).
\end{equation}
\end{cor}

\begin{remark}
Taking $E_1(\vecx)=E(\vecx)$ as in \eqref{KANDEDEF} is indeed a valid
choice in Corollary \ref{also2},
for note that $\vece_1 E(\vecx)^{-1}=\Bigl(\cos \|\vecx\|, 
-\frac{\sin\|\vecx\|}{\|\vecx\|}\vecx\Bigr)$, and one checks that 
this map has nonsingular differential except when
$\|\vecx\|\in\{\pi,2\pi,3\pi,\ldots\}$.
\end{remark}

\begin{proof}[Proof of Corollary \ref{also2}.]
We first prove that if $\vecx_0\in\Edomain$ is any point where the map
$\vecx\mapsto\vece_1 E_1(\vecx)^{-1}$ has nonsingular differential,
then there is some open neighborhood 
$\Edomain_0\subset\Edomain$ of $\vecx_0$ such that \eqref{also2identity}
holds when $\Edomain$ is replaced 
by $\Edomain_0$ or by any Borel subset of $\Edomain_0$.

To see this, write $E_0=E_1(\vecx_0)$ and
\begin{align}
E_2(\vecx):=E_0^{-1} E_1(\vecx)=\begin{pmatrix} c & \vecw \\ 
\trans\vecv & A \end{pmatrix}
=\begin{pmatrix} c(\vecx) & \vecw(\vecx) \\ 
\trans\vecv(\vecx) & A(\vecx) \end{pmatrix},
\qquad c\in\R,\: \vecv,\vecw\in\R^{d-1}.%
\end{align}
Then $E_2(\vecx_0)=1_d$ and thus $c(\vecx_0)=1$ and $\vecv(\vecx_0)=\bn$.
Furthermore the map
$\vecx\mapsto (c(\vecx),\vecv(\vecx)) \in \S^{d-1}_1$ has
nonsingular differential at $\vecx=\vecx_0$,
since $(c(\vecx),\vecv(\vecx))=\vece_1 E_2(\vecx)^{-1}
=(\vece_1 E_1(\vecx)^{-1})E_0$, and thus also the map
$\vecx \mapsto \tx:=-c(\vecx)^{-1}\vecv(\vecx) \in \R^{d-1}$
must have nonsingular differential at $\vecx=\vecx_0$.
Hence there exists some bounded open neighborhood
$\Edomain_0'$ of $\vecx_0$ with $\overline{\Edomain_0'}\subset \Edomain$ such that
$c(\vecx)>1/2$ for all $\vecx\in\Edomain_0'$ and such that
$\vecx \mapsto \tx$ %
is a diffeomorphism of $\Edomain_0'$ onto a bounded open subset
$\widetilde{\Edomain}_0'\subset\R^{d-1}$.
Now for each $\vecx\in\Edomain_0'$ we have
\begin{equation}
\begin{split}
\Bigl(E_2(\vecx),\bn\Bigr)
&=n_-(\tx) \left(\matr 1{-\tx}{^t\bn}{1_{d-1}}
\matr{c}{\vecw}{\trans\vecv}{A},\bn\right)
=n_-(\tx) \left(\matr{c^{-1}}{\bn}{\trans\vecv}A,\bn\right),
\end{split}
\end{equation}
since writing out %
$\trans E_2(\vecx) E_2(\vecx)=1_d$ one gets the relations
$c\vecw+\vecv A=\bn$ and $c^2+\vecv\trans\vecv=1$,
viz.\ $\vecw-\tx A=\bn$ and $c-\tx\trans\vecv=c^{-1}$.

Hence also
\begin{equation} \label{also2keyid}
\Bigl( E_1(\vecx),\bn \Bigr) \Phi^t = \bigl(E_0,\bn\bigr) 
n_-(\tx) \Phi^t \left( 
\matr{c(\vecx)^{-1}}{\bn}{\trans\vecv(\vecx) e^{-dt}}{A(\vecx)},\bn\right).
\end{equation}

Now fix $\Edomain_0$ as an open neighborhood of $\vecx_0$ such that
$\overline{\Edomain}_0\subset\Edomain_0'$,
and consider any Borel subset $B$ of $\Edomain_0$.
Write $\widetilde{B}\subset \widetilde{\Edomain}_0\subset \widetilde{\Edomain}_0'$ 
for the images of $B$ and $\Edomain_0$ under $\vecx\to\tx$.
Let us assume $\lambda(B)>0$, and
let $\widetilde\lambda$ be the measure on $\R^{d-1}$
which corresponds to $\lambda(B)^{-1}\lambda|_{B}$ under the 
diffeomorphism $\vecx\to\tx$;
then $\widetilde\lambda$ is
a Borel probability measure with bounded support 
and absolutely continuous with respect to Lebesgue measure.
Since $\overline{\widetilde{\Edomain}}_0\subset\widetilde{\Edomain}_0'$,
we may choose a continuous cutoff function $h:\R^{d-1}\to [0,1]$ such that 
$\chi_{\widetilde{\Edomain}_0}\leq h \leq \chi_{\widetilde{\Edomain}_0'}$.

If %
$f_t$ and $f$ are given as in the statement of the 
corollary, we may define a family of continuous 
functions $\widetilde f_t:\R^{d-1}\times X\to\R$  and a continuous function
$\widetilde f:\R^{d-1}\times X\to\R$ through
\begin{align}
& \widetilde f_t(\tx,g) := h(\tx) f_t\left(\vecx,g
\left(\matr{c(\vecx)^{-1}}{\bn}{\trans\vecv(\vecx) e^{-dt}}{A(\vecx)}
,\bn\right)\right) & & \text{if } \:
\tx\in\widetilde{\Edomain}_0';
\\ \notag
& \widetilde f(\tx,g) := h(\tx) f\left(\vecx,g
\left(\matr{c(\vecx)^{-1}}{\bn}{\trans\bn}{A(\vecx)}
,\bn\right)\right) & & \text{if } \:
\tx\in\widetilde{\Edomain}_0';
\\ \notag
& \widetilde f_t(\tx,g)=\widetilde f(\tx,g) := 0 & & \text{if } \:
\tx\notin\widetilde{\Edomain}_0'.
\end{align}
(We here view $\vecx\in\Edomain_0'$ as a %
function of $\tx\in\widetilde{\Edomain}_0'$.)
We then have
$\widetilde f_t(\tx,g) \to \widetilde f(\tx,g)$ as $t\to\infty$,
uniformly on compacta.  %
Applying Theorem \ref{equiThm} for $\widetilde{\lambda}$,
$\widetilde f_t$, $\widetilde f$, and with $M$ replaced by $ME_0$,
we get
\begin{equation} 
\lim_{t\to\infty} \int_{\R^{d-1}} \widetilde 
f_t\Bigl(\tx,(1_d,\vecalf)(ME_0,\bn)
n_-(\tx)\Phi^t\Bigr)\, d\widetilde{\lambda}(\tx)  
= \int_{\R^{d-1}\times X} \widetilde f(\tx,g) 
\, d\mu(g)d\widetilde{\lambda}(\tx).
\end{equation}
Here the left hand side equals,
using $\widetilde{\lambda}=\widetilde{\lambda}|_{\widetilde B}$ 
and \eqref{also2keyid},
\begin{equation}
\begin{split}
& \lim_{t\to\infty} \int_{\widetilde B} 
f_t\Bigl(\vecx,(1_d,\vecalf)(M,\bn)
\bigl(E_1(\vecx),\bn\bigr)\Phi^t \Bigr) \, d\widetilde{\lambda}(\tx)  
\\ & 
=\lambda(B)^{-1}\lim_{t\to\infty} \int_{B} 
f_t\Bigl(\vecx,(1_d,\vecalf)(M,\bn)
\bigl(E_1(\vecx),\bn\bigr)\Phi^t \Bigr) \, d\lambda(\vecx),
\end{split}
\end{equation}
and right hand side equals (using the right invariance of $\mu$)
\begin{align}
\int_{\widetilde{B}\times X}
f(\vecx,g)\,
d\mu(g)d\widetilde{\lambda}(\tx)
=\lambda(B)^{-1} \int_{B\times X}
f\left(\vecx,g\right) \, d\mu(g)d\lambda(\vecx).
\end{align}
This proves our claim: 
\eqref{also2identity} holds when $\Edomain$ is replaced by any
Borel subset $B$ of $\Edomain_0$.
We have proved this under the assumption $\lambda(B)>0$,
but it is trivially true also in the case $\lambda(B)=0$.

Now the proof of Corollary \ref{also2} is completed by a
simple covering argument: Given $\ve>0$ there is some compact subset 
$K\subset\Edomain$ such that $\lambda(K)>1-\ve$ and the map
$\Edomain\ni\vecx\mapsto \vece_1 E_1(\vecx)^{-1}\in\S^{d-1}_1$ 
has non-singular differential at every $\vecx\in K$.
Then by what we have proved and since $K$ is compact, there exists a 
finite family $\Edomain_1,\ldots,\Edomain_n$ of open subsets of $\Edomain$
which cover $K$ and which have the same property as $\Edomain_0$ above.
Set $B_1:=\Edomain_1\cap K$
and, recursively, $B_j:=(\Edomain_j\cap K)\setminus (B_1\cup\ldots\cup B_{j-1})$
for $j=2,\ldots,n$. Then each $B_j$ is a Borel subset of $\Edomain_j$
so that \eqref{also2identity} holds when $\Edomain$ is replaced by $B_j$.
Furthermore $K$ is the disjoint union of $B_1,\ldots,B_n$;
hence by adding we obtain that
\eqref{also2identity} holds when $\Edomain$ is replaced by $K$.
Using $\lambda(K)>1-\ve$ and our assumption that the family $f_t$ 
is uniformly bounded, we obtain \eqref{also2identity} upon letting
$\ve\to 0$.
\end{proof}

\subsection{Characteristic functions} \label{charfcnssec}

We recall the definition of limits of a family of sets 
$\{\scrE_t\}_{t\geq t_0}$, where $t_0$ is a fixed real constant:
\begin{equation}
	\liminf \scrE_t := \bigcup_{t\geq t_0} \bigcap_{s\geq t} \scrE_s , \qquad
	\limsup \scrE_t := \bigcap_{t\geq t_0} \bigcup_{s\geq t} \scrE_s .
\end{equation}
We will also use the notation
\begin{equation} \label{LIMINFSUPTOP}
	\lim(\inf \scrE_t)^\circ := \bigcup_{t\geq t_0} \bigg(\bigcap_{s\geq t} \scrE_s\bigg)^\circ , \qquad
	\lim\overline{\sup \scrE_t} := \bigcap_{t\geq t_0} \overline{\bigcup_{s\geq t} \scrE_s} .
\end{equation}
Note that $\lim(\inf \scrE_t)^\circ$ is open and
$\lim\overline{\sup \scrE_t}$ is closed.

If $\{\scrE_t\}_{t\geq t_0}$ is a decreasing family and $\scrE = \bigcap_{t\geq t_0} \scrE_t$ we write $\scrE_t\downarrow \scrE$; 
if $\{\scrE_t\}_{t\geq t_0}$ is an increasing family and $\scrE = \bigcup_{t\geq t_0} \scrE_t$ we write $\scrE_t\uparrow \scrE$. \label{decinc}

\begin{thm}\label{charThm}
Let $\lambda$ be a Borel probability measure on $\RR^{d-1}$ which is absolutely continuous with respect to Lebesgue measure, and let $\scrE_t$ be a family of subsets of $\RR^{d-1}\times X$. Then, for $\vecalf\in\RR^d\setminus\QQ^d$ and $M\in\SL(d,\RR)$,
\begin{equation}\label{inf}
	\liminf_{t\to\infty} \int_{\RR^{d-1}} \chi_{\scrE_t}(\vecx,(1_d,\vecalf)(M,\vecnull)n_-(\vecx) \Phi^t) d\lambda(\vecx) \geq \int_{\lim(\inf \scrE_t)^\circ} d\mu(g)d\lambda(\vecx) ,
\end{equation}
and
\begin{equation}\label{sup}
	\limsup_{t\to\infty} \int_{\RR^{d-1}} \chi_{\scrE_t}(\vecx,(1_d,\vecalf)(M,\vecnull)n_-(\vecx) \Phi^t) d\lambda(\vecx) \leq \int_{\lim\overline{\sup \scrE_t}} d\mu(g)d\lambda(\vecx) .
\end{equation}
If furthermore the set $\lim\overline{\sup \scrE_t}\setminus\lim(\inf \scrE_t)^\circ$ has measure zero,
\begin{equation}\label{lim}
	\lim_{t\to\infty} \int_{\RR^{d-1}} \chi_{\scrE_t}(\vecx,(1_d,\vecalf)(M,\vecnull)n_-(\vecx) \Phi^t) d\lambda(\vecx) = \int_{\limsup \scrE_t} d\mu(g)d\lambda(\vecx) .
\end{equation}
\end{thm}

\begin{proof}
We begin with the proof of \eqref{sup}. Define the closed set
\begin{equation}
	\widetilde\scrE_t := \overline{\bigcup_{s\geq t} \scrE_s} .
\end{equation}
Clearly $\scrE_t\subset\widetilde\scrE_t\subset \widetilde\scrE_{t_1}$ for $t\geq t_1$. So
\begin{multline}
	\limsup_{t\to\infty} \int_{\RR^{d-1}} \chi_{\scrE_t}(\vecx,(1_d,\vecalf)(M,\vecnull)n_-(\vecx) \Phi^t) d\lambda(\vecx) \\
\leq \limsup_{t_1\to\infty} \limsup_{t\to\infty} \int_{\RR^{d-1}} \chi_{\widetilde\scrE_{t_1}}(\vecx,(1_d,\vecalf)(M,\vecnull)n_-(\vecx) \Phi^t) d\lambda(\vecx) .
\end{multline}
It follows from Theorem \ref{equiThm} (for a constant family of test functions $f=f_t$) by a standard probabilistic argument in which characteristic functions are approximated by bounded continuous functions $f$ (see e.g. \cite{Shiryaev}, Chap.~III) that
\begin{equation}
	\limsup_{t\to\infty} \int_{\RR^{d-1}} \chi_{\widetilde\scrE_{t_1}}(\vecx,(1_d,\vecalf)(M,\vecnull)n_-(\vecx) \Phi^t) d\lambda(\vecx)
	\leq \int_{\widetilde\scrE_{t_1}} d\mu(g)d\lambda(\vecx) .
\end{equation}

Since $\widetilde\scrE_{t_1} \downarrow \lim\overline{\sup \scrE_t}$, 
\begin{equation}
	\limsup_{t_1\to\infty}\int_{\widetilde\scrE_{t_1}} d\mu(g)d\lambda(\vecx) =
	\int_{\lim\overline{\sup \scrE_t}} d\mu(g)d\lambda(\vecx),
\end{equation}
and \eqref{sup} follows. Relation \eqref{inf} is established by taking complements, and \eqref{lim} then follows from \eqref{inf} and \eqref{sup}.
\end{proof}

\begin{remark} \label{charCor}
Let $E_1:\Edomain\to\SO(d)$ be any map as in Corollary \ref{also2};
then the assertions of Theorem \ref{charThm} also hold with
$n_-(\vecx)$ replaced by $(E_1(\vecx),\bn)$:
Let $\lambda$ be a Borel probability measure on $\Edomain$, absolutely 
continuous with respect to Lebesgue measure,
and let $\scrE_t$ be a family of subsets of $\RR^{d-1}\times X$.
Then, for $\vecalf\in\RR^d\setminus\QQ^d$ and $M\in\SL(d,\RR)$,
\begin{equation}\label{cCinf}
	\liminf_{t\to\infty} \int_{\Edomain} \chi_{\scrE_t}\bigl(\vecx,(1_d,\vecalf)(M,\vecnull)(E_1(\vecx),\bn) \Phi^t\bigr) \, d\lambda(\vecx) \geq \int_{\lim(\inf \scrE_t)^\circ} \, d\mu(g)d\lambda(\vecx) ,
\end{equation}
and we have corresponding analogues of \eqref{sup} and \eqref{lim}.
The proof is exactly as the proof of Theorem \ref{charThm},
except that Corollary \ref{also2} is used in place
of Theorem \ref{equiThm}.
\end{remark}

\subsection{Corresponding results for $\SLR$}\label{CORRESSLR}

By following the same line of arguments as for $\ASLR$, one can prove the analogous equidistribution results for any homogeneous space $\Gamma\backslash\SLR$ with $\Gamma$ a lattice in $\SLR$. 
The lattices relevant for our application are the congruence subgroups 
$\Gamma=\Gamma(q)$. 
The main results are as follows
(cf.\ Theorem \ref{equiThm}, Corollary~\ref{also2}, Theorem \ref{charThm}
and Remark \ref{charCor} above).
\begin{thm}\label{equiThm-rat}
Let $\lambda$ be a Borel probability measure on $\RR^{d-1}$ which is absolutely continuous with respect to Lebesgue measure. Let $f:\RR^{d-1}\times X_q\to\RR$ be bounded continuous and $f_t:\RR^{d-1}\times X_q\to\RR$ a family of uniformly bounded, continuous functions such that $f_t\to f$ as $t\to\infty$, uniformly on compacta.
Then, for every $M\in\SL(d,\RR)$,
\begin{equation} \label{equiThm-rat-formula}
	\lim_{t\to\infty} \int_{\RR^{d-1}} f_t\left(\vecx,M 
\matr 1{\vecx}{\trans\vecnull}{1_{d-1}}   
\matr{\e^{-(d-1)t}}{\bn}{\trans\bn}{\e^{t}1_{d-1}}\right)
\, d\lambda(\vecx) 
\\ = \int_{\RR^{d-1}\times X_q} f(\vecx,M) \, d\mu_q(M)d\lambda(\vecx) .
\end{equation}
\end{thm}

\begin{cor}\label{also2-rat}
Let $E_1:\Edomain\to\SO(d)$ be any map as in Corollary \ref{also2},
let $\lambda$ be a Borel probability measure on $\Edomain$,
absolutely continuous with respect to Lebesgue measure,
and let $f:\Edomain\times X_q\to \R$ and $f_t:\Edomain\times X_q\to \R$
be bounded continuous functions such that $f_t\to f$ as $t\to\infty$,
uniformly on compacta. Then, for every $M\in\SL(d,\R)$,
\begin{equation} \label{also2identity-rat}
\lim_{t\to\infty} \int_{\Edomain} f_t\left(\vecx,M E_1(\vecx)
\matr{\e^{-(d-1)t}}{\bn}{\trans\bn}{\e^{t}1_{d-1}}\right)
\, d\lambda(\vecx) \\ 
= \int_{\Edomain\times X_q} f(\vecx,M) \, d\mu_q(M)d\lambda(\vecx).
\end{equation}
\end{cor}

\begin{thm}\label{charThm-rat}
Let $\lambda$ be a Borel probability measure on $\RR^{d-1}$ which is absolutely continuous with respect to Lebesgue measure, and let $\scrE_t$ be a family of subsets of $\RR^{d-1}\times X_q$. Then, for every $M\in\SL(d,\RR)$,
\begin{equation}\label{inf-rat}
\liminf_{t\to\infty} \int_{\RR^{d-1}} \chi_{\scrE_t}\left(\vecx,M
\matr 1{\vecx}{\trans\vecnull}{1_{d-1}}   
\matr{\e^{-(d-1)t}}{\bn}{\trans\bn}{\e^{t}1_{d-1}}\right)\, d\lambda(\vecx) 
\geq \int_{\lim(\inf \scrE_t)^\circ} d\mu_q(M)d\lambda(\vecx) ,
\end{equation}
and
\begin{equation}\label{sup-rat}
	\limsup_{t\to\infty} \int_{\RR^{d-1}} \chi_{\scrE_t}\left(\vecx,M
\matr 1{\vecx}{\trans\vecnull}{1_{d-1}}   
\matr{\e^{-(d-1)t}}{\bn}{\trans\bn}{\e^{t}1_{d-1}}\right)\, d\lambda(\vecx) 
\leq \int_{\lim\overline{\sup \scrE_t}} d\mu_q(M)d\lambda(\vecx) .
\end{equation}
If furthermore the set $\lim\overline{\sup \scrE_t}\setminus\lim(\inf \scrE_t)^\circ$ has measure zero,
\begin{equation}\label{lim-rat}
\lim_{t\to\infty} \int_{\RR^{d-1}} 
\chi_{\scrE_t}\left(\vecx,M
\matr 1{\vecx}{\trans\vecnull}{1_{d-1}}   
\matr{\e^{-(d-1)t}}{\bn}{\trans\bn}{\e^{t}1_{d-1}}\right)\, d\lambda(\vecx) 
= \int_{\limsup \scrE_t} d\mu_q(M)d\lambda(\vecx) .
\end{equation}
\end{thm}

\begin{remark}\label{charCor-rat}
Let $E_1:\Edomain\to\SO(d)$ be any map as in Corollary \ref{also2};
then the assertions of Theorem \ref{charThm-rat} hold with
$\smatr 1{\vecx}{\trans\vecnull}{1_{d-1}}$ replaced with $E_1(\vecx)$:
Let $\lambda$ be a Borel probability measure on $\Edomain$, 
absolutely continuous with respect to Lebesgue measure,
let $\scrE_t$ be a family of subsets of $\RR^{d-1}\times X_q$,
and let $M\in\SL(d,\RR)$. Then
\begin{equation}\label{cCinf-rat}
\liminf_{t\to\infty} \int_{\Edomain} \chi_{\scrE_t}\left(\vecx,M E_1(\vecx)
\matr{\e^{-(d-1)t}}{\bn}{\trans\bn}{\e^{t}1_{d-1}}\right)\, d\lambda(\vecx) 
\geq \int_{\lim(\inf \scrE_t)^\circ} d\mu_q(M)d\lambda(\vecx) ,
\end{equation}
and we have corresponding analogues of \eqref{sup-rat} and \eqref{lim-rat}.
\end{remark}

It should be noted that these statements for $\SLR$ are in fact consequences of the mixing property of diagonal one-parameter subgroups of $\SLR$ on $\Gamma\backslash\SLR$ (cf. the arguments used in \cite{Eskin93}, \cite{Marklof00}), and do not require an application of Ratner's theory.

\section{Lattice points in thin sets}\label{secThin}

\subsection{Affine lattices with irrational $\vecalf$}\label{secThinirr}

In the following we consider subsets $\fB$ of $\RR^{d-1}\times\RR^d$; we use the notation
\begin{equation}\label{FBXDEF}
	\fB|_\vecx=(\{ \vecx \} \times \RR^d) \cap \fB 
\end{equation}
which we identify with a subset of $\RR^d$ by projection onto the $\RR^d$ component.
Our goal in this section is to study, for a given affine lattice,
the limit distribution of the number of lattice points contained in such a set
$\fB|_\vecx$ after it has been deformed, thinly stretched, and then
sheared (or rotated) by a random amount. As we will see in 
Section \ref{secProofs}, the problems discussed in
Sections \ref{secVisible0} and \ref{secVisible} correspond to
special cases of the present question.

\begin{thm}\label{thinThm}
Let $\lambda$ be a Borel probability measure on $\RR^{d-1}$ which is
absolutely continuous with respect to Lebesgue measure, and let $\fB_t$ be a 
family of subsets of $\RR^{d-1}\times\RR^d$ such that
$\cup_t \fB_t$ is bounded.
Then, for $r\in\ZZ_{\geq 0}$, $\vecalf\in\RR^d\setminus\QQ^d$ and $M\in\SL(d,\RR)$,
\begin{multline}\label{inf2}
	\liminf_{t\to\infty} \lambda\Bigl(\Bigl\{ \vecx\in\RR^{d-1} \col
\#\bigl( \fB_t|_\vecx\Phi^{-t} n_-(\vecx) \cap (\ZZ^d+\vecalf)M \bigr) 
\geq r\Bigr\} \Bigr)   \\
	\geq 
	(\lambda\times\mu)\Bigl(\Bigl\{ (\vecx,g)\in\RR^{d-1}\times X
\col \#\bigl( (\lim(\inf\fB_t)^\circ)|_\vecx\cap \ZZ^d g \bigr) \geq r
\Bigr\} \Bigr)  ,
\end{multline}
and
\begin{multline}\label{sup2}
	\limsup_{t\to\infty} \lambda\Bigl( \Bigl\{ \vecx\in\RR^{d-1}: \#\bigl( \fB_t|_\vecx\Phi^{-t} n_-(\vecx) \cap (\ZZ^d+\vecalf)M \bigr) \geq r\Bigr\} 
\Bigr)   \\
	\leq 
	(\lambda\times\mu)\Bigl( \Bigl\{ (\vecx,g)\in\RR^{d-1}\times X: 
\#\bigl( (\lim\overline{\sup\fB_t})|_\vecx\cap \ZZ^d g \bigr) \geq r
\Bigr\} \Bigr)  .
\end{multline}
If furthermore the set $\lim\overline{\sup\fB_t}\setminus\lim(\inf\fB_t)^\circ$ has Lebesgue-measure zero, then
\begin{multline}\label{lim2}
	\lim_{t\to\infty} \lambda\Bigl(\Bigl\{ \vecx\in\RR^{d-1}\col \#\bigl( \fB_t|_\vecx\Phi^{-t} n_-(\vecx) \cap (\ZZ^d+\vecalf)M \}\bigr)\geq r\Bigr\} \Bigr)   \\
	= 
	(\lambda\times\mu)\Bigl( \Bigl\{(\vecx,g)\in\RR^{d-1}\times X \col \#\bigl( (\limsup\fB_t)|_\vecx\cap \ZZ^d g \bigr) \geq r\Bigr\} \Bigr)  .
\end{multline}
\end{thm}

We will require the following lemma for the proof of Theorem \ref{thinThm}.
Given a set $\fB\subset\RR^{d-1}\times\RR^d$ and an integer $r\in\ZZ_{> 0}$, we define the subset 
\begin{equation} \label{SCREDEF}
	\scrE(\fB,r) = \Bigl\{ (\vecx,g)\in \RR^{d-1}\times X \col
\#\bigl(\fB|_\vecx \cap \ZZ^d g\bigr)\geq r \Bigr\} .
\end{equation}

\begin{lem}\label{limlem}
Fix $r\in\ZZ_{>0}$. Then the following statements hold.
\begin{enumerate}
	\item[(i)] If $\fA\subset\fB$, then $\scrE(\fA,r)\subset\scrE(\fB,r)$.
	\item[(ii)] If $\fB_t$ is a decreasing family of
bounded sets, then   %
$\:\cap_t \scrE(\fB_t,r)=\scrE(\cap_t \fB_t,r)$.
	\item[(iii)] If $\fB_t$ is an increasing family of
sets then   %
$\:\cup_t \scrE(\fB_t,r)=\scrE(\cup_t \fB_t,r)$.
	\item[(iv)] If $\fB$ is open, then $\scrE(\fB,r)$ is open.
	\item[(v)] If $\fB$ is closed and bounded, then $\scrE(\fB,r)$ is closed.
	\item[(vi)] If $\fB$ has zero Lebesgue measure, then $\scrE(\fB,r)$ has zero measure with respect to $\vol_{\RR^{d-1}}\times\mu$.
\end{enumerate}
\end{lem}

\begin{proof}[Proof of {\rm(i)}]
Clear.
\end{proof}

\begin{proof}[Proof of {\rm(ii)}]
It follows from (i) that
$\cap_t \scrE(\fB_t,r)\supset\scrE(\cap_t \fB_t,r)$.
To prove the opposite inclusion, let $(\vecx,g)\in \R^{d-1}\times X$
be an arbitrary point outside $\scrE(\cap_t \fB_t,r)$,
where $g\in\ASL(d,\R)$ is a fixed representative for a point in $X$.
Then $\#\bigl((\cap_t \fB_t)|_\vecx \cap \ZZ^d g\bigr)<r$.
Because of our assumptions there is
a bounded set $\fC\subset \R^{d}$
such that $\fB_t|_\vecx\subset \fC$ for all $t\geq t_0$
(for some constant $t_0\in\RR$).
Let $F$ be the finite set $F:=\{\vecm\in\ZZ^d \col \vecm g\in \fC\}$,
and let $F':=\{\vecm\in\Z^d \col \vecm g \in (\cap_t \fB_t)|_\vecx\}
\subset F$.
Then $\# F'<r$. For each $\vecm\in F\setminus F'$ there is some 
$t\geq t_0$ such that $\vecm g\notin \fB_t|_\vecx$;
thus for all sufficiently large $t$ we have
$\vecm g\notin \fB_t|_\vecx$ for all $\vecm\in F\setminus F'$.
Hence for these $t$ we have $\# \bigl(\fB_t|_\vecx \cap \ZZ^d g\bigr)
\leq\# F'<r$.
Hence $(\vecx,g)\notin \cap_t\scrE(\fB_t,r)$.
\end{proof}

\begin{proof}[Proof of {\rm(iii)}] 
It follows from (i) that $\cup_t \scrE(\fB_t,r)\subset\scrE(\cup_t \fB_t,r)$.
To prove the other inclusion, take an arbitrary point $(\vecx,g)
\in\scrE(\cup_t \fB_t,r)$. Then there are $r$ distinct vectors
$\vecm_1,\ldots,\vecm_r\in\ZZ^d$ with $\vecm_j g \in (\cup_t \fB_t)|_\vecx
=\cup_t (\fB_t|_\vecx)$. Hence for $t$ sufficiently large
we have $\vecm_jg\in\fB_t|_\vecx$ for all $j=1,\ldots,r$.
Hence $(\vecx,g)\in \cup_t \scrE(\fB_t,r)$.
\end{proof}

\begin{proof}[Proof of {\rm(iv)}]
Assume that $\fB$ is open. Take $(\vecx_0,g_0)\in \scrE(\fB,r)$,
where $g_0\in\ASL(d,\RR)$ is a fixed representative for a point
in $X$.
Then there exist $r$ distinct points $\vecm_1,\ldots,\vecm_r\in\Z^d$
satisfying $\vecm_j g_0 \in \fB|_{\vecx_0}$, i.e.\ 
$(\vecx_0,\vecm_j g_0)\in\fB$. %
Writing $\Omega=\cap_{j=1}^r f_j^{-1}(\fB)$ where
$f_j:\R^{d-1} \times \ASL(d,\RR) \ni (\vecx,g)
\mapsto (\vecx,\vecm_j g)\in \RR^{d-1}\times \RR^d$,
we have $(\vecx_0,g_0)\in\Omega$, and each 
$(\vecx,g)\in \Omega$ projects to a point in $\scrE(\fB,r)$.
Also $\Omega$ is an open subset of $\R^{d-1}\times\ASL(d,\R)$,
each $f_j$ being continuous.
Since $(\vecx_0,g_0)$ was arbitrary in $\scrE(\fB,r)$
we conclude that $\scrE(\fB,r)$ is open.
\end{proof}

\begin{proof}[Proof of {\rm(v)}]
Assume that $\fB$ is closed and bounded.
Take $(\vecx_0,g_0)\in \R^{d-1}\times X$ outside $\scrE(\fB,r)$,
where again $g_0\in\ASL(d,\RR)$ is a fixed representative for a point
in $X$.
Then $\#(\fB|_{\vecx_0} \cap \Z^dg_0)<r$.

Let $U_1$ be a neighborhood of the identity in $\SL(d,\RR)$ such that
$\|\vecy M-\vecy\|\leq\sfrac 12\|\vecy\|$ for all 
$\vecy\in\R^d$, $M\in U_1$.
Let $R=\sup \, \bigl\{\|\vecy\| \col \vecy\in\cup_{\vecx\in\R^{d-1}}
\fB|_\vecx\bigr\}$. Then $U=U_1\times \scrB_R^d$ 
is a neighborhood of the identity
in $\ASL(d,\RR)=\SL(d,\R)\ltimes\R^d$, and for each
$\vecy\in\RR^d$ with $\|\vecy\|>4R$ and
$g=(M,\vecxi)\in U$ we have
\begin{equation}
\begin{split}
\|\vecy g\|=\|\vecy M+\vecxi\|
\geq \|\vecy\|-\|\vecy M-\vecy\|-\|\vecxi\|
>\sfrac 12\|\vecy\|-R>R.
\end{split}
\end{equation}
Hence $\vecy g\notin \fB|_{\vecx}$ holds automatically for all
$g\in U$, $\vecx\in\RR^{d-1}$ and all $\vecy\in\RR^d$ with
$\|\vecy\|>4R$. 
Let $F$ be the finite set of points
$\vecm\in\ZZ^d$ which satisfy $\|\vecm g_0\|\leq 4R$
and $\vecm g_0\notin \fB|_{\vecx_0}$.
For each $\vecm \in F$ we choose some open sets
$V_\vecm\subset\RR^{d-1}$ and $V'_\vecm\subset\RR^d$ 
such that $(\vecx_0,\vecm g_0)\in V_\vecm \times V'_\vecm
\subset \complement \fB$.
Now set
\begin{equation}
\begin{split}
U'=(g_0 U) \: \cap \: \Bigl( \bigcap_{\vecm\in F}
\{g\in \ASL(d,\RR) \col \vecm g \in V'_\vecm \}\Bigr);
\qquad V=\bigcap_{\vecm \in F} V_\vecm.
\end{split}
\end{equation}
These are open subsets of $\ASL(d,\RR)$ and $\RR^d$, respectively,
and $(\vecx_0,g_0)\in V\times U'$.
Furthermore, if $(\vecx,g)\in V\times U'$ then 
by construction $\vecm g\notin \fB|_\vecx$ for each
$\vecm\in\Z^d$ with $\vecm g_0\notin\fB|_{\vecx_0}$,
and thus $\#(\Z^dg \cap \fB|_{\vecx})<r$
since $\#(\Z^dg_0 \cap \fB|_{\vecx_0})<r$.
Hence each $(\vecx,g)\in V\times U'$ projects to
a point in $\R^{d-1}\times X$ \textit{outside} $\scrE(\fB,r)$.

Since $(\vecx_0,g_0)$ was an arbitrary point outside
$\scrE(\fB,r)$ we conclude that $\scrE(\fB,r)$ is closed.
\end{proof}

\begin{proof}[Proof of {\rm(vi)}]
Assume that $\fB$ has Lebesgue measure zero.
Note that
\begin{equation} \label{MEASUREZERO}
\begin{split}
&(\vol_{\RR^{d-1}}\times\mu)\Bigl(\bigl\{
(\vecx,g)\in\RR^{d-1}\times\ASL(d,\RR) \col \fB|_\vecx \cap \ZZ^d g \neq
\emptyset\bigr\}\Bigr)
\\
& \le \sum_{\vecm \in \ZZ^d}
(\vol_{\RR^{d-1}}\times\mu)\Bigl(\bigl\{
(\vecx,g)\in\RR^{d-1}\times\ASL(d,\RR) \col \vecm g\in\fB|_\vecx
\bigr\}\Bigr)
\\
&=\sum_{\vecm \in \ZZ^d} \int_{\RR^{d-1}} \int_{\SL(d,\R)} \int_{\RR^d}
I\Bigl(\vecm (M,\vecxi) \in \fB|_\vecx\Bigr) \, d\vol_{\R^d}(\vecxi)
d\mu_1(M) d\vol_{\R^{d-1}}(\vecx),
\end{split}
\end{equation}
where $I$ is the indicator function.
The innermost integral equals $\vol_{\RR^d}(\fB|_\vecx)$,
since $\vecm (M,\vecxi)=\vecm M+\vecxi$.
But $\vol(\fB|_\vecx)=0$ holds for almost every $\vecx\in\RR^{d-1}$,
and thus the total integral is zero.
Hence, a fortiori, $\scrE(\fB,r)$ has measure zero.
\end{proof}

\begin{proof}[Proof of Theorem \ref{thinThm}]
If $r=0$ then the statements are trivial; thus from now on we may assume $r>0$.
Define the decreasing family of sets
\begin{equation}\label{widehatE}
	\widehat\scrE_t := \scrE\bigg( \overline{\bigcup_{s\geq t} \fB_s},r \bigg) .
\end{equation}
These sets are clearly closed (cf. Lemma \ref{limlem} (v)).
Then
\begin{equation}\label{sup3}
\begin{split}
	\limsup_{t\to\infty} & \: \lambda\Bigl( \Bigl\{ \vecx\in\RR^{d-1}: \#\bigl( \fB_t|_\vecx\Phi^{-t} n_-(\vecx) \cap (\ZZ^d+\vecalf)M \bigr) \geq r\Bigr\} \Bigr)   \\
	& \leq 
	\limsup_{t\to\infty} \int_{\RR^{d-1}} \chi_{\widehat\scrE_t}\bigl(
\vecx,(1_d,\vecalf)(M,\vecnull)n_-(-\vecx) \Phi^t\bigr) \, d\lambda(\vecx)  \\
	& \leq \int_{\limsup\widehat\scrE_t} d\mu(g)d\lambda(\vecx) ,
\end{split}
\end{equation}
due to Theorem \ref{charThm}. 
(To be precise, to treat ``$n_-(-\vecx)$'' as above,
one applies Theorem \ref{charThm} to $\lambda'$ and
$\widehat\scrE_t'$,
defined through $\lambda'(B)=\lambda(-B)$ for $B\subset \RR^{d-1}$,
and $\widehat\scrE_t'=\{(-\vecx,g) : (\vecx,g)\in\widehat\scrE_t\}$.)
In view of Lemma \ref{limlem} (ii),
\begin{equation}
	\limsup\widehat\scrE_t=\bigcap_t \widehat\scrE_t 
= \scrE\bigg( \bigcap_t \overline{\bigcup_{s\geq t} \fB_s},r\bigg)
	= \scrE( \lim\overline{\sup\fB_t},r ) ,
\end{equation}
and hence
\begin{align} \notag
	\int_{\limsup\widehat\scrE_t} d\mu(g)d\lambda(\vecx) & 
	=
	\int_{\scrE( \lim\overline{\sup\fB_t},r )} d\mu(g)d\lambda(\vecx) 
\\ \label{thinThmstep1}
	& =(\lambda\times\mu)\Bigl(\Bigl\{ (\vecx,g)\in\RR^{d-1}\times X: 
\#\bigl( (\lim\overline{\sup\fB_t})|_\vecx\cap \ZZ^d g \bigr) \geq r\Bigr\} 
\Bigr) ,
\end{align}
which proves \eqref{sup2}. 
The proof of \eqref{inf2} is analogous, 
using Lemma \ref{limlem} (iii) and (iv).
Finally \eqref{lim2} follows using Lemma \ref{limlem} (vi) for
$r=1$, since $\lambda$ is absolutely continuous with respect to 
$\vol_{\R^{d-1}}$, and
\begin{equation}
\scrE\big(\lim\overline{\sup \fB_t},r\bigr) \setminus
\scrE\bigl(\lim(\inf \fB_t)^\circ,r\bigr) \: \subset \:
\scrE\bigl(\lim\overline{\sup \fB_t} \setminus \lim(\inf \fB_t)^\circ,1\bigr).
\end{equation}
\end{proof}

Theorem \ref{thinThm} is easily generalized to multiple families of sets:
\begin{thm}\label{thinThm-multi}
Let $\lambda$ be a Borel probability measure on $\RR^{d-1}$ which is
absolutely continuous with respect to Lebesgue measure.
For each $j=1,\ldots,m$, let $\fB_t^{(j)}$ be a 
family of subsets of $\RR^{d-1}\times\RR^d$ such that
$\cup_t \fB_t^{(j)}$ is bounded.
Then, for any $r_1,\ldots,r_m\in\ZZ_{\geq 0}$, 
$\vecalf\in\RR^d\setminus\QQ^d$ and $M\in\SL(d,\RR)$,
\begin{multline}\label{inf2-multi}
	\liminf_{t\to\infty} \lambda\Bigl(\Bigl\{ \vecx\in\RR^{d-1} \col
\#\bigl( \fB_t^{(j)}|_\vecx\Phi^{-t} n_-(\vecx) \cap (\ZZ^d+\vecalf)M \bigr) 
\geq r_j, \; j=1,\ldots,m \Bigr\} \Bigr)   \\
	\geq 
	(\lambda\times\mu)\Bigl(\Bigl\{ (\vecx,g)\in\RR^{d-1}\times X
\col \#\bigl( (\lim(\inf\fB_t^{(j)})^\circ)|_\vecx\cap \ZZ^d g \bigr) \geq r_j,
\; j=1,\ldots,m \Bigr\} \Bigr)  ,
\end{multline}
and
\begin{multline}\label{sup2-multi}
	\limsup_{t\to\infty} \lambda\Bigl( \Bigl\{ \vecx\in\RR^{d-1}: \#\bigl( \fB_t^{(j)}|_\vecx\Phi^{-t} n_-(\vecx) \cap (\ZZ^d+\vecalf)M \bigr) \geq r_j, \; j=1,\ldots,m \Bigr\} 
\Bigr)   \\
	\leq 
	(\lambda\times\mu)\Bigl( \Bigl\{ (\vecx,g)\in\RR^{d-1}\times X: 
\#\bigl( (\lim\overline{\sup\fB_t^{(j)}})|_\vecx\cap \ZZ^d g \bigr) \geq r_j, \; j=1,\ldots,m \Bigr\} \Bigr)  .
\end{multline}
If furthermore each set $\lim\overline{\sup\fB_t^{(j)}}\setminus\lim(\inf\fB_t^{(j)})^\circ$ ($j=1,\ldots,m$) has Lebesgue-measure zero, then
\begin{multline}\label{lim2-multi}
	\lim_{t\to\infty} \lambda\Bigl(\Bigl\{ \vecx\in\RR^{d-1}\col \#\bigl( \fB_t^{(j)}|_\vecx\Phi^{-t} n_-(\vecx) \cap (\ZZ^d+\vecalf)M \}\bigr)\geq r_j, \; j=1,\ldots,m\Bigr\} \Bigr)   \\
	= 
	(\lambda\times\mu)\Bigl( \Bigl\{(\vecx,g)\in\RR^{d-1}\times X \col \#\bigl( (\limsup\fB_t^{(j)})|_\vecx\cap \ZZ^d g \bigr) \geq r_j, \; j=1,\ldots,m\Bigr\} \Bigr)  .
\end{multline}
\end{thm}

\begin{proof}
We may throw away each $j$ for which $r_j=0$.
Thus from now on $r_j>0$ for each $j$.
Define the decreasing family of sets
\begin{equation}\label{widehatE2}
	\widehat\scrE_t%
:= \bigcap_{j=1}^m \scrE\bigg( 
\overline{\bigcup_{s\geq t} \fB_s^{(j)}},r_j \bigg) .
\end{equation}
These sets are clearly closed (cf. Lemma \ref{limlem} (v)).
Now \eqref{sup3} generalizes in the obvious way.
In view of Lemma \ref{limlem} (ii),
\begin{equation}
	\limsup\widehat\scrE_t
=\bigcap_t \widehat\scrE_t 
=\bigcap_{j=1}^m \bigcap_t \scrE\bigg( 
\overline{\bigcup_{s\geq t} \fB_s^{(j)}},r_j \bigg)
=\bigcap_{j=1}^m 
\scrE\bigg( \bigcap_t \overline{\bigcup_{s\geq t} \fB_s^{(j)}},r_j\bigg)
=\bigcap_{j=1}^m \scrE\Bigl( \lim\overline{\sup\fB_t^{(j)}},r_j \Bigr) ,
\end{equation}
and hence \eqref{thinThmstep1} carries over to give a proof of 
\eqref{sup2-multi}.
The proof of \eqref{inf2-multi} is analogous, 
using Lemma \ref{limlem} (iii) and (iv),
and noticing that
\begin{align}
\bigcup_t \bigcap_{j=1}^m \scrE\bigg( \Big(
\bigcap_{s\geq t} \fB_s^{(j)} \Big)^\circ,r_j \bigg)
=\bigcap_{j=1}^m \bigcup_t \scrE\bigg( \Big(
\bigcap_{s\geq t} \fB_s^{(j)} \Big)^\circ,r_j \bigg).
\end{align}
Finally \eqref{lim2-multi} follows using Lemma \ref{limlem} (vi) for
$r=1$, since $\lambda$ is absolutely continuous with respect to 
$\vol_{\R^{d-1}}$, and\enlargethispage{10pt}
\begin{equation}
\Big(
\bigcap_{j=1}^m\scrE\big(\lim\overline{\sup \fB_t^{(j)}},r_j\bigr) \Big)
\setminus
\Big(
\bigcap_{j=1}^m\scrE\bigl(\lim(\inf \fB_t^{(j)})^\circ,r_j\bigr) \Big)
\: \subset \:
\bigcup_{j=1}^m
\scrE\bigl(\lim\overline{\sup \fB_t^{(j)}} 
\setminus \lim(\inf \fB_t^{(j)})^\circ,1\bigr).
\end{equation}
\end{proof}

\begin{remark}\label{thinCor}
The assertions of Theorem \ref{thinThm} and  Theorem \ref{thinThm-multi}
also hold if $n_-(\vecx)$ is replaced by %
$(E_1(\vecx),\bn)^{-1}$
where $E_1:\Edomain\to\SO(d)$ is any map as in Corollary \ref{also2}.
Specifically, if $\lambda$ is any Borel probability measure on $\Edomain$,
absolutely continuous with respect to Lebesgue measure, then
for any given families $\scrB_t^{(j)}\subset\R^{d-1}\times\R^d$ as above, 
and any $r_j\in\Z_{\geq 0}$,
$\vecalf\in\R^d\setminus\Q^d$ and $M\in\SL(d,\R)$, we have
\begin{multline}\label{inf2-multi-cor}
	\liminf_{t\to\infty} \lambda\Bigl(\Bigl\{ \vecx\in\Edomain \col
\#\bigl( \fB_t^{(j)}|_\vecx\Phi^{-t} (E_1(\vecx),\bn)^{-1} 
\cap (\ZZ^d+\vecalf)M \bigr) 
\geq r_j, \; j=1,\ldots,m \Bigr\} \Bigr)   \\
	\geq 
	(\lambda\times\mu)\Bigl(\Bigl\{ (\vecx,g)\in\Edomain\times X
\col \#\bigl( (\lim(\inf\fB_t^{(j)})^\circ)|_\vecx\cap \ZZ^d g \bigr) \geq r_j,
\; j=1,\ldots,m \Bigr\} \Bigr)  ,
\end{multline}
and corresponding relations for the $\limsup$ and the limes, cf.\
\eqref{sup2-multi} and \eqref{lim2-multi}.
The proof is exactly as the proofs of Theorem \ref{thinThm}
and Theorem \ref{thinThm-multi},
except that Remark~\ref{charCor} is used in place of Theorem \ref{charThm}.
\end{remark}

\subsection{The case of rational $\vecalf$}

Using the same strategy of proof, the above results can be readily established for $\vecalf\in\QQ^d$, if the space $X$ is replaced by $X_q$ and the measure $\mu$ by $\mu_q$, for some $q$ with $\vecalf\in q^{-1}\Z^d$.
In the proofs one uses the following analogue of \eqref{SCREDEF}:
\begin{equation} \label{SCREQDEF}
	\scrE_q(\fB,r) = \Bigl\{ (\vecx,M)\in \RR^{d-1}\times X_q \col 
\#\bigl(\fB|_\vecx \cap (\ZZ^d+\vecalf)M\setminus\{\bn\}\bigr)\geq r \Bigr\} .
\end{equation}
The reason for removing the point $\bn$ is so as to make all of
Lemma \ref{limlem} valid in the present setting. (Specifically,
in the proof of the analogue of Lemma \ref{limlem} (vi) we need
to note that $\int_{\SL(d,\R)} I \bigl(\vecm M\in \fC\bigr) \, d\mu_1(M)=0$
holds for each subset $\fC\subset \R^d$ of measure $0$.
This is true for each $\vecm\in\R^d$ except $\vecm=\bn$.)

We thus have the following.

\begin{thm}\label{thinThm-rat-multi}
Let $\lambda$ be a Borel probability measure on $\RR^{d-1}$ which is absolutely continuous with respect to Lebesgue measure.
For each $j=1,\ldots,m$, let $\fB_t^{(j)}$ be a 
family of subsets of $\RR^{d-1}\times\RR^d$ such that
$\cup_t \fB_t^{(j)}$ is bounded.
Then, for any $r_1,\ldots,r_m\in\ZZ_{\geq 0}$, $\vecalf=\frac{\vecp}{q}\in\QQ^d$ and $M\in\SL(d,\RR)$,
\begin{multline}\label{inf2-rat}
	\liminf_{t\to\infty} \lambda\Bigl( \Bigl\{\vecx\in\RR^{d-1}: \#\bigl( \fB_t^{(j)}|_\vecx\Phi^{-t} n_-(\vecx) \cap (\ZZ^d+\vecalf)M\setminus\{\vecnull\} \bigr) \geq r_j,\; j=1,\ldots,m \Bigr\} \Bigr)   \\
	\geq 
	(\lambda\times\mu_q)\Bigl( \Bigl\{(\vecx,M')\in\RR^{d-1}\times X_q: \#\bigl( (\lim(\inf\fB_t^{(j)})^\circ)|_\vecx\cap (\ZZ^d+\vecalf) M'\setminus\{\vecnull\} \bigr) \geq r_j,\; j=1,\ldots,m \Bigr\} \Bigr)  ,
\end{multline}
and
\begin{multline}\label{sup2-rat}
	\limsup_{t\to\infty} \lambda\Bigl( \Bigl\{\vecx\in\RR^{d-1}: \#\bigl( \fB_t^{(j)}|_\vecx\Phi^{-t} n_-(\vecx) \cap (\ZZ^d+\vecalf)M\setminus\{\vecnull\} \bigr) \geq r_j,\; j=1,\ldots,m \Bigr\} \Bigr)   \\
	\leq 
	(\lambda\times\mu_q)\Bigl(\Bigl\{ (\vecx,M')\in\RR^{d-1}\times X_q: \#\bigl( (\lim\overline{\sup\fB_t^{(j)}})|_\vecx\cap (\ZZ^d+\vecalf) M'\setminus\{\vecnull\} \bigr) \geq r_j,\; j=1,\ldots,m \Bigr\} \Bigr)  .
\end{multline}
If furthermore each set $\lim\overline{\sup\fB_t^{(j)}}\setminus\lim(\inf\fB_t^{(j)})^\circ$ ($j=1,\ldots,m$) has Lebesgue-measure zero, then
\begin{multline}\label{lim2-rat}
	\lim_{t\to\infty} \lambda\Bigl(\Bigl\{ \vecx\in\RR^{d-1}: \#\bigl( \fB_t^{(j)}|_\vecx\Phi^{-t} n_-(\vecx) \cap (\ZZ^d+\vecalf)M\setminus\{\vecnull\} \bigr) \geq r_j,\; j=1,\ldots,m \Bigr\} \Bigr)   \\
	= 
	(\lambda\times\mu_q)\Bigl(\Bigl\{ (\vecx,M')\in\RR^{d-1}\times X_q: \#\bigl( (\limsup\fB_t^{(j)})|_\vecx\cap (\ZZ^d+\vecalf) M'\setminus\{\vecnull\} \bigr) \geq r_j,\; j=1,\ldots,m \Bigr\} \Bigr)  .
\end{multline}
\end{thm}

\begin{remark}\label{thinCor-rat}
The assertion of Theorem \ref{thinThm-rat-multi}
holds if $n_-(\vecx)$ is replaced by $(E_1(\vecx),\bn)^{-1}$,
where $E_1:\Edomain\to\SO(d)$ is any map as in Corollary \ref{also2}.
Compare Remark \ref{thinCor}.
\end{remark}

\subsection{Visible lattice points} \label{SUBSEC:PRIMLATP}

In the case of rational $\vecalf$, all results are equally valid for 
$\widehat\ZZ_\vecalf^d$ in place of $(\ZZ^d+\vecalf)\setminus\{\bn\}$.

\begin{thm}\label{thinThm-null}  \label{THM:PRIMLATP}
Let $\lambda$ be a Borel probability measure on $\RR^{d-1}$ which is absolutely continuous with respect to Lebesgue measure.
For each $j=1,\ldots,m$, let $\fB_t^{(j)}$ be a family of subsets of 
$\RR^{d-1}\times\RR^d$ such that $\cup_t \fB_t^{(j)}$ is bounded. 
Then, for any $r_1,\ldots,r_m\in\Z_{\geq 0}$, $\vecalf=\frac{\vecp}q\in\Q^d$ 
and $M\in\SL(d,\RR)$,
\begin{multline}\label{inf2-null}
	\liminf_{t\to\infty} \lambda( \{\vecx\in\RR^{d-1}\col  \#( \fB_t^{(j)}|_\vecx\Phi^{-t} n_-(\vecx) \cap \widehat\ZZ_\vecalf^d M ) \geq r_j,
\: j=1,\ldots,m \} )   \\
	\geq 
	(\lambda\times\mu_q)( \{ (\vecx,M')\in\RR^{d-1}\times X_q\col  \#( (\lim(\inf\fB_t^{(j)})^\circ)|_\vecx\cap \widehat\ZZ_\vecalf^d M' ) \geq r_j,
\: j=1,\ldots,m\} )  ,
\end{multline}
and
\begin{multline}\label{sup2-null}
	\limsup_{t\to\infty} \lambda(\{ \vecx\in\RR^{d-1}\col  \#( \fB_t^{(j)}|_\vecx\Phi^{-t} n_-(\vecx) \cap \widehat\ZZ_\vecalf^d M ) \geq r_j,
\: j=1,\ldots,m\} )   \\
	\leq 
	(\lambda\times\mu_q)( \{ (\vecx,M')\in\RR^{d-1}\times X_q\col  \#( (\lim\overline{\sup\fB_t^{(j)}})|_\vecx\cap \widehat\ZZ_\vecalf^d M' ) \geq r_j,
\: j=1,\ldots,m\} )  .
\end{multline}
If furthermore each set $\lim\overline{\sup\fB_t^{(j)}}\setminus\lim(\inf\fB_t^{(j)})^\circ$ ($j=1,\ldots,m$) has Lebesgue-measure zero, then
\begin{multline}\label{lim2-null}
	\lim_{t\to\infty} \lambda( \{ \vecx\in\RR^{d-1}\col  \#( \fB_t^{(j)}|_\vecx\Phi^{-t} n_-(\vecx) \cap \widehat\ZZ_\vecalf^d M ) \geq r_j,
\: j=1,\ldots,m\} )   \\
	= 
	(\lambda\times\mu_q)(\{ (\vecx,M')\in\RR^{d-1}\times X_q\col  \#( (\limsup\fB_t^{(j)})|_\vecx\cap \widehat\ZZ_\vecalf^d M' ) \geq r_j,
\: j=1,\ldots,m\} )  .
\end{multline}
\end{thm}

\begin{remark}\label{thinThm-null-cor}
Just as in previous remarks, the assertion of Theorem \ref{thinThm-null}
holds if $n_-(\vecx)$ is replaced by $(E_1(\vecx),\bn)^{-1}$,
where $E_1:\Edomain\to\SO(d)$ is any map as in Corollary \ref{also2}.
\end{remark}

\section{Integration formulas on $X$ and $X_q$} \label{FOLIATIONSEC}

In this section we prove some formulas for integrals
and volumes in the spaces $(X,\mu)$ and $(X_q,\mu_q)$, 
which we will need to be able to generalize a technique which was introduced 
in Elkies and McMullen \cite{Elkies04} in the case of $d=2$ and $(X,\mu)$
(cf.\ also \cite{SV}).
The goal is to obtain a
more precise understanding of the explicit limit functions described in 
our main theorems; we will achieve this in Section \ref{PROPLIMFCNSEC}.

Recall that we have fixed $\mu_q$ as the Haar measure on $\SL(d,\R)$
normalized to be a probability measure on $X_q=\Gamma(q)\backslash\SL(d,\R)$.
This implies, via a well-known volume formula by Siegel,
that $\mu_1$ can be explicitly given as the measure on $\SL(d,\R)$
which satisfies
\begin{align} \label{SLDRHAAR}
d\mu_1(M) \frac{dt}t = \Bigl( \zeta(2)\zeta(3)\cdots\zeta(d)\Bigr)^{-1}
\bigl( \det (x_{ij}) \bigr)^{-d} dx_{11} dx_{12} \cdot \ldots \cdot dx_{dd}
\end{align}
when parametrizing $\GL^+(d,\R)$ as $(x_{ij})=t^{1/d}M \in \GL^+(d,\R)$,
with $M\in\SL(d,\R)$, $t>0$;
cf.,\ e.g.,\ \cite{siegel}, \cite{Marklof00}.
From this it follows that
\begin{align} \label{SLDRHAARQ}
\mu_q=I_q^{-1} \mu_1; \qquad \text{where } \: 
I_q:=[\Gamma(q):\Gamma(1)],
\end{align}
and also that the Haar measure $\mu$ on $\ASL(d,\R)$ which we 
have normalized by $\mu(X)=1$, is explicitly given by 
\begin{align} \label{ASLDRHAAR}
d\mu(M,\vecxi)=d\mu_1(M) d\vecxi, \qquad
(M,\vecxi)\in\ASL(d,\R),
\end{align}
where $d\vecxi=d\xi_1\cdot\ldots\cdot d\xi_d$ is the standard Lebesgue
measure on $\R^d$.

The following lattice average formula is also due to Siegel 
(at least on $X_1$). Recall that we always keep $d\geq 2$.

\begin{prop} \label{SIEGELQPROP}
Let $F\in \L^1(\R^d)$, $q\in\Z_{>0}$ and $\vecalf\in q^{-1}\Z^d$. Then
\begin{align}
\int_{X_q} \sum_{\veck\in\Z^d+\vecalf\setminus\{\bn\}} F(\veck M) d\mu_q(M)
=\int_{\R^d}F(\vecx)d\vecx.
\end{align}
\end{prop}
\begin{proof}
If $\vecalf\in\Z^d$ then one easily reduces to the case $q=1$,
and then the formula is proved in Siegel, \cite{siegel}.
(Cf.\ also \cite[Sect.\ 3.7]{Marklof00}.)

From now on we assume $\vecalf\notin\Z^d$
(and thus $\Z^d+\vecalf\setminus\{\bn\}=\Z^d+\vecalf$).
Write $\vecalf =\frac{\vecp}q$ with $\vecp\in\Z^d$.
Let $\F\subset\SL(d,\R)$ be a fundamental domain for 
$\SL(d,\Z)\backslash\SL(d,\R)$ and choose representatives
$T_j\in\SL(d,\Z)$ so that $\SL(d,\Z)=\bigsqcup_{j=1}^{I_q} \Gamma(q) T_j$
(with $\bigsqcup$ denoting disjoint union);
then $\bigsqcup_{j=1}^{I_q} T_j \F$ is a fundamental domain for $\Gamma(q)$.
Hence
\begin{equation}
\begin{split}
\int_{X_q} \sum_{\veck\in\Z^d+\vecalf} F(\veck M) \, d\mu_q(M)
& =I_q^{-1} \sum_{j=1}^{I_q} \int_{\F}
\sum_{\veck\in\Z^d+\vecalf} F(\veck T_j M) \, d\mu_1(M)
\\
&%
= I_q^{-1} \int_{\F} \sum_{\ell=1}^\infty
F(q^{-1} \vecn_\ell M) \, d\mu_1(M),
\end{split}
\end{equation}
where $\vecn_1,\vecn_2,\ldots\in\Z^d\setminus\{\bn\}$ is an enumeration 
(with multiplicities taken into account)
of the points $\vecm T_j$, for $\vecm\in \vecp+q\Z^d$, $j\in\{1,\ldots,I_q\}$.
For every $\gamma\in\SL(d,\Z)$, the list
$\vecn_1\gamma,\vecn_2\gamma,\ldots$ can be obtained as a 
permutation of $\vecn_1,\vecn_2,\ldots$.
(To see this, note that given $\gamma\in\SL(d,\Z)$ there are elements 
$\gamma_1,\ldots,\gamma_{I_q}
\in \Gamma(q)$ and a permutation $\rho$ of $\{1,\ldots,I_q\}$ such that
$T_j \gamma =\gamma_j T_{\rho(j)}$ for all $j\in\{1,\ldots,I_q\}$.
Also note $(\vecp+q\Z^d)\gamma_j=\vecp+q\Z^d$.)
Recall that each orbit for the action of $\SL(d,\Z)$ 
on $\Z^d\setminus\{\bn\}$ equals $t\widehat\Z^d$
for some $t\in\Z_{>0}$, where $\widehat\Z^d$ as before denotes the set of
primitive lattice points in $\Z^d$.
It follows that for each $t\in\Z_{>0}$ there is some multiplicity
$m_t\in\Z_{\geq 0}$ such that the sequence $\vecn_1,\vecn_2,\ldots$ visits 
each point in $t\widehat\Z^d$ exactly $m_t$ times.
Now the above integral may be rewritten as
\begin{equation}
I_q^{-1} \int_{X_1} \sum_{t=1}^\infty m_t \sum_{\vecc\in\widehat\Z^d}
F(q^{-1} t\vecc M) \, d\mu_1(M).
\end{equation}
Arguing as in \cite[pp.\ 1150--1151]{Marklof00} we find that this
equals
\begin{equation}\label{lastid}
\frac 1{I_q\zeta(d)} \sum_{t=1}^\infty \frac{m_t}{t^d} 
\int_{\R^d} F(q^{-1}\vecx) \, d\vecx
=\frac {q^d}{I_q\zeta(d)} \sum_{t=1}^\infty \frac{m_t}{t^d} 
\int_{\R^d} F(\vecx) \, d\vecx
\end{equation}
Finally an argument as in \cite[p.\ 1152(top)]{Marklof00}
shows that the constant in front of the integral must actually be 1,
i.e.\ $\sum_{t=1}^\infty \frac{m_t}{t^d}=q^{-d}I_q\zeta(d)$, and the proof is
complete.
\end{proof}

The identity $\sum_{t=1}^\infty \frac{m_t}{t^d}=q^{-d}I_q\zeta(d)$ 
can of course also be proved by a more explicit computation:
One easily reduces the situation to the case where $q$ is the minimal
denominator of the given $\vecalf\in\Q^d$; in other words
$\vecalf=\frac{\vecp}{q}$ where $\vecp=(p_1,\ldots,p_d)\in\Z^d$
has $\gcd(q,p_1,\ldots,p_d)=1$.
Then note that the $\SL(d,\Z/q\Z)$-orbit of $\vecp+q\Z^d$
in $\Z^d/q\Z^d$ equals
\begin{align} \label{VDEF}
V=\bigl\{\vecx+q\Z^d \col \vecx=(x_1,\ldots,x_d)\in\Z^d, \:
\gcd(q,x_1,\ldots,x_d)=1\bigr\}\subset \Z^d/q\Z^d,
\end{align}
and since $\#\SL(d,\Z/q\Z)=I_q$ we see that the sequence
$\vecn_1,\vecn_2,\ldots$ visits exactly those points in $\Z^d$
which belong to the preimage of $V$, and each such point is visited 
exactly $I_q/\#V$ times.
Hence
\begin{align}
\sum_{t=1}^\infty\frac{m_t}{t^d}
=\frac{I_q}{\# V}\sum_{t\geq 1,\: (t,q)=1}t^{-d}
=\frac{I_q}{\# V}\zeta(d) \sum_{e\mid q} \mu(e) e^{-d}.
\end{align}
On the other hand $\# V=q^d\sum_{e\mid q} \mu(e) e^{-d}$,
and the identity follows.

\subsection{The submanifolds $X_q(\vecy)$ of $X_q$}\label{SUBMSECTION}

Fix $\vecalf=\vecp/q\in\Q^d$.
Given any $\vecy\in\R^d\setminus\{\bn\}$ we define
\begin{align} \label{XQYDEF}
X_q(\vecy):=\bigl\{M\in X_q \col \vecy\in (\Z^d+\vecalf)M\bigr\}.
\end{align}
This set can be given the structure of an embedded
submanifold in $X_q$ of codimension $d$, and with a countably
infinite number 
of connected components. To see this we first note that $X_q(\vecy)=
\bigcup_{\veck\in\Z^d+\vecalf \setminus\{\bn\}} X_q(\veck,\vecy)$,
where
\begin{align} \label{XQKYDEF}
X_q(\veck,\vecy):=
\bigl\{\Gamma(q)M\in X_q \col M \in \SL(d,\R), \: \veck M =\vecy\bigr\}.
\end{align}
One checks that for any $\veck,\veck'\in \Z^d+\vecalf \setminus\{\bn\}$,
we have $X_q(\veck,\vecy)=X_q(\veck',\vecy)$ if $\veck'\in\veck \Gamma(q)$;
whereas $X_q(\veck,\vecy) \cap X_q(\veck',\vecy)=\emptyset$
whenever $\veck'\notin \veck \Gamma(q)$. Hence if $S$ is any subset of
$\Z^d+\vecalf \setminus\{\bn\}$
containing exactly one representative from each orbit 
of the right action of $\Gamma(q)$ on $\Z^d+\vecalf \setminus\{\bn\}$,
then we can express $X_q(\vecy)$ as a disjoint union
\begin{align} \label{XQYDISJUNION}
X_q(\vecy)=\bigsqcup_{\veck\in S} X_q(\veck,\vecy).
\end{align}

To describe each $X_q(\veck,\vecy)$ further we set
\begin{align} \label{HAV}
H=\bigl\{g\in\SL(d,\R) \col \vece_1 g =\vece_1\bigr\}
=\Bigl\{\begin{pmatrix} 1 & \bn \\ \trans\vecv & A\end{pmatrix} \col
\vecv\in\R^{d-1}, \: A\in\SL(d-1,\R) \Bigr\}.
\end{align}
This is a closed subgroup of $\SL(d,\R)$ which is isomorphic 
with $\ASL(d-1,\R)$ (as defined in \eqref{ASLMULTLAW}) through
\begin{align}
H \ni \begin{pmatrix} 1 & \bn \\ \trans\vecv & A\end{pmatrix}
\mapsto (\trans A^{-1},\vecv \trans A^{-1})\in\ASL(d-1,\R).
\end{align}
We let $\mu_H$\label{MUHDEF} be the Haar measure on $H$ given by
$d\mu_H(g)=d\mu_1^{(d-1)}(A) \, d\vecv$, with 
$A,\vecv$ as in \eqref{HAV}, $\mu_1^{(d-1)}$ the
Haar measure on $\SL(d-1,\R)$ from \eqref{SLDRHAAR}, and
$d\vecv$ the standard Lebesgue measure on $\R^{d-1}$.
In dimension $d=2$ we have $H=\bigl\{\smatr 10v1 \col v\in\R\bigr\}$ and we set
$d\mu_H=dv$.

Now fix some $M_\veck,$ $M_\vecy\in SL(d,\R)$ such that
$\veck=\vece_1 M_\veck$, $\vecy=\vece_1 M_\vecy$.
Then $X_q(\veck,\vecy)$ is the image of
$M_\veck^{-1} H M_\vecy\subset \SL(d,\R)$ under the standard projection
$\pi:\SL(d,\R)\to X_q$, and $h_1,h_2\in H$ give the same point
$\pi(M_\veck^{-1} h_1 M_\vecy)=\pi(M_\veck^{-1} h_2 M_\vecy)$ if and only if
$h_1$ and $h_2$ belong to the same left 
$(M_\veck \Gamma(q) M_\veck^{-1}\cap H)$-coset.
This gives an identification of sets
\begin{align} \label{XQKYID}
X_q(\veck,\vecy)=M_\veck^{-1} \Bigl(
\bigl(M_\veck \Gamma(q) M_\veck^{-1}\cap H\bigr)\backslash H \Bigr) M_\vecy.
\end{align}
Since $M_\veck \Gamma(q) M_\veck^{-1}\cap H$ is a discrete subgroup of $H$,
the quotient space
$\bigl(M_\veck \Gamma(q) M_\veck^{-1}\cap H\bigr)\backslash H$ is a
connected $(d^2-d-1)$-dimensional manifold, and hence
$X_q(\veck,\vecy)$ inherits a natural structure as a 
connected $(d^2-d-1)$-dimensional manifold. One verifies that this structure 
does not depend on the choice of $M_\vecy$ or $M_\veck$
(since left or right multiplication by any fixed $H$-element gives a
diffeomorphism of $H$).
Since the map $H\ni h\mapsto M_\veck^{-1} h M_\vecy \in \SL(d,\R)$
is an immersion we see that $X_q(\veck,\vecy)$ is a
connected immersed submanifold of $X_q$.
Hence since the union \eqref{XQYDISJUNION} is disjoint we have now
given $X_q(\vecy)$ a structure as an immersed submanifold of $X_q$ 
with a countably infinite number of connected components.
($X_q(\vecy)$ is even an embedded submanifold of $X_q$,
but we will not need this fact.)

Note that $\mu_H$ induces a Borel measure on each quotient space
$\bigl( M_\veck \Gamma(q) M_\veck^{-1}\cap H\bigr)\backslash H$,
which we also call $\mu_H$.
\textit{We endow $X_q(\vecy)$ with the Borel measure $\nu_\vecy$\label{NUYDEF}
defined on each $X_q(\veck,\vecy)$ as coming from
$(I_q\zeta(d))^{-1}\mu_H$
under the identification \eqref{XQKYID}.}
This measure $\nu_\vecy$ is independent of the choices
of $S$ and matrices $M_\veck$, $M_\vecy$, as is easily seen from the
fact that $\mu_H$ is both left and right invariant.

\begin{lem} \label{CANMFORMINVARIANCE}
For any $\vecy\in\R^d\setminus\{\bn\}$, $T\in\SL(d,\R)$ and
any Borel subset $\scrE\subset X_q(\vecy)$ we have
$\nu_\vecy(\scrE)=\nu_{\vecy T}(\scrE T)$.
\end{lem}
\begin{proof}
For any given subset $\scrE'\subset X_q(\veck,\vecy)$ we have
$\scrE' T\subset X_q(\veck,\vecy T)$, and if we choose
$M_{\vecy T}=M_\vecy T$ in the above definitions then
these two subsets correspond to exactly the same
subset of $(M_\veck\Gamma(q)M_\veck^{-1}\cap H)\backslash H$ under the 
identification(s) \eqref{XQKYID}. The lemma follows trivially from this.
\end{proof}

\begin{prop} \label{FOLINTPROP}
Let $\scrE\subset X_q$ be any Borel set; then 
$\vecy\mapsto \nu_\vecy(\scrE \cap X_q(\vecy))$
is a measurable function of $\vecy\in\R^d\setminus\{\bn\}$.
If $U\subset\R^d\setminus\{\bn\}$ is any Borel set
such that $\scrE\subset \bigcup_{\vecy\in U} X_q(\vecy)$,
then
\begin{align} \label{FOLINTGENREL}
\mu_q(\scrE) \leq \int_U \nu_\vecy(\scrE \cap X_q(\vecy)) \, d\vecy.
\end{align}
Furthermore, if 
$\forall \vecy_1\neq \vecy_2\in U: X_q(\vecy_1)\cap X_q(\vecy_2) \cap \scrE
=\emptyset$, then \textrm{equality} holds in \eqref{FOLINTGENREL}.
\end{prop}

The following lemma will be required for the proof.

\begin{lem} \label{FULLHYDECLEMMA}
For each $\vecy\in\R^d\setminus\{\bn\}$, choose some
$M_\vecy\in\SL(d,\R)$ with $\vece_1 M_\vecy=\vecy$.
Then for every $f\in \L^1(\SL(d,\R),\mu_q)$ we have
\begin{align} \label{FULLHYDECLEMMAFORMULA}
\int_{\SL(d,\R)} f(M)\, d\mu_q(M)=
\frac 1{I_q\zeta(d)} \int_{\R^d\setminus\{\bn\}}
\Bigl(\int_H f(hM_\vecy) \, d\mu_H(h)\Bigr)\, d\vecy.
\end{align}
\end{lem}

\begin{proof}
First of all the integral
$\int_H f(hM_\vecy) \, d\mu_H(h)$ (if it exists) only depends on $f$ and
$\vecy$, and not on the choice of $M_\vecy$, since for a given $\vecy$
the matrix $M_\vecy$
is uniquely determined up to left multiplication by an element
from $H$, and $\mu_H$ is right $H$-invariant.
Hence we may fix the following specific choices of $M_\vecy$,
for $\vecy=(y_1,\ldots,y_d)$ with $y_1>0$:
\begin{align} \label{M0YDEF}
M_\vecy=M^{(0)}_\vecy:=
\begin{pmatrix} y_1 & \vecy' \\ \trans\bn & y_1^{-\frac 1{d-1}}1_{d-1}
\end{pmatrix}
\quad \text{with } \: \vecy'=(y_2,\ldots,y_d);
\end{align}
and for $\vecy=(y_1,\ldots,y_d)$ with $y_1<0$:
\begin{align}
M_\vecy=M_{\vecy K_0}^{(0)}K_0,
\qquad \text{where }\:K_0=\diag[-1,-1,1,\ldots,1] \in\SL(d,\R).
\end{align}
We may leave $M_\vecy$ unspecified when $y_1=0$, since these $\vecy$'s form
a subset of $\R^d\setminus\{\bn\}$ of Lebesgue measure zero.

Write $G=\SL(d,\R)$, $G^+=\{(m_{jk})\in G\col m_{11}>0\}$ and
$G^-=\{(m_{jk})\in G\col m_{11}<0\}$. 
Then the map $\langle h,\vecy\rangle \mapsto M=hM_\vecy^{(0)}$
gives a diffeomorphism from 
$H\times\{\vecy\in\R^d \col y_1>0\}$ onto $G^+$
(indeed, the inverse is easily computed explicitly and seen to be smooth).
Furthermore in this parametrization we have,
via a standard computation
similar to, e.g., \cite[(3.70), case $r=1$]{Marklof00},
$d\mu_q(M)=(I_q\zeta(d))^{-1} d\mu_H(h)d\vecy$.
Hence %
$\int_{G^+} f(M)\, d\mu_q(M)=
(I_q\zeta(d))^{-1} \int_{\{\vecy\in\R^d \col y_1>0\}}
\int_H f(hM_\vecy) \, d\mu_H(h)\, d\vecy.$
Similarly one verifies
$\int_{G^-} f(M)\, d\mu_q(M)=$\linebreak
$(I_q\zeta(d))^{-1} \int_{\{\vecy\in\R^d \col y_1<0\}}
\int_H f(hM_\vecy) \, d\mu_H(h)\, d\vecy$,
and \eqref{FULLHYDECLEMMAFORMULA} follows by addition of these two.
\end{proof}

\begin{proof}[Proof of Proposition \ref{FOLINTPROP}.]
Let $\F\subset\SL(d,\R)$ be a (measurable) fundamental domain for
$\Gamma(q)\backslash\SL(d,\R)$, in the set theoretic sense. That is, we assume
$\F \cap \gamma \F=\emptyset$ for all $\gamma\in\Gamma(q)$, and
$\bigcup_{\gamma\in\Gamma(q)} \gamma\F=\SL(d,\R)$.
For each $\vecy\in\R^d\setminus\{\bn\}$, fix some
$M_\vecy\in\SL(d,\R)$ with $\vece_1 M_\vecy=\vecy$.
Now for any $\vecy\in\R^d\setminus\{\bn\}$ we have, via 
\eqref{XQYDISJUNION}, \eqref{XQKYID} and the definition of $\nu_\vecy$,
\begin{align}
\nu_\vecy(\scrE \cap X_q(\vecy))
=(I_q\zeta(d))^{-1} \sum_{\veck\in S}
\int_{\F_1} \chi_{_{\scrE_0}}\bigl(M_\veck^{-1}hM_\vecy\bigr) \, d\mu_H(h),
\end{align}
where $\F_1\subset H$ is any fixed fundamental domain 
for $(M_\veck\Gamma(q)M_\veck^{-1}\cap H) \backslash H$, 
${\scrE_0}$ denotes the pre-image in $\SL(d,\R)$ of  
$\scrE\subset X_q$, 
and $\chi_{_{\scrE_0}}$ is its characteristic function.
We may choose %
$\F_1=H\cap\F_2$ where $\F_2\subset\SL(d,\R)$ is any
fixed fundamental domain for 
$(M_\veck\Gamma(q)M_\veck^{-1}\cap H) \backslash \SL(d,\R)$, and such an
$\F_2$ may be fixed as
$\F_2= M_\veck\bigl(\bigsqcup_{\gamma\in S^{(\veck)}}\gamma \F\bigr)
M_\vecy^{-1}$,
where $S^{(\veck)}\subset\Gamma(q)$ is any set of coset representatives for
$(\Gamma(q)\cap M_\veck^{-1}HM_\veck)\backslash\Gamma(q)$. 
Hence, since ${\scrE_0}\subset\SL(d,\R)$ is left $\Gamma(q)$ invariant,
\begin{align} \label{FOLINTPROPSTEP1}
\nu_\vecy(\scrE \cap X_q(\vecy))
=(I_q\zeta(d))^{-1} \sum_{\veck\in S}
\sum_{\gamma\in S^{(\veck)}}
\int_{H} \chi_{_{\F\cap{\scrE_0}}}\bigl(\gamma^{-1}M_\veck^{-1}hM_\vecy\bigr)
\, d\mu_H(h).
\end{align}
But for each $\veck\in S$ and $\gamma\in S^{(\veck)}$ we have
$\vece_1 M_\veck \gamma=\veck \gamma=\vece_1 M_{\veck\gamma}$
and thus $M_\veck \gamma=h_0 M_{\veck\gamma}$ for some $h_0\in H$;
hence using the invariance of $\mu_H$ we see that we may
replace $\gamma^{-1} M_\veck^{-1}$ with $M_{\veck\gamma}^{-1}$
inside the integrand.
Furthermore, given $\gamma,\gamma'\in\Gamma(q)$ we have
the following chain of equivalent statements:
\begin{align}
(\Gamma(q)\cap M_\veck^{-1}HM_\veck)\gamma=
(\Gamma(q)\cap M_\veck^{-1}HM_\veck)\gamma'
\Longleftrightarrow
\gamma'\gamma^{-1} \in M_\veck^{-1}HM_\veck 
\hspace{50pt} &
\\ \notag
\Longleftrightarrow
\vece_1 M_\veck \gamma'\gamma^{-1} M_\veck^{-1}=\vece_1
\Longleftrightarrow
\veck \gamma'=\veck \gamma. &
\end{align}
Hence by the definition of $S$ and $S^{(\veck)}$, as $\veck$ and $\gamma$
run through the double sum in \eqref{FOLINTPROPSTEP1},
$\veck \gamma$ visits each vector in $\Z^d+\vecalf\setminus\{\bn\}$ exactly once.
Hence
\begin{align} \label{FOLINTPROPSTEP2}
\nu_\vecy(\scrE \cap X_q(\vecy))
=(I_q\zeta(d))^{-1} \sum_{\veck\in \Z^d+\vecalf\setminus\{\bn\}}
\int_{H} \chi_{_{\F\cap{\scrE_0}}}\bigl(M_\veck^{-1}hM_\vecy\bigr)
\, d\mu_H(h).
\end{align}
Here for each $\veck$ the function
$\R^d\setminus\{\bn\}\ni\vecy\mapsto
\int_H \chi_{_{M_\veck(\F\cap{\scrE_0})}}(hM_\vecy)\, d\mu_H(h)$ is measurable
(this is implicit in Lemma \ref{FULLHYDECLEMMA});
hence also the above sum \eqref{FOLINTPROPSTEP2} is measurable as
a function from $\vecy\in\R^d\setminus\{\bn\}$
into $\R_{\geq 0}\cup\{\infty\}$.

Now to prove \eqref{FOLINTGENREL}, note that the assumption on $U$
implies $\scrE_0=\bigcup_{\veck\in\Z^d+\vecalf\setminus\{\bn\}} \scrE_\veck$,
where $\scrE_\veck:=\{M\in\scrE_0\col\veck M\in U\}$.
We have by \eqref{FOLINTPROPSTEP2},
\begin{align}
\int_U \nu_\vecy(\scrE \cap X_q(\vecy)) \, d\vecy
=(I_q\zeta(d))^{-1} \sum_{\veck\in \Z^d+\vecalf\setminus\{\bn\}}
\int_U \int_{H} \chi_{_{\F\cap{\scrE_0}}}\bigl(M_\veck^{-1}hM_\vecy\bigr)
\, d\mu_H(h) \, d\vecy,
\end{align}
and for any $\veck,\vecy,h$ appearing in the above expression we
have $\veck(M_\veck^{-1} h M_\vecy)=\vecy\in U$, so that
$M_\veck^{-1} h M_\vecy\in\scrE_\veck$ must hold whenever
$M_\veck^{-1} h M_\vecy\in\scrE_0$.
Also for every $\vecy\in \R^d\setminus (U\cup\{\bn\})$
we have $\veck(M_\veck^{-1} h M_\vecy)=\vecy\notin U$, so that
$M_\veck^{-1} h M_\vecy\notin\scrE_\veck$. Hence
\begin{align}
& \int_U \nu_\vecy(\scrE \cap X_q(\vecy)) \, d\vecy
=(I_q\zeta(d))^{-1} \sum_{\veck\in \Z^d+\vecalf\setminus\{\bn\}}
\int_{\R^d\setminus\{\bn\}} \int_{H} 
\chi_{_{\F\cap{\scrE_\veck}}}\bigl(M_\veck^{-1}hM_\vecy\bigr)
\, d\mu_H(h) \, d\vecy
\\ \notag
& = \sum_{\veck\in \Z^d+\vecalf\setminus\{\bn\}}
\mu_q\bigl(M_\veck(\F\cap{\scrE_\veck})\bigr)
= \sum_{\veck\in \Z^d+\vecalf\setminus\{\bn\}}
\mu_q\bigl(\F\cap{\scrE_\veck}\bigr)
\geq \mu_q\bigl(\F\cap{\scrE_0}\bigr)=\mu_q(\scrE),
\end{align}
where we used Lemma \ref{FULLHYDECLEMMA}, the invariance of $\mu_q$,
and $\scrE_0=\bigcup_{\veck\in\Z^d+\vecalf\setminus\{\bn\}} \scrE_\veck$.
(To avoid any confusion in the last step: Recall that we use $\mu_q$ to 
denote both a Haar measure on $\SL(d,\R)$ and the corresponding 
probability measure on $X_q$.)
Hence \eqref{FOLINTGENREL} is proved.
To prove the final statement about equality, note that the condition
$\forall \vecy_1\neq \vecy_2\in U: X_q(\vecy_1)\cap X_q(\vecy_2) \cap \scrE
=\emptyset$ implies that the sets $\scrE_\veck$ are pairwise disjoint,
and thus 
$\sum_{\veck\in \Z^d+\vecalf\setminus\{\bn\}}
\mu_q\bigl(\F\cap{\scrE_\veck}\bigr)%
=\mu_q(\scrE)$.
\end{proof}

\begin{prop} \label{XQVOLONEPROP}
For each $\vecy\in\R^d\setminus\{\bn\}$ we have
$\nu_\vecy(X_q(\vecy))=1$.
\end{prop}
\begin{proof}
Let us write $\vecalf=\frac{\vecp}q$ with $\vecp=(p_1,\ldots,p_d)\in\Z^d$.
We first show that without loss of generality
we may assume $\gcd(q,p_1,\ldots,p_d)=1$, i.e.\ that
$q$ is the minimal denominator of the given vector $\vecalf\in\Q^d$.
Indeed, any \textit{other} denominator of $\vecalf$ can be written as
$q'=qq_1$, with $q_1\in\Z_{>0}$; for each such $q'$ there is a canonical 
covering map $\pi:X_{q'}\to X_q$ of index $[\Gamma(q'):\Gamma(q)]=I_{q'}/I_q$,
and it follows straight from the definition \eqref{XQYDEF} that
$X_{q'}(\vecy)=\pi^{-1}(X_q(\vecy))$, i.e.\ 
$X_{q'}(\vecy)$ is a covering of the manifold $X_q(\vecy)$ of index
$[\Gamma(q'):\Gamma(q)]$. 
Furthermore %
the measure $\nu_\vecy^{(q)}$ on $X_{q}(\vecy)$ lifts to
$[\Gamma(q'):\Gamma(q)]\nu_\vecy^{(q')}$ on $X_{q'}(\vecy)$
(in an obvious notation).
Hence if $\nu_\vecy^{(q)}(X_q(\vecy))=1$ then also
$\nu_\vecy^{(q')}(X_{q'}(\vecy))=1$, as desired.

Thus from now on we assume $\gcd(q,p_1,\ldots,p_d)=1$.
By \eqref{XQYDISJUNION} and \eqref{XQKYID} we have
\begin{align} \label{XQVOLONEPROPSTEP0}
\nu_\vecy(X_q(\vecy))
=(I_q\zeta(d))^{-1} \sum_{\veck\in S}
\mu_H\bigl((M_\veck\Gamma(q)M_\veck^{-1}\cap H) \backslash H\bigr).
\end{align}
Given $\veck=(k_1,\ldots,k_d)\in\Z^d+\vecalf\setminus\{\bn\}$,
set $t_\veck:=\gcd(qk_1,qk_2,\ldots,qk_d)\in\Z_{>0}$. \label{TVECKDEF}
Then $(q/t_\veck)\veck$ is a primitive vector in $\Z^d$, %
and hence
there is some  $\gamma\in\Gamma(1)$ so that
$(q/t_\veck)\veck=\vece_1\gamma$.
For each $\delta>0$ we define $g_\delta=\diag[\delta,\delta^{-1},1,
\ldots,1]\in\SL(d,\R)$.
Then we may choose $M_\veck$ as $M_\veck:=g_{t_\veck/q} \gamma$
(since this gives $\vece_1M_\veck = \veck$).
With this choice we have 
$M_\veck\Gamma(q)M_\veck^{-1}=g_{t_\veck/q}\Gamma(q)g_{{t_\veck}/q}^{-1}$, 
since $\Gamma(q)$ is normal in $\Gamma(1)$.
Note that $\alpha:H\ni h \mapsto g_{{t_\veck}/q}hg_{{t_\veck}/q}^{-1}\in H$
gives an automorphism of $H$, and hence
$M_\veck\Gamma(q)M_\veck^{-1}\cap H=\alpha(\Gamma(q) \cap H)$.
Furthermore one verifies by a quick computation that $\alpha$
scales the Haar measure with a factor
$(q/{t_\veck})^d$, i.e.\ $\mu_H(\alpha(A))=(q/{t_\veck})^d \mu_H(A)$ for any measurable
$A\subset H$. Hence
\begin{align} \label{XQVOLONEPROPSTEP1}
\nu_\vecy(X_q(\vecy))
=\frac{q^d}{I_q\zeta(d)} \sum_{\veck\in S} t_\veck^{-d}
\mu_H\bigl((\Gamma(q)\cap H) \backslash H\bigr).
\end{align}

For each $\veck=(k_1,\ldots,k_d)\in\Z^d+\vecalf\setminus\{\bn\}$
we have $(t_\veck,q)=1$, since $q\veck\in \vecp %
+q\Z^d$ and $\gcd(q,p_1,\ldots,p_d)=1$.
On the other hand, given any $t\in\Z_{>0}$ with $(t,q)=1$
we may choose $t^*\in\Z$ so that $t^*t\equiv 1\minmod q$;
then $\gcd(q,t^*p_1,\ldots,t^*p_d)=1$ since $(q,t^*)=1$,
and thus there exists some primitive vector $\vecm$ in $t^*\vecp+q\Z^d$,
\footnote{This\label{DIRICHLETFOOTNOTE} 
can for example be shown using Dirichlet's theorem
on arithmetic progressions, for by that theorem 
we may find $m_j\in t^*p_j+q\Z$, $j=1,\ldots,d$ such that
$m_j=r_j\gcd(p_j,q)$ with prime numbers $q<r_1<r_2<\ldots<r_d$; then 
$\vecm=(m_1,\ldots,m_d)$ lies in $t^*\vecp+q\Z^d$ and is primitive.}
and then $\veck=(t/q)\vecm\in\Z^d+\vecalf\setminus\{\bn\}$ has $t_\veck=t$.
Furthermore, we claim that $t_{\veck_1}=t_{\veck_2}$ holds if and only if
$\veck_1\Gamma(q)=\veck_2\Gamma(q)$.
To prove the nontrivial direction of this claim we assume that
$\veck_1,\veck_2\in\Z^d+\vecalf\setminus\{\bn\}$ have 
$t:=t_{\veck_1}=t_{\veck_2}$. Then $(q/t)\veck_j$ is a primitive vector
in $\Z^d$ and hence there are some $\gamma_1,\gamma_2\in\Gamma(1)$
with $(q/t)\veck_j=\vece_1\gamma_j$.
Now both vectors $\vece_1\gamma_j$ belong to
$t^*\vecp+q\Z^d$ with $t^*$ as before;
hence $\vece_1\gamma_1\gamma_2^{-1}\equiv\vece_1 \minmod q\Z^d$,
so that $\gamma_1\gamma_2^{-1}=\begin{pmatrix} x_1 & \vecx'
\\ \trans\vecv & A\end{pmatrix}$ with $x_1\equiv 1 \minmod q$ and
$\vecx'\in q\Z^{d-1}$. Reducing mod $q$ we also see that
$A \minmod q$ lies in $\SL(d-1,\Z/q\Z)$; 
hence there is some $A'\in\SL(d-1,\Z)$ so that $A'\equiv A \minmod q$
\cite[p.\ 21]{Shimura}.
Now $\begin{pmatrix} 1 & \bn \\ \trans\vecv & A'\end{pmatrix} \in 
\Gamma(1)$, and this matrix has the same projection as 
$\gamma_1\gamma_2^{-1}$ in $\SL(d,\Z/q\Z)\cong \Gamma(1)/\Gamma(q)$. 
Hence $\gamma_0:=\gamma_1^{-1} 
\begin{pmatrix} 1 & \bn \\ \trans\vecv & A'\end{pmatrix} \gamma_2$
belongs to $\Gamma(q)$, and we see that 
$\vece_1\gamma_1\gamma_0=\vece_1\gamma_2$,
and thus $\veck_1\Gamma(q)=\veck_2\Gamma(q)$, as desired.

It follows that \eqref{XQVOLONEPROPSTEP1} may be rewritten as
\begin{align} 
\nu_\vecy(X_q(\vecy))
=\frac{q^d \, \mu_H\bigl((\Gamma(q)\cap H) \backslash H\bigr)}{I_q\zeta(d)} 
\sum_{\substack{t\geq 1 \\ (t,q)=1}} t^{-d}.
\end{align}
But note $\mu_H\bigl((\Gamma(q)\cap H) \backslash H\bigr)=
\#\bigl ( (\Gamma(q)\cap H)\backslash(\Gamma(1)\cap H) \bigr) \cdot
\mu_H\bigl((\Gamma(1)\cap H) \backslash H\bigr)$,
where the second factor equals one by the definition of $\mu_H$,
and the first factor is seen to equal
$\#H(\Z/q\Z)$ with
$H(\Z/q\Z):=
\Bigl\{\begin{pmatrix} 1 & \bn \\ \trans\vecv & A \end{pmatrix}
\in\SL(d,\Z/q\Z)\Bigr\}$
(for this one uses the surjectivity of the map
$\SL(d-1,\Z)\to\SL(d-1,\Z/q\Z)$, cf.\ \cite[p.\ 21]{Shimura}).
Also note that we have a decomposition of $\SL(d,\Z/q\Z)$ analogous to
the decomposition of $\SL(d,\R)$ in the proof of Lemma \ref{FULLHYDECLEMMA}:
Take $V\subset \Z^d/q\Z^d$ to be as in \eqref{VDEF}.
For each $\vecy\in V$
we fix a matrix $M_\vecy\in\SL(d,\Z/q\Z)$ whose first row equals $\vecy$.
(Such a matrix exists, for we may lift $y_1,\ldots,y_d$ to
integers satisfying $\gcd(y_1,\ldots,y_d)=1$,
cf.\ footnote \ref{DIRICHLETFOOTNOTE} above, and then apply
\cite[VIII.1-2]{siegel}.)
One then verifies that the map
$H(\Z/q\Z)\times V \ni \langle h,\vecy \rangle
\mapsto hM_\vecy \in \SL(d,\Z/q\Z)$ is a bijection.
Hence $I_q=\#\SL(d,\Z/q\Z)=\# H(\Z/q\Z)\cdot\# V$.
Finally recall that $\#V=q^d\sum_{e\mid q} \mu(e)e^{-d}$
and
$\sum_{t\geq 1,\: (t,q)=1} t^{-d}=
\zeta(d)\sum_{e\mid q} \mu(e)e^{-d}$. Hence we obtain
\begin{align} 
\nu_\vecy(X_q(\vecy))
=\frac{q^d\, \# H(\Z/q\Z)}{\#V \cdot \# H(\Z/q\Z)} 
\sum_{e\mid q} \mu(e)e^{-d}=1.
\end{align}
\end{proof}

We next prove an analogue for $X_q(\vecy)$ 
of Siegel's lattice average formula, Proposition \ref{SIEGELQPROP}.

\begin{prop} \label{XQYSIEGELPROP}
Assume $d\geq 3$ and
$\vecalf=\frac{\vecp}q$ with $\vecp=(p_1,\ldots,p_d)\in\Z^d$ and
$\gcd(q,p_1,\ldots,p_d)=1$.
If $F:\R^d\to\R$ is a bounded measurable function of compact support,
then for any $\vecy\in\R^d\setminus\{\bn\}$ we have
\begin{align} \label{XQYSIEGELPROPFORMULA}
\int_{X_q(\vecy)} \sum_{\vecm\in\Z^d} F((\vecm+\vecalf)M) \,
d\nu_\vecy(M)
= \int_{\vecx\in\R^d} F(\vecx)\, d\vecx
+\sum_{\substack{t\geq 1\\ (t,q)=1}}
t^{-d}
\sum_{\substack{a\in t+q\Z \\ (a,t)=1}}
F\Bigl(\frac{a}{t}\vecy\Bigr),
\end{align}
where all sums and integrals are absolutely convergent.
\end{prop}

We require the following lemma.

\begin{lem} \label{HSIEGELBASICLEMMA}
Let $F:\R^d\to\R$ be a bounded measurable function of compact support.
If $d=2$ then we furthermore require that $F(x\vece_1+m\vece_2)$
is a measurable function of $x\in\R$ for each fixed $m\in\Z$.
Let $\vecalf=(\alpha_1,\ldots,\alpha_d)$ with $\alpha_1\in \R$
and $\alpha_2,\ldots,\alpha_d\in\Z$. Then
\begin{align} \label{HSIEGELBASICLEMMAFORMULA}
& \int_{(\Gamma(q)\cap H) \backslash H}
\sum_{\vecm\in\Z^d} F((\vecm+\vecalf) M) \, d\mu_H(M)
\\ \notag
& =q^{d-1}\begin{cases}
I_q^{(d-1)}\bigl(\sum_{{\ell}\in\Z} F(({\ell}+\alpha_1)\vece_1)
+\int_{\vecx\in\R^d} F(\vecx) \, d\vecx\bigr) & d\geq 3
\\
\sum_{{\ell}\in\Z} F(({\ell}+\alpha_1)\vece_1) +\sum_{m\in\Z\setminus\{0\}}
\int_{\R} F(x\vece_1+m\vece_2) \, dx & d=2.
\end{cases}
\end{align}
\end{lem}

\begin{proof}[Proof of Lemma \ref{HSIEGELBASICLEMMA}.]
The right hand side in \eqref{HSIEGELBASICLEMMAFORMULA} is
clearly absolutely convergent; it will be clear from the proof
that the left hand side is also absolutely convergent.

We first give the proof in the case $d\geq 3$.
Write $\vecalf=(\alpha_1,\vecalf')\in\R\times\Z^{d-1}$ and express $M\in H$ as
$M=\begin{pmatrix} 1 & \bn \\ \trans\vecv & M_1 \end{pmatrix}$;
then 
\begin{align}
(\Z^d+\vecalf) M
=\bigsqcup_{{\ell}\in\Z} \bigsqcup_{\vecm\in\Z^{d-1}}
\Bigl(({\ell}+\alpha_1)+(\vecm+\vecalf')\trans\vecv,(\vecm+\vecalf')M_1\Bigr),
\end{align}
and a fundamental domain for $(\Gamma(q)\cap H) \backslash H$ is
given by $\bigl\{M\in H\col \vecv\in [0,q)^{d-1},\:M_1\in\F\bigr\},$
where $\F$ is any fixed fundamental domain for 
$\Gamma(q)\backslash\SL(d-1,\R)$.
Set $F_1(x,\vecy)=\sum_{{\ell}\in\Z} F(x+{\ell},\vecy)$
where in the right hand side we identify $\R^d$ with $\R\times\R^{d-1}$
in the obvious way.
Since $\vecalf'\in \Z^{d-1}$, the integral in the left hand side of 
\eqref{HSIEGELBASICLEMMAFORMULA} can now be expressed as
\begin{align} \label{HSIEGELBASICLEMMASTEP1}
\int_{\F}\int_{[0,q)^{d-1}}\sum_{\vecm\in\Z^{d-1}} 
F_1(\alpha_1+\vecm\trans\vecv,\vecm M_1) \, d\mu_1(M_1) d\vecv.
\end{align}
Note that $\int_{[0,q)} F_1(a+b x,\vecy) \, dx=
q\int_{\R/\Z} F_1(x,\vecy) \, dx$ for any $a\in\R$ and $b\in\Z_{\neq 0}$.
Thus, defining $F_2(\vecy):=\int_{\R/\Z} F_1(x,\vecy) \, dx
=\int_\R F(x,\vecy) \, dx$ (so that $F_2(\vecy)$ is defined for almost every
$\vecy\in\R^{d-1}$, and the function $F_2$ is measurable, by Fubini's Theorem),
we have $\int_{\vecv \in [0,q)^{d-1}} 
F_1(\alpha_1+\vecm \trans\vecv,\vecy) \,
d\vecv=q^{d-1} F_2(\vecy)$ for each 
$\vecm\in\Z^{d-1}\setminus\{\bn\}$, and hence
\eqref{HSIEGELBASICLEMMASTEP1} equals
\begin{align} \label{HSIEGELBASICLEMMASTEP2}
q^{d-1} \int_{\F} \Bigl( F_1(\alpha_1,\bn)+
\sum_{\vecm\in\Z^{d-1}\setminus\{\bn\}} 
F_2(\vecm M_1) \Bigr) \, d\mu_1(M_1).
\end{align}
The integrand in \eqref{HSIEGELBASICLEMMASTEP2}
only depends on the $\Gamma(1)$-coset of $M_1$,
i.e.\ the integration over $\F$ may be replaced by 
$I_q^{(d-1)}$ times an integral over $\Gamma(1)\backslash\SL(d-1,\R)
=X_1^{(d-1)}$; hence by Proposition~\ref{SIEGELQPROP} 
(applied for ``$d-1$'' and ``$q=1$'') we get
\begin{align} \label{HSIEGELGENANSWER}
=q^{d-1}I_q^{(d-1)}
\Bigl(F_1(\alpha_1,\bn)+\int_{\vecy\in\R^{d-1}} F_2(\vecy)\, d\vecy\Bigr),
\end{align}
which gives the formula \eqref{HSIEGELBASICLEMMAFORMULA}.

In the remaining case $d=2$ we obtain as before \eqref{HSIEGELBASICLEMMASTEP1}
and \eqref{HSIEGELBASICLEMMASTEP2},
but with the inner integration sign removed and instead just taking $M_1=1$ 
in the formulas. 
Now \eqref{HSIEGELBASICLEMMASTEP2} agrees with 
\eqref{HSIEGELBASICLEMMAFORMULA}, and we are done.
\end{proof}

\begin{proof}[Proof of Proposition \ref{XQYSIEGELPROP}.]
We first prove the absolute convergence. It will be clear from the
proof below that it suffices to prove that the right hand side
of \eqref{XQYSIEGELPROPFORMULA} is absolutely convergent. 
This is clear for the integral; thus we turn to the double sum.
Assume $|F(\vecx)|\leq B$ for all $\vecx\in\R^d$ and take $C>0$ such that
$F(r\vecy)=0$ whenever $r\leq -C$ or $r\geq C$ (for our given $\vecy\in\R^d
\setminus\{\bn\}$).
Then
\begin{align} \label{ABSCONVXQYSIEGELPROP}
\sum_{\substack{t\geq 1\\ (t,q)=1}}
t^{-d} \sum_{\substack{a\in t+q\Z \\ (a,t)=1}}
\Bigl | F\Bigl(\frac{a}{t}\vecy\Bigr) \Bigr |
\leq \sum_{t\geq 1}t^{-d} (1+2Cq^{-1}t) B.
\end{align}
This is obviously absolutely convergent, since $d\geq 3$.

We now prove the identity.
In view of \eqref{XQYDISJUNION} and \eqref{XQKYID}
the left hand side of \eqref{XQYSIEGELPROPFORMULA} decomposes as
\begin{align}
(I_q\zeta(d))^{-1}
\sum_{\veck\in S} \int_{(M_\veck\Gamma(q)\M_\veck^{-1}\cap H)\backslash H}
\sum_{\vecm\in\Z^d} F((\vecm+\vecalf)M_{\veck}^{-1}hM_\vecy) \, d\mu_H(h).
\end{align}
For each fixed $\veck\in S$ we now perform the same manipulations 
as in the proof of Proposition~\ref{XQVOLONEPROP}, just before
\eqref{XQVOLONEPROPSTEP1}; we thus get
(since $\Z^d\gamma^{-1}=\Z^d$ for every $\gamma\in\Gamma(1)$)
\begin{align}
=\frac{q^d}{I_q\zeta(d)} 
\sum_{\veck\in S} t_\veck^{-d} \int_{(\Gamma(q)\cap H)\backslash H}
\sum_{\vecm\in\Z^d} F((\vecm+\vecalf \gamma_\veck^{-1})hg_{t_\veck/q}^{-1}
M_\vecy) \, d\mu_H(h),
\end{align}
where $\gamma_\veck$ is any map in $\Gamma(1)$ with
$(q/t_\veck)\veck=\vece_1\gamma_\veck$.
Now note (for each $\veck\in S\subset \Z^d+\vecalf\setminus\{\bn\}$)
that $\vecalf\gamma_\veck^{-1}\in (\veck+\Z^d)\gamma_\veck^{-1}
=(t_\veck/q)\vece_1+\Z^d$. 
Hence Lemma \ref{HSIEGELBASICLEMMA} applies, giving
\begin{align}
& =\frac{q^{2d-1}I_q^{(d-1)}}{I_q^{(d)}\zeta(d)} 
\sum_{\veck\in S} t_\veck^{-d} \Bigl(
\sum_{\ell\in\Z} F\Bigl(\bigl(\ell+\frac{t_\veck}q\bigr)
\vece_1 g_{t_\veck/q}^{-1} M_\vecy \Bigr) + 
\int_{\vecx\in\R^d} F\bigl(\vecx g_{t_\veck/q}^{-1} M_\vecy\bigr) \, d\vecx
\Bigr)
\\ \notag
& =\frac{q^{2d-1}I_q^{(d-1)}}{I_q^{(d)}\zeta(d)} 
\sum_{\veck\in S} t_\veck^{-d} \Bigl(
\sum_{\ell\in\Z} F\Bigl(\bigl(\frac{\ell q}{t_\veck}+1\bigr)\vecy\Bigr)+
\int_{\vecx\in\R^d} F\bigl(\vecx\bigr) \, d\vecx \Bigr).
\end{align}
But we saw in the proof of Proposition \ref{XQVOLONEPROP} that when
$\veck$ runs through $S$ then $t_\veck$ visits each $t\in\Z_{>0}$ with
$(t,q)=1$ exactly once, and no other numbers. 
Also from that proof we have
$I^{(d)}_q=\#H(\Z/q\Z)\cdot \#V=
\bigl(q^{d-1} I_q^{(d-1)}\bigr) \cdot q^d\sum_{e\mid q} \mu(e)e^{-d}$,
and recall $\sum_{t\geq 1, \: (t,q)=1} t^{-d}
=\zeta(d)\sum_{e\mid q} \mu(e)e^{-d}$.
Hence we get
\begin{align}
& =\int_{\vecx\in\R^d} F\bigl(\vecx\bigr) \, d\vecx 
+\Bigl(\sum_{t\geq 1, \: (t,q)=1} t^{-d}\Bigr)^{-1}
\sum_{t\geq 1, \:(t,q)=1} t^{-d} 
\sum_{\ell\in\Z} F\Bigl(\bigl(\frac{\ell q}{t}+1\bigr)\vecy\Bigr).
\end{align}
In the last double sum we substitute $e=(\ell,t)$;
thus $\ell=e\ell_1$, $t=et_1$ with $(\ell_1,t_1)=1$, and the double sum
becomes
\begin{align}
& \sum_{\substack{e\geq 1\\ (e,q)=1}} \,
\sum_{\substack{t_1\geq 1\\ (t_1,q)=1}} (et_1)^{-d} 
\sum_{\substack{\ell_1\in\Z\\ (\ell_1,t_1)=1}} 
F\Bigl(\bigl(\frac{\ell_1 q}{t_1}+1\bigr)\vecy\Bigr)
=\Bigl(\sum_{\substack{e\geq 1\\ (e,q)=1}} e^{-d}\Bigr)
\sum_{\substack{t_1\geq 1\\ (t_1,q)=1}} t_1^{-d} 
\sum_{\substack{a\in t_1+q\Z\\ (a,t_1)=1}} 
F\Bigl(\frac{a}{t_1}\vecy\Bigr).
\end{align}
Hence we obtain the desired formula.
\end{proof}

We next turn to the case $d=2$. In this case the integral in the left hand side of
\eqref{XQYSIEGELPROPFORMULA} typically \textit{diverges}.
This is e.g.\ true for every continuous function $F\geq 0$
which is not identically zero along the line $\R\vecy$,
as is seen by following the proof of Proposition \ref{XQYSIEGELPROP}
and noting that the sum
$\sum_{\substack{t\geq 1\\ (t,q)=1}}t^{-2}
\sum_{\substack{a\in t+q\Z \\ (a,t)=1}}
F\Bigl(\frac{a}{t}\vecy\Bigr)$ now diverges.
However we can prove that the integral in the left hand side of
\eqref{XQYSIEGELPROPFORMULA} is finite if
$X_q(\vecy)$ is replaced by any subset 
\begin{align} \label{XQT0DEF}
X_q^{(t_0)}(\vecy):=
\bigsqcup_{\veck\in S; \, t_\veck\leq t_0} X_q(\veck,\vecy),
\qquad (t_0\in\Z_{>0}).
\end{align}
\begin{prop} \label{XQYSIEGELPROPD2}
Let $\vecalf=\vecp/q$ with $\vecp=(p_1,p_2)\in\Z^2$
and $\gcd(q,p_1,p_2)=1$.
Let $\vecy\in\R^2\setminus\{\bn\}$ and let
$\widetilde{\vecy}\in\R^2$ be any of the two vectors orthogonal to $\vecy$ with
$\|\widetilde{\vecy}\|=\|\vecy\|^{-1}$.
Let $F:\R^2\to\R$ be a non-negative, bounded measurable function 
of compact support such that $F(x\vecy+u\widetilde{\vecy})$ is a measurable function of
$x\in\R$ for each fixed $u\in\R$.
Then for any $t_0\in\Z_{>0}$ we have
\begin{align} \label{XQYSIEGELPROPD2BOUND}
& \int_{X_q^{(t_0)}(\vecy)} \sum_{\vecm\in\Z^2} F((\vecm+\vecalf)M) \,
d\nu_\vecy(M)
\\ \notag
& \leq \sum_{\substack{1\leq t\leq t_0\\ (t,q)=1}}
t^{-2} \sum_{\substack{a\in t+q\Z \\ (a,t)=1}}
F\Bigl(\frac{a}{t}\vecy\Bigr)
+q^{-1} %
\sum_{v\in\Z\setminus\{0\}} \Bigl(\sum_{\substack{t\mid v\\ (t,q)=1}} 
t^{-1}\Bigr)
\int_\R F\Bigl(x\vecy + \frac{v}q \widetilde{\vecy}\Bigr)\, dx.
\end{align}
\end{prop}

\begin{proof}
This follows by imitating the proof of Proposition \ref{XQYSIEGELPROP} but 
noting the special form of Lemma \ref{HSIEGELBASICLEMMA} when $d=2$,
and using the restriction $t_\veck\leq t_0$ from \eqref{XQT0DEF}
in the treatment of the $\ell$-sum from \eqref{HSIEGELBASICLEMMAFORMULA}.
When treating the constant factor in front of the (``new'') second term,
one uses the fact that $\sum_{(t,q)=1} t^{-2}>1$.
\end{proof}

\subsection{The submanifolds $X(\vecy)$ of $X$} \label{XYSEC}

These are analogous to the submanifolds $X_q(\vecy)$ of $X_q$, but we will 
see that many details are quite a bit simpler.
Given any $\vecy\in\R^d$ we define
\begin{align} \label{XYDEF}
X(\vecy):=\bigl\{g\in X \col \vecy\in \Z^dg\bigr\}.
\end{align}
We will write $\Gamma=\ASL(d,\Z)$ throughout this section. Since $\Z^d=\bn \Gamma$
we actually have
\begin{align}
X(\vecy)=\bigl\{\Gamma g \col g \in \ASL(d,\R),\: \bn g=\vecy\bigr\}
=\bigl\{\Gamma (M,\vecy) \col M\in\SL(d,\R)\bigr\}.
\end{align}
Furthermore one checks that $M_1,M_2\in\SL(d,\R)$ give the same point
$\Gamma(M_1,\vecy)=\Gamma(M_2,\vecy)$ in $X$ if and only if
$\SL(d,\Z)M_1=\SL(d,\Z)M_2$. Hence we get an identification of the sets
$X(\vecy)$ and $X_1=\SL(d,\Z)\backslash\SL(d,\R)$, through
\begin{align} \label{XYID}
X(\vecy)=\bigl\{(M,\vecy)\col M\in X_1\bigr\}.
\end{align}
This gives $X(\vecy)$ the structure of an %
embedded submanifold of $X$, of dimension $d^2-1$.
\textit{We endow $X(\vecy)$ with the Borel probability measure 
$\nu_\vecy$\label{NUYDEF2}
which comes from $\mu_1$ on $X_1$ under the identification \eqref{XYID}.}
Hence, automatically, $\nu_\vecy(X(\vecy))=1$.

\begin{lem} \label{XYCANMFORMINVARIANCE}
For any $\vecy\in\R^d$, $h\in\ASL(d,\R)$ and
any Borel subset $\scrE\subset X(\vecy)$ we have
$\nu_\vecy(\scrE)=\nu_{\vecy h}(\scrE h)$.
\end{lem}
\begin{proof}
This follows easily using the fact that $\mu_1$ is invariant under the
(right) action of $\SL(d,\R)$ on $X_1$.
\end{proof}

\begin{prop} \label{XYFOLINTPROP}
Let $\scrE\subset X$ be any Borel set; then 
$\vecy\mapsto \nu_\vecy(\scrE \cap X(\vecy))$
is a measurable function from $\R^d$ to $\R$.
If $U\subset\R^d$ is any Borel set
such that $\scrE\subset \bigcup_{\vecy\in U} X(\vecy)$,
then
\begin{align} \label{XYFOLINTGENREL}
\mu(\scrE) \leq \int_U \nu_\vecy(\scrE \cap X(\vecy)) \, d\vecy.
\end{align}
Furthermore, if 
$\forall \vecy_1\neq \vecy_2\in U: X(\vecy_1)\cap X(\vecy_2) \cap \scrE
=\emptyset$, then \textrm{equality} holds in \eqref{FOLINTGENREL}.
\end{prop}

\begin{proof}
Let $\F_1\subset\SL(d,\R)$ be a (measurable) fundamental domain for
$\Gamma(1)\backslash\SL(d,\R)$, in the set theoretic sense.
Then
\begin{align} \label{ASLFUNDDOM}
\F:=\bigl\{(M,\vecxi)\col M\in\F_1,\: \vecxi\in [0,1)^dM\bigr\}
\end{align}
is a fundamental domain for $\Gamma\backslash\ASL(d,\R)$.
Now by the definition of $\nu_\vecy$ we have for each $\vecy\in\R^d$,
\begin{align} \label{XYFOLINTPROPSTEP2}
\nu_\vecy(\scrE\cap X(\vecy))
=\int_{\F_1} \chi_{_{\scrE_0}} (M,\vecy) \: d\mu_1(M)
=\int_{\SL(d,\R)} \chi_{_{(\F_1 \times \R^d)\cap\scrE_0}}
 (M,\vecy) \: d\mu_1(M),
\end{align}
where $\scrE_0$ denotes the pre-image in $\ASL(d,\R)$ of $\scrE\subset X$.
But the set $(\F_1 \times \R^d)\cap\scrE_0$ is $\mu$-measurable,
and recall from \eqref{ASLDRHAAR} that
$d\mu(M,\vecxi)=d\mu_1(M) d\vecxi$; hence by Fubini's theorem,
\eqref{XYFOLINTPROPSTEP2} shows that 
$\vecy\mapsto \nu_\vecy(\scrE\cap X(\vecy))$ is a measurable function
with respect to the Lebesgue measure on $\R^d$.

Next, to prove \eqref{XYFOLINTGENREL} we note that 
\eqref{XYFOLINTPROPSTEP2} implies
\begin{align}
\int_U \nu_\vecy(\scrE \cap X(\vecy)) \, d\vecy
=\int_{\R^d} \int_{\SL(d,\R)} \chi_{_{(\F_1\times U)\cap \scrE_0}}(M,\vecy) 
\: d\mu_1(M) \, d\vecy=\mu\bigl((\F_1\times U) \cap \scrE_0\bigr),
\end{align}
where we again used \eqref{ASLDRHAAR} in the last step.
But it follows from our assumption 
$\scrE\subset \bigcup_{\vecy\in U} X(\vecy)$
that each point in $\scrE\subset X$ has at least one representative in
$(\F_1\times U) \cap \scrE_0\subset\SL(d,\R)$.
Hence $\mu\bigl((\F_1\times U) \cap \scrE_0\bigr)\geq \mu(\scrE)$
and \eqref{XYFOLINTGENREL} is proved.
To prove the final statement about equality, note that
the condition
$\forall \vecy_1\neq \vecy_2\in U: X(\vecy_1)\cap X(\vecy_2) \cap \scrE
=\emptyset$ implies that each point in $\scrE$ has \textit{exactly}
one representative in $(\F_1\times U) \cap \scrE_0$,
and thus $\mu\bigl((\F_1\times U) \cap \scrE_0\bigr)=\mu(\scrE)$.
\end{proof}

\begin{prop} \label{XYSIEGELPROP}
If $F\in \L^1(\R^d)$ and $\vecy\in\R^d$ then
\begin{align}
\int_{X(\vecy)} \sum_{\vecm\in\Z^d} F(\vecm g) \, d\nu_\vecy(g)
=F(\vecy)+\int_{\R^d} F(\vecx)\,d\vecx.
\end{align}
\end{prop}

\begin{proof}
This follows directly from \eqref{XYID} and Proposition \ref{SIEGELQPROP}
(with $\vecalf=\bn$, $q=1$).
\end{proof}

\subsection{A thin region seldom contains an extra lattice point}
It will be important for our applications of Proposition \ref{FOLINTPROP} 
and Proposition \ref{XYFOLINTPROP} 
to know
that if a bounded set $U\subset \R^d$ is 
thin in at least one direction (i.e. contained between two
parallel hyperplanes close together) then a random lattice with a 
vertex in $U$ is unlikely to have \textit{another} vertex in $U$.
Precisely, we will need an upper bound on the integral in
\eqref{FEWWITHTWOLEMMA2BOUND} below. 
Since this integral is obviously monotone with respect to the set $U$,
it suffices to consider the case when $U$ is a translate of 
a cylinder $\fZ(c_1,c_2,\sigma)$ (cf.\ \eqref{FZC1C2DEF})
with $c_2-c_1$ small.

\begin{lem} \label{FEWWITHTWOLEMMA2}
Assume $\vecalf\in q^{-1}\Z^d$, fix $C>1$ and write $U=\vecz+\fZ(c_1,c_2,C)$.
Then if $d\geq 3$ we have
\begin{align} \label{FEWWITHTWOLEMMA2BOUND}
\int_U \nu_\vecy \Bigl(\Bigl\{ M\in X_q(\vecy) \col
\#\bigl(U \cap (\Z^d+\vecalf) M\bigr) \geq 2\Bigr\}\Bigr) 
\, d\vecy \ll (c_2-c_1)^2,
\end{align}
uniformly over all $\vecz\in\{0\}\times\R^{d-1}$ and $C^{-1}\leq c_1<c_2$.
If $d=2$ then the same integral is
\begin{align} \label{FEWWITHTWOLEMMA2BOUNDD2}
\ll (c_2-c_1)^2 \log \bigl(2+(c_2-c_1)^{-1}\bigr),
\end{align}
uniformly over all $\vecz\in\{0\}\times [-C,C]$ and
$C^{-1}\leq c_1<c_2\leq C$.
(In the first bound the implied constant depends only on $C,d$;
in the second bound it depends only on $C,q$.)
\end{lem}

\begin{proof}
Just as in the proof of Proposition \ref{XQVOLONEPROP} we may assume
$\gcd(q,p_1,\ldots,p_d)=1$, without loss of generality.
Take $\vecz\in\{0\}\times\R^{d-1}$, $C^{-1}\leq c_1<c_2$
and let $U=\vecz+\fZ(c_1,c_2,C)$. 
For each $\vecy\in U$ and $M\in X_q(\vecy)$ we have
$\sum_{\vecm\in\Z^d} \chi_U((\vecm+\vecalf)M)\geq 1$
by the definition of $X_q(\vecy)$,
and the same sum is $\geq 2$ whenever $\#(U\cap (\Z^d+\vecalf)M)\geq 2$.
Hence, using also $\nu_\vecy(X_q(\vecy))=1$
(see Proposition \ref{XQVOLONEPROP}), we have for each $\vecy\in U$,
\begin{align} \label{FEWWITHTWOSTEP1}
\nu_\vecy \Bigl(\Bigl\{ M\in X_q(\vecy) \col
\#\bigl(U \cap (\Z^d+\vecalf) M\bigr) \geq 2\Bigr\}\Bigr) 
\leq
-1+\int_{X_q(\vecy)} \sum_{\vecm\in\Z^d} \chi_U((\vecm+\vecalf)M) 
\, d\nu_\vecy(M).
\end{align}
If $d\geq 3$ then this is, by Proposition \ref{XQYSIEGELPROP},
\begin{align} \label{FEWWITHTWOSTEP2}
&= -1+\vol(U)+\sum_{\substack{t\geq 1\\(t,q)=1}} t^{-d} 
\sum_{\substack{a\in t+q\Z\\(a,t)=1}}
\chi_U\Bigl(\frac at \vecy\Bigr)
\leq O\bigl(c_2-c_1\bigr)
+\sum_{t=2}^\infty t^{-d} \sum_{\substack{a\in\Z\\(a,t)=1}}
\chi_U\Bigl(\frac at \vecy\Bigr),
\end{align}
where in the second step we used
$C^{-1}\leq c_1<c_2$ to get $\sum_{a\in\Z}\chi_U(a\vecy)\leq 1+O(c_2-c_1)$.
If some $t\geq 2$ gives non-vanishing contribution to the last
sum then we must have $\frac at \vecy\in U$ either for $a=t+1$ or $a=t-1$.
In the first case it follows that
$\frac{t+1}t c_1< \frac{t+1}t y_1< c_2$ so that
$t> \frac{c_1}{c_2-c_1}$; in the second case it follows that
$\frac{t-1}t c_2> \frac{t-1}t y_1> c_1$ so that
$t> \frac{c_2}{c_2-c_1}$.
Hence all $t$-values which contribute to the sum must satisfy
$t> \frac{c_1}{c_2-c_1}$.
For each such $t$, a given $a\in\Z$ gives non-vanishing contribution only if
$\frac at y_1 < c_2$ (implying $a< t(c_2/c_1)$)
and $\frac at y_1 > c_1$ (implying $a> t(c_1/c_2)$);
the number of such $a$'s is $\leq \#\bigl(\Z \cap \bigl(t\frac{c_1}{c_2},t\frac{c_2}{c_1}\bigr)\bigr)
\leq 1+t(\frac{c_2}{c_1}-\frac{c_1}{c_2})$; hence the sum in 
\eqref{FEWWITHTWOSTEP2} is
\begin{equation}
\leq \sum_{t\geq \max(2,\frac{c_1}{c_2-c_1})} t^{-d} 
\Bigl(1+t\bigl(\frac{c_2}{c_1}-\frac{c_1}{c_2}\bigr)\Bigr)
= \sum_{t\geq \max(2,x^{-1})} t^{-d} 
\Bigl(1+tx\frac{2+x}{1+x}\Bigr)
\end{equation}
where $x=\frac{c_2}{c_1}-1$.
Hence if $x\leq \sfrac 12$ then the full expression in 
\eqref{FEWWITHTWOSTEP2} is (when $d\geq 3$)
\begin{equation}
\leq O(c_2-c_1)
+O\bigl(x^{d-1}\bigr)
+O\bigl(x^{d-1}\bigr)
=O\Bigl(c_2-c_1+\Bigl( \frac{c_2-c_1}{c_1}\Bigr)^{d-1}\Bigr),
\end{equation}
whereas if $x>\frac 12$ we get 
\begin{equation}
\leq O(c_2-c_1)+O(1+x)
=O\Bigl(c_2-c_1+\frac{c_2}{c_1}\Bigr).
\end{equation}
Using  $C^{-1}\leq c_1<c_2$ the above is
$\leq O(c_2-c_1)$, in both cases.
Hence we have proved
\begin{align}
\nu_\vecy \Bigl(\Bigl\{ M\in X_q(\vecy) \col
\#\bigl(U \cap (\Z^d+\vecalf) M\bigr) \geq 2\Bigr\}\Bigr) 
\leq O(c_2-c_1),\qquad \forall \vecy\in U,
\end{align}
where the implied constant depends only on $C$ and $d$.
Since this bound is uniform over $\vecy\in U$ we obtain 
\eqref{FEWWITHTWOLEMMA2BOUND} by integration.

We now turn to the case $d=2$.
We take $\vecz\in \{0\}\times [-C,C]$ and $C^{-1}\leq c_1<c_2\leq C$.
Take $t_0\in \Z_{\geq 10}$, to be fixed later.
Recall the definition of $X_q^{(t_0)}(\vecy)$, \eqref{XQT0DEF}.
The left hand side in \eqref{FEWWITHTWOSTEP1} is
\begin{align} \label{FEWWITHTWOSTEP1D2}
\leq \int_{X^{(t_0)}_q(\vecy)} 
\Bigl(-1+\sum_{\vecm\in\Z^2} \chi_U((\vecm+\vecalf)M)\Bigr) \, d\nu_\vecy(M)
+\vol\Bigl(X_q(\vecy)\setminus X_q^{(t_0)}(\vecy)\Bigr).
\end{align}
Imitating the proof of Proposition \ref{XQVOLONEPROP} one shows that
the last volume is $\ll t_0^{-1}$.
Hence by Proposition \ref{XQYSIEGELPROPD2} the above is
\begin{align} \label{FEWWITHTWOSTEP2D2}
\leq -1+O(t_0^{-1})+
\sum_{1\leq t \leq t_0} t^{-2} \sum_{\substack{a\in\Z\\ (a,t)=1}}
\chi_U\Bigl(\frac at \vecy\Bigr)
+%
q^{-1}\sum_{v\in\Z\setminus\{0\}} \Bigl(\sum_{t\mid v} t^{-1}\Bigr)
\int_\R \chi_U\Bigl( x\vecy+\frac vq \widetilde{\vecy}\Bigr)\, dx.
\end{align}
Arguing along the same lines as before we find, with $x=\frac{c_2}{c_1}-1$,
\begin{align}
-1+\sum_{1\leq t \leq t_0} t^{-2} \sum_{\substack{a\in\Z\\ (a,t)=1}}
\chi_U\Bigl(\frac at \vecy\Bigr)
\leq O(c_2-c_1)+\sum_{\max(2,x^{-1})\leq t\leq t_0} t^{-2} 
\Bigl(1+tx\frac{2+x}{1+x}\Bigr) & 
\\ \notag
\leq O((c_2-c_1)\log t_0) & .
\end{align}

Finally we treat the last sum in \eqref{FEWWITHTWOSTEP2D2}.
Let $L=\sqrt{4C^2+(c_2-c_1)^2}$, the length of the diagonal of $U$.
If $\|\frac vq \widetilde{\vecy}\|>L$
then $\int_\R \chi_U\Bigl( x\vecy+\frac vq \widetilde{\vecy}\Bigr)\, dx=0$
for all $\vecy\in U$.
Hence only $v\in \Z\setminus\{0\}$ with
$|v|\leq Lq\|\widetilde{\vecy}\|^{-1}= Lq\|\vecy\|$ give contributions
to the last sum in \eqref{FEWWITHTWOSTEP2D2},
and since $\|\vecy\|<c_2+\|\vecz\|+C\leq 3C$ and $L\leq \sqrt 5 C$
it follows that these $v$'s are bounded in absolute value by a constant
which only depends on $C,q$.
Hence the last sum in \eqref{FEWWITHTWOSTEP2D2} is
\begin{align} \label{FEWWITHTWOSTEP3D2}
\leq O(1) \sum_{v\in\Z\setminus\{0\}}
\int_\R \chi_U\Bigl( x\vecy+\frac vq \widetilde{\vecy}\Bigr)\, dx.
\end{align}
Now for each $v\in\Z_{>0}$ for which the integral is non-zero, there exists some 
$x'\in\R$ such that $x'\vecy+\frac{v-1}q \widetilde{\vecy}\in U$
(since $\vecy\in U$ and $U$ is convex); hence if the contribution from our
$v$ equals $\int_{x_1}^{x_2} \, dx$ then %
$U$ must contain the triangle with vertices
$x'\vecy+\frac{v-1}q \widetilde{\vecy}$,
$x_1\vecy+\frac{v}q \widetilde{\vecy}$ and
$x_2\vecy+\frac{v}q \widetilde{\vecy}$, which has area
$\frac 12 (x_2-x_1)\|\vecy\| \cdot \frac 1q \|\widetilde{\vecy}\|
=\frac 1{2q} (x_2-x_1)$.
Note also that distinct $v$'s %
lead to pairwise disjoint triangles inside $U$; 
hence the total contribution in \eqref{FEWWITHTWOSTEP3D2} from positive
$v$'s is $\leq 2q \text{Area}(U)$.
Similarly for the negative $v$'s.
Combining our bounds we have now proved that for each $\vecy\in U$
the left hand side in \eqref{FEWWITHTWOSTEP1} is
\begin{align}
\leq O\bigl(t_0^{-1}+(c_2-c_1)\log t_0\bigr).
\end{align}
Choosing $t_0=\max(10,[(c_2-c_1)^{-1}])$
and integrating over all $\vecy\in U$
we obtain the bound \eqref{FEWWITHTWOLEMMA2BOUNDD2}.
\end{proof}

The corresponding bound in the case $\vecalf\notin\Q^d$ is as follows:

\begin{lem} \label{XYFEWWITHTWOLEMMA}
Let $d\geq 2$ and $C>1$ and write $U=\vecz+\fZ(c_1,c_2,C)$.
Then 
\begin{align} \label{XYFEWWITHTWOLEMMABOUND}
\int_U \nu_\vecy \Bigl(\Bigl\{ g\in X(\vecy) \col
\#\bigl(U \cap \Z^d g\bigr) \geq 2\Bigr\}\Bigr) 
\, d\vecy \ll (c_2-c_1)^2,
\end{align}
uniformly over all $\vecz\in\{0\}\times\R^{d-1}$ and $c_1<c_2$.
(The implied constant depends only on $C$, $d$.)
\end{lem}

\begin{proof}
This follows by arguing as in the first part of the proof of 
Lemma \ref{FEWWITHTWOLEMMA2} (up until \eqref{FEWWITHTWOSTEP2}) but
using Proposition \ref{XYSIEGELPROP}
in place of Proposition \ref{XQYSIEGELPROP}.
\end{proof}

\section{Properties of the limit functions}\label{PROPLIMFCNSEC}

\subsection{An important volume function, for $\vecalf\in\Q^d$} 
\label{IMPORTANTVOLUMESEC}
In this section we will prove some ``quasi-continuity'' properties of
the limit function $\Phi_\vecalf(\xi,\vecw,\vecz)$ in Theorem \ref{exactpos1},
and for some more general functions.
These considerations will be of importance for the proof of 
Theorem~\ref{exactpos1}.

Given $r\in\Z_{\geq 0}$ and $\vecalf \in q^{-1}\Z^d$
we introduce the function
\begin{align} \label{FR5VARDEF}
f_r(c_1,c_2,\sigma,\vecz,\vecy)
:=\nu_\vecy\bigl(\bigl\{M\in X_q(\vecy) \col 
\#\bigl((\Z^d+\vecalf)M \cap (\fZ(c_1,c_2,\sigma)+\vecz)\bigr)
=r\bigr\}\bigr)
\end{align}
in the domain
\begin{align} \label{OMEGADEF}
\Omega=\bigl\{\langle c_1,c_2,\sigma,\vecz,\vecy\rangle \in \R\times 
\R\times \R \times \bigl(\{0\}\times\R^{d-1}\bigr) \times \R^d \col
0\leq c_1<c_2\leq y_1, \: 0\leq \sigma \bigr\}.
\end{align}
Arguing as in the first paragraph of the 
proof of Proposition \ref{XQVOLONEPROP} 
we see that although the function $f_r$ depends on the given vector
$\vecalf\in\Q^d$, it does
not depend on the choice of denominator $q$ of $\vecalf$;
hence from now on in this section we will always assume that $q$ is
the minimal denominator of $\vecalf$, so that Propositions
\ref{XQYSIEGELPROP}, \ref{XQYSIEGELPROPD2} apply.
We also write, for $\xi>0$ and $\vecz,\vecw \in \{0\}\times\R^{d-1}$,
\begin{align} \label{FR3VARDEF}
F_r(\xi,\vecw,\vecz):=f_r(0,\xi,1,\vecz,\xi\vece_1+\vecw+\vecz).
\end{align}
Thus the function $\Phi_\vecalf(\xi,\vecw,\vecz)$ %
in Theorem \ref{exactpos1} is the same as $F_0(\xi,\vecw,\vecz)$.

\begin{lem} \label{F5SYMMLEMMA}
For any $K=\begin{pmatrix} 1 & \bn \\ \trans\bn & K_1 \end{pmatrix}\in \O(d)$ 
we have 
\begin{align}
f_r(c_1,c_2,\sigma,\vecz K,\vecy K)=f_r(c_1,c_2,\sigma,\vecz,\vecy),
\end{align}
and for any $\delta>0$ we have
\begin{align}
f_r(c_1\delta^{d-1},c_2\delta^{d-1},\sigma\delta^{-1},\delta^{-1}\vecz,
\delta^{-1}\vecy)=f_r(c_1,c_2,\sigma,\vecz,\vecy).
\end{align}
\end{lem}
\begin{proof}
If $K_1\in\SO(d-1)$ then the first claim follows immediately from 
Lemma \ref{CANMFORMINVARIANCE} with $T=K$,
using $\fZ(c_1,c_2,\sigma) K^{-1}=\fZ(c_1,c_2,\sigma)$.
Similarly the second claim follows from
Lemma \ref{CANMFORMINVARIANCE} using 
\begin{align} \label{FZC1C2SIGMAINV}
\fZ(c_1,c_2,\sigma) T_\delta^{-1}=
\fZ(c_1\delta^{d-1},c_2\delta^{d-1},\sigma\delta^{-1}),
\qquad \text{for } \:
T_\delta:=\diag[\delta^{1-d},\delta,\ldots,\delta].
\end{align}
To extend the first claim to general $K_1\in\O(d-1)$ it now suffices to treat
the single case $K=K_0:=\diag[1,\ldots,1,-1]$.
Fix some $\gamma\in\SL(d,\Z)$ such that $\vecalf \gamma K_0=\vecalf$,
and thus $(\Z^d+\vecalf)\gamma K_0=\Z^d+\vecalf$.
Then $a:M\mapsto \gamma K_0 M K_0$ gives a well-defined diffeomorphism from
$X_q$ onto $X_q$, and one checks by a straightforward computation 
that for any Borel subset
$\scrE\subset X_q(\vecy)$ we have $a(\scrE)\subset X_q(\vecy K_0)$ and
$\nu_\vecy(\scrE)=\nu_{\vecy K_0}(a(\scrE))$.
Applying this with 
$\scrE=\{M\in X_q(\vecy) \col 
\#((\Z^d+\vecalf)M \cap (\fZ(c_1,c_2,\sigma)+\vecz))=r\}$ we get
$f_r(c_1,c_2,\sigma,\vecz,\vecy)=
f_r(c_1,c_2,\sigma,\vecz K_0,\vecy K_0)$, as desired.
\end{proof}

\begin{remark}\label{invarem}
It follows that $F_r(\xi,\vecw K,\vecz K)=F_r(\xi,\vecw,\vecz)$
for all $K$ as in the lemma,
and hence $F_r(\xi,\vecw,\vecz)$ only depends on the four real
numbers $\xi,\|\vecz\|,\|\vecw\|,\vecz\cdot\vecw$.
\end{remark}

We will now prove our main technical result about 
$f_r(c_1,c_2,\sigma,\vecz,\vecy)$ being not too far from continuous.
For $N\in\Z_{\geq 2}$ we let $\mathfrak{F}_N$ be the set of rational
numbers strictly between $0$ and $1$ and with denominator $\leq N$,
that is,
\begin{equation}\label{FsubN}
\mathfrak{F}_N=\bigl\{\sfrac hk \col h,k \in \Z, \: 0<h<k\leq N\bigr\}.
\end{equation}
Given $\langle c_1,c_2,\sigma,\vecz,\vecy\rangle \in \Omega$
and $\delta\in\mathfrak{F}_N$ we define
\begin{align} \label{SDELTADEF}
s(\delta)=s_{\langle c_1,c_2,\sigma,\vecz,\vecy\rangle}(\delta)
=\begin{cases} 1 & \text{if } \:
\vecy\in \delta^{-1}(\vecz+\fZ(c_1,c_2,\sigma))
\\
0 & \text{if } \:
\vecy\notin \delta^{-1}(\vecz+\fZ(c_1,c_2,\sigma)). \end{cases}
\end{align}
For $C>1$ we write
\begin{align} \label{OMEGACDEF}
\Omega_C:=\begin{cases}
\{\langle c_1,c_2,\sigma,\vecz,\vecy\rangle \in \Omega
\:\col\: \sigma,\|\vecz\|,\|\vecy\| \leq C;\:\: 
C^{-1}\leq |y_1|,|y_2|\} & \text{if } \: d=2
\\
\{\langle c_1,c_2,\sigma,\vecz,\vecy\rangle \in \Omega
\:\col\: \sigma,\|\vecz\|,\|\vecy\| \leq C;\:\: 
C^{-1}\leq \|\vecy\|\} & \text{if } \: d\geq 3. \end{cases}
\end{align}

\begin{prop} \label{CONT3DPROP}
Fix $d\geq 2$ and $r\in \Z_{\geq 0}$.
Given $C>1$ and $\ve>0$ there exist some $\eta>0$
and $N\in \Z_{\geq 2}$ such that 
\begin{equation}
\Bigl | f_r(c_1,c_2,\sigma,\vecz,\vecy)-f_r(c'_1,c'_2,\sigma',\vecz',\vecy')\Bigr |
\leq \ve
\end{equation}
holds for all 
$\langle c_1,c_2,\sigma,\vecz,\vecy\rangle,
\langle c'_1,c'_2,\sigma',\vecz',\vecy'\rangle\in\Omega_C$
satisfying $|c_1-c'_1|\leq \eta$, $|c_2-c'_2|\leq \eta$,
$|\sigma-\sigma'|\leq \eta$,
$\|\vecz-\vecz'\|\leq \eta$,
$\|\vecy-\vecy'\|\leq \eta$
and $s_{\langle c_1,c_2,\sigma,\vecz,\vecy\rangle}(\delta)=
s_{\langle c'_1,c'_2,\sigma',\vecz',\vecy'\rangle}(\delta)$
for all $\delta\in\Far_N$.
\end{prop}

\begin{proof}
For certain technical statements in the following proof 
to be correct we need to introduce the
notation $\widetilde{\fZ}(c_1,c_2,\sigma):=\{x_1\vece_1 \col c_1<x_1<c_2\}$ \label{wideZ}
when $\sigma=0$, but $:=\fZ(c_1,c_2,\sigma)$ when $\sigma>0$.
Let $C>1$ and $\ve>0$ be given.
If $d\geq 3$ then we choose $0<\eta_1<1$ so small that
$\vol(\overline{\scrB_{\eta_1}^d}+\partial \widetilde{\fZ}(c_1,c_2,\sigma))
\leq \sfrac{\ve}2$
for all $\langle c_1,c_2,\sigma,\vecz,\vecy\rangle\in\Omega_C$
(this is possible since
$\langle c_1,c_2,\sigma,\vecz,\vecy\rangle\in\Omega_C$ 
implies $0\leq c_1< c_2\leq C$ and $\sigma\leq C$);
if $d=2$ then we instead set 
$\eta_1=\min\bigl(1,
\ve/(20C\sum_{1\leq |v|\leq 4C^2q} \sum_{t\mid v} t^{-1})\bigr)$.
We will denote by $\|A\|$ the operator norm of any $d\times d$ matrix 
$A$, viz.\ $\|A\|=\sup_{\vecv\in\S^{d-1}_1} \|\vecv A\|.$
Take $\eta\in\bigl(0,\min(\sfrac{\eta_1}{10},\sfrac 1C)\bigr)$ so small 
that $\|M_\vecw^{(0)}-I\|\leq\sfrac{\eta_1}{40C}$
for all $\vecw\in\vece_1+\overline{\scrB_{C\eta}^d}$,
where $M_\vecw^{(0)}$ is as in \eqref{M0YDEF}.
If $d\geq 3$ we take $N$ so large that 
$\sum_{t\geq N} t^{1-d}<\frac{\ve}2$;
if $d=2$ we take $N$ so large that
$\nu_\vecy\bigl(
X_q(\vecy)\setminus X_q^{(N)}(\vecy)\bigr)\leq \frac{\ve}{2}$
(cf.\ \eqref{XQT0DEF}).

Let $\langle c_1,c_2,\sigma,\vecz,\vecy\rangle,
\langle c'_1,c'_2,\sigma',\vecz',\vecy'\rangle$
be any two points satisfying all assumptions in the proposition,
for our fixed $\eta,N$.
Then $\|\vecy-\vecy'\|\leq\eta\leq C\eta \|\vecy\|$, and hence by our 
choice of $\eta$ we can find some $T\in\SL(d,\R)$ such that
\begin{align}
\vecy'=\vecy T\quad\text{and}\quad \|T-I\|\leq \sfrac{\eta_1}{40C}
\qquad (<\sfrac 1{40})
\end{align}
(namely: let $T=K^{-1} M_{\|\vecy\|^{-1}\vecy' K^{-1}}^{(0)} K$
for some $K\in\SO(d)$ with $\vecy=\|\vecy\|\vece_1 K$).
Then also $\|T^{-1}-I\|\leq \sfrac{\|T-I\|}{1-\|T-I\|}<
\sfrac{\eta_1}{39C}$.
Hence, since the constraints in $\Omega_C$ imply that
$\vecz+\fZ(c_1,c_2,\sigma)$ is contained in $\scrB_{3C}^d$,
we have:
\begin{align} \label{TMOVESSHORT}
\|\vecx T -\vecx\| < \sfrac{\eta_1}{10}; \qquad
\|\vecx T^{-1} -\vecx\| < \sfrac{\eta_1}{10}, \qquad
\forall \vecx\in\vecz+\fZ(c_1,c_2,\sigma)
\end{align}
(and similarly for $\vecz'+\fZ(c_1',c_2',\sigma')$).

Now by Lemma \ref{CANMFORMINVARIANCE} we have
\begin{align} \notag
f_r(c'_1,c'_2,\sigma',\vecz',\vecy')
&=\nu_{\vecy}\bigl(\bigl\{M\in X_q(\vecy) \col 
\#\bigl((\Z^d+\vecalf)MT
\cap (\vecz'+\fZ(c'_1,c'_2,\sigma'))\bigr)
=r\bigr\}\bigr)
\\
& =\nu_{\vecy}\bigl(\bigl\{M\in X_q(\vecy) \col 
\#\bigl((\Z^d+\vecalf)M
\cap (\vecz'+\fZ(c'_1,c'_2,\sigma'))T ^{-1}\bigr)
=r\bigr\}\bigr),
\end{align}
and hence
\begin{align} \label{CONTUPPERBOUND}
\Bigl | f_r(c_1,c_2,\sigma,\vecz,\vecy)-f_r(c'_1,c'_2,\sigma',\vecz',\vecy')\Bigr |
\leq \nu_{\vecy}\bigl(\bigl\{M\in X_q(\vecy) \col 
\#\bigl((\Z^d+\vecalf)M
\cap U \bigr)\geq 1\bigr\}\bigr),
\end{align}
where $U$ is the symmetric set difference
\begin{align}
U=(\vecz'+\fZ(c'_1,c'_2,\sigma'))T^{-1} \: \triangle\:
(\vecz+\fZ(c_1,c_2,\sigma)).
\end{align}
But %
\eqref{CONTUPPERBOUND} is $\leq \int_{X_q(\vecy)} \sum_{\vecm\in \Z^d} 
\chi_U\bigl((\vecm+\vecalf) M\bigr) \, d\nu_{\vecy}$,
and by Propositions \ref{XQYSIEGELPROP}, \ref{XQYSIEGELPROPD2} this is
\begin{align} \label{BOUNDFROMSIEGELPROP}
& \text{if $d\geq 3$:} \quad
\leq \vol(U)+\sum_{t=1}^\infty t^{-d} \sum_{a\in\Z} 
\chi_U\Bigl(\frac at \vecy\Bigr);
\\ \label{BOUNDFROMSIEGELPROPD2}
& \text{if $d=2$:} \quad
\leq \frac{\ve}2+\sum_{t=1}^N t^{-d} \sum_{a\in\Z} 
\chi_U\Bigl(\frac at \vecy\Bigr)+
q^{-1}\sum_{v\in\Z\setminus\{0\}} \bigl(\sum_{t\mid v} t^{-1} \bigr)
\int_\R \chi_U\Bigl(x\vecy + \frac{v}q \widetilde{\vecy}\Bigr)\, dx.
\end{align}

We now claim that
\begin{align} \label{CONT3DPROPKEY1}
U\subset \overline{\scrB_{\eta_1}^d}
+\partial\bigl(\vecz+\widetilde{\fZ}(c_1,c_2,\sigma)\bigr).
\end{align}
Indeed, using 
$|c_1-c'_1|,$ $|c_2-c'_2|,$ $|\sigma-\sigma'|,$
$\|\vecz-\vecz'\|$ $\leq \eta$ one verifies
\begin{align} \label{CONT3DPROPKEYINCL}
& \vecz+\fZ(c_1,c_2,\sigma)\subset 
\bigl(\vecz'+\widetilde{\fZ}(c_1',c_2',\sigma')\bigr)+\scrB_{3\eta}^d
\qquad  \text{and}
\\  \notag
& \vecz'+\fZ(c_1',c_2',\sigma')\subset 
\bigl(\vecz+\widetilde{\fZ}(c_1,c_2,\sigma)\bigr)+\scrB_{3\eta}^d.
\end{align}
Hence using \eqref{TMOVESSHORT} and $\eta<\sfrac{\eta_1}{10}$ we have
\begin{align} \notag
& \bigl(\vecz+\fZ(c_1,c_2,\sigma)\bigr) T \subset 
\bigl(\vecz'+\widetilde{\fZ}(c_1',c_2',\sigma')\bigr)+\scrB_{\eta_1/2}^d
\qquad  \text{and}
\\ \label{CONT3DPROPKEY1P1}
& \bigl(\vecz'+\fZ(c_1',c_2',\sigma')\bigr) T^{-1} \subset 
\bigl(\vecz+\widetilde{\fZ}(c_1,c_2,\sigma)\bigr)+\scrB_{\eta_1/2}^d,
\end{align}
and since $\|T-I\|<\sfrac 1{40}$ implies
$\scrB_{\eta_1/2}^d T^{-1}\subset \scrB_{\eta_1}^d$
we also get
\begin{align} \label{CONT3DPROPKEY1P2}
& \bigl(\vecz+\fZ(c_1,c_2,\sigma)\bigr) \subset 
\bigl(\vecz'+\widetilde{\fZ}(c_1',c_2',\sigma')\bigr)T^{-1}+\scrB_{\eta_1}^d.
\end{align}
Our claim \eqref{CONT3DPROPKEY1} follows easily from
\eqref{CONT3DPROPKEY1P1} and \eqref{CONT3DPROPKEY1P2},
using also the convexity of the set
$\bigl(\vecz'+\widetilde{\fZ}(c_1',c_2',\sigma')\bigr)T^{-1}$.

To see this take $\vecx\in U$. Then \textit{either} 
$\vecx\in \bigl(\vecz'+\fZ(c_1',c_2',\sigma')\bigr)T^{-1}$
and
$\vecx\notin \vecz+\fZ(c_1,c_2,\sigma)$;
and in this case \eqref{CONT3DPROPKEY1P1} shows that there exists a point
$\vecx'\in \vecz+\widetilde{\fZ}(c_1,c_2,\sigma)$ with
$\|\vecx'-\vecx\|<\eta_1/2$. Then some point on the line segment
between $\vecx$ and $\vecx'$ must lie in
$\partial\bigl(\vecz+\widetilde{\fZ}(c_1,c_2,\sigma)\bigr)$---q.e.d.
\textit{Or else} we have
$\vecx\notin \bigl(\vecz'+\fZ(c_1',c_2',\sigma')\bigr)T^{-1}$
and $\vecx\in \vecz+\fZ(c_1,c_2,\sigma)$.
(Thus $\sigma>0$ and $\fZ(c_1,c_2,\sigma)=\widetilde{\fZ}(c_1,c_2,\sigma)$.)
Then, since $\bigl(\vecz'+\widetilde{\fZ}(c_1',c_2',\sigma')\bigr)T^{-1}$
is convex, there is a hyperplane $\Pi\subset\R^d$ through $\vecx$ such that
$\bigl(\vecz'+\widetilde{\fZ}(c_1',c_2',\sigma')\bigr)T^{-1}$ lies in one
of the closed half spaces determined by $\Pi$. Let $\vecx'$ be that point
which lies in the \textit{other} half space, on the normal line to $\Pi$ 
through $\vecx$, with $\|\vecx'-\vecx\|=\eta_1$. Then
\eqref{CONT3DPROPKEY1P2} implies 
$\vecx'\notin \vecz+\fZ(c_1,c_2,\sigma)$ and hence by
our assumption on $\vecx$, some point on the line segment
between $\vecx$ and $\vecx'$ must lie in
$\partial\bigl(\vecz+\widetilde{\fZ}(c_1,c_2,\sigma)\bigr)$---q.e.d.

If $d\geq 3$ then \eqref{CONT3DPROPKEY1} implies that 
$\vol(U)\leq \frac{\ve}2$, by our choice of $\eta_1$.

Next we will show that $\delta\vecy \in U$ with $\delta\in\Q$ implies that
$\delta$ has a large denominator.
For each $\delta\geq 1$ we have
$\delta\vecy \notin \vecz+\fZ(c_1,c_2,\sigma)$ since
$y_1\geq c_2$, and also
$\delta\vecy \notin (\vecz'+\fZ(c'_1,c'_2,\sigma'))T^{-1}$ since
$\delta\vecy T=\delta\vecy' \notin (\vecz'+\fZ(c'_1,c'_2,\sigma'))$;
hence $\delta\vecy\notin U$.
Similarly $\delta\vecy\notin U$ for each $\delta\leq 0$.
Also if $\delta\in\Far_N$ then our assumption
$s_{\langle c_1,c_2,\sigma,\vecz,\vecy\rangle}(\delta)=
s_{\langle c'_1,c'_2,\sigma',\vecz',\vecy'\rangle}(\delta)$ implies that
the point $\delta\vecy$ either lies in both sets 
$\vecz+\fZ(c_1,c_2,\sigma)$ and
$(\vecz'+\fZ(c'_1,c'_2,\sigma'))T^{-1}$, or else in none of them;
thus $\delta\vecy\notin U$.
Hence it follows that $\delta\vecy\in U$ for rational $\delta$
can only hold if $0<\delta<1$ and $\delta$'s denominator is larger than $N$.

It follows from this that if $d\geq 3$ then the sum in
\eqref{BOUNDFROMSIEGELPROP} is
$\leq \sum_{t\geq N}t^{-d}t<\sfrac{\ve}2$ and hence since
$\vol(U)\leq \frac{\ve}2$ we have now shown that 
\eqref{CONTUPPERBOUND} is $\leq \ve$, i.e.\ the proof of the 
proposition is complete.

If $d=2$ then it follows that the first sum in
\eqref{BOUNDFROMSIEGELPROPD2} \textit{vanishes},
and it remains to bound the second sum in \eqref{BOUNDFROMSIEGELPROPD2}.
Since $U\subset \scrB_{4C}^2$ 
we get non-vanishing contributions in that sum only when
$|v|\leq 4Cq\|\vecy\|\leq 4C^2q$.
Furthermore it follows from \eqref{CONT3DPROPKEY1} that
$U$ is contained in the union of two translates of
$[0,c_2-c_1+2\eta_1]\times [0,2\eta_1]$ and
two translates of $[0,2\eta_1]\times [0,2\sigma+2\eta_1]$.
Using now the condition $|y_2|\geq C^{-1}$
we see that for each translate $B$ of $[0,c_2-c_1+2\eta_1]\times [0,2\eta_1]$
and any $\vecw\in\R^2$, the interval
$\{x\in\R\col x\vecy+\vecw\in B\}$
has length $\leq 2\eta_1/|y_2| \leq 2C\eta_1$,
and hence the total contribution from $B$ to
the $v$-sum in \eqref{BOUNDFROMSIEGELPROPD2} is
$\leq \sum_{1\leq |v|\leq 4C^2q} (\sum_{t\mid v}t^{-1})2C\eta_1$, 
and by our choice of $\eta_1$ this is
$\leq \sfrac{\ve}{10}$.
Similarly using $|y_1| \geq C^{-1}$ one finds that
the total contribution from each \textit{vertical} side is also
$\leq \sfrac{\ve}{10}$.
Hence in total \eqref{BOUNDFROMSIEGELPROPD2} is 
$\leq \frac{\ve}2+0+\frac{\ve}{10}+\frac{\ve}{10}+\frac{\ve}{10}+\frac{\ve}{10}
<\ve$, and the proof is complete.
\end{proof}

We will now point out several consequences of Proposition \ref{CONT3DPROP}.
First, the following technical lemma will be quite convenient to use
in our proof of Theorem \ref{exactpos1}.

\begin{lem} \label{FrfrRELLEMMA}
Given any $C,\ve$ and corresponding $\eta,N$ as in 
Proposition \ref{CONT3DPROP}, then for all
$c,\xi>0$ and $\vecw,\vecz\in\{0\}\times\R^{d-1}$ 
satisfying 
$C^{-1}\leq c \leq \xi\leq c+\min(\eta,c/N)$ and 
$\xi+\|\vecw\|+\|\vecz\|\leq C$
[and if $d=2$: $\|\vecw+\vecz\|\geq C^{-1}$],
we have
\begin{align}
& \bigl | f_r(0,c,1,\vecz,\xi\vece_1+\vecw+\vecz)
-F_r(\xi,\vecw,\vecz) \bigr |
\\ \notag
& =\bigl | f_r(0,c,1,\vecz,\xi\vece_1+\vecw+\vecz)
-f_r(0,\xi,1,\vecz,\xi\vece_1+\vecw+\vecz) \bigr |
\leq \ve.
\end{align}
\end{lem}
\begin{proof}
The assumptions imply that both
$\langle 0,c,1,\vecz,\xi\vece_1+\vecw+\vecz\rangle$ and 
$\langle 0,\xi,1,\vecz,\xi\vece_1+\vecw+\vecz\rangle$ belong to $\Omega_C$,
and these 5-tuples differ only in the second coordinate,
by an amount $\leq \eta$;
hence by Proposition \ref{CONT3DPROP} we only need to check that
$s_{\langle 0,c,1,\vecz,\xi\vece_1+\vecw+\vecz\rangle}(\delta)
=s_{\langle 0,\xi,1,\vecz,\xi\vece_1+\vecw+\vecz\rangle}(\delta)$ holds
for every $\delta\in\Far_N$.
Fix $\delta\in\Far_N$; our task is now to prove that the point
$\xi\vece_1+\vecw+\vecz$ either belongs to both or none of the two sets
$\delta^{-1}(\vecz+\fZ(0,c,1))$ and $\delta^{-1}(\vecz+\fZ(0,\xi,1))$.
Note that $\delta\leq \frac {N-1}N$; thus using 
$0\leq\xi-c\leq c/N$ we have
$\xi<\delta^{-1} c$ as well as $\xi<\delta^{-1}\xi$.
Hence the two containment relations are \textit{both} equivalent with
$\|\vecw+\vecz-\delta^{-1}\vecz\|<\delta^{-1}$, and we are done.
\end{proof}

We next prove several lemmas relating directly to the function $F_r$.

\begin{lem} \label{FRMEASURABLELEM}
$F_r(\xi,\vecw,\vecz)$ is Borel measurable. %
\end{lem}
\begin{proof}
We first take $d\geq 3$. 
It suffices to prove that the restriction of $F_r$ to any given
compact subset $K$ of 
$\R_{>0} \times (\{0\}\times \R^{d-1})\times (\{0\}\times \R^{d-1})$
is Borel measurable.
Using Proposition \ref{CONT3DPROP} we see that on $K$ we can obtain
$F_r$ as a uniform limit of functions which take only a finite number of
values, each level set being a finite union of sets of the form
\begin{align}
B\cap \{\langle \xi,\vecw,\vecz\rangle \in K
\col s_{\langle 0,\xi,1,\vecz,\xi\vece_1+\vecw+\vecz\rangle}(\delta)=
s_0(\delta),\: \forall \delta\in\Far_{N}\},
\end{align}
with $B$ a box region and $s_0$ some function from $\Far_N$ to $\{0,1\}$.
Since each such level set is a Borel set we have thus expressed $(F_r)|_K$
as a uniform limit of Borel measurable functions, and we are done.

We now turn to the case $d=2$. In this case an application of
Proposition \ref{CONT3DPROP} as above shows that
the restriction of $F_r$ to any given compact subset $K$ of 
$\{\langle \xi,w\vece_2,z\vece_2\rangle \col w+z\neq\bn\}$
is Borel measurable. Next, by a computation using the
set-up from Proposition \ref{XQVOLONEPROP} one finds
\begin{align} 
& F_r(\xi,-z\vece_2,z\vece_2)
\\ \notag
& = \Bigl(\sum_{\substack{t\geq 1\\(t,q)=1}} t^{-2}\Bigr)^{-1}
\sum_{\substack{t\geq 1\\(t,q)=1}} t^{-2}
\int_{\R/\Z} I\Bigl(\sum_{\substack{n\in\Z\\
q\xi(z-1)<nt<q\xi(z+1)}}
\#\bigl(\Z\cap(nx,nx+\sfrac{t}{q})\bigr)=r\Bigr) \, dx.
\end{align}
In particular $F_r(\xi,-z\vece_2,z\vece_2)$ is constant on any set of the
form
\begin{align}
M_{a_1,a_2}=
\bigl\{(\xi,z)\col q\xi(z-1)\in [a_1,a_1+1),\: q\xi(z+1)\in (a_2,a_2+1]\bigr\}
\qquad (a_1,a_2\in\Z).
\end{align}
This implies that
also the restriction of $F_r$ to
$\{\langle \xi,w\vece_2,z\vece_2\rangle \col w+z=\bn\}$ is 
Borel measurable, and we are done.
\end{proof}
In particular this proves the claim about Borel measurability in 
Remark \ref{PALFSYMMREMARK}. This shows that we may freely
change order of integration in the right hand side of the
limit formula \eqref{exactpos1eq}.

Next we prove the claim about continuity in Remark \ref{PALFSYMMREMARK}.
\begin{lem} \label{FCONTZWLTO}
If we keep $\|\vecw\|<1$ and $\|\vecz\|\leq 1$ [and if $d=2$:
$\vecz+\vecw\neq\bn$] then the function
$F_r(\xi,\vecw,\vecz)$ is jointly continuous in 
all three variables.
\end{lem}
\begin{proof}
This is a simple consequence of Proposition \ref{CONT3DPROP}
once we note that 
$s_{\langle 0,\xi,1,\vecz,\xi\vece_1+\vecw+\vecz\rangle}(\delta)=1$
holds for any $\xi>0$, $\vecw\in\{0\}\times\scrB_1^{d-1}$
$\vecz\in\{0\}\times\overline{\scrB_1^{d-1}}$ and
any $\delta\in\Far_N$.
This fact follows from $0<\xi<\delta^{-1}\xi$ and
$\|\vecw+\vecz-\delta^{-1}\vecz\|
\leq \|\vecw\|+(\delta^{-1}-1)\|\vecz\|
<1+(\delta^{-1}-1)=\delta^{-1}$.
\end{proof}

\begin{lem} \label{FCONTINXI}
For any fixed $\vecz,\vecw$ [if $d=2$: assume $\vecz+\vecw\neq\bn$], 
the function
$F_r(\xi,\vecw,\vecz)$ is continuous in the variable $\xi>0$.
\end{lem}
\begin{proof}
This follows directly from Proposition \ref{CONT3DPROP}
once we note that for any $\delta\in\Far_N$, the function
$s_{\langle 0,\xi,1,\vecz,\xi\vece_1+\vecw+\vecz\rangle}(\delta)$
is independent of $\xi$. Indeed, since $\delta<1$,
$s_{\langle 0,\xi,1,\vecz,\xi\vece_1+\vecw+\vecz\rangle}(\delta)=1$
holds if and only if $\|\vecw+\vecz-\delta^{-1}\vecz\|< \delta^{-1}$.
\end{proof}

\begin{lem} \label{FWINTCONT}
Let $\scrW$ be any bounded Borel subset of $\{0\}\times\R^{d-1}$;
then the integral $\int_{\scrW} F_r(\xi,\vecw,\vecz) \, d\vecw$
exists for all $\xi>0$, $\vecz\in\{0\}\times\R^{d-1}$, and is jointly
continuous in these two variables.
In fact, given any $\ve>0$ and $B>1$ there is some $\nu>0$ such that 
\begin{align} \label{FWINTSTRONGCONT}
\int_{\scrW} \bigl | F_r(\xi,\vecw,\vecz)-F_r(\xi',\vecw,\vecz')\bigr | 
\, d\vecw<\ve
\end{align}
holds for all $\xi,\xi'\in [B^{-1},B]$, 
$\vecz,\vecz'\in\{0\}\times\scrB_B^{d-1}$ satisfying $|\xi-\xi'|<\nu$
and $\|\vecz-\vecz'\|<\nu$.
\end{lem}
\begin{proof}
Since $0\leq F_r(\xi,\vecw,\vecz)\leq 1$, 
the existence of the integral follows from the Borel measurability
proved in Lemma \ref{FRMEASURABLELEM}.

To prove \eqref{FWINTSTRONGCONT}, let $\ve>0$ and $B>1$ be given.
Applying Proposition \ref{CONT3DPROP} with 
$\ve':=(2+\vol_{d-1}(\scrW))^{-1}\ve$ in place of $\ve$
and with $C=\max(2B+\sup_{\vecw\in\scrW}\|\vecw\|,4/\ve')$, we get that there
are some $\eta>0$ and $N\in\Z_{\geq 2}$ such that
$\bigl | F_r(\xi,\vecw,\vecz)-F_r(\xi',\vecw,\vecz')\bigr |\leq\ve'$
holds for all $\xi,\xi'\in [B^{-1},B]$, $\vecw\in\scrW$ and
$\vecz,\vecz'\in\{0\}\times\scrB_B^{d-1}$ satisfying $|\xi-\xi'|<\frac{\eta}2$
and $\|\vecz-\vecz'\|<\frac{\eta}2$
and $s_{\langle 0,\xi',1,\vecz',\xi'\vece_1+\vecw+\vecz'\rangle}(\delta)=
s_{\langle 0,\xi,1,\vecz,\xi\vece_1+\vecw+\vecz\rangle}(\delta)$,
$\forall \delta\in\Far_N$.
If $d=2$ then we must also require $\|\vecw+\vecz\|\geq C^{-1}$
and $\|\vecw+\vecz'\|\geq C^{-1}$.
The $s$-conditions are seen to hold if and only if,
for each $\delta\in\Far_N$, either
both or none of $\|\vecw-(\delta^{-1}-1)\vecz\|<\delta^{-1}$ and
$\|\vecw-(\delta^{-1}-1)\vecz'\|<\delta^{-1}$ are true.
For each $\delta\in\Far_N$, the set of \textit{exceptional}
$\vecw$'s is thus seen to lie in a union of two translates of
the region
$\delta^{-1} \bigl(\scrB_{1+\|\vecz-\vecz'\|}^{d-1}
\setminus\scrB_1^{d-1}\bigr)$.
Hence, since $\delta^{-1}\leq N$ and $\Far_N$ is finite,
there is some $\nu\in(0,\frac{\eta}2]$ such that the volume of the total set of 
exceptional $\vecw$'s is less than $\ve'$ whenever $\|\vecz-\vecz'\|<\nu$.
For $d=2$ we also have to consider the set of exceptional $\vecw$'s
satisfying $\|\vecw+\vecz\|< C^{-1}$ or $\|\vecw+\vecz'\|< C^{-1}$;
this set has volume $\leq 4C^{-1}\leq\ve'$.
Hence, since the integrand in \eqref{FWINTSTRONGCONT} is everywhere $\leq 1$,
we see that for any $\xi,\xi'\in [B^{-1},B]$ and
$\vecz,\vecz'\in\{0\}\times\scrB_B^{d-1}$ satisfying $|\xi-\xi'|<\nu$
(or just $<\frac{\eta}2$) and $\|\vecz-\vecz'\|<\nu$, the integral in
\eqref{FWINTSTRONGCONT} is $\leq (2+\vol_{d-1}(\scrW))\ve'=\ve$, as desired.
\end{proof}

\subsection{An important volume function, for $\vecalf\notin\Q^d$} 
\label{IMPORTANTVOLUMEIRRSEC}

The questions treated in the last section become much simpler if we consider the submanifolds $X(\vecy)$ in place of $X_q(\vecy)$. Indeed, let us define, in analogy with \eqref{FR5VARDEF} above:
\begin{align} \label{XYFR5VARDEF}
f_r(c_1,c_2,\sigma,\vecz,\vecy)
:=\nu_\vecy\bigl(\bigl\{g\in X(\vecy) \col 
\#\bigl(\Z^d g \cap (\fZ(c_1,c_2,\sigma)+\vecz)\bigr)
=r\bigr\}\bigr),
\end{align}
with the same domain $\Omega$ as before, and
for $\xi>0$ and $\vecz,\vecw \in \{0\}\times\R^{d-1}$,
\begin{align} \label{XYFR3VARDEF}
F_r(\xi,\vecw,\vecz):=f_r(0,\xi,1,\vecz,\xi\vece_1+\vecw+\vecz).
\end{align}

It will be clear from the context which case of functions $f_r$,
$F_r$ (\eqref{FR5VARDEF}, \eqref{FR3VARDEF} or 
\eqref{XYFR5VARDEF}, \eqref{XYFR3VARDEF}) we are referring to.

\begin{lem} \label{F5VARSYMMLEMMA}
$f_r(c_1,c_2,\sigma,\vecz,\vecy)$ in \eqref{XYFR5VARDEF}
satisfies the same invariance relations
as in the $X_q(\vecy)$-case (see Lemma \ref{F5SYMMLEMMA}),
and also
$f_r(c_1,c_2,\sigma,\vecz,\vecy)=f_r(c_1,c_2,\sigma,\bn,\vecy-\vecz)$.
\end{lem}
\begin{proof}
Cf.\ the proof of Lemma \ref{F5SYMMLEMMA} but use
Lemma \ref{XYCANMFORMINVARIANCE} in place of
Lemma \ref{CANMFORMINVARIANCE}, and also use the transformation
$h=(1_d,-\vecz)\in\ASL(d,\R)$.
\end{proof}

Hence $F_r(\xi,\vecw,\vecz)$ in fact only depends on $\xi$ and $\|\vecw\|$.
(In particular this is true for $\Phi_\vecalf(\xi,\vecw,\vecz)=
F_0(\xi,\vecw,\vecz)$, as pointed out in Remark \ref{PALFSYMMREMARK}.)

\begin{prop} \label{XYCONT3DPROP}
The function $f_r(c_1,c_2,\sigma,\vecz,\vecy)$ in \eqref{XYFR5VARDEF}
is continuous everywhere in $\Omega$.
\end{prop}

\begin{proof}
This follows by the same method of proof as in Proposition \ref{CONT3DPROP},
but the details are much simpler:
Using Proposition \ref{XYSIEGELPROP} in place of 
Proposition \ref{XQYSIEGELPROP} one finds that 
\eqref{BOUNDFROMSIEGELPROP} is now replaced by
\begin{align}
\Bigl | f_r(c_1,c_2,\sigma,\vecz,\vecy)-f_r(c'_1,c'_2,\sigma',\vecz',\vecy')\Bigr | \leq \vol(U)+\chi_U(\vecy),
\end{align}
and as before one sees that $\chi_U(\vecy)=0$ and that
$\vol(U)$ can be made arbitrarily small by taking
$\langle c'_1,c'_2,\sigma',\vecz',\vecy' \rangle$ close to
$\langle c_1,c_2,\sigma,\vecz,\vecy \rangle$.
\end{proof}

We end by remarking some relations which will be useful in 
Proposition \ref{GENC2PROP} below and in our discussion
of explicit formulas in \cite{partIII}.
First, using \eqref{XYID} and the definition of $\nu_\vecy$ just below
\eqref{XYID} we see that
\begin{align} \label{FRASLFORMULA1}
& f_r(c_1,c_2,\sigma,\vecz,\vecy)
= \mu_1\bigl(\bigl\{M\in X_1 \col
\#\bigl(\Z^d M \cap (\vecz-\vecy+\fZ(c_1,c_2,\sigma))\bigr)
=r\bigr\}\bigr).
\end{align}
In particular we have
\begin{align} \label{FRASLFORMULA1a}
& F_r(\xi,\vecw,\vecz)
= \mu_1\bigl(\bigl\{M\in X_1 \col
\#\bigl(\Z^d M \cap (-\xi\vece_1-\vecw+\fZ(0,\xi,1))\bigr)
=r\bigr\}\bigr).
\end{align}
Here $-\xi\vece_1-\vecw+\fZ(0,\xi,1)$ may be replaced by its 
pointwise negate,
$\xi\vece_1+\vecw-\fZ(0,\xi,1)$, and since $\vecw\in \{0\}\times\R^{d-1}$
this set is seen to equal $\vecw+\fZ(0,\xi,1)$. Hence 
\begin{align} \label{FRASLFORMULA2}
F_r(\xi,\vecw,\vecz) = \mu_1\bigl(\bigl\{M\in X_1 \col
\#\bigl(\Z^d M \cap (\vecw+\fZ(0,\xi,1))\bigr)
=r\bigr\}\bigr).
\end{align}
One may note that this volume is a special case of the limit function
$F_{c,\vecalf,\vecbeta}(r,\sigma)$
obtained in Theorem \ref{visThm2} for $\vecalf=\bn$.
Indeed, using the relation
$\bigl(\vecw+\fZ(0,\xi,1)\bigr)\matr{\xi^{-1}}\bn{\trans\bn}{\xi^{1/(d-1)}}
=\xi^{\frac 1{d-1}}\vecw+\fZ(0,1,\xi^{\frac 1{d-1}})$
we see
$F_r(\xi,\vecw,\vecz)=F_{0,\bn,\vecbeta}(r,\xi^{\frac 1{d-1}})$ 
holds for any choice of function $\vecbeta(\vecv)$ such that
$\bigl\|\Proj_{\{\vecv\}^\perp}\vecbeta(\vecv)\bigr\|=\|\vecw\|$
for all $\vecv\in\S_1^{d-1}$.

\subsection{Differentiability properties} \label{subsecC1}

\begin{prop} \label{GENC1PROP}
For any fixed $\vecalf,\vecbeta,\lambda,r$ (and $c=0$)
as in Theorem \ref{visThm2} with $\vecalf\in\Q^d$, 
the function $F_{0,\vecalf,\vecbeta}(r,\sigma)$ 
defined in \eqref{defF} is $\C^1$ with respect 
to $\sigma>0$. %
\end{prop}

\begin{proof}
In analogy with \eqref{FZVCSDEF} we define   %
$\fZ_\vecv(c_1,c_2,\sigma):=\fZ(c_1,c_2,\sigma)+\sigma
\bigl\|\Proj_{\{\vecv\}^\perp} \vecbeta(\vecv)\bigr\|\cdot \vece_2$,
so that $\fZ_\vecv(c,\sigma)=\fZ_\vecv(c,1,\sigma)$.
Then
$\fZ_\vecv(c'c,c',{c'}^{-\frac 1{d-1}}\sigma)=\fZ_\vecv(c,1,\sigma)
\begin{pmatrix} c' & \bn \\ \trans\bn & c'^{-1/(d-1)} 1_{d-1} \end{pmatrix}$
for all $c'>0$, and hence,
using also the invariance of $\mu_q$, we have
\begin{align} 
&F_{c,\vecalf,\vecbeta}(r,\sigma)
=(\mu_q\times\lambda)\bigl(\bigl\{(M,\vecv)\in X_q\times \S_1^{d-1}
\col \# \bigl( (\Z^d+\vecalf)M \cap 
\fZ_\vecv(c\sigma^{d-1},\sigma^{d-1},1)\bigr) = r 
\bigr\}\bigr),
\end{align}

To simplify the notation we write $\sigma=\xi^{\frac 1{d-1}}$.
Now, for any $\xi>0$ and $h>0$,
\begin{align} \label{C1PROOFSTEP1}
& \bigl(F_{0,\vecalf,\vecbeta}(r,(\xi+h)^{\frac 1{d-1}})
-F_{0,\vecalf,\vecbeta}(r,\xi^{\frac 1{d-1}})\bigr) /h
\\ \notag
& =h^{-1}\int_{\S_1^{d-1}}
\mu_q\Bigl(\Bigl\{ M\in X_q \col 
\# \bigl( (\Z^d+\vecalf)M \cap 
\fZ_\vecv(0,\xi,1)\bigr) < r, \:
\\ \notag
& \hspace{140pt} \# \bigl( (\Z^d+\vecalf)M \cap 
\fZ_\vecv(0,\xi+h,1)\bigr) = r
\Bigr\}\Bigr)\, d\lambda(\vecv)
\\ \notag
& -h^{-1}\int_{\S_1^{d-1}}
\mu_q\Bigl(\Bigl\{ M\in X_q \col 
\# \bigl( (\Z^d+\vecalf)M \cap 
\fZ_\vecv(0,\xi,1)\bigr) = r, \:
\\ \notag
& \hspace{140pt} \# \bigl( (\Z^d+\vecalf)M \cap 
\fZ_\vecv(0,\xi+h,1)\bigr) > r
\Bigr\}\Bigr)\, d\lambda(\vecv).
\end{align}
If $r\geq 1$, then
using Proposition \ref{FOLINTPROP} and Lemma \ref{FEWWITHTWOLEMMA2}
we find (cf.\ the discussion of \eqref{exactposLIMINFBOUND3} below)
that the first term in the right hand side of \eqref{C1PROOFSTEP1} equals,
as $h\to 0$,
\begin{align}
O\bigl(h\log(h^{-1})\bigr)
+ h^{-1}\int_{\S_1^{d-1}}
\int_{\xi}^{\xi+h} \int_{\{0\}\times \scrB_1^{d-1}}
f_{r-1}(0,\xi,1,\vecz_\vecv,\xi'\vece_1+\vecw+\vecz_\vecv)
\, d\vecw \, d\xi' \, d\lambda(\vecv),
\end{align}
where $\vecz_\vecv:=
\bigl\|\Proj_{\{\vecv\}^\perp} \vecbeta(\vecv)\bigr\|\cdot \vece_2$.
This tends to $\int_{\S_1^{d-1}}\int_{\{0\}\times \scrB_1^{d-1}} 
F_{r-1}(\xi,\vecw,\vecz_\vecv)\,d\vecw \, d\lambda(\vecv)$
as $h\to 0$, by Lemma \ref{FrfrRELLEMMA}.
Treating the second term in \eqref{C1PROOFSTEP1} in the same way
we obtain 
\begin{align} \label{C1PROOFSTEP2}
\lim_{h\to 0^+} & \quad h^{-1}\Bigl(
F_{0,\vecalf,\vecbeta}(r,(\xi+h)^{\frac 1{d-1}})
-F_{0,\vecalf,\vecbeta}(r,\xi^{\frac 1{d-1}})\Bigr)
\\ \notag
& =\int_{\S_1^{d-1}}\int_{\{0\}\times \scrB_1^{d-1}}
\Bigl( F_{r-1}(\xi,\vecw,\vecz_\vecv)-F_{r}(\xi,\vecw,\vecz_\vecv) \Bigr)
\, d\vecw \, d\lambda(\vecv) .
\end{align}
This is valid also for $r=0$ if we define $F_{-1}:\equiv 0$.
Inspecting the proof just carried out and using the uniformity in the
statements of Lemma \ref{FEWWITHTWOLEMMA2} and Lemma \ref{FrfrRELLEMMA} 
we see that the convergence in \eqref{C1PROOFSTEP2} is uniform
with respect to $\xi$ in any compact subset of $\R_{>0}$.
Hence the formula \eqref{C1PROOFSTEP2} is also valid in the limit $h\to 0^-$,
and Lemma~\ref{FWINTCONT} gives that $F_{0,\vecalf,\vecbeta}(r,\sigma)$
is indeed $\C^1$ with respect to $\sigma$.
We also note that \eqref{C1PROOFSTEP2} gives an explicit formula for the 
derivative.
\end{proof}

\begin{remark} \label{AFTERGENC1PROP}
The explicit formula for the derivative, \eqref{C1PROOFSTEP2},
specializes to the formula \eqref{PALFBETSPECREMFORMULA} in 
Remark \ref{PALFBETSPECREM}
in the case $r=0$.
(For recall \eqref{PALFBETDEF}, 
$F_0(\xi,\vecw,\vecz)=\Phi_\vecalf(\xi,\vecw,\vecz)$,
and Remark~\ref{invarem}.)

We also note that the argument in the above proof applies without
changes to the case when $\lambda$ is a (not absolutely continuous) 
probability measure which gives mass one to a single point. Hence
for each $\vecz\in\{0\}\times\R^{d-1}$ we have
\begin{align} \label{C1PROOFSTEP2SINGLETON}
\frac d{d\xi} \mu_q\bigl(\bigl\{M \in X_q\col
(\Z^d+\vecalf)M \cap (\fZ(0,\xi,1)+\vecz)=\emptyset \bigr\}\bigr)
=-\int_{\{0\}\times\scrB_1^{d-1}} \Phi_\vecalf(\xi,\vecw,\vecz)\, d\vecw;
\end{align}
in particular the derivative in the left hand side %
is a continuous function of $\xi$, cf.\ Lemma \ref{FWINTCONT}.
The set in the left hand side of
\eqref{C1PROOFSTEP2SINGLETON} has $\mu_q$-measure tending to $1$ as
$\xi\to 0^+$ and tending to $0$ as $\xi\to\infty$,
cf.\ the proof of Remark \ref{GENBETAUPPERBOUNDREM}
in Section \ref{GENBETAUPPERBOUNDPROOFSEC}.
Hence, integrating \eqref{C1PROOFSTEP2SINGLETON} over $\xi\in\R_{>0}$
we deduce the formula \eqref{PALFINTEQ1} in Remark \ref{PALFBETSPECREM},
$\int_0^\infty \int_{\{0\}\times\scrB_1^{d-1}} \Phi_\vecalf(\xi,\vecw,\vecz)
\, d\vecw \,d\xi=1$.
\end{remark}

Next we turn to the case $\vecalf\notin\Q^d$.
Recall that in this case $F_{c,\vecalf,\vecbeta}(r,\sigma)$
is independent of $\vecbeta,\lambda,\vecalf$, and we have introduced
the notation $F_c(r,\sigma)$ for this function.
Proposition \ref{GENC2PROP} and the ensuing remarks carry over directly
to the case $\vecalf\notin\Q^d$, with the usual changes of notation.
However, we can say more:
\begin{prop} \label{GENC2PROP}
For any fixed $0\leq c<1$ and $r\in\Z_{\geq 0}$
the function $F_{c}(r,\sigma)$ 
is $\C^2$ with respect to $\sigma>0$. %
\end{prop}

\begin{proof}
The function $F_{c}(r,\sigma)$ satisfies the invariance relation
$F_{c}(r,\sigma)=F_{0}(r,\sigma(1-c)^{\frac 1{d-1}})$,
which follows directly from the definition \eqref{FCALFRSIGMDEF},
using the right $\ASL(d,\R)$-invariance of $\mu$.
Hence we may from now on assume $c=0$.

Arguing as in the proof of Proposition \ref{GENC1PROP} we prove that
$F_{0}(r,\sigma)$ is $\C^1$ with respect to
$\sigma$. The explicit formula \eqref{C1PROOFSTEP2}
is still valid (with $F_r(\xi,\vecw,\vecz)$ now being given by
\eqref{XYFR3VARDEF}, \eqref{XYFR5VARDEF}), although the integration over
$\S_1^{d-1}$ may be skipped since in this case $F_r(\xi,\vecw,\vecz)$ is 
independent of $\vecz$.
Rewriting \eqref{C1PROOFSTEP2} using \eqref{FRASLFORMULA2} we get
\begin{align} \label{C2PROOFSTEP1}
\frac{d}{d\xi}F_{0}(r,\xi^{\frac 1{d-1}})
=&\int_{\{0\}\times \scrB_1^{d-1}} \mu_1\bigl(\bigl\{M\in X_1 \col
\#\bigl(\Z^d M \cap (\vecw'+\fZ(0,\xi,1))\bigr)
=r-1\bigr\}\bigr)\,d\vecw'
\\ \notag
&-\int_{\{0\}\times \scrB_1^{d-1}} \mu_1\bigl(\bigl\{M\in X_1 \col
\#\bigl(\Z^d M \cap (\vecw'+\fZ(0,\xi,1))\bigr)
=r\bigr\}\bigr)\,d\vecw'.
\end{align}
But here the right hand side can again be differentiated with respect to $\xi$,
by repeating the argument in the proof of Proposition \ref{GENC1PROP}
(with %
``$\vecalf=\bn$'' and letting $\vecw'$ play the role of $\vecz_\vecv$ in
that proof);
this leads to
\begin{align} \label{DDFFORMULA}
\frac{d^2}{d\xi^2}F_{0}(r,\xi^{\frac 1{d-1}})
=\int_{\{0\}\times \scrB_1^{d-1}} \int_{\{0\}\times \scrB_1^{d-1}} 
\Bigl( F_{r-2}^{(\vecalf=\bn)}(\xi,\vecw,\vecw')
-2F_{r-1}^{(\vecalf=\bn)}(\xi,\vecw,\vecw')&
\\ \notag
+F_{r}^{(\vecalf=\bn)}(\xi,\vecw,\vecw')&\Bigr)
\, d\vecw \, d\vecw',
\end{align}
where ``$F_{r}^{(\vecalf=\bn)}$'' means ``$F_r$ as in 
\eqref{FR3VARDEF}, \eqref{FR5VARDEF} with $\vecalf=\bn$, $q=1$''
(and we understand $F_{-2}^{(\vecalf=\bn)}:\equiv 0$ and 
$F_{-1}^{(\vecalf=\bn)}:\equiv 0$).
Hence (for our $\vecalf\notin\Q^d$) $F_{0}(r,\sigma)$
is indeed $\C^2$ with respect to $\sigma$, cf.\ Lemma~\ref{FWINTCONT}.
\end{proof}
\begin{remark}
The formula \eqref{DDFFORMULA} generalizes \cite[Eq.\ (34)]{SV} from $d=2$
to general $d$.
\end{remark}

\subsection{A uniform bound} \label{GENBETAUPPERBOUNDPROOFSEC}

In this section we prove the two bounds in Remark \ref{GENBETAUPPERBOUNDREM}.
If $\vecalf\in\Q^d$ we note that for each $\vecv\in\S_1^{d-1}$ we have,
by Proposition \ref{SIEGELQPROP},
\begin{align} \notag
\mu_q\bigl(\bigl\{M\in X_q\col (\Z^d+\vecalf)M\cap\fZ_\vecv(c,\sigma)
=\emptyset\bigr\}\bigr)
\geq 1-\int_{X_q} \#\bigl((\Z^d+\vecalf)M\cap\fZ_\vecv(c,\sigma)\bigr)\,
d\mu_q(M)
\\ \label{GENBETAUPPERBOUNDPROOF1}
=1-\vol\bigl(\fZ_\vecv(c,\sigma)\bigr)=1-v_d(1-c)\sigma^{d-1}.
\end{align}
Integrating over $\vecv\in\S_1^{d-1}$ with respect to the measure
$\lambda$ (cf.\ the definition \eqref{defF}) we obtain the first bound in
\eqref{GENBETAUPPERBOUND1}; the second one follows using
$\sum_{r=0}^\infty F_{c,\vecalf,\vecbeta}(r,\sigma)=1$.

In the case $\vecalf\notin\Q^d$ the bound \eqref{GENBETAUPPERBOUND1} 
follows using \eqref{ASLDRHAAR}, \eqref{ASLFUNDDOM}
and a computation as in \eqref{GENBETAUPPERBOUNDPROOF1}, noticing
$\int_{[0,1)^dM} \#\bigl((\Z^dM+\vecxi)\cap\fZ_\vecv(c,\sigma)\bigr)\, 
d\vecxi=\vol\bigl(\fZ_\vecv(c,\sigma)\bigr)$ for each $M\in\F_1$.

The bound \eqref{GENBETAUPPERBOUND2} is a direct consequence of the 
following lemma.

\begin{lem} \label{TRIVBOUNDVLARGEGENLEMMA}
If $r\in\Z_{\geq 0}$ and $B$ is any translate
of a cylinder $\fZ(c_1,c_2,\sigma)$ (cf.\ \eqref{FZC1C2DEF}) in $\R^d$ 
of volume $V$, then
\begin{align} \label{TRIVBOUNDVLARGEGEN}
& \mu_q\bigl(\bigl\{M\in X_q\col \#((\Z^d+\vecalf)M\cap B)\leq r\bigr\}\bigr)
\ll V^{-1}, & &
\forall \vecalf\in q^{-1}\Z^d;
\\ \notag
& \text{and } \quad
\mu\bigl(\bigl\{g\in X\col \#(\Z^dg\cap B)\leq r\bigr\}\bigr)\ll V^{-1}.
\end{align}
The implied constants depend only on $r,d$.
\end{lem}
\begin{proof}
The proof uses the methods in 
\cite[section 3.6]{Marklof00}, but note that we work with a slightly different
notation in the present paper.
We will prove the first bound in \eqref{TRIVBOUNDVLARGEGEN};
the proof of the second bound is quite similar.
Since both sides in the inequality remain invariant if $B$ is replaced by
$B M_0$ for any $M_0\in\SL(d,\R)$, we may assume without loss of
generality that $B$ is a translate
of a cylinder $\fZ(c_1,c_2,\sigma)$ with $c_2-c_1=\sigma$.

Every element $M\in\SL(d,\R)$ has a unique Iwasawa decomposition
$M=\nn \aa \kk$, where $\nn$ belongs to the group $N$ of upper triangular 
matrices with $1$s on the diagonal, $\aa$
is diagonal with positive diagonal
elements, %
and $\kk\in\SO(d)$.
We let $\F_N$ be the set of all matrices in $N$ for which all entries
above the diagonal lie in the interval $(-\frac 12,\frac 12]$,
and introduce the following Siegel set 
(denoting $\aa=\text{diag}[a_1,\ldots,a_d]$):
\begin{align}
\Si:=\Bigl\{\nn \aa \kk \col \nn\in \mathcal{F}_N,\: 
0<a_{j+1} \leq \sfrac{2}{\sqrt 3}a_j \: (j=1,\ldots,d-1), 
\: \kk \in \SO(d) \Bigr\}.
\end{align}
It is known that $\Si$ contains a fundamental domain for
$X_1=\SL(d,\Z)\backslash\SL(d,\R)$; we fix $\F\subset\Si$ to be one
such fundamental domain (in the set-theoretic sense).
Choose representatives
$T_j\in\SL(d,\Z)$ so that $\SL(d,\Z)=\bigsqcup_{j=1}^m \Gamma(q) T_j$
(disjoint union);
then $\bigsqcup_{j=1}^m T_j \F$ is a fundamental domain for $X_q=\Gamma(q)
\backslash\SL(d,\R)$.

Now let $M$ be any element in $\bigsqcup_{j=1}^m T_j \F$. Choose $j$
so that $T_j^{-1}M\in\F$, let the Iwasawa decomposition of this matrix 
be $T_j^{-1}M=\nn\aa\kk$, and let the row vectors of the same matrix be
$\vecb_1,\ldots,\vecb_d\in\R^d$. Then using $\nn\in\F_N$ and
$\nn\aa\kk\in\Si$ we see that $\|\vecb_k\|\leq\sum_{j=1}^d a_j
\ll_d a_1$ for each $k=1,\ldots,d$.
Using $T_j\in\SL(d,\Z)$ we see that
$\#((\Z^d+\vecalf)M\cap B)=
\#((\Z\vecb_1+\ldots+\Z\vecb_d) \cap (B-\vecalf M))$.
Choose $\xi_1,\ldots,\xi_d\in\R$ so that $\xi_1\vecb_1+\ldots+\xi_d\vecb_d$ 
is the center of the cylinder $B-\vecalf M$,
and take $m_1,\ldots,m_d\in\Z$ so that $\xi_k-m_k\in (-\frac 12,\frac 12]$
for each $k$.
Then the distance from $\xi_1\vecb_1+\ldots+\xi_d\vecb_d$ 
to any of the lattice points $m_1\vecb_1+\ldots+m_d\vecb_d+j\vecb_d$,
for $j=0,\ldots,r$,
is $\leq\frac 12\bigl(\|\vecb_1\|+\ldots+\|\vecb_d\|\bigr)
+r\|\vecb_d\|\ll_{d,r} a_1$.  %
Hence using our assumption $c_2-c_1=\sigma$, we see that if 
$a_1\ll_{d,r} V^{1/d}$ then all these lattice points lie in 
$B-\vecalf M$, so that $\#((\Z^d+\vecalf)M\cap B)>r$.
Hence the left hand side in \eqref{TRIVBOUNDVLARGEGEN} is
$\leq \sum_j \mu_q(\{M \col T_j^{-1}M=\nn\aa\kk\in\F,\: a_1\gg V^{1/d}\})$.
Using \eqref{SLDRHAARQ} and the invariance of $\mu_1$ we see that this is
$\leq \mu_1(\{M=\nn\aa\kk\in\Si\col a_1\gg V^{1/d}\})$,
and as in \cite[section 3.6]{Marklof00} %
we see that this is $\ll V^{-1}$.
\end{proof}

\subsection{Analogous results for Section \ref{secVisible0}}
\label{IMPORTANTSECV0}
In this section we indicate how most parts of the development in sections 
\ref{IMPORTANTVOLUMESEC}--\ref{subsecC1} carry over to the setting
of Section \ref{distrinsmalldiscssec}, leading to a proof of
the claim in Remark \ref{visThm0rem} that the function 
$E_{0,\vecalf}(r,\sigma)$ is $\C^1$ with respect to $\sigma$.

For any $0\leq c_1\leq c_2$ and $\sigma\geq 0$ we let 
\begin{align} \label{FCCCSDEF}
	\fC(c_1,c_2,\sigma) =\bigg\{(x_1,\ldots,x_d)\in\RR^d \col c_1< x_1 < c_2, \: \|(x_2,\ldots,x_d)\|< \sigma x_1  \bigg\},
\end{align}
so that in particular $\fC(c,\sigma)$ (cf.\ \eqref{FCCSDEF})
equals $\fC(c,1,A(c,\sigma))$ up to a set of measure zero.
(The reason for using ``$\leq$'' in \eqref{FCCSDEF} is to make \eqref{holds}
below true without modification also when $\sigma_\infty=0$.)
Given $r\in\Z_{\geq 0}$ and $\vecalf \in q^{-1}\Z^d$
we now introduce the function
\begin{align} \label{GR5VARDEF}
g_r(c_1,c_2,\sigma,\vecz,\vecy)
:=\nu_\vecy\bigl(\bigl\{M\in X_q(\vecy) \col 
\#\bigl((\Z^d+\vecalf)M \cap (\fC(c_1,c_2,\sigma)+\vecz)\bigr)
=r\bigr\}\bigr)
\end{align}
with domain $\Omega$ as in \eqref{OMEGADEF}.
Thus $g_r(c_1,c_2,\sigma,\vecz,\vecy)$ is defined exactly as
$f_r(c_1,c_2,\sigma,\vecz,\vecy)$ in \eqref{FR5VARDEF} except that
we use $\fC(c_1,c_2,\sigma)$ in place of $\fZ(c_1,c_2,\sigma)$.
We also write, in analogy with \eqref{FR3VARDEF}, 
for $\xi>0$ and $\vecz,\vecw \in \{0\}\times\R^{d-1}$,
\begin{align} \label{GR3VARDEF}
G_r(\xi,\vecw,\vecz)=g_r(0,\xi,1,\vecz,\xi\vece_1+\vecw+\vecz).
\end{align}
Now the discussion in Section \ref{IMPORTANTVOLUMESEC} up to and including
Proposition \ref{CONT3DPROP} carries over to the
case of $g_r(c_1,c_2,\sigma,\vecz,\vecy)$ with very minor changes. 
In particular, if we replace
$\fZ(c_1,c_2,\sigma)$ by $\fC(c_1,c_2,\sigma)$ in the definition of
$s(\delta)$, \eqref{SDELTADEF}, and replace the definition of 
$\Omega_C$ in the case $d=2$ (cf.\ \eqref{OMEGACDEF}) by
\begin{align}\label{altOmegaC}
\Omega_C:=
\{\langle c_1,c_2,\sigma,\vecz,\vecy\rangle \in \Omega
\:\col\: \sigma,\|\vecz\|,\|\vecy\| \leq C;\:\: 
C^{-1}\leq |y_1|,|y_2\pm\sigma y_1|\},
\end{align}
then \textit{the statement of Proposition \ref{CONT3DPROP} holds with 
$g_r(c_1,c_2,\sigma,\vecz,\vecy)$ in place of
$f_r(c_1,c_2,\sigma,\vecz,\vecy)$.}
Using this, one then also proves that \textit{the statements of 
Lemma \ref{FrfrRELLEMMA}, Lemma \ref{FRMEASURABLELEM} and
Lemma~\ref{FWINTCONT} hold with
$g_r,G_r$ in place of $f_r,F_r$}, with the only difference %
that
\textit{in Lemma \ref{FrfrRELLEMMA} the condition 
``[and if $d=2$: $\|\vecw+\vecz\|\geq C^{-1}$]'' must be replaced with
``[and if $d=2$: $\bigl | \|\vecw+\vecz\|-\xi \bigr | \geq C^{-1}$]''.}

Similarly, in the case $\vecalf\notin\Q^d$, the discussion in 
Section \ref{IMPORTANTVOLUMEIRRSEC} up to and including
Proposition~\ref{XYCONT3DPROP} carries over in the obvious way
to the function
\begin{align} \label{XYGR5VARDEF}
g_r(c_1,c_2,\sigma,\vecz,\vecy)
:=\nu_\vecy\bigl(\bigl\{g\in X(\vecy) \col 
\#\bigl(\Z^d g \cap (\fC(c_1,c_2,\sigma)+\vecz)\bigr)
=r\bigr\}\bigr).
\end{align}
Also the formulas \eqref{FRASLFORMULA1} and \eqref{FRASLFORMULA1a} carry
over, but \textit{(\ref{FRASLFORMULA2}) does not carry over},
since the cone $\fC(0,\xi,1)$ does not have the necessary symmetry.

Now Proposition \ref{GENC1PROP} carries over,
i.e.\ the function $E_{0,\vecalf}(r,\sigma)$ is $\C^1$ with respect to 
$\sigma>0$ for any fixed $\vecalf\in\R^d$ and $r\in\Z_{\geq 0}$,
as claimed.
We remark that in the proof of this we actually only need 
\eqref{GR5VARDEF} with $\vecz=\bn$. The analog of the formula
\eqref{C1PROOFSTEP2} is
\begin{align} \label{GC1PROOFSTEP2}
\frac{d}{d\xi}
E_{0,\vecalf}\Bigl(r,\sfrac{\vol(\scrB_1^{d-1})}{d}\xi^{d}\Bigr)
=\int_{\{0\}\times \scrB_\xi^{d-1}}
\Bigl( G_{r-1}(\xi,\vecw,\bn)-G_{r}(\xi,\vecw,\bn) \Bigr)
\, d\vecw &.
\end{align}

\vspace{5pt}

Finally we turn to the special case $d=2$ and $\vecalf\in\Q^2$
(say $\vecalf\in q^{-1}\Z^2$ with $q$ minimal).
We intend to prove \eqref{E0ALF0SMALLS} in Section \ref{VISLPTS},
i.e.\ that $E_{0,\vecalf}(0,\sigma)=1-\kappa_q \sigma$ holds for all 
$\sigma\in \bigl[0,(2q)^{-1}\bigr]$.
Clearly, by \eqref{GC1PROOFSTEP2}, it suffices to prove that
if $0<\xi\leq (2q)^{-1/2}$ and $|w|<\xi$ then
$G_0(\xi,w\vece_2,\bn)=\kappa_q$, i.e.,
\begin{align} \label{E0ALF0SMALLSSTEP1}
\nu_\vecw\bigl(\bigl\{M\in X_q(\vecw) \col 
(\Z^2+\vecalf)M \cap \fC(0,\xi,1)=\emptyset\bigr\}\bigr)
=\kappa_q\qquad
(\text{where }\: \vecw=\xi\vece_1+w\vece_2).
\end{align}
Let $M\in\SL(d,\R)$ be a representative for an arbitrary element
in $X_q(\vecw)$. Then there is some 
$\veck\in\Z^2+\vecalf\setminus\{\bn\}$ such that $\veck M=\vecw$.
Set $t=\gcd(q\veck)$; then $\frac qt\veck$ is a primitive vector in $\Z^2$,
and thus $\Z^2=\Z\frac qt\veck+\Z\vech$ for some $\vech\in\Z^2$.
Hence $(\Z^2+\vecalf)M=(\Z^2+\veck)M=
\Z\frac qt\veck M +\Z\vech M+\veck M
\subset \R\vecw+\Z\vech M$. Also $\R\vecw+n\vech M
=\R\vecw\pm n\frac t{q\xi}\vece_2$,
and from this one verifies
(using $0<\xi\leq (2q)^{-1/2}$, $|w|<\xi$) that each line $\R\vecw+n\vech M$
($n\in\Z\setminus\{0\}$) lies outside $\fC(0,\xi,1)$.
Hence
\begin{align}
(\Z^2+\vecalf)M \cap \fC(0,\xi,1)
=\R\vecw \cap (\Z^2+\vecalf)M \cap \fC(0,\xi,1)
=(\vecw+\Z\sfrac qt\vecw)\cap \fC(0,\xi,1).
\end{align}
This set is empty if and only if $t\leq q$.
Hence by mimicking
the proof of Proposition \ref{XQVOLONEPROP} we find that the left hand side
of \eqref{E0ALF0SMALLSSTEP1} equals
\begin{align}
\frac{q^2 \, \mu_H\bigl((\Gamma(q)\cap H) \backslash H\bigr)}{I_q\zeta(2)} 
\sum_{\substack{1\leq t\leq q \\ (t,q)=1}} t^{-2}
=\Bigl(\sum_{\substack{n\geq 1 \\ (n,q)=1}} n^{-2}\Bigr)^{-1}
\sum_{\substack{1\leq t\leq q \\ (t,q)=1}} t^{-2}
=\kappa_q,
\end{align}
and we are done.

\section{Proof of the limit theorems in Sections \ref{secIntro}--\ref{secLorentz}} \label{secProofs}

\subsection{Proofs for Section \ref{secVisible}} \label{secVisproofssec}

We first prove Theorem \ref{visThm2}
(and thus Theorem \ref{visThm}, which is a special case).
Theorem \ref{visThm2} will be derived as a direct consequence of
the general limit theorems in Section \ref{secThin}, and our 
only serious task in the present section 
will be to compute the upper and lower limits of an
appropriate family of subsets of $\R^{d-1}\times\R^d$
(see Lemma \ref{lemCon2fU} below).
In fact we will carry this out for a generalized version of
Theorem \ref{visThm2}, see Theorem \ref{visThm2gen} below.
This generalization is interesting in its own right,
and its proof is also a useful preparation for the demonstration 
of Theorem \ref{exactpos1} in Section \ref{exactpos1proofsec}.

First let us fix a parametrization of the sphere:
Let $\Edomain$\label{EDOMDEF2} be a bounded open subset of $\R^{d-1}$ and let
$E_1:\Edomain\to\SO(d)$\label{E1DEF2} be a smooth map such that 
$\vecv=\vecv(\vecx):=\vece_1 E_1(\vecx)^{-1} \in \S_1^{d-1}$
gives a diffeomorphism from $\Edomain$ to $\S_1^{d-1}$ minus one 
point.\footnote{For example, we may choose $\Edomain=\scrB_\pi^{d-1}$ and
$E_1(\vecx)=K_0^{-1}E(-\vecx)$ for any fixed $K_0\in\SO(d)$, where 
$E(\vecx)=\exp \smatr 0{\vecx}{-\trans\vecx}{0_{d-1}}$.}
The fact that we miss one point in $\S_1^{d-1}$ will not matter for 
us since the measure $\lambda$ is absolutely continuous.

Now for any subset $\fU\subset\HS=\{\vecz\in\S_1^{d-1}\col \vecz\cdot\vece_1>0\}$
and any $\vecw\in\R^d$, $\rho>0$ such that
$\rho \vecw$ lies outside all the balls
$\scrB_\rho^d+\vecy$ ($\vecy\in\Lalf\setminus\{\bn\}$),
we define
\begin{equation} \label{asinfu}
	\scrN_{c,T}^{(\fU)}(\rho,\vecx,\vecw) 
:= \#\bigg\{ \vecy \in (\Lalf \cap \scrB^d_T(c)\setminus\{\vecnull\})-\rho \vecw \col 
\vecy E_1(\vecx) \in \R_{>0}\vece_1+\rho \fU_\perp \bigg\},
\end{equation}
where we write $\fU_\perp:=\{\vecz_\perp \col \vecz\in\fU\}$
with $\vecz_\perp:=\vecz-(\vecz\cdot\vece_1)\vece_1=(0,z_2,\ldots,z_d)$
for any $\vecz=(z_1,\ldots,z_d)\in\R^d$.
Note that $\scrN_{c,T}^{(\fU)}(\rho,\vecx,\vecw)$ 
is the number of points 
$\vecy\in\Lalf\cap\scrB_T^d(c)\setminus\{\bn\}$ such that 
the ray $\rho\vecw+\R_{>0}\vecv$ ($\vecv=\vecv(\vecx)$) hits
the ball $\scrB_\rho^d+\vecy$, with the extra condition 
that $-\vecw_\vecy E_1(\vecx)\in\fU$, where $\vecw_\vecy=\rho^{-1}(\rho\vecw+\tau_\vecy\vecv-\vecy)\in\S^{d-1}_1$ and $\tau_\vecy=\inf\{t>0\col \rho \vecw + t\vecv\in\scrB_\rho^d+\vecy\}$.
Here $\vecw_\vecy$ is the location of the point where the ray
first hits the $\vecy$-sphere, relative to its center $\vecy$.
Hence, similarly as in Section \ref{firstcoll}, $\vecw_\vecy$
\textit{always} satisfies $-\vecw_\vecy E_1(\vecx)\in\HS$.
In particular we have
\begin{equation} \label{NCTOKGEN}
\scrN_{c,T}(\rho,\vecv(\vecx),\vecw)
=\scrN_{c,T}^{(\HS)}(\rho,\vecx,\vecw),
\end{equation}
so that $\scrN_{c,T}^{(\fU)}(\rho,\vecx,\vecw)$ generalizes our
notation from \eqref{asin}.
We will write $\lambda$ and $\vecbeta$ also for the lifts of
$\lambda$ and $\vecbeta$ to the variable $\vecx$.
Thus $\lambda$ is a Borel probability measure on $\R^{d-1}$ 
with bounded support (in fact $\lambda=\lambda|_\Edomain$),
which is absolutely continuous with respect to Lebesgue measure. Furthermore
$\vecbeta$ is a continuous function from $\Edomain$ to $\R^d$.

\begin{thm}\label{visThm2gen}
For every subset $\fU\subset\HS$ with 
$\vol_{\S_1^{d-1}}(\partial \fU)=0$
and for all $\sigma\geq 0$ and $r\in\ZZ_{\geq 0}$, the limit 
\begin{equation} 
	F^{(\fU)}_{c,\vecalf,\vecbeta}(r,\sigma):=\lim_{T\to\infty} 
\lambda(\{ \vecx\in\Edomain \col \scrN^{(\fU)}_{c,T}(\sigma T^{-1/(d-1)},\vecx,\vecbeta(\vecx))=r \})
\end{equation}
exists, and for fixed $\vecalf,\vecbeta,\lambda,r,\fU$ 
the convergence is uniform with respect to $\sigma$ in any 
compact subset of $\RR_{\geq 0}$ and with respect to $c\in [0,1]$.
The limit function is given by
\begin{equation} \label{defFFU}
\begin{split}
&F^{(\fU)}_{c,\vecalf,\vecbeta}(r,\sigma)
\\
&=	\begin{cases}
	(\mu_q\times\lambda)(\{ (M,\vecx)\in X_q\times\Edomain\col
	\#( (\ZZ^d+\frac{\vecp}{q}) M \cap \fZ^{(\fU)}(c,1,\sigma,\vecbeta)|_\vecx)= r \}) 
	& \text{if $\vecalf=\frac{\vecp}{q}\in\QQ^d$}\\
	\mu(\{ (M,\vecxi)\in X\col \#((\ZZ^d M +\vecxi) \cap \fZ^{(\fU)}(c,1,\sigma))= r \}) & \text{if $\vecalf\notin\QQ^d$,}
	\end{cases}
\end{split}
\end{equation}
where
\begin{align} \label{FZFULIMITDEF}
& \fZ^{(\fU)}(c_1,c_2,\sigma)=\bigl\{\vecy=(y_1,\ldots,y_d)\in\R^d \col c_1<y_1<c_2, \: \vecy_\perp\in\sigma\fU_\perp\bigr\};
\\ \notag
& \fZ^{(\fU)}(c_1,c_2,\sigma,\vecbeta)=\bigl\{(\vecx,\vecy)\in\Edomain
\times \R^d\col \vecy\in \fZ^{(\fU)}(c_1,c_2,\sigma)+
(\sigma\vecbeta(\vecx)E_1(\vecx))_\perp\bigr\}.
\end{align}
In particular $F^{(\fU)}_{c,\vecalf,\vecbeta}(r,\sigma)$ is continuous in $\sigma$ and independent of $\scrL$, and if $\vecalf\notin\QQ$ then it is also
independent of $\vecbeta$ and $\lambda$.
\end{thm}

Theorem \ref{visThm2} follows from Theorem \ref{visThm2gen}
by taking $\fU=\HS$. Indeed, $\fZ^{(\HS)}(c,1,\sigma)=\fZ(c,\sigma)$
(except if $\sigma=0$, but then both sets are of measure zero),
and in the case $\vecalf\in\Q^d$ the volume in \eqref{defFFU} equals
\begin{equation}
\int_{\Edomain}
\mu_q\Bigl(\Bigl\{M\in X_q\col\#\bigl((\Z^d+\vecalf)M\cap(\fZ(c,\sigma)
+(\sigma\vecbeta(\vecx)E_1(\vecx))_\perp)\bigr)=r\Bigr\}\Bigr)
\, d\lambda(\vecx).
\end{equation}
Here we may replace ``$(\sigma\vecbeta(\vecx)E_1(\vecx))_\perp$''
with ``$\sigma \|\Proj_{\{\vecv(\vecx)\}^\perp} \vecbeta(\vecx)\|
\cdot\vece_2$'', since (if $d\geq 3$) there is a rotation 
$\begin{pmatrix} 1 & \bn\\\trans\bn & K \end{pmatrix}\in\SO(d)$ which
takes the second vector to the first, and
$\mu_q$ is invariant under the diffeomorphism
$X_q\ni M\mapsto M\begin{pmatrix} 1 & \bn\\\trans\bn & K \end{pmatrix}
\in X_q$.
(If $d=2$: Then either the two vectors are equal, or they correspond to
each other under $\smatr 100{-1}$; in the latter case one chooses
$\gamma_0\in\SL(2,\Z)$ with $\vecalf\gamma_0\smatr 100{-1}=\vecalf$
and then uses the fact
that $M\mapsto \gamma_0\smatr 100{-1} M \smatr 100{-1}$ is a well-defined
automorphism of $X_q$ onto itself, which preserves $\mu_q$.)
Hence we obtain the volume in \eqref{defF}.

\begin{proof}[Proof of Theorem \ref{visThm2gen}]
To prove the desired uniformity,
it suffices to show that, given any continuous functions
$\RR_{>0} \ni T\mapsto \sigma_T \in \RR_{\geq 0}$ and
$\RR_{>0} \ni T\mapsto c_T \in [0,1]$ such that
$\sigma_\infty=\lim_{T\to\infty} \sigma_T$ and
$c_\infty=\lim_{T\to\infty} c_T$ exist, we have
\begin{equation} \label{visThm2wtp}
\lim_{T\to\infty} 
\lambda(\{ \vecx\in\RR^{d-1} : \scrN^{(\fU)}_{c_T,T}(\sigma_T T^{-1/(d-1)},\vecx,\vecbeta(\vecx))=r \})=
F^{(\fU)}_{c_\infty,\vecalf,\vecbeta}(r,\sigma_\infty).
\end{equation}
where the right hand side is given by \eqref{defFFU}.

In the following we let $\SLR$ and $\ASLR$ act on $\RR^{d-1}\times\RR^d$
by leaving the first entry fixed and acting as usual on the second entry:
\begin{equation}
	g: \RR^{d-1}\times\RR^d \to \RR^{d-1}\times\RR^d, \qquad (\vecx,\vecy)\mapsto (\vecx,\vecy g).
\end{equation}
Set, for any $\sigma\geq 0$, $0\leq c_1\leq c_2$, $T>0$,
\begin{align} \label{FZFUTC1C2DEF}
& \fZ^{(\fU)}_T(c_1,c_2,\sigma,\vecbeta) = \bigg\{ (\vecx,\vecy)\in
\Edomain\times\R^d \col \; \; \quad c_1 T \leq \|\vecy\| < c_2 T, \:
\\ \notag & \hspace{100pt} 
\vecy-\sigma T^{-\frac 1{d-1}}\vecbeta(\vecx)E_1(\vecx)
\in \R_{>0} \vece_1+\sigma T^{-\frac 1{d-1}} \fU_\perp
\bigg\}
\begin{pmatrix} 
T^{-1} & \vecnull \\ \trans\vecnull & T^{1/(d-1)} 1_{d-1} 
\end{pmatrix}
\end{align}
We then have for all $\vecx\in\Edomain$,
\begin{equation} \label{SCRNINFZFU}
	\scrN^{(\fU)}_{c,T}(\sigma T^{-1/(d-1)},\vecx,\vecbeta(\vecx))
	= \#\big( \fZ_T^{(\fU)}(c,1,\sigma,\vecbeta)|_\vecx \Phi^{-t} 
\bigl(E_1(\vecx)^{-1},\bn\bigr) \cap (\ZZ^d+\vecalf)M_0 \setminus\{\bn\}\big),
\end{equation}
with %
$T=\e^{(d-1)t}$, 
so long as $T$ is large enough so that the left hand side is defined.

Now taking Lemma \ref{lemCon2fU} (with $c_{2,T}= 1$)
below into account, we see that
\eqref{visThm2wtp} and Theorem~\ref{visThm2gen} follow
immediately from the Theorems of Section \ref{secThin}.
\end{proof}

The flexibility of taking $c_{2,T}\not\equiv 1$ in the
following lemma is not needed for the proof of Theorem \ref{visThm2gen},
but it will be convenient later.

\begin{lem}\label{lemCon2fU}
Let $\sigma_T,c_{1,T},c_{2,T}$ be continuous 
functions of $T>0$ with
$\sigma_T\geq 0$, $0\leq c_{1,T}\leq c_{2,T}$ for all $T>0$,
and such that all three limits $\sigma_\infty=\lim_{T\to\infty}\sigma_T$,
$c_{1,\infty}=\lim_{T\to\infty} c_{1,T}$ and
$c_{2,\infty}=\lim_{T\to\infty} c_{2,T}$ exist.
Then the union
$\cup_{T\geq 1}\fZ^{(\fU)}_T(c_{1,T},c_{2,T},\sigma_T,\vecbeta)$
is bounded, and
\begin{equation}\label{holds2fUa}
\lim(\inf\fZ^{(\fU)}_T(c_{1,T},c_{2,T},\sigma_T,\vecbeta))^\circ 
\supset \widetilde{\fZ}^{(\fU)}(c_{1,\infty},c_{2,\infty},\sigma_\infty,\vecbeta)^\circ
\end{equation}
and
\begin{equation}\label{holds2fUb}
\lim\overline{\sup\fZ^{(\fU)}_T(c_{1,T},c_{2,T},\sigma_T,\vecbeta)} 
\subset  \overline{\widetilde{\fZ}^{(\fU)}(c_{1,\infty},c_{2,\infty},\sigma_\infty,\vecbeta)} 
\end{equation}
(closures and limits taken \textrm{in $\R^{d-1}\times\R^d$}),
where $\widetilde{\fZ}^{(\fU)}(c_1,c_2,\sigma,\vecbeta):=
\fZ^{(\fU)}(c_1,c_2,\sigma,\vecbeta)$ (cf.\ \eqref{FZFULIMITDEF})
if $c_1<c_2$, but $\widetilde{\fZ}^{(\fU)}(c_1,c_1,\sigma,\vecbeta):=
\bigl\{(\vecx,\vecy)\in\Edomain\times\R^d\col
\vecy\in (\{c_1\}\times\sigma\fU_\perp)
+(\sigma\vecbeta(\vecx) E_1(\vecx))_\perp\bigr\}.$
Furthermore the boundary of $\widetilde{\fZ}^{(\fU)}(c_1,c_2,\sigma,\vecbeta)$
intersects $\Edomain\times\R^d$ in a set of Lebesgue measure zero.
\end{lem}

\begin{proof}
Let $C=1+\sup_{\Edomain} \|\vecbeta\|$.
Take $T>0$ and consider an arbitrary point 
$(\vecx,\vecy)\in\fZ^{(\fU)}_T(c_{1,T},c_{2,T},\sigma_T,\vecbeta)$.
Set
$\vecy'=Ty_1\vece_1+T^{-\frac 1{d-1}}\vecy_\perp$; then
$c_{1,T}T\leq\|\vecy'\|< c_{2,T}T$
and $\vecy'-\sigma_T T^{-\frac 1{d-1}}\vecbeta(\vecx)E_1(\vecx)
\in\R_{>0}\vece_1+\sigma_T T^{-\frac 1{d-1}}\fU_\perp$.
From these we conclude
\begin{align} \label{lemCon2fUrel1}
-\sigma_T T^{-\frac d{d-1}} \sup \|\vecbeta\|<y_1<c_{2,T}
\quad\text{and}\quad
(\vecy-\sigma_T \vecbeta(\vecx)E_1(\vecx))_\perp
\in \sigma_T \fU_\perp.
\end{align}
Since $\fU_\perp\subset \scrB_1^d$ the last relation
implies $\|\vecy_\perp\|\leq C\sigma_T$.
The first claim of the lemma follows from the inequalities noted so far.

Now let $\eta>0$ be given, and take $T_0$ so large that
$c_{1,T}>c_{1,\infty}-\frac{\eta}2$,
$c_{2,T}<c_{2,\infty}+\eta$,
$C\sigma_T T^{-\frac d{d-1}}<\frac{\eta}8$ and
$|\sigma_T-\sigma_\infty|<\eta/C$ hold for all $T\geq T_0$.
Let $T\geq T_0$ and consider any point 
$(\vecx,\vecy)\in\fZ^{(\fU)}_T(c_{1,T},c_{2,T},\sigma_T,\vecbeta)$.
Then by \eqref{lemCon2fUrel1} we have $y_1\geq -C\sigma_T T^{-\frac d{d-1}}
>-\frac{\eta}8$, but using $\|\vecy'\|\geq c_{1,T}T$ we also
conclude $|y_1|\geq c_{1,T}-C\sigma_T T^{-\frac d{d-1}}
>c_{1,\infty}-\frac{5\eta}8$.
Together these two inequalities imply in particular that
$y_1>c_{1,\infty}-\eta$.
Also, by \eqref{lemCon2fUrel1},
$y_1<c_{2,T}<c_{2,\infty}+\eta$.
From \eqref{lemCon2fUrel1} we also see that there is some $\vecw\in\fU_\perp$
such that
$(\vecy-\sigma_T\vecbeta(\vecx)E_1(\vecx))_\perp=\sigma_T\vecw$.
Thus
\begin{equation}
(\vecy-\sigma_\infty\vecbeta(\vecx)E_1(\vecx))_\perp
=\sigma_\infty\vecw+(\sigma_T-\sigma_\infty)(\vecw
+(\vecbeta(\vecx)E_1(\vecx))_\perp),
\end{equation}
and here
$\bigl\| \vecw +(\vecbeta(\vecx)E_1(\vecx))_\perp \bigr\|<C$
and $|\sigma_T-\sigma_\infty|<\eta/C$,
so that
\begin{equation}
(\vecy-\sigma_\infty\vecbeta(\vecx)E_1(\vecx))_\perp
\in \sigma_\infty \fU_\perp + \scrB_\eta^d.
\end{equation}
Hence we have proved that for each $T\geq T_0$ we have
\begin{align}  \notag
\fZ^{(\fU)}_T(c_{1,T},c_{2,T},\sigma_T,\vecbeta) \subset
\Big\{ (\vecx,\vecy)\in\Edomain\times\R^d \col
c_{1,\infty}-\eta < \vecy\cdot\vece_1< c_{2,\infty}+\eta ,\; 
\hspace{70pt}&
\\
(\vecy-\sigma_\infty \vecbeta(\vecx)E_1(\vecx))_\perp
\in \sigma_\infty \fU_\perp + \scrB_\eta^d \Big\} &.
\end{align}
We have seen that such a $T_0$ exists for any $\eta>0$;
this fact leads easily to \eqref{holds2fUb}.

We now turn to \eqref{holds2fUa}.
Assume $(\vecx_0,\vecy_0)\in \widetilde{\fZ}^{(\fU)}(c_{1,\infty},c_{2,\infty},
\sigma_\infty,\vecbeta)^\circ$, and take 
$\eta>0$ so that
\begin{align} \label{lemCon2fUrel2}
(\vecx_0+\scrB_{2\eta}^{d-1}) \times (\vecy_0+\scrB_{2\eta}^d)\subset
\widetilde{\fZ}^{(\fU)}(c_{1,\infty},c_{2,\infty},\sigma_\infty,\vecbeta).
\end{align}
Then we must have $\sigma_\infty>0$ and $c_{1,\infty}<c_{2,\infty}$.
Take $T_0$ so large that each of the following five inequalities
hold when $T\geq T_0$:
\begin{equation}
\begin{split}
& \sigma_T>0; \qquad
|\frac{\sigma_\infty}{\sigma_T}-1|<\frac{\eta}{C\sigma_\infty}; \qquad
\sigma_T T^{-\frac d{d-1}}<\frac{\eta}C;
\\
& c_{1,T}\leq c_{1,\infty}+\eta; \qquad
c_{2,T}-C\sigma_\infty T^{-\frac d{d-1}} > c_{2,\infty}-\eta.
\end{split}
\end{equation}
We then claim 
\begin{align} \label{lemCon2fUrel3}
(\vecx_0+\scrB_{\eta}^{d-1}) \times (\vecy_0+\scrB_{\eta}^d)\subset
\fZ_T^{(\fU)}(c_{1,T},c_{2,T},\sigma_T,\vecbeta),
\qquad \forall T\geq T_0.
\end{align}
This implies 
$(\vecx_0,\vecy_0)\in
\lim(\inf\fZ^{(\fU)}_T(c_{1,T},c_{2,T},\sigma_T,\vecbeta))^\circ$,
and hence \eqref{holds2fUa} will be proved,
since $(\vecx_0,\vecy_0)$ was arbitrary in 
$\widetilde{\fZ}^{(\fU)}(c_{1,\infty},c_{2,\infty},\sigma_\infty,\vecbeta)^\circ$.

To prove \eqref{lemCon2fUrel3}, let $(\vecx,\vecy)$ be an arbitrary
point in $(\vecx_0+\scrB_{\eta}^{d-1}) \times (\vecy_0+\scrB_{\eta}^d)$,
and take $T\geq T_0$.
Write $\vecy'=Ty_1\vece_1+T^{-\frac 1{d-1}}\vecy_\perp$.
Using $C\sigma_T T^{-\frac d{d-1}}<\eta$ we get
\begin{align} \label{NEEDED3A}
Ty_1-\sigma_T T^{- \frac 1{d-1}}\vecbeta(\vecx)E_1(\vecx)\cdot \vece_1
\geq T\eta-\sigma_T T^{-\frac 1{d-1}} \sup \|\vecbeta\|>0.
\end{align}
Next \eqref{lemCon2fUrel2} implies
$(\vecy-\sigma_\infty \vecbeta(\vecx)E_1(\vecx))_\perp
+\bigl(\{0\}\times \scrB_\eta^{d-1}\bigr)
\subset \sigma_\infty \fU_\perp$.
In particular $\|\vecy_\perp\|<C\sigma_\infty$,
and using
$|\frac{\sigma_\infty}{\sigma_T}-1|<\frac{\eta}{C\sigma_\infty}$
we get $\bigl | \frac{\sigma_\infty}{\sigma_T}-1 \bigr |\cdot 
\|\vecy_\perp\| <\eta$ and hence
\begin{align}
(\vecy-\sigma_\infty \vecbeta(\vecx)E_1(\vecx))_\perp
+\Bigl(\frac{\sigma_\infty}{\sigma_T}-1\Bigr) \vecy_\perp
\in \sigma_\infty \fU_\perp.
\end{align}
In other words $(\vecy-\sigma_T \vecbeta(\vecx)E_1(\vecx))_\perp
\in \sigma_T \fU_\perp$, and thus
\begin{align}\label{NEEDED3B}
(\vecy'-\sigma_T T^{-\frac 1{d-1}}\vecbeta(\vecx)E_1(\vecx))_\perp
\in \sigma_T T^{-\frac 1{d-1}} \fU_\perp.
\end{align}
Finally \eqref{lemCon2fUrel2} gives
$c_{1,\infty}+\eta\leq y_1 \leq c_{2,\infty}-\eta$, and 
using $c_{1,T}\leq c_{1,\infty}+\eta$ and
$c_{2,\infty}-\eta< c_{2,T}-C\sigma_\infty T^{-\frac d{d-1}}$
we obtain
\begin{align} \label{NEEDED2}
c_{1,T}T\leq \|\vecy'\|< c_{2,T}T.
\end{align} 
But \eqref{NEEDED2}, \eqref{NEEDED3A}, \eqref{NEEDED3B}
imply
$(\vecx,\vecy)\in\fZ_T^{(\fU)}(c_{1,T},c_{2,T},\sigma_T,\vecbeta)$,
and hence \eqref{lemCon2fUrel3} is proved.

Finally, the fact that $(\Edomain\times \R^d)\cap 
\partial \widetilde{\fZ}^{(\fU)}(c_1,c_2,\sigma,\vecbeta)$ has
Lebesgue measure zero follows from 
\begin{align} \label{BOUNDARYCLAIMFU}
& (\Edomain\times \R^d) \cap 
\partial \widetilde{\fZ}^{(\fU)}(c_1,c_2,\sigma,\vecbeta)
\\ \notag
& \subset
\bigl\{(\vecx,\vecy)\in \Edomain\times \R^d \col
\vecy\cdot\vece_1 \in \{c_1,c_2\}, \:
\vecy_\perp \in \sigma (\vecbeta(\vecx) E_1(\vecx))_\perp
+\bigl(\{0\}\times\overline{\scrB_{\sigma}^{d-1}}\bigr)\bigr\}
\\ \notag
& \qquad \cup \:
\bigl\{(\vecx,\vecy)\in \Edomain\times \R^d \col
c_1<\vecy\cdot\vece_1< c_2, \:
\vecy_\perp \in \sigma (\vecbeta(\vecx) E_1(\vecx))_\perp
+ \partial (\sigma \fU_\perp)\bigr\},
\end{align}
using $\partial (\fU_\perp)=(\partial \fU)_\perp$,
and our assumption that $\vol_{\S^{d-1}_1}(\partial \fU)=0$.
\end{proof}

The proof of Theorem \ref{prim-visThm2} is almost identical to the
proof of Theorem \ref{visThm2}, using the theorems of Section 
\ref{SUBSEC:PRIMLATP}.

We proceed to the proofs of Theorems \ref{visThm3} and \ref{prim-visThm3}. 
To be in line with the notation used in the previous proofs, we 
again write $\vecv=\vece_1 E_1(\vecx)^{-1}$ ($\vecx\in\Edomain$), and write
$\lambda$ also for the lift of $\lambda$ to the variable $\vecx$.
Set
\begin{multline}\label{ZcQ}
\fZ_T(c,\scrQ) = \big\{ (\vecx,\vecy)\in\Edomain\times\RR^d \col c T \leq \|\vecy\| < T,\; 
\\
\R_{>0}\vece_1  \cap (\scrQ_T E_1(\vecx) + \vecy) \neq \emptyset \big\}
\begin{pmatrix} 
T^{-1} & \vecnull \\ \trans\vecnull & T^{1/(d-1)} 1_{d-1}
\end{pmatrix} .
\end{multline}
For the counting function defined in \eqref{asin2} we have
\begin{equation} \label{NQTFORMULA}
	\scrN_{c,T}(\scrQ,\vece_1 E_1(\vecx)^{-1})
	= \#\big( \fZ_T(c,\scrQ)|_\vecx \Phi^{-t} (E_1(\vecx)^{-1},\bn) \cap (\ZZ^d+\vecalf)M_0\setminus\{\bn\} \big) 
\end{equation}
with $T=\e^{(d-1)t}$.
The primitive case is analogous.

Theorems \ref{visThm3} and \ref{prim-visThm3} are again a consequence of the theorems in Section \ref{secThin} and the following lemma.

\begin{lem}\label{lemCon3}
The union $\cup_{T\geq 1} \fZ_T(c,\scrQ)$ is bounded, and we have
\begin{equation}\label{holds3}
	\lim(\inf\fZ_T(c,\scrQ))^\circ \supset  \fZ(c,\scrQ)^\circ, \qquad
	\lim\overline{\sup\fZ_T(c,\scrQ)} \subset  \overline{\fZ(c,\scrQ)} ,
\end{equation}
where
\begin{equation} \label{NQ}
\fZ(c,\scrQ)  := \big\{ (\vecx,\vecy)\in\Edomain\times\RR^d \col  c < y_1 < 1 ,\; (y_2,\ldots,y_d)\in -(\scrQ E_1(\vecx))_\perp \big\} 
\end{equation}
is a bounded set whose boundary intersects $\Edomain\times\RR^d$ in a set
of Lebesgue measure zero.
\end{lem}

\begin{proof}
This is very similar to the proof of Lemma \ref{lemCon2fU}
(but slightly easier, since $c$ and $\scrQ$ are kept fixed).
To prove the last statement one first verifies that
\begin{align} \notag
(\Edomain\times\R^d)\cap\partial\fZ(c,\scrQ)
\subset &\Bigl\{(\vecx,\vecy)\in\Edomain\times\R^d\col y_1\in\{c,1\},
\: \vecy_\perp\in -\overline{(\scrQ E_1(\vecx))_\perp}\Bigr\}
\\ \label{lexCon3step8}
& \cup \Bigl\{(\vecx,\vecy)\in\Edomain\times\R^d\col y_1\in [c,1],
\: \vecy_\perp\in -\partial\bigl((\scrQ E_1(\vecx))_\perp\bigr)\Bigr\}.
\end{align}
Here the first set clearly has measure zero, and the second set has measure
\begin{align}
(1-c)\int_{\vecx\in\Edomain} \vol_{\R^{d-1}} \bigl(
\partial\bigl((\scrQ E_1(\vecx))_\perp\bigr)\bigr)\,d\vecx,
\end{align}
which is
zero exactly because of the technical assumption made just below \eqref{QTDEF}.
\end{proof}

Lemma \ref{lemCon3} is applied in the following way:
If $\vecalf\in\Q^d$ then by \eqref{NQTFORMULA},
Remark \ref{thinCor-rat} and Lemma \ref{lemCon3} the
limit in \eqref{visThm3limit} exists, and equals
\begin{align} \label{visThm3proofstep1}
\int_\Edomain \int_{X_q} I \Bigl(
\#\bigl(\fZ(c,\scrQ)|_\vecx \cap (\Z^d+\vecalf)M\bigr)=r\Bigr)
\, d\mu_q(M)\, d\lambda(\vecx)
\end{align}
But we have from \eqref{NQ}, since $\vecv=\vece_1 E_1(\vecx)^{-1}$:
\begin{equation}
\begin{split}
\fZ(c,\scrQ)|_\vecx 
& = \big\{ \vecy\in\RR^d \col  c < \vecy\cdot\vece_1 < 1 ,\; \RR \vece_1 \cap (\scrQ E_1(\vecx) + \vecy) \neq \emptyset \big\} \\
& = \big\{ \vecy\in\RR^d \col  c < \vecy\cdot\vecv < 1 ,\; \RR \vecv \cap (\scrQ + \vecy) \neq \emptyset \big\} E_1(\vecx).
\end{split}
\end{equation}
Hence by substituting $M=M'E_1(\vecx)$ in the inner integral in
\eqref{visThm3proofstep1} we obtain the formula stated in Theorem
\ref{visThm3}. The proof in the case $\vecalf\notin\Q^d$ is entirely
similar, and so is the proof of Theorem \ref{prim-visThm3}.

\subsection{Averaging over $\vecalf$}\label{fT1corproofsec}

Naturally, one can also prove $\vecalf$-averaged 
(or $\vecq$-averaged)  %
versions of all the limit results 
obtained in the present paper. In this section we discuss this
to the extent necessary to give a proof of Theorem \ref{freeThm1cor}.

We first give an averaged version of Corollary \ref{freeCor1}.
Recall that if $\vecalf\notin\Q^d$ then %
$\Phi_\vecalf(\xi)$ %
is independent of $\vecalf$,
and we write $\Phi(\xi)$ for this function. %
\begin{cor} \label{freeCor1-ave}
Fix a lattice $\scrL=\Z^d M_0$ and let $\lambda$ be a Borel probability 
measure on $\T^1(\R^d)=\R^d\times\S_1^{d-1}$ which is absolutely continuous 
with respect to Lebesgue measure $\vol_{\R^d}\times\vol_{\S_1^{d-1}}$.
Then, for every $\xi\geq 0$,
\begin{equation} \label{freeCor1-ave-equ}
\lim_{\rho\to 0} 
\lambda(\{ (\vecq,\vecv)\in\T^1(\scrK_\rho) \col  \rho^{d-1} 
\tau_1(\vecq,\vecv;\rho)\geq \xi \})
= \int_\xi^\infty \Phi(\xi') \, d\xi' .
\end{equation}
\end{cor}
\begin{proof}
By the Theorem of Radon-Nikodym %
we have
$d\lambda(\vecq,\vecv)=f(\vecq,\vecv) \, d\vecq \,d\!\vol_{\S_1^{d-1}}(\vecv)$
for some non-negative function $f\in \L^1(\R^d\times\S_1^{d-1})$
with $\|f\|_{\L^1}=1$.
By Fubini's Theorem, the left hand side of \eqref{freeCor1-ave-equ} equals
\begin{align} \label{freeCor1-ave-equ-step1}
\lim_{\rho\to 0} \int_{\R^d} \Bigl( \int_{\S_1^{d-1}}
I\bigl( \rho^{d-1} \tau_1(\vecq,\vecv;\rho)\geq \xi \bigr)
f(\vecq,\vecv) \, d\!\vol_{\S_1^{d-1}}(\vecv) \Bigr)\, d\vecq,
\end{align}
where the indicator function $I\bigl(\ldots\bigr)$ is interpreted as zero
whenever $\vecq\notin\scrK_\rho$.
For almost every $\vecq\in\R^d$ we have 
$f(\vecq,\cdot)\in \L^1(\S_1^{d-1})$ and $-\vecq M_0^{-1}\notin\Q^d$,
and for each such (fixed) point $\vecq$,
Corollary \ref{freeCor1} implies that
the inner integral in \eqref{freeCor1-ave-equ-step1} tends to 
\begin{align}
\Bigl(\int_{\S_1^{d-1}} f(\vecq,\vecv)\,d\!\vol_{\S_1^{d-1}}(\vecv)\Bigr)
\cdot \int_\xi^\infty \Phi(\xi')\, d\xi' 
\qquad \text{as } \: \rho\to 0.
\end{align}
By Lebesgue's Bounded Convergence Theorem 
(with
$\vecq\mapsto \int_{\S_1^{d-1}} f(\vecq,\vecv)\,d\!\vol_{\S_1^{d-1}}(\vecv)$
as a majorant function), we may change the 
order between $\lim_{\rho\to 0}$ and $\int_{\R^d}$ in 
\eqref{freeCor1-ave-equ-step1}, thus obtaining \eqref{freeCor1-ave-equ}.
\end{proof}

\begin{proof}[Proof of Theorem \ref{freeThm1cor}]
Let $M$ be the set of non-negative functions
$f\in \L^1(\R^d\times\S_1^{d-1})$ with $\|f\|_{\L^1}=1$.
By the Theorem of Radon-Nikodym and \eqref{TAU1TAU1DEF}, 
our task is to prove that for each $f\in M$ we have
\begin{align} \label{freeThm1cor-step1}
\lim_{\rho\to 0} \int_{\scrK_\rho} \int_{\S_1^{d-1}} 
I\Bigl( \rho^{d-1}\tau_1(\vecq,\vecv;\rho)\geq\xi\Bigr) 
\rho^{d(d-1)} f(\rho^{d-1}\vecq,\vecv) \,d\!\vol_{\S_1^{d-1}}(\vecv)\, d\vecq 
= %
\int_\xi^\infty \Phi(\xi') \, d\xi' & .
\end{align}
In fact it suffices to prove \eqref{freeThm1cor-step1}
when $f\in M$ is continuous and of compact support, since the
subset of such functions is dense in $M$
with respect to the $\L^1$-norm.

Using the $\scrL$-periodicity of $\tau_1(\cdot,\vecv;\rho)$,
the double integral in \eqref{freeThm1cor-step1} can be expressed as
\begin{align} \label{freeThm1cor-step2}
& \int_{F\cap\scrK_\rho} \int_{\S_1^{d-1}} 
I\Bigl( \rho^{d-1}\tau_1(\vecq_0,\vecv;\rho)\geq\xi\Bigr) 
\Bigl\{\rho^{d(d-1)} \sum_{\vecq\in\vecq_0+\scrL} f(\rho^{d-1}\vecq,\vecv)
\Bigr\}
\,d\!\vol_{\S_1^{d-1}}(\vecv)\, d\vecq_0,
\end{align}
where $F\subset\R^d$ is a fundamental parallelogram for $\scrL$.
But for $f$ continuous and of compact support,
the expression within the brackets in \eqref{freeThm1cor-step2} tends to
$h(\vecv):=\int_{\R^d} f(\vecq,\vecv)\,d\vecq$ as $\rho\to 0$, 
uniformly with respect to $\vecv\in\S_1^{d-1}$ and $\vecq_0\in F$.
Hence Theorem \ref{freeThm1cor} follows from Corollary~\ref{freeCor1-ave}, 
applied with $d\lambda(\vecq,\vecv)=\chi_F(\vecq)h(\vecv) \, d\vecq\,
d\!\vol_{\S_1^{d-1}}(\vecv)$.
\end{proof}

\subsection{Proofs for Section \ref{secLorentz}}\label{exactpos1proofsec}

\begin{proof}[Proof of Theorem \ref{exactpos1}]

As in Section \ref{secVisproofssec} we fix a smooth map
$E_1:\Edomain\to\SO(d)$ such that $\vecv=\vecv(\vecx)=\vece_1 E_1(\vecx)^{-1}
\in\S_1^{d-1}$ gives a diffeomorphism between the bounded open set
$\Edomain\subset\R^{d-1}$ and $\S_1^{d-1}$ minus one point.
However we now make the extra requirement that
$E_1(\vecx)=K(\vecv(\vecx))$ for all $\vecx\in\Edomain$.\footnote{For example, we may choose $\Edomain=\scrB_\pi^{d-1}$ and 
$E_1(\vecx)=K(\vece_1 E(\vecx)K_0)$ where 
$E(\vecx)=\exp \smatr 0{\vecx}{-\trans\vecx}{0_{d-1}}$ and
$K_0$ is any fixed matrix in $\SO(d)$ such that $\vecv=-\vece_1 K_0$
is the unique point where $K(\vecv)$ is not smooth.}

We again write $\lambda$ and $\vecbeta$ also for the lifts of $\lambda$
and $\vecbeta$ to the variable $\vecx$.
Now the measure appearing in the limit in \eqref{exactpos1eq} equals,
with $\vecq_{\rho,\vecbeta}(\vecx)=\vecq+\rho\vecbeta(\vecx)$:
\begin{align} \label{exactpos1lambdarestat}
\lambda\bigl(\bigl\{ \vecx\in\Edomain \col 
\rho^{d-1} \tau_1(\vecq_{\rho,\vecbeta}(\vecx),\vecv(\vecx);\rho)
\in [\xi_1,\xi_2), \:
-\vecw_1(\vecq_{\rho,\vecbeta}(\vecx),\vecv(\vecx);\rho)
\in \fU E_1(\vecx)^{-1}\bigr\}\bigr).
\end{align}
This is well defined for $\rho$ small;
more specifically, if $\rho$ is sufficiently
small then $(\vecq_{\rho,\vecbeta}(\vecx),\vecv(\vecx))\in \T^1(\scrK_\rho)$ 
for all $\vecx\in\Edomain$,
so that $\tau_1(\vecq_{\rho,\vecbeta}(\vecx),\vecv(\vecx);\rho)$ and
(if $\tau_1<\infty$)
$\vecw_1(\vecq_{\rho,\vecbeta}(\vecx),\vecv(\vecx);\rho)$ are defined.
(For recall that if $\vecq\in\scrL$ then by our assumption on $\vecbeta$
we have $\|\vecbeta(\vecx)\|\geq 1$ everywhere.)

For technical reasons we will prove Theorem \ref{exactpos1} under the
extra assumption that $\xi_1>0$. This is no loss of generality, for once
that proof is complete, the remaining case $\xi_1=0$ follows
by a simple limit argument, using Corollary \ref{freeCor2} in the form
$\lim_{\rho\to 0} \lambda(\{\vecv\in\S_1^{d-1}\col \rho^{d-1}\tau_1<\xi\})
=1-F_{0,\vecalf,\vecbeta}(0,\xi^{1/(d-1)})$ together with the fact that
$\lim_{\xi\to 0} F_{0,\vecalf,\vecbeta}(0,\xi^{1/(d-1)})=1$
(cf.\ Remark \ref{GENBETAUPPERBOUNDREM}).

The measure in \eqref{exactpos1lambdarestat} can be bounded from
above and below using the counting function 
$\scrN^{(\fU)}_{c,T}(\rho,\vecx,\vecw)$
(cf.\ \eqref{asinfu}),
taken with respect to the affine lattice $\Lalf=\scrL-\vecq$, as follows.
We will use the shorthand notation
$\scrN_{c,T}(\rho,\vecx,\vecw):=\scrN^{(\HS)}_{c,T}(\rho,\vecx,\vecw)$,
which is natural in view of \eqref{NCTOKGEN}.
Let $C=1+\sup_{\Edomain} \|\vecbeta\|$.
Now for any $0<\xi_1<\xi_2$ and any $\rho>0$ so small that
$\xi_1\rho^{1-d}-C\rho>0$, $\xi_1 \rho^{1-d}+C\rho<\xi_2 \rho^{1-d}-C\rho$
and $(\vecq_{\rho,\vecbeta}(\vecx),\vecv(\vecx))\in \T^1(\scrK_\rho)$ 
for all $\vecx\in\Edomain$, we have:
\begin{align} \label{exactpos2ul}
\lambda\bigl(\bigl\{ \vecx\in\Edomain \col 
\scrN_{0,T_1}(\rho,\vecx,\vecbeta(\vecx))=0, \:
\scrN^{(\fU)}_{c_2,T_2}(\rho,\vecx,\vecbeta(\vecx))\geq 1, \:
\scrN_{c_3,T_3}(\rho,\vecx,\vecbeta(\vecx))\leq 1\bigr\}\bigr) \hspace{40pt}&
\\ \notag
\leq
\lambda\bigl(\bigl\{ \vecx\in\Edomain \col 
\rho^{d-1} \tau_1(\vecq_{\rho,\vecbeta}(\vecx),\vecv(\vecx);\rho)
\in [\xi_1,\xi_2), \:
-\vecw_1(\vecq_{\rho,\vecbeta}(\vecx),\vecv(\vecx);\rho)
\in \fU E_1(\vecx)^{-1}\bigr\}\bigr) &
\\ \notag
\leq \lambda\bigl(\bigl\{ \vecx\in\Edomain \col 
\scrN_{0,T_4}(\rho,\vecx,\vecbeta(\vecx))=0, \:
\scrN^{(\fU)}_{c_5,T_5}(\rho,\vecx,\vecbeta(\vecx))\geq 1\bigr\}\bigr),
\hspace{80pt}&
\end{align}
where $T_j>0$, $c_j\in [0,1]$ are defined through
$T_1=c_2T_2=c_3T_3=\xi_1 \rho^{1-d}+C\rho$,
$T_2=\xi_2 \rho^{1-d}-C\rho$,
$T_3=T_5=\xi_2 \rho^{1-d}+C\rho$,
$T_4=c_5T_5=\xi_1 \rho^{1-d}-C\rho$.

To prove \eqref{exactpos2ul},
let $\vecx$ be any point in $\Edomain$ with
$\scrN_{0,T_1}(\rho,\vecx,\vecbeta(\vecx))=0,$
$\scrN^{(\fU)}_{c_2,T_2}(\rho,\vecx,\vecbeta(\vecx))\geq 1$ and
$\scrN_{c_3,T_3}(\rho,\vecx,\vecbeta(\vecx))\leq 1$.
To show the first inequality in \eqref{exactpos2ul}
it suffices to prove that these conditions imply
$\tau_1:=\tau_1(\vecq_{\rho,\vecbeta}(\vecx),\vecv(\vecx);\rho)
\in [\rho^{1-d} \xi_1,\rho^{1-d} \xi_2)$ and
$\vecw_1:=\vecw_1(\vecq_{\rho,\vecbeta}(\vecx),\vecv(\vecx);\rho)
\in-\fU E_1(\vecx)^{-1}$.

It follows from $\scrN^{(\fU)}_{c_2,T_2}(\rho,\vecx,\vecbeta(\vecx))\geq 1$ 
that there is some 
$\vecy\in\Lalf\setminus\{\bn\}$ with
$\xi_1\rho^{1-d}+C\rho \leq \|\vecy\| <\xi_2\rho^{1-d}-C\rho$
and $(\vecy-\rho \vecbeta(\vecx)) E_1(\vecx)
\in \R_{>0}\vece_1+\rho\fU_\perp$.
Since $\fU\subset\HS$ it follows that there exist 
$\vecw\in\fU$ and $t>-\rho$ such that
$(\vecy-\rho \vecbeta(\vecx)) E_1(\vecx)=t\vece_1+\rho \vecw$.
This implies in particular that $\|\vecy\|-C\rho\leq t\leq \|\vecy\|+C\rho$,
and thus
\begin{align}
\xi_1\rho^{1-d}\leq t <\xi_2\rho^{1-d}.
\end{align}
Set $\vecy':=\vecy+\vecq\in\scrL\setminus\{\vecq\}$ and recall
$\vecv=\vecv(\vecx)=\vece_1 E_1(\vecx)^{-1}$;
then our equality says
\begin{align}
\vecq_{\rho,\vecbeta}(\vecx)+t\vecv=\vecy'-\rho \vecw E_1(\vecx)^{-1}.
\end{align}
This implies $\tau_1\leq t$.
Furthermore, using $\scrN_{0,T_1}(\rho,\vecx,\vecbeta(\vecx))=0$ 
together with our requirement that if $\vecq\in\scrL$ then
$(\vecbeta(\vecv)+\R_{>0}\vecv)\cap\scrB_1^d=\emptyset$
for all $\vecv\in\S_1^{d-1}$, we conclude $\xi_1 \rho^{1-d}\leq \tau_1$.

We claim that in fact $\tau_1=t$ holds. Assume the opposite;
then we have $\xi_1 \rho^{1-d}\leq \tau_1<t<\xi_2\rho^{1-d}$.
By the definition of $\tau_1$ there exists some 
$\vecy''\in\Lalf\setminus\{\bn\}$ such that 
$\rho\vecbeta(\vecx)+(\tau_1+\ve)\vecv\in \scrB_\rho^d+\vecy''$ for all 
sufficiently small $\ve>0$.
Then $\|\vecy''\|\leq\tau_1+C\rho$, and also since
$\scrN_{0,T_1}(\rho,\vecx,\vecbeta(\vecx))=0$ we must have
$\|\vecy''\|\geq T_1=c_3T_3$;
hence we see that $\vecy''-\rho\vecbeta(\vecx)$ lies in the set defining
$\scrN_{c_3,T_3}(\rho,\vecx,\vecbeta(\vecx))$.
But $\vecy-\rho\vecbeta(\vecx)$ \textit{also} lies in this
set, and from $\tau_1<t$ we see that $\vecy\neq\vecy''$.
Hence $\scrN_{c_3,T_3}(\rho,\vecx,\vecbeta(\vecx))\geq 2$,
contradicting our assumptions.
Having thus proved $\tau_1=t$ we obtain
$\vecw_1 %
=-\vecw E_1(\vecx)^{-1}$
by the definition of $\vecw_1$, and hence both
$\tau_1\in [\xi_1\rho^{1-d},\xi_2\rho^{1-d})$ and
$\vecw_1 \in -\fU E_1(\vecx)^{-1}$.
Hence the proof of the first inequality in \eqref{exactpos2ul}
is completed.

The proof of the second inequality in \eqref{exactpos2ul} is easier,
and we leave it to the reader.

Continuing onwards, let us note the following mild generalization
of \eqref{SCRNINFZFU}.
For all $\sigma\geq 0$, $c\geq 0$, $c'>0$,
$\vecx\in\Edomain$ and any $T>0$ so large that the left hand
side is defined, we have
\begin{equation} \label{SCRNINFZFUC1C2}
	\scrN^{(\fU)}_{c,T}(\sigma T^{-\frac 1{d-1}},\vecx,\vecbeta(\vecx))
	= \#\Bigl( \fZ_{T/c'}^{(\fU)}(c'c,c',{c'}^{-\frac 1{d-1}}\sigma,\vecbeta)|_\vecx \Phi^{-t} \bigl(E_1(\vecx)^{-1},\bn\bigr) \cap (\ZZ^d+\vecalf)M_0 \setminus\{\bn\}\Bigr),
\end{equation}
with $T/c'=\e^{(d-1)t}$.
This follows directly from
\eqref{SCRNINFZFU} combined with the invariance relation 
\begin{align} \label{FZCINFZC1C2}
\fZ_{T/c'}^{(\fU)}(c'c,c',{c'}^{-\frac 1{d-1}}\sigma,\vecbeta)
=\fZ_{T}^{(\fU)}(c,1,\sigma,\vecbeta)
\begin{pmatrix} c' & \bn \\ \trans\bn & c'^{-1/(d-1)} 1_{d-1} \end{pmatrix},
\end{align}
which can be verified straight from the definition \eqref{FZFUTC1C2DEF}.

In \eqref{exactpos2ul}, introduce $\sigma_1,\ldots,\sigma_5$ through
$\rho=\sigma_j T_j^{-\frac 1{d-1}}$.
Using \eqref{SCRNINFZFUC1C2} and $T_1=c_2T_2=c_3T_3$ we see that
when $\rho$ is sufficiently small,
the left hand side in \eqref{exactpos2ul} can be expressed as
\begin{align} \notag
\lambda\Bigl(\Bigl\{\vecx\in\Edomain\col
\#\bigl(\fZ_{T_1/\xi_1}(0,\xi_1,\xi_1^{-\frac 1{d-1}}\sigma_1,\vecbeta)|_\vecx \Phi^{-t} \bigl(E_1(\vecx)^{-1},\bn\bigr) \cap (\ZZ^d+\vecalf)M_0  \setminus\{\bn\}\big)=0,\: \hspace{30pt} &
\\ \label{exactpos2ullowref}
\#\bigl(\fZ_{T_1/\xi_1}^{(\fU)}(\xi_1,\sfrac{\xi_1}{c_2},(\sfrac{\xi_1}{c_2})^{-\frac 1{d-1}}\sigma_2,\vecbeta)|_\vecx \Phi^{-t} \bigl(E_1(\vecx)^{-1},\bn\bigr) \cap (\ZZ^d+\vecalf)M_0  \setminus\{\bn\}\big)\geq 1,\: \hspace{10pt} &
\\ \notag
\#\bigl(\fZ_{T_1/\xi_1}(\xi_1,\sfrac{\xi_1}{c_3},(\sfrac{\xi_1}{c_3})^{-\frac 1{d-1}}\sigma_3,\vecbeta)|_\vecx \Phi^{-t} \bigl(E_1(\vecx)^{-1},\bn\bigr) \cap (\ZZ^d+\vecalf)M_0  \setminus\{\bn\}\big)\leq 1\Bigr\}\Bigr) &
\end{align}
with $\e^{(d-1)t}=T_1/\xi_1$, and using the notation
$\fZ_{c,T}(c_1,c_2,\sigma,\vecbeta):=
\fZ^{(\HS)}_{c,T}(c_1,c_2,\sigma,\vecbeta)$.
Recall that all $c_j,\sigma_j,T_j$ are functions of $\rho$,
and, when $\rho\to 0$, we have $T_j\to\infty$,
$\sigma_1\to \xi_1^{\frac 1{d-1}}$,
$\sigma_2,\sigma_3\to \xi_2^{\frac 1{d-1}}$,
and $c_2,c_3\to\xi_1/\xi_2$.
Note that $e^{(d-1)t}=T_1/\xi_1=\rho^{1-d}+\sfrac{C}{\xi_1}\rho$ 
is strictly decreasing as a function
of $\rho$ for small $\rho>0$; hence for small $\rho$ 
($\Leftrightarrow$ large $t$)
we may instead view $\rho$ as a function of $t$;
then also all $c_j$, $\sigma_j$, $T_j$ are functions of $t$.
Now, using an obvious shorthand notation, we have the
following sieving type identity for \eqref{exactpos2ullowref}:
\begin{align} \notag
& \lambda \bigl(\bigl\{\vecx \in \Edomain\col
\# F_{t,\vecx}^{(1)}=0, \:
\# F_{t,\vecx}^{(2)}\geq 1, \:
\# F_{t,\vecx}^{(3)}\leq 1\bigr\}\bigr) %
\\ 
& =
\lambda\bigl(\bigl\{\vecx \col \# F_{t,\vecx}^{(2)}\geq 1\bigr\}\bigr)
-
\lambda\bigl(\bigl\{\vecx \col \# F_{t,\vecx}^{(1)}\geq 1, \:
\# F_{t,\vecx}^{(2)}\geq 1 \bigr\}\bigr)
\\ \notag
& \:\: -
\lambda\bigl(\bigl\{\vecx \col \# F_{t,\vecx}^{(2)}\geq 1, \:
\# F_{t,\vecx}^{(3)}\geq 2\bigr\}\bigr)
+
\lambda\bigl(\bigl\{\vecx \col \# F_{t,\vecx}^{(1)}\geq 1, \:
\# F_{t,\vecx}^{(2)}\geq 1, \:
\# F_{t,\vecx}^{(3)}\geq 2\bigr\}\bigr). %
\end{align}
To each of the four terms in the right hand side we can now apply
the $E_1(\vecx)$-variant of Theorem \ref{thinThm-multi} 
(see Remark \ref{thinCor}) and its analogue for rational $\vecalf$
(Theorem \ref{thinThm-rat-multi}, Remark \ref{thinCor-rat}), 
in conjunction with Lemma \ref{lemCon2fU}.
If $\vecalf\in q^{-1}\Z^d$ then we obtain that as $\rho\to 0$, 
\eqref{exactpos2ullowref} tends to
\begin{align} \notag
& (\lambda \times \mu_q) \Bigl(\Bigl\{ (\vecx,M)\in\Edomain\times X_q \col
\#\bigl(\fZ(0,\xi_1,1,\vecbeta)|_\vecx \cap (\Z^d+\vecalf)M\bigr)=0, \:
\\ \label{exactpos2LIMINFBOUND}
& \hspace{180pt}
\#\bigl(\fZ^{(\fU)}(\xi_1,\xi_2,1,\vecbeta)|_\vecx \cap (\Z^d+\vecalf)M\bigr)
\geq 1,\:
\\ \notag
& \hspace{200pt}
\#\bigl(\fZ(\xi_1,\xi_2,1,\vecbeta)|_\vecx \cap (\Z^d+\vecalf)M\bigr)
\leq 1\Bigr\}\Bigr).
\end{align}
(Note that here we need not remove $\bn$ from the
set $\Z^d+\vecalf$, since $\bn$ is anyway not contained in any of the sets
$\fZ(0,\xi_1,1,\vecbeta)|_\vecx$ or $\fZ(\xi_1,\xi_2,1,\vecbeta)|_\vecx$.)
In the case $\vecalf\notin\Q^d$ we obtain the same expression but with
$\mu$, $X$ and $\Z^d$ in place of $\mu_q$, $X_q$ and $\Z^d+\vecalf$. 

Similarly, the right hand side in \eqref{exactpos2ul} can
be expressed as (using also $T_4=c_5T_5$)
\begin{align} \notag
\lambda\Bigl(\Bigl\{\vecx\in\Edomain\col
\#\bigl(\fZ_{T_4/\xi_1}(0,\xi_1,\xi_1^{-\frac 1{d-1}}\sigma_4,\vecbeta)|_\vecx \Phi^{-t} \bigl(E_1(\vecx)^{-1},\bn\bigr) \cap (\ZZ^d+\vecalf)M_0  \setminus\{\bn\}\big)=0,
\hspace{50pt} &
\\
\#\bigl(\fZ_{T_4/\xi_1}^{(\fU)}(\xi_1,\sfrac{\xi_1}{c_5},(\sfrac{\xi_1}{c_5})^{-\frac 1{d-1}}\sigma_5,\vecbeta)|_\vecx \Phi^{-t} \bigl(E_1(\vecx)^{-1},\bn\bigr) \cap (\ZZ^d+\vecalf)M_0  \setminus\{\bn\}\big)\geq 1\Bigr\}\Bigr) &
\end{align}
with $\e^{(d-1)t}=T_4/\xi_1$,
and as $\rho\to 0$ this is seen to tend to (if $\vecalf\in q^{-1}\Z^d$)
\begin{align} \notag
& (\lambda \times \mu_q) \Bigl(\Bigl\{ (\vecx,M)\in\Edomain\times X_q \col
\#\bigl(\fZ(0,\xi_1,1,\vecbeta)|_\vecx \cap (\Z^d+\vecalf) M\bigr)=0, \:
\\ \label{exactpos2LIMSUPBOUND}
& \hspace{200pt}
\#\bigl(\fZ^{(\fU)}(\xi_1,\xi_2,1,\vecbeta)|_\vecx \cap (\Z^d+\vecalf)M\bigr)
\geq 1\Bigr\}\Bigr).
\end{align}

Hence we conclude: Given any $0<\xi_1<\xi_2$, the $\liminf$ 
of the expression \eqref{exactpos1lambdarestat} 
as $\rho\to 0$ is bounded below by
\eqref{exactpos2LIMINFBOUND}, and the $\limsup$ is bounded above by
\eqref{exactpos2LIMSUPBOUND} (both with the usual modifications if
$\vecalf\notin\Q^d$).
In order to get successively sharper bounds we will now \textit{split}
the original interval $[\xi_1,\xi_2)$ into many small parts,
and apply the bounds just proved to each part.
We will also use the results on integrals over $(X,\mu)$ and $(X_q,\mu_q)$
which we developed in Sections~\ref{FOLIATIONSEC} and \ref{PROPLIMFCNSEC}.
We will give the details for the case $\vecalf\in q^{-1}\Z^d$,
but exactly the same proof with very small changes of notation
works also in the case $\vecalf\notin\Q^d$; in particular
all expressions below containing $f_0(\ldots)$ of $F_0(\ldots)$
will remain unchanged, except that they refer to the definitions
\eqref{XYFR5VARDEF}, \eqref{XYFR3VARDEF} in place of \eqref{FR5VARDEF},
\eqref{FR3VARDEF}; also some of the continuity issues below are slightly
easier in the case $\vecalf\notin\Q^d$ since we can refer to 
Proposition \ref{XYCONT3DPROP} for all that we need.

Thus from now on we assume $\vecalf\in q^{-1}\Z^d$.
Recall that we have defined (\eqref{FR5VARDEF} with $r=0$)
\begin{align} \label{F0RECALLDEF}
f_0(c_1,c_2,\sigma,\vecz,\vecy)
=\nu_\vecy\bigl(\bigl\{M\in X_q(\vecy) \col 
(\fZ(c_1,c_2,\sigma)+\vecz) \cap (\Z^d+\vecalf)M=\emptyset\bigr\}\bigr)
\end{align}
and $F_0(\xi,\vecw,\vecz)=f_0(0,\xi,1,\vecz,\xi\vece_1+\vecw+\vecz)$.
(And $F_0(\xi,\vecw,\vecz)$ is the same as $\Phi_\vecalf(\xi,\vecw,\vecz)$
in \eqref{exactpos1limitrat} in Theorem \ref{exactpos1}.)
Our goal now is to prove that 
the expression in \eqref{exactpos1lambdarestat} tends to %
\begin{align} \label{exactpos2answer}
\int_{\vecx\in\Edomain} \int_{\xi_1}^{\xi_2} \int_{\fU_\perp}
F_0(\xi,\vecw,\vecz_\vecx)
\, d\vecw d\xi d\lambda(\vecx),
\end{align}
where 
$\vecz_\vecx:=(\vecbeta(\vecx)E_1(\vecx))_\perp$.
Recall that we have already seen in Lemma \ref{FRMEASURABLELEM}
that the function $F_0(\xi,\vecw,\vecz)$ is Borel measurable on the
$\langle \xi,\vecw,\vecz\rangle$-product space;
in particular we are allowed to freely change order of integration 
in \eqref{exactpos2answer}; hence our present aim is equivalent with proving
the limit formula \eqref{exactpos1eq} in Theorem \ref{exactpos1}.

Let $0<\xi_1<\xi_2$ be given once and for all.
Take $\ve>0$ arbitrary (we will take $\ve\to 0$ in the end).
Fix a constant $C$ so large that 
$C\geq 1+\xi_2+\sup_{\Edomain} \|\vecbeta\|$, $C\geq \xi_1^{-1}$
and if $d=2$ then also require $2C^{-1} \leq \ve$.
Next choose $\eta>0$ and $N\in\Z_{\geq 2}$
as in Proposition \ref{CONT3DPROP}, for $r=0$ and our fixed $C$ and $\ve$;
if necessary shrink $\eta$ further so that $\eta<\xi_1/N$.
By Lemma \ref{FEWWITHTWOLEMMA2} we may \textit{also} assume,
after possibly shrinking $\eta$ further, that for every set
$U=\fZ(c_1,c_2,1)+\vecz$ with $\vecz\in \{0\}\times\scrB_C^{d-1}$ 
and $c_1<c_2$ satisfying $\xi_1\leq c_1<c_2\leq\xi_2$ and $c_2-c_1\leq \eta$,
we have
\begin{align} \label{FEWWITHTWOLEMMA2CONS}
\int_U \nu_\vecy \Bigl(\Bigl\{ M\in X_q(\vecy) \col
\#\Bigl(U \cap (\Z^d+\vecalf)M\Bigr) \geq 2\Bigr\}\Bigr) \, d\vecy
\leq \ve(c_2-c_1).
\end{align}

We fix a splitting $\xi_1=\theta_1<\theta_2<\ldots<\theta_n=\xi_2$
of the interval $[\xi_1,\xi_2)$ such that
$\theta_{j+1}-\theta_j<\eta$ for each $j=1,2,\ldots,n-1$.
Note that \eqref{exactpos1lambdarestat} can be expressed as
\begin{equation}
\sum_{j=1}^{n-1} \lambda\bigl(\bigl\{ \vecx\in\Edomain \col 
\rho^{d-1} \tau_1(\vecq_{\rho,\vecbeta}(\vecx),\vecv(\vecx);\rho)
\in [\theta_j,\theta_{j+1}), \:
\vecw_1(\vecq_{\rho,\vecbeta}(\vecx),\vecv(\vecx);\rho)
\in -\fU E_1(\vecx)^{-1}\bigr\}\bigr).
\end{equation}
We now apply \eqref{exactpos2LIMINFBOUND} and \eqref{exactpos2LIMSUPBOUND}
for the $\liminf$ and $\limsup$ of each term in this sum.
We get that the $\liminf$ of the total expression is
\begin{align} \notag
\geq \sum_{j=1}^{n-1} \int_{\vecx\in\Edomain} 
\mu_q \Bigl(\Bigl\{ M\in X_q \col
& \#\bigl((\fZ(0,\theta_j,1)+\vecz_\vecx)
\cap (\Z^d+\vecalf)M\bigr)=0, \hspace{72pt}
\\[-7pt]\notag
& \#\bigl((\fZ^{(\fU)}(\theta_j,\theta_{j+1},1)+\vecz_\vecx) \cap 
(\Z^d+\vecalf)M\bigr) \geq 1,
\\ \label{exactposLIMINFBOUND3}
& \hspace{20pt} \#\bigl((\fZ(\theta_j,\theta_{j+1},1)+\vecz_\vecx) \cap 
(\Z^d+\vecalf)M\bigr) \leq 1
\Bigr\}\Bigr) \, d\lambda(\vecx),
\end{align}
where %
$\fZ^{(\fU)}(c_1,c_2,\sigma)$ is defined as in \eqref{FZFULIMITDEF}.

We will next apply Proposition \ref{FOLINTPROP} to bound each term from below.
Let us fix \mbox{$j\in\{1,\ldots,n-1\}$} and $\vecx\in\Edomain$
for the moment, set
\begin{equation}
\begin{split}
S=\Bigl\{M\in X_q \col
& \#\bigl((\fZ(0,\theta_j,1)+\vecz_\vecx)
\cap (\Z^d+\vecalf)M\bigr)=0,
\\
& \#\bigl((\fZ^{(\fU)}(\theta_j,\theta_{j+1},1)+\vecz_\vecx) \cap 
(\Z^d+\vecalf)M\bigr) \geq 1\Bigr\}
\end{split}
\end{equation}
and denote by $S'$ the subset of $S$ which appears in
\eqref{exactposLIMINFBOUND3} for our fixed $j,\vecx$.
Set $U=\fZ^{(\fU)}(\theta_j,\theta_{j+1},1)+\vecz_\vecx$;
then $S\subset\bigcup_{\vecy\in U} X_q(\vecy)$ and also
$\forall \vecy_1\neq\vecy_2\in U: X_q(\vecy_1)\cap X_q(\vecy_2)\cap S'
=\emptyset$,
since $U\subset \fZ(\theta_j,\theta_{j+1},1)+\vecz_\vecx$.
Hence Proposition \ref{FOLINTPROP} applies, yielding
\begin{equation}
\begin{split}
\mu_q(S')&=\int_U \nu_\vecy(S'\cap X_q(\vecy))\, d\vecy
\\
& \geq \int_U \nu_\vecy(S\cap X_q(\vecy))\, d\vecy
\\
& \hspace{50pt}
-\int_U \nu_\vecy \Bigl(\Bigl\{ M\in X_q(\vecy) \col
\#\Bigl(\bigl(\fZ(\theta_j,\theta_{j+1},1)+\vecz_\vecx\bigr) \cap 
(\Z^d+\vecalf)M\Bigr)
\geq 2\Bigr\}\Bigr) \, d\vecy.
\end{split}
\end{equation}
Here the first integral in the right hand side equals
$\int_{\theta_j}^{\theta_{j+1}} \int_{\fU_\perp}
f_0(0,\theta_j,1,\vecz_\vecx,\xi\vece_1+\vecw+\vecz_\vecx) \, d\vecw \, d\xi$
(recall \eqref{F0RECALLDEF}), since each $M\in X_q(\vecy)$ 
with $\vecy\in U$ automatically fulfills
$\#\bigl((\fZ^{(\fU)}(\theta_j,\theta_{j+1},1)+\vecz_\vecx) \cap 
(\Z^d+\vecalf)M\bigr) \geq 1$;
and the second integral is bounded from above by $\ve(\theta_{j+1}-\theta_j)$,
by \eqref{FEWWITHTWOLEMMA2CONS}.
Adding this over all $j$ and $\vecx$ we have now proved that the total
expression in \eqref{exactposLIMINFBOUND3} is
\begin{align} \label{exactposLIMINFBOUND4}
& \geq -\ve(\xi_2-\xi_1)+
\int_{\vecx\in\Edomain}\sum_{j=1}^{n-1}
\int_{\theta_j}^{\theta_{j+1}} \int_{\fU_\perp}
f_0(0,\theta_j,1,\vecz_\vecx,\xi\vece_1+\vecw+\vecz_\vecx) \, d\vecw \, d\xi
\, d\lambda(\vecx).
\end{align}
Now for each $\langle \vecx,j,\xi,\vecw\rangle$ which
appears in the above integral, and which satisfies 
$\|\vecw+\vecz_\vecx\|\geq C^{-1}$ if $d=2$,
Lemma \ref{FrfrRELLEMMA} applies, and yields
\begin{align} \label{CONTRELFORLIMSUP}
\bigl | f_0(0,\theta_j,1,\vecz_\vecx,\xi\vece_1+\vecw+\vecz_\vecx)
-F_0(\xi,\vecw,\vecz_\vecx)\bigr |\leq \ve.
\end{align}
If $d=2$ we note that the set 
$\{\vecw \col \|\vecw+\vecz_\vecx\|< C^{-1}\}$
has measure $\leq 2C^{-1} \leq \ve$
(viz., the 1-dimensional Lebesgue measure $d\vecw$),
and for these $\vecw$'s the difference in \eqref{CONTRELFORLIMSUP}
is certainly $\leq 1$, since $0\leq f_0,F_0\leq 1$ everywhere.
Hence \eqref{exactposLIMINFBOUND4} is
\begin{align}
& \geq -2 \ve (\xi_2-\xi_1)+
\int_{\vecx\in\Edomain}
\int_{\xi_1}^{\xi_2} \int_{\fU_\perp}
\Bigl(-\ve+F_0(\xi,\vecw,\vecz_\vecx) \Bigr) 
\, d\vecw \, d\xi \, d\lambda(\vecx).
\end{align}
In conclusion, we have proved that this last expression is a lower bound
for the $\liminf$ of \eqref{exactpos1lambdarestat}.
But this is true for any $\ve>0$; hence the $\liminf$ is in fact
\begin{equation}
\begin{split}
& \geq \int_{\vecx\in\Edomain}
\int_{\xi_1}^{\xi_2} \int_{\fU_\perp}
F_0(\xi,\vecw,\vecz_\vecx) \, d\vecw \, d\xi \, d\lambda(\vecx).
\end{split}
\end{equation}

The treatment of the $\limsup$ is similar but a bit easier:
With $S$ and $U$ as before we need only notice that by
the upper bound in Proposition \ref{FOLINTPROP} we have
\begin{equation}
\begin{split}
\mu_q(S)& \leq \int_U \nu_\vecy(S\cap X_q(\vecy))\, d\vecy
\\
& =\int_{\vecx\in\Edomain}\sum_{j=1}^{n-1}
\int_{\theta_j}^{\theta_{j+1}} \int_{\fU_\perp}
f_0(0,\theta_j,1,\vecz_\vecx,\xi\vece_1+\vecw+\vecz_\vecx) \, d\vecw \, d\xi
\, d\lambda(\vecx).
\end{split}
\end{equation}
Now Lemma \ref{FrfrRELLEMMA} is applied as before,
and we obtain that the $\limsup$ of \eqref{exactpos1lambdarestat} is
\begin{equation}
\leq \ve (\xi_2-\xi_1)+
\int_{\vecx\in\Edomain}
\int_{\xi_1}^{\xi_2} \int_{\fU_\perp} \Bigl( \ve+
F_0(\xi,\vecw,\vecz_\vecx)\Bigr) \, d\vecw \, d\xi \, d\lambda(\vecx).
\end{equation}
Hence, by letting $\ve\to 0$ and combining with our result for $\liminf$,
we have finally proved our claim that \eqref{exactpos1lambdarestat}
tends to \eqref{exactpos2answer} as $\rho\to 0$.
This completes the proof of Theorem \ref{exactpos1}.
\end{proof}

\begin{proof}[Proof of Theorem \ref{exactpos2-1hit}]
Let $\lambda$ and $f$ be given as in the statement of the theorem.
By \eqref{VECV1DEF}, the left hand side of \eqref{exactpos2eq-1hit} equals
\begin{align} \label{exactpos2-1hit-proof1}
\lim_{\rho\to 0}  \int_{\S_1^{d-1}} g\big(\vecv_0, \rho^{d-1} \tau_1(\vecq_{\rho,\vecbeta}(\vecv_0),\vecv_0;\rho), 
\vecw_1(\vecq_{\rho,\vecbeta}(\vecv_0),\vecv_0;\rho)\big) d\lambda(\vecv_0),
\end{align}
where $g(\vecv_0,\xi,\vecw_1)=
f\big(\vecv_0, \xi,\vecv_0-2(\vecv_0\cdot\vecw_1)\vecw_1\big)$.
Using Corollary \ref{exactpos2} 
we obtain %
\begin{align} \notag
=\int_{\HS} \int_{\R_{>0}} \int_{\S_1^{d-1}}
f\bigl(\vecv_0,\xi,\vecv_0-2(\vecv_0\cdot(\vecomega K(\vecv_0)^{-1}))
\vecomega K(\vecv_0)^{-1}\bigr)  &
\\ \label{exactpos2-1hit-proof2}
\times
\Phi_\vecalf\bigl(\xi,\vecomega_\perp,(\vecbeta(\vecv_0)K(\vecv_0))_\perp\bigr)
\, \omega_1 \, d\lambda(\vecv_0) \, & d\xi \, d\!\vol_{\S^{d-1}_1}(\vecomega).
\end{align}
Now change the order of integration by moving the integral over
$\vecomega\in\HS$ to the innermost position, and then apply the
variable substitution \eqref{V1OMEGASUBST}
in the innermost integral; note that this
gives a diffeomorphism $\vecomega\mapsto\vecv_1$ from $\HS$ onto
$\S_1^{d-1}\setminus\{\vecv_0\}$ (the inverse map is given by
$\vecomega=%
\frac{\vece_1-\vecv_1 K(\vecv_0)}{\|\vecv_0-\vecv_1\|}$).
Recalling \eqref{exactpos2-1hit-tpdef} we then see that
\eqref{exactpos2-1hit-proof2} equals the right hand side of
\eqref{exactpos2eq-1hit}, and we are done.
\end{proof}

\subsection{Proofs for Section \ref{secVisible0}} \label{diff0}
Introduce $E_1:\Edomain\to\SO(d)$ as in Section \ref{secVisproofssec} and
write $\lambda$ also for the lift of $\lambda$ to the variable 
$\vecx\in\R^{d-1}$, as before.
Set
\begin{equation} \label{FCTCSDEF}
\fC_T(c,\sigma) = \bigl\{ \vecy\in\RR^d\setminus\{\vecnull\} \col c T \leq \|\vecy\| < T,\; \|\vecy\|^{-1} \vecy \in\fD_T(\sigma) \bigr\}
\begin{pmatrix} 
T^{-1} & \vecnull \\ \trans\vecnull & T^{1/(d-1)} 1_{d-1} \end{pmatrix} .
\end{equation}
Then
\begin{equation}
	\scrN_{c,T}(\sigma,\vece_1 E_1(\vecx)^{-1}) 
	= \#\big( \fC_T(c,\sigma) \Phi^{-t} (E_1(\vecx)^{-1},\bn) \cap (\ZZ^d+\vecalf)M_0\setminus\{\bn\} \big) ,
\end{equation}
with $T=\e^{(d-1)t}$.
As before, Theorem \ref{visThm0} and Theorem \ref{prim-visThm0} now follow from the theorems in Section \ref{secThin} and the following lemma.

\begin{lem}\label{lemCon1}
Fix $0\leq c<1$.
Let $\sigma_T$ be a continuous 
non-negative function of $T>0$ 
such that the limit $\sigma_\infty=\lim_{T\to\infty}\sigma_T$ exist. 
Then the union $\cup_{T\geq 1}\fC_T(c,\sigma_T)$ is bounded, and
\begin{equation}\label{holds}
\lim(\inf\fC_T(c,\sigma_T))^\circ \supset \fC(c,\sigma_\infty)^\circ, \qquad
\lim\overline{\sup\fC_T(c,\sigma_T)} \subset  \overline{\fC(c,\sigma_\infty)} ,
\end{equation}
where $\fC(c,\sigma)$ is as in \eqref{FCCSDEF}.
The boundary of $\fC(c,\sigma)$ has Lebesgue measure zero.
\end{lem}

\begin{proof}
From \eqref{DISCSIZE} we have 
$\fD_T(\sigma_T)=(\vece_1+\scrB_{r_T}^d)\cap \S_1^{d-1}$ where
\begin{align} \label{lemCon1rasympt}
T^{\frac d{d-1}}r_T\to \Bigl(\frac{d\sigma_\infty}{(1-c^d)\vol(\scrB_1^{d-1})}\Bigr)^{\frac 1{d-1}}=A(c,\sigma_\infty)
\qquad \text{as } \: T\to\infty,
\end{align}
In particular, for $T$ sufficiently large, if
$\vecy$ is any point in $\fC_T(c,\sigma_T)$,
and $\vecy'=Ty_1\vece_1+T^{-\frac 1{d-1}} \vecy_\perp$, 
then $\|\vecy'\|^{-1}\vecy'\in\fD_T(\sigma_T)$ implies
$y_1'>0$ and $\|\vecy'_\perp\|\leq \frac{r_T\sqrt{4-r_T^2}}{2-r_T^2} y_1'$, 
thus $y_1>0$ and
$\|\vecy_\perp\|\leq (A(c,\sigma_\infty)+\eta)y_1$, where $\eta>0$ can be made
arbitrarily small.
With these observations the proof of Lemma \ref{lemCon1} is easily completed
by mimicking the proof of Lemma \ref{lemCon2fU}.
\end{proof}

\section*{Index of notations}

\begin{center}
\begin{footnotesize}
\begin{longtable}{llr}
$\perp$ & $\vecx_\perp=\vecx-(\vecx\cdot\vece_1)\vece_1$ & \pageref{perpos}
\\
$\ASLR$ & $=\SLR\ltimes \RR^d$ & \pageref{ASLMULTLAW}
\\
$A(c,\sigma)$ & the constant in \eqref{ACSDEF} & \pageref{ACSDEF} 
\\
$\scrB^d_\rho$ & open ball of radius $\rho$, centered at the origin & \pageref{openball}
\\
$\scrB^d_T(c)$ & $=\{ \vecx\in\RR^d : c T \leq \|\vecx\| < T \}$, spherical shell & \pageref{shell}
\\
$\fB|_\vecx$ & $=(\{ \vecx \} \times \RR^d) \cap \fB$ & \pageref{FBXDEF}
\\
$\fC(c,\sigma)$ & the cone in \eqref{FCCSDEF} & \pageref{FCCSDEF}
\\
$\fC(c_1,c_2,\sigma)$ & the cone in \eqref{FCCCSDEF} & \pageref{FCCCSDEF}
\\
$\fC_T(c,\sigma)$ & the set in \eqref{FCTCSDEF} & \pageref{FCTCSDEF}
\\
$\Edomain$ & open subset of $\R^{d-1}$ & \pageref{also2}, \pageref{EDOMDEF2}
\\
$\fD_T(\sigma,\vecv)$ & small disc on $\S_1^{d-1}$ of radius $\asymp T^{-d/(d-1)}$ &  \pageref{DISCSIZE}
\\
$\fD_T(\sigma)$ & $=\fD_T(\sigma,\vece_1)$ &
\\
$\fD(\epsilon,\vecv)$ & small disc on $\S_1^{d-1}$ of radius $\sim\epsilon$ & \pageref{disceps} 
\\
$\fD(\epsilon)$ & $=\fD(\epsilon,\vece_1)$ &
\\
$\vece_1$ & $=(1,0,\ldots,0)$ &
\\
$\vece_2$ & $=(0,1,0,\ldots,0)$ &
\\
$\scrE(\fB,r)$ & $=\bigl\{ (\vecx,g)\in \RR^{d-1}\times X \col
\#(\fB|_\vecx \cap \ZZ^d g)\geq r \bigr\}$ & \pageref{SCREDEF}
\\
$\widehat\scrE_t$ & 
the family of sets in \eqref{widehatE} or \eqref{widehatE2} & \pageref{widehatE}, \pageref{widehatE2}
\\
$E_1$ & a map $\Edomain\to\SO(d)$ & \pageref{also2}, \pageref{E1DEF2}, \pageref{exactpos1proofsec}
\\
$E_{c,\vecalf}(r,\sigma)$ & limiting probability for $\scrN_{c,T}(\sigma,\vecv)=r$ & \pageref{visThm0}
\\
$\widehat E_{c,\vecalf}(r,\sigma)$ & limiting probability for $\widehat \scrN_{c,T}(\sigma,\vecv)=r$ & \pageref{prim-visThm0}
\\
$F_{c,\vecalf}(r,\sigma)$ & limiting probability for $\scrN_{c,T}(\rho,\vecv)=r$ & \pageref{visThm}
\\
$F_{c}(r,\sigma)$ & $=F_{c,\vecalf}(r,\sigma)$ for any $\vecalf\notin\Q^d$,
universal limiting probability & \pageref{FCRSDEF}
\\
$F_{c,\vecalf,\vecbeta}(r,\sigma)$ & limiting probability for $\scrN_{c,T}(\rho,\vecv,\vecbeta(\vecv))=r$ & \pageref{defF}
\\
$F^{(\fU)}_{c,\vecalf,\vecbeta}(r,\sigma)$ & limiting probability for $\scrN^{(\fU)}_{c,T}(\rho,\vecx,\vecbeta(\vecx))=r$ & \pageref{defFFU}
\\
$F_{c,\vecalf}(r,\scrQ)$ & limiting probability for $\scrN_{c,T}(\scrQ,\vecv)=r$ & \pageref{visThm3limit}
\\
$f_r(c_1,c_2,\sigma,\vecz,\vecy)$ & the volume function in \eqref{FR5VARDEF} or \eqref{XYFR5VARDEF}  & \pageref{FR5VARDEF}, \pageref{XYFR5VARDEF}
\\
$F_r(\xi,\vecw,\vecz)$ & $=f_r(0,\xi,1,\vecz,\xi\vece_1+\vecw+\vecz)$ & \pageref{FR3VARDEF}, \pageref{XYFR3VARDEF}
\\
$\mathfrak{F}_N$ & $=\bigl\{\sfrac hk \col h,k \in \Z, \: 0<h<k\leq N\bigr\}$ & \pageref{FsubN}
\\
$g_r(c_1,c_2,\sigma,\vecz,\vecy)$ & the volume function in \eqref{GR5VARDEF} or \eqref{XYGR5VARDEF}  & \pageref{GR5VARDEF}, \pageref{XYGR5VARDEF}
\\
$G_r(\xi,\vecw,\vecz)$ & $=g_r(0,\xi,1,\vecz,\xi\vece_1+\vecw+\vecz)$ & \pageref{GR3VARDEF}
\\
$H$ & $=\bigl\{g\in\SL(d,\R) \col \vece_1 g =\vece_1\bigr\}$ & \pageref{HAV}
\\
$I_q$ & $=[\Gamma(q):\Gamma(1)]$ & \pageref{SLDRHAARQ}
\\
$\scrK_\rho$ & complement of the set $\scrB^d_\rho + \scrL$ in $\RR^d$ (the ``billiard domain'') & \pageref{Krho} 
\\
$\lim(\inf \scrE_t)^\circ$ & $=\cup_{t\geq t_0} \big(\cap_{s\geq t} \scrE_s\big)^\circ$ & \pageref{LIMINFSUPTOP}
\\[2pt]
$\lim\overline{\sup \scrE_t}$ & $=\cap_{t\geq t_0} \overline{\cup_{s\geq t} \scrE_s}$  & \pageref{LIMINFSUPTOP}
\\
$\scrL$ & $=\ZZ^d M_0$, euclidean lattice of covolume 1 & \pageref{secVisible0}, \pageref{latt}
\\
$\scrL_\vecalf$ & $=(\ZZ^d+\vecalf)M_0$, affine lattice of covolume 1 & \pageref{afflatt}
\\
$\widehat\scrL_\vecalf$ & set of visible lattice points & \pageref{VISIBLECHAR}
\\
$n_-(\vecx)$ & element in $\ASLR$ & \pageref{NMINUSDEF}
\\
$\scrN_{c,T}(\sigma,\vecv)$ & $=\#\{\vecy\in\scrP_T \col
\|\vecy\|^{-1}\vecy\in \fD_T(\sigma,\vecv)\}$ & \pageref{disc00}
\\
$\scrN_{c,T}(\sigma,K)$ & $=\#(\scrP_T \cap \fD_T(\sigma) K)$ & \pageref{disc1}
\\
$\scrN_{c,T}(\rho,\vecv)$ & number of spheres in direction $\vecv$ & \pageref{asdef}
\\
$\scrN_{c,T}(\rho,\vecv,\vecw)$ & as above, but includes shift by $\rho\vecw$ & \pageref{asin}
\\
$\scrN_{c,T}^{(\fU)}(\rho,\vecx,\vecw)$ &
generalized version of $\scrN_{c,T}(\rho,\vecv,\vecw)$ & \pageref{asinfu}
\\
$\widehat{\scrN}_{c,T}(\sigma,\vecv)$ 
& analogue of $\scrN_{c,T}(\sigma,\vecv)$ for visible lattice points & \pageref{WHNDEF}
\\
$\scrN_{c,T}(\scrQ,\vecv)$ & number of $\scrQ$'s in direction $\vecv$ & \pageref{asin2}
\\
$p_{\vecalf,\vecbeta}(\vecv_0,\xi,\vecv_1)$ & joint limiting distribution for free path lengths and velocities & \pageref{exactpos2-1hit-tpdef}
\\
$P_\vecalf(s)$ & limiting gap distribution for directions of lattice points & \pageref{thmGap}
\\
$\widehat P_\vecalf(s)$ & limiting gap distribution for directions of visible lattice points & \pageref{prim-thmGap}
\\
$\scrP_T$ & $=\Lalf\cap \scrB^d_T(c)\setminus\{\vecnull\}$ & \pageref{shell}
\\
$\widehat\scrP_T$ & $=\widehat\scrL_\vecalf \cap \scrB_T^d(c)$ & \pageref{WHPDEF}
\\
$\Proj_{\{\vecv\}^\perp}$ & orthogonal projection from
$\R^d$ onto the orthogonal complement of $\vecv$ & \pageref{FZVCSDEF}
\\
$\vecq_1(\vecq,\vecv;\rho)$ & location of first collision in $\RR^d$ & \pageref{firstcoll}
\\
$\vecq_{\rho,\vecbeta}(\vecv)$ & initial position $\vecq+\rho\vecbeta(\vecv)$ & \pageref{inpo}
\\
$\S_1^{d-1}$ & unit sphere in $\RR^d$ & 
\\
$\HS$ & the hemisphere $\{\vecv=(v_1,\ldots,v_d)\in\S_1^{d-1} \col v_1>0\}$ & \pageref{HS}
\\
$s(\delta)$ & see \eqref{SDELTADEF} & \pageref{SDELTADEF}
\\
$t_\veck$ & $=\gcd(qk_1,qk_2,\ldots,qk_d)$ & \pageref{TVECKDEF}
\\
$\vecw_1(\vecq,\vecv;\rho)$ & location of first collision on $\S_1^{d-1}$ & \pageref{firstcoll}
\\
$X$ & $=\ASLASL$, space of affine lattices & \pageref{ASLMULTLAW}
\\
$X_1$ & $=\SLSL$, space of lattices & \pageref{secVisible0}
\\
$X_q$ & $=\Gamma(q)\backslash\SLR$ & \pageref{ASLMULTLAW}
\\
$X(\vecy)$ & submanifold of $X$ & \pageref{inpo}
\\
$X_q(\vecy)$ & submanifold of $X_q$ & \pageref{inpo}
\\
$X_q(\veck,\vecy)$ & connected component of $X_q(\vecy)$ & \pageref{XQKYDEF}
\\
$X_q^{(t_0)}(\vecy)$ & subset of $X_q(\vecy)$ & \pageref{XQT0DEF}
\\
$\ZZ_*^d$ & $=\ZZ^d\setminus\{\vecnull\}$ & \pageref{ZZSTARDDEF}
\\
$\widehat\Z^d$ & $=\widehat\Z_\vecnull^d$, set of primitive lattice points & \pageref{VISIBLECHAR}
\\
$\widehat\Z_\vecalf^d$ & $= \{\vecx\in(\Z^d+\vecalf)\setminus\{\bn\} \col
\gcd(q\vecx)\leq q\}$ & \pageref{VISIBLECHAR}
\\
$\fZ(c,\sigma)$ & the cylinder in \eqref{zyl} & \pageref{zyl}
\\
$\fZ_\vecv(c,\sigma)$ & the cylinder in \eqref{FZVCSDEF} & \pageref{FZVCSDEF}
\\
$\fZ(c_1,c_2,\sigma)$ & the cylinder in \eqref{FZC1C2DEF} & \pageref{FZC1C2DEF}
\\
$\widetilde{\fZ}(c_1,c_2,\sigma)$ & modification of $\fZ(c_1,c_2,\sigma)$ & \pageref{wideZ}
\\
$\fZ^{(\fU)}(c_1,c_2,\sigma)$ & generalized cylinder in \eqref{FZFULIMITDEF} & \pageref{FZFULIMITDEF}
\\
$\fZ^{(\fU)}(c_1,c_2,\sigma,\vecbeta)$ & generalized cylinder in \eqref{FZFULIMITDEF} & \pageref{FZFULIMITDEF}
\\
$\fZ^{(\fU)}_T(c_1,c_2,\sigma,\vecbeta)$ & 
the set in \eqref{FZFUTC1C2DEF} & \pageref{FZFUTC1C2DEF}
\\
$\fZ(c,\scrQ)$ & the set in \eqref{NQ} & \pageref{NQ} 
\\
$\fZ_T(c,\scrQ)$ & the set in \eqref{ZcQ} & \pageref{ZcQ} 
\\
$\Gamma(q)$ & principal congruence subgroup & \pageref{Gamq}
\\
$\kappa_q$ & the constant in \eqref{KAPPAQDEF} & \pageref{KAPPAQDEF}
\\
$\mu$ & Haar measure on $\ASL(d,\RR)$, probability measure on $X$ & \pageref{ASLMULTLAW}
\\
$\mu_H$ & Haar measure on $H$ & \pageref{MUHDEF}
\\
$\mu_q$ & Haar measure on $\SL(d,\RR)$, probability measure on $X_q$ & \pageref{ASLMULTLAW}
\\
$\nu$ & Liouville measure & \pageref{Liouville}
\\
$\nu_\vecy$ & volume measure on $X_q(\vecy)$ or $X(\vecy)$ & \pageref{NUYDEF}, \pageref{NUYDEF2}
\\
$\tau_1(\vecq,\vecv;\rho)$ & $= \inf\{ t>0 : \vecq+t\vecv \notin\scrK_\rho \}$, free path length  & \pageref{TAU1DEF0} 
\\
$\varphi_t$ & the Lorentz flow & \pageref{Lorentzflow}
\\
$\Phi(\xi)$, $\Phi_{\scrL,\vecq}(\xi)$ & limiting distribution for the free path length & \pageref{freeThm1}
\\
$\Phi_\vecalf(\xi)$ & alternative notation for $\Phi_{\scrL,\vecq}(\xi)$ & \pageref{P1}
\\
$\Phi_{\vecalf,\vecbeta}(\xi)$ & limiting disitribution for the free path length & \pageref{PALFBETDEF}
\\
$\Phi_\vecalf(\xi,\vecw,\vecz)$ & joint limiting disitribution for free path length and impact location & \pageref{exactpos1limitrat}
\\
$\Phi^t$ & element in $\ASLR$ & \pageref{PHITDEF}
\\
$\chi_{\scrA}$ & characteristic function of a set $\scrA$ & 
\\
$\Omega$ & domain of $f_r(c_1,c_2,\sigma,\vecz,\vecy)$ & \pageref{FR5VARDEF}
\\
$\Omega_C$ & $\Omega$ truncated & \pageref{OMEGACDEF}, \pageref{altOmegaC}
\end{longtable}
\end{footnotesize}
\end{center}

\end{document}